\documentclass[]{siamart171218}

\usepackage{amsmath}
\usepackage{amssymb}
\usepackage{amsfonts}
\usepackage{graphicx}
\usepackage{thmtools}
\usepackage{caption}
\usepackage{tikz-cd}
\usepackage{framed}
\usepackage{bm}
\usepackage[all]{nowidow}
\usetikzlibrary{calc}

\newcommand{\seq}[1]{\left \{ {#1} \right \}}
\newtheorem{prop}{Proposition}
\newtheorem{example}{Example}

\usepackage{lipsum}

\usepackage{epstopdf}
\usepackage{algorithmic}

\def\presuper#1#2%
  {\mathop{}%
   \mathopen{\vphantom{#2}}^{#1}%
   \kern-\scriptspace%
   #2}

\title{Koopman Theory and Linear Approximation Spaces}

\author{Andrew J. Kurdila \thanks{W.M. Johnson Professor of Mechanical Engineering, Virginia Tech}\and Parag S. Bobade \thanks{Postdoctoral Research Fellow, Aerospace Engineering, University of Michigan}} 

\begin{document}

\maketitle


\begin{abstract}
Koopman theory studies dynamical systems in terms of operator theoretic properties 
 of the Perron-Frobenius and Koopman operators $\mathcal{P}$ and $\mathcal{U}$, respectively. This paper derives rates of convergence for   estimates of these operators, corresponding to generally nonlinear dynamical systems,  under a variety of situations. We also derive convergence rates for some probability measures associated with  these operators, as well as for specific data-driven algorithms constructed from them.    This paper  introduces  a suitably general  class of priors, which describes the information available for constructing approximations, one  that facilitates the development of error estimates in many applications of interest. These priors are defined in terms of the action of $\mathcal{P}$ or $\mathcal{U}$ on certain  linear  approximation spaces. For cases where it is feasible to obtain eigenfunctions of either the Perron-Frobenius or Koopman operator, priors are defined in terms of the {\em spectral approximation space} $A^{r,2}_{\lambda}(H):=A^{r,2}_{\lambda(T)}(H)\subset H$ with $r>0$ the approximation rate, $T$ a given self-adjoint and compact operator on $H$, $\lambda:=\lambda(T)$ the eigenvalues of $T$,   and $H$ a Hilbert space of functions over the domain $\Omega\subseteq \mathbb{R}^d$.     More generally, we utilize priors expressed in terms of the {\em Banach spaces of linear approximation}  $A^{r,q}(X):=A^{r,q}(X,\left \{A_j \right\}_{j\in \mathbb{N}_0})\subset X$ with $X$ a  Banach  space of functions on the domain $\Omega$, $r>0$ a measure of  the  approximation rate (or smoothness),  $\left \{A_j \right\}_{j\in \mathbb{N}_0}$ a sequence of approximant spaces, and   $1\leq q \leq \infty$.  
The most common cases studied in the paper choose the Hilbert space $H$ to be $U:=L^2_\mu(\Omega)$ with $\mu$ a measure or $V\subset U$ a reproducing kernel Hilbert space (RKHS).  This paper characterizes the rates of convergence of approximations of $\mathcal{P}$ or $\mathcal{U}$  that are generated by  finite dimensional bases of wavelets, multiwavelets, and eigenfunctions,  as well as approaches that use samples of the input and output of the system in conjunction with these bases.  Since the wavelets and multiwavelets are  selected to reproduce piecewise polynomials of a certain order,  the results of this analysis also can be used to  understand the best attainable approximation rates that are achievable by  common spline or finite element approximations in (Petrov-)Galerkin approximations of the Frobenius-Perron or Koopman operators.   When the estimates of the operators are generated by samples, it is shown that the error in approximation of the Perron-Frobenius  or Koopman  operators can be decomposed into two parts, the approximation and sample errors.  This  result emphasizes that  sample-based estimates of Perron-Frobenius and Koopman operators are subject to  the well-known trade-off  between the  bias and variance that contribut to the error, a balance  that also features in  nonlinear regression and  statistical learning theory.   
\end{abstract}

%

\section{Introduction}
Data-driven approaches for the study of dynamical systems have flourished over the past decade. Applications using such methodologies have arisen in  several fields of science and engineering, spanning  biology \cite{Shlizerman2013Olfactory}, neurosciences \cite{BinBrunton2016Extracting}, chemistry \cite{Kevrekidis2004Equationfree}, solid mechanics \cite{Ortiz2016Data}, and fluid mechanics \cite{Noack2016Machine, Kutz2017Deep}. The underlying philosophy of building accurate, yet simple,  models  that predict the output of a dynamical system based on  observations from experiments has been a central theme of science and mathematics. While there is an ongoing debate about  the fidelity and robustness of such data-driven models as compared to those derived from the first principles of physics, it undeniable that data-driven approaches  have been demonstrated to be highly efficient and effective. This success is due in part  to advances in computational technologies, especially with respect to  data-driven modeling.

The theory underlying many of these data-driven algorithms is referred to as Koopman theory, and has emerged as an important discipline in the study of dynamical systems. The approach studies  dynamical systems via  an operator theoretic framework, and it dates back to the papers by  Koopman in \cite{koopman1931,koopman1932}. It has also found use as a theoretical tool for the study of Markov chains since the 1960's, see  \cite{meyn} for an overall account of the operators induced by transition probability kernels. The theoretic background of Koopman theory can be found in treatises on operator or ergodic theory such as in \cite{ergodicnagel}  or in references on Markov semigroups \cite{markovsemigroups}.   Most recently, it  has been popularized for data-driven analysis of high dimensional dynamical systems in  the seminal work of Mezic in \cite{Mezic2005spectal,Mezic2004comparison,MezicAFR,MezicKoopmanism}. 

This paper examines the approximation of the pair of  dual or adjoint 
Perron-Frobenius and Koopman operators $\mathcal{P}$ and $\mathcal{U}$,  respectively, in the context of the study of dynamical systems.   
A large collection of papers over the past few years have studied the use of Perron-Frobenius and Koopman operators to analyze evolutions in discrete and continuous time,   \cite{Mezic2017ddsa}, \cite{Giannikis2017}, and \cite{Das2018}. Just over the past two years,  studies such as  \cite{Korda2017Convergence,Nitin2018,klus2016,klusTO,klusdmd,mezic2017kosda}, and the references cited  therein,  provide careful analyses  that establish the convergence of  approximations in many situations.  Results are stated for discrete or continuous  systems, as well as deterministic or stochastic flows. Both semiflows and flows are studied, and convergence in various norms is examined.  As is natural, there is a trade-off between the strength and the generality of the conclusions regarding convergence of approximations in these many references.  Still, it is fair to say that while convergence is often studied and proven, the rates of convergence are rarely studied  in these references
{\em per se}. 

We limit the study of the Perron-Frobenius and Koopman operators in this paper to those that are easily associated to certain discrete, deterministic flows or stochastic flows generated by  Markov chains. While many of the results in this paper can also be used as a foundation for the study of continuous flows, we leave this topic for a future study. 
We discuss in more detail in Section \ref{sec:operators} how  the  form of the operators  vary  depending on whether they are being used to study the dynamics of deterministic or  stochastic flows. For now, it suffices to note  that we  will focus primarily, but not exclusively, on  two particular classes of applications. In the first class of examples,  the Perron-Frobenius operator takes the form 
\begin{equation}
(\mathcal{P}f)(x):=\int p(y,x)f(y)\mu(dy) \label{eq:simpleP}
\end{equation}
for $x\in \Omega$, $f\in L^2_\mu(\Omega)$,  the kernel $p:\Omega \times \Omega \rightarrow \mathbb{R}$, and a measure $\mu$ on the domain $\Omega \subseteq \mathbb{R}^d$. This form of the operator $\mathcal{P}$ arises in the study of Markov chains that have a transition probabability density $p(y,x)$ so that the transition probability kernel is given by $\mathbb{P}(dy|x):=p(y,x)\mu(dy)$. Here the Koopman operator is the adjoint $\mathcal{U}=\mathcal{P}^*$ and is induced by the adjoint  kernel relative to the Hilbert space $L^2_\mu(\Omega)$.

A second class of examples are related to studies of the Koopman operator 
\begin{equation}
\label{eq:simpleU}
(\mathcal{U}f)(x)=(f\circ w)(x)
\end{equation}
for $x\in \Omega$ and  a  mapping $w:\Omega \rightarrow \Omega$. This operator arises in connection to discrete, deterministic flows in Equation \ref{eq:det_Phi}  and also for some stochastic flows determined by a  Markov chain as  in Equation \ref{eq:stoch_Phi}.
In these cases the Perron-Frobenius and Koopman operators are transposes of one another, $\mathcal{P}'=\mathcal{U}$, with respect to the dual pairing $<\cdot,\cdot>_{X^* \times X}$ for a Banach space $X$.  In one setting the transpose is defined  in terms of the duality pairing between the continuous functions $C(\Omega)$ and its topological dual space  $C^*(\Omega)$, a space of signed measures. This case is quite familiar in the study of Markov chains. \cite{meyn} This is not the only common choice of $X$.  It is also possible to define these operators as dual with respect to the pairing $<\cdot,\cdot>_{L^\infty_\mu \times L^1_\mu}$ as in the popular treatise \cite{lasota}, or as adjoints with respect to the inner product on $L^2_\mu(\Omega)$.  

The limitation of the examples to these two general types  enables a cohesive and reasonably general theory of approximations of Perron-Frobenius and Koopman operators. Still, the details of the analysis of these examples remains complex owing the variety of choices for the domain $\Omega$, the density $p(\cdot,\cdot)$, the mapping $w:\Omega \rightarrow \Omega$, and the measure $\mu$. In addition, the problem of approximating the operators $\mathcal{P}, \mathcal{U}$ is made all the more difficult because the exact form of the kernel $p(\cdot,\cdot)$, measure $\mu$, mapping $w$, or even the domain $\Omega$ may be uncertain or unknown.
It may be the case that the kernel $p$ is known, but the domain $\Omega\subseteq \mathbb{R}^d$ that supports the dynamics is unknown.  Or,  the domain $\Omega$ may be known, but the form of the measure $\mu$ over $\Omega$ is uncertain. It might be the case that both $\mu$ and $\Omega$ are unknown. 
It is clear that there are a large number of approximation problems that can arise depending on what combination of these  constituents are known or unknown. We do not systematically consider all combinations of such problems in this paper, but focus on only a few of the key scenarios. These typical cases then can inform the study of problems subject to other types of  uncertainty. We briefly discuss some of the  specific cases we study in this paper in the next section.  

\subsection{Types of Approximation Problems}
\label{sec:types}
This paper introduces a novel approach for the approximation of a particular class of Perron-Frobenius and Koopman operators whose regularity is characterized by their action on  different types of  approximation spaces. The most general of these is the Banach space $A^{r,q}(X):=A^{r,q}(X,\left \{A_j\right \}_{j\in \mathbb{N}_0})$ contained in a Banach space $X$ of functions on $\Omega$.  Here the parameter $r>0$ measures the rate that certain {\em linear} approximations converge in this space,  $\left \{A_j\right\}_{j\in \mathbb{N}_0}$ are a family of approximant subspaces each having dimensions $n_j:=\#(A_j)$ from which estimates are constructed, and $1\leq q \leq \infty$. 
These spaces have come to assume a central role in understanding  many recent advances in the theoretical foundations of signal and image processing, denoising, compressed sensing, optimal recovery, nonlinear regression, and harmonic analysis. \cite{devorenonlinear}
We also construct estimates using the spectral approximation spaces $A^{r,2}_\lambda(H)$ over a Hilbert space $H$, which can be shown to be a special case of the $A^{r,q}(X)$ under some circumstances. 
We show that by combining approximation space theory and probabilistic error estimates based on confidence functions, an overall convergence rate is obtained for approximation of $\mathcal{P}$,  $\mathcal{U}$,  and some of their associated 
data-driven algorithms for the study of dynamical systems. We  describe below several  specific cases treated in this paper.

\medskip
\noindent \underline{ The case when $p,w,\Omega,\mu$ are known:} In this case we  want to define algorithms and  create finite dimensional estimates $\mathcal{U}_jf$ or $\mathcal{P}_jf$  that are guaranteed to converge to  $\mathcal{U}f$ or $\mathcal{P}f$, respectively,  for all $f$ in a given class of functions. Here the finite dimensional operators $\mathcal{U}_j$, for instance,  are mappings onto some $n_j$ dimensional approximant space $A_j$ contained in the family of functions over which $\mathcal{U}$ is defined. Ideally,  we would like to know at what rate $r>0$  the estimates $\mathcal{P}_jf$ or $\mathcal{U}_jf$  converge to true values for different classes of functions $f$. Specifically, we  seek   conditions that guarantee  the error decays such as $\|(\mathcal{P}-\mathcal{P}_j)f \|_X\lesssim ao(r,j)$ for all  functions $f$ in a given class. Here $ao(r,j)$ is the
approximation order of rate $r$ for estimates constructed from the subspace $A_j$.  Examples of this form with $ao(r,j)=2^{-rj}$ or $ao(r,j)=n_j^{-r}$  can be found in Theorems \ref{th:th1}, \ref{th:Ar2_NSA_a}, and  \ref{th:approx_meas}. Similar types of convergence rates are also derived in Theorems \ref{th:spectral_approx} and \ref{th:spectral_approx_ns} in terms of powers of eigenvalues $ao(r,j)={\lambda^{r/2}_j}$ when the bases used for approximations are eigenfunctions of $\mathcal{P}$ or $\mathcal{U}$. We also are interested in determining conditions that $\mathcal{U}_j,\mathcal{P}_j$ converge to $\mathcal{U},\mathcal{P}$ in some suitable operator norm, as stated in Theorem  \ref{th:th1}.  In most recent  studies of Koopman theory, the approximations  are  constructed using bases that consist of algebraic polynomials, trigonometric polynomials, piecewise polynomial spaces such as splines or finite element spaces, as well as eigenfunctions  if their calculation is tractable.  

This first category of results is important to  Koopman theory for a few reasons.  As a general rule, we have already observed that  it is  more common that  convergence of approximations is established in many recent  studies of Koopman theory  \cite{Korda2017Convergence,Mezic2017ddsa,Nitin2018,Das2018,mezic2017kosda,Mezic2017Random,Tu2014ExactDMD,Williams2015EDMD}, but not the rate of convergence. 
It is also important to note that the case here, which assumes exact knowledge of all the problem data $p,w,\Omega,\mu$, serves as the foundation for treating much more difficult analyses when some problem data  are unknown or uncertain. This case is the starting point for the next two  cases that feature uncertainty in the problem data.

\medskip
\noindent \underline{The case when $p$, $w$, $\Omega$ are known,  the measure $\mu$ is unknown:}   Even for the study of deterministic systems, Koopman theory includes aspects of probability theory. The role of a measure, or measures,  is central to the approach. It should come as little surprise then that  there is an interplay of deterministic and stochastic contributions to the overall error in many of the associated  approximation problems.   The class of problems when the measure $\mu$ is unknown plays an  important theoretical role. It serves to bridge the theoretical gap between the aforementioned case when all problem data $p,w,\Omega,\mu$ are  known, and the most  complex case studied next when all problem data are unknown or uncertain. It is because the measure is often unknown in applications that algorithms based on samples of the underlying dynamical system are so popular. 

Examples of this case are plentiful when we have an accepted model of the discrete evolution, and we seek to understand or characterize the subset or submanifold of the full configuration  space over which evolutions are concentrated. 
One such case arises in developing lower order models from  studies of computational fluid dynamics.  A fluid flow discrete model arising from a numerical approximation of the Navier-Stokes equations might evolve on a state space having dimension  $d\approx O(10^6)$ or $ O(10^7)$. So,  there does in  principle exist an exact model, given by the Navier-Stokes equations.  Sufficiently high dimensional grids can be assumed to yield approximations with small, perhaps negligible,  error of the large scale flow dynamics. It is not practical to use the  direct numerical simulation in many applications that require interactive or near real-time  predictions. Such applications can include models of aerodynamic loads for use in air vehicle design or flight control synthesis. In some studies the goal is to gain an understanding of the inherent or underlying mechanics of a particular flow. The determination of regions of recirculation, the identification of coherent structures, or the determination of overall lift and drag trends from the dynamics of shed vortices are but a few of the common examples.   Low dimensional proxies are needed for these applications.   It is not unusual that low dimensional approximations are generated from data-driven models that evolve on a state space having  $d\approx O(10)$ to $ O(10^2)$ degrees of freedom.  In this application  $\mu$ describes the concentration or support of the dynamics on the small subset of  $\mathbb{R}^d$, and some low order models can be interpreted in terms of its approximation.


\medskip
\noindent \underline{The case when  $p$, $w$, $\Omega$, $\mu$ are unknown:} In applications it is perhaps of the greatest interest when the only explicit information regarding the system under study is a collection of observations $z:=\{(x_i,y_i)\}_{i\leq m}\subset \Omega \times \Omega$ from an  experiment  that measures  the  output state $y\in \Omega$ obtained from the input state $x\in \Omega$.  
We again consider the study of  fluid flows for an application. Some fluid flow systems are so complex that it is prohibitively difficult,  time-consuming, or just infeasible to construct an accurate numerical simulation of the Navier-Stokes equations without supporting experiments. It is then common that  time-indexed experimental measurements  of the velocity  field are made using techniques such as particle image velocimetry (PIV), or observations are made of pressure distributions at points on surfaces, say, from pressure sensitive paints. These samples can be the source of discrete approximations of the fluid dynamics. The  discrete dynamics between consecutive experimentally observed velocity fields, for instance, can  have a dimension  that is still large, although  perhaps not has large as that in a full resolution simulation of the Navier-Stokes equations. In effect  both the measure $\mu$ describing the concentration of trajectories in the configuration space and the underlying dynamic model may be  largely unknown in these cases.   One goal in studying  these systems is the generation of data-driven models of the Perron-Frobenius and Koopman operators from the observations $z=\{(x_i,y_i)\}_{i\leq m}$ of the input-output pairs $(x_i,y_i)$ of the system.

In this situation we  want  to derive practical estimates $\mathcal{P}_{j,z}f$ or $\mathcal{U}_{j,z}f$ that depend on the samples $z$ and the approximant subspace $A_j$ and to prove in what sense these estimates converge to $\mathcal{P}f$ and $\mathcal{U}f$. These types of  estimators have been proposed and investigated in many of the recent studies \cite{Williams2015EDMD,Tu2014ExactDMD,Kutz2016MRDMD,Mezic2017Random} and are embodied in  popular algorithms such as the Dynamic Mode Decomposition (DMD) method and the Extended Dynamic Mode Decomposition (EDMD).  These  studies are important in that they make clear the 
structural relationship between the EDMD algorithm and methods to approximate the Koopman or Frobenius Perron operators, and they establish cases in which convergence is proven. This paper will explore in what sense rates of convergence depend on the number $m$ of samples and dimension $n_j$ of the approximant spaces $A_j$. 

\subsection{Approach and Philosophy}
\label{sec:philosophy} 
The primary aim of this paper is to introduce a common theoretical framework for Koopman theory and associated data-driven algorithms that are used in the study of dynamical systems. We focus on  determination of the  rates of convergence of approximations of the Koopman and Frobenius operators, as well as  some classes of  measures that arise in the study of associated data-driven algorithms. These results facilitate the study of the rates of  convergence of  practical data-driven algorithms.

An essential feature of the theory in this paper is the introduction and use of a  class of priors that determine the rates of convergence. 
The  set of  priors define what information is available about the function, measure, or  operator to be approximated \cite{devore2006approxmethodsuperlearn}. (We are not referring to the notion of priors that arises in Bayesian estimation or stochastic filtering.)  In the most general analysis,  this paper defines the priors in terms of  the  {\em linear} approximation spaces $A^{r,q}(X)$ of order $r>0$ with $1\leq q\leq \infty$ that are contained in a Banach space $X$ of functions defined on a domain $\Omega$. We also study the {\em linear} spectral approximation spaces $A_\lambda^{r,2}(U)$ that are relevant to many approaches in Koopman theory, which can be understood as a special case of the spaces $A^{r,q}(X)$ in some situations.  
While we have noted the broad spectrum of advances that have been facilitated by this theory, 
as far as the authors can tell, the systematic use  of such priors   has not been pursued  in the study of Koopman theory. 
The intended audience of this paper is primarily the researchers, engineers, scientists, and academics that want to understand  or use Koopman theory for the  study of dynamical systems.
We have worked carefully to try and  balance the generality of the approach in this paper, the practicality of the theory for  the study  of data-driven algorithms, and the intuitive understanding of what membership in an approximation space means,  pragmatically speaking. Those who are familiar with the underlying theory of approximation spaces will be well-aware that the methodology introduced here can be generalized substantially. 

When they are defined axiomatically, the approximation spaces $A^{r,q}(X)$ for a (quasi-)Banach space $X$ can seem quite abstract. A discussion of the  theory can be found in the brief paper \cite{piestch}, and a full account of the theory is given in \cite{devore1993constapprox}. Reference  \cite{devorenonlinear} is a particularly good introduction to the theory and gives a readable, interesting  motivation  for  the approach.  Also, Dahmen's work in \cite{dahmenmultiscale} is particularly relevant to most of the theory as it is applied in this paper. The approach there utilizes biorthogonal families of bases for the construction of approximation spaces, and the approach here based on orthonormal wavelets or multiwavelets to construct $A^{r,2}(U)$ is  but a specific case of the biorthogonal construction. 

Intuition regarding the approximation spaces is perhaps most easily developed  considering $A^{r,2}(H)$ when $H$ is a Hilbert space of functions over a domain $\Omega$, the functions $\{\psi_i\}_{i\in \mathbb{N}}$ are an $H-$orthonormal basis for $H$,  and $A_j:=\{\psi_i\}_{i\leq j}$ is the approximant space from which approximations are built.  This space can be understood in terms of the generalized Fourier coefficients $\left \{ (f,\psi_i)_H\right \}_{i\in \mathbb{N}}$ of $f\in H$ in the expansion 
$$
f:=\sum_{i\in \mathbb{N}} (f,\psi_i)_H 
\psi_i.
$$
While the precise definition of the approximation space $A^{r,2}(H)$ in Sections \ref{sec:overview}, \ref{sec:approx_spaces_Arq}, or \ref{sec:approx_spaces} might seem a  bit lengthy, it is important to keep in mind a one simple fact. If $f\in A^{r,2}(H)$, then the approximation $\Pi_jf$ that is obtained by truncating the orthonormal decomposition  $\Pi_jf=\sum_{i<j}(f,\psi_i)\psi_i$,  is an example of a {\em linear approximation method}.  These linear   approximation strategies  are explained in the context of the  more general theory of nonlinear  approximation spaces in  Appendix \ref{sec:linear_approx} or in references \cite{devorenonlinear,piestch}.  Linear  strategies have  an error for $f\in A^{r,2}(H)$ that decays like 
$$
\|(I-\Pi_j)f\|_H \lesssim n_j^{-r}|f|_{A^{r,2}(H)}.
$$
with $|\cdot|_{A^{r,2}(H)}$ the seminorm on $A^{r,2}(H)$, which is defined in Equation  \ref{eq:A_S}, and $n_j=\#(A_j)$. \footnote{In this particular exemplary case $n_j:=\#(A_j)=O(j)$, but  in many examples $n_j$ is some  other function of $j$. If we are approximating functions $f:\Omega\subseteq \mathbb{R}^d \rightarrow \mathbb{R}^d$ using a tensor product wavelet or dyadic spline  basis for each coordinate direction, one  often obtains an expression such as  $n_j\approx O(d\cdot 2^{dj})$, for instance.}
In fact a function $f\in A^{r,2}(H)$ if and only if  the sequence $\left \{i^r (f,\psi_i)_H \right \}_{i\in \mathbb{N}}$ is square summable, that is, it is contained in  $\ell^2(\mathbb{N})$. It follows that the greater the approximation rate $r>0$, the faster  the generalized Fourier coefficients must converge to zero for a function $f\in A^{r,2}(H)$.

We   emphasize again  that  approaches we study in this paper are  referred to as linear methods of approximation in the references on approximation theory \cite{devorenonlinear}.  {\em This should not be confused with the types of dynamical systems that are to be studied using these techniques.} The discrete flow, whether deterministic or stochastic,  will generally be nonlinear. Our canonical example of a deterministic system evolves according to the recursion
\begin{equation}
\label{eq:det_Phi}
x_{n+1}=w(x_n),
\end{equation}
where  the function $w:\Omega \rightarrow \Omega$ is generally nonlinear.
This evolution is discussed in more detail following Equation \ref{eq:flow}.   We use linear approximation methods to estimate the Perron-Frobenius or Koopman operators generated by this  discrete, nonlinear, deterministic  dynamics. An analogous observation is true for our examples of discrete stochastic dynamics. Example \ref{ex:meas2} considers an iterated function system, or IFS, {\cite{barnsley,lasota}}. It is the special case of the discrete,  stochastic evolution
\begin{equation}
\label{eq:stoch_Phi}
x_{n+1}=w(x_n,\lambda_n) 
\end{equation}
where $w: \Omega \times \Lambda \rightarrow \Omega$ is  generally a nonlinear function and $\{\lambda_n\}_{n\in\mathbb{N}}$ is a sequence of  random values  in a symbol space $\Lambda$. We describes rates of convergence of linear methods of approximation of the Koopman and Perron-Frobenius operators generated by the above  discrete, nonlinear, stochastic dynamics.  

Intuition about approximation spaces is also improved by noting that they are in many instances equivalent to other function spaces that may be much more familiar.  We only use a few of the  simplest of the numerous equivalent characterizations of these spaces  \cite{devore1993constapprox} in this paper.  Again, in this motivating  introduction, we mostly restrict consideration to $A^{r,2}(H)$ for a Hilbert space $H$. For  analysts that study evolutionary partial differential equations, it is frequently the case that studies are carried out in terms of the Sobolev space $W^{r}(L^2(\Omega))$. The Sobolev space $W^{r}(L^2(\Omega))$ contains all functions in the  Lebesgue space $L^2(\Omega)$ that have  weak derivatives in $L^2(\Omega)$ of order less than or equal to $r$. The Banach space $C^{r}(\Omega)$, the set of functions having  classical derivatives through order $r$, is a subset of the corresponding   Sobolev space, $C^{r}(\Omega)\subseteq W^{r}(L^2(\Omega))$.  It will be important in many of our examples that  linear approximation spaces are often equivalent to Sobolev spaces, $A^{r,2}(L^2(\Omega),\left \{A_j \right\}_{j\in \mathbb{N}_0}) \approx W^{r}(L^2(\Omega))$ for a range of $r$ that depends on the smoothness of the basis for $A_j$. This is made precise via an argument summarized in Appendix \ref{sec:approx_spaces} and in the book    \cite{devore1993constapprox}.   This equivalence is the reason why the index $r$ is understood heurstically as a measure of  smoothness and rate of approximation.  A practical implication of this fact is that a bound in the error in terms of the approximation space norm (which may seem rather abstract) implies that the bound holds for the more common or conventional space of $r-$time continuously differentiable functions.

Analysts who study evolutionary partial or ordinary differential equations encounter Lipschitz conditions in many theorems that guarantee the existence of unique solutions of initial value problems.
The Lipschitz spaces play an important role in this paper, and particularly in the examples, in that they are also often equivalent to approximation spaces. For connecting the approximation spaces to certain Lipschitz spaces, we rely on the generalized Lipschitz space $\text{Lip}^*(r,L^p(\Omega))$. This space is not usually encountered in theorems that guarantee solutions to initial value problems, but features prominently in approximation theory.  
As summarized in Section \ref{sec:notation},  the generalized   Lipschitz spaces include the spaces $\text{Lip}(r,C(\Omega))$ and $\text{Lip}(r,L^p(\Omega))$ as special cases for some ranges of smoothness $r$. 
In the former,   functions  satisfy the   pointwise Lipschitz inequality that is so common in existence theorems for ODEs.    In the latter the Lipschitz condition in defined in terms of an integral in $L^p(\Omega)$.  When $0<r<1$, for instance, we have 
$ \text{Lip}^*(r,L^p(\Omega)) \approx \text{Lip}(r,L^p(\Omega))$ and $\text{Lip}(r,L^\infty(\Omega))\approx \text{Lip}(r,C(\Omega))$. 
We have found that the two special cases of  Lipschitz spaces make for simpler computations such as in the Example \ref{ex:warp3}, while the generalized Lipschitz spaces are particularly amenable to determine their relationship to linear approximation spaces over a wider range of smoothness $r$ via Theorem \ref{th:approx_besov} in Appendix \ref{sec:approx_splines}.    

One potential limitation of the approach in this paper might be that our choice to use orthonormal bases  for the construction of the approximations is too restrictive. In fact, in this paper we  mostly limit our analysis to approximations associated with projection onto families of orthonormal wavelets, multiwavelets, and eigenfunctions.  By far, most studies of convergence of approximations of Koopman and Perron-Frobenius operators use finite dimensional spaces of polynomials, spaces of piecewise polynomials, or eigenfunctions. While the orthogonal eigenfunctions of compact self-adjoint operators do fit nicely within the theory outlined in this paper, the choice of higher order  piecewise polynomial approximations in the references are seldom  constructed from   orthonormal projections. It is much more common, for example, that piecewise polynomial approximations are based on quadratures, an interpolation  formula, or on a (Petrov-)Galerkin approximation.  There are at least two points we want to make  regarding this issue. 

First, the  general theory of approximation spaces is not cast in terms of orthogonal projections onto finite dimensional subspaces spanned by orthonormal bases. Instead the general theory is presented in many  equivalent  forms in the literature. A pragmatic general theory can be  expressed in terms of quasi-interpolant projections on spline spaces \cite{devore1993constapprox} or   biorthogonal bases of splines or wavelets \cite{dahmenmultiscale}. Not too surprisingly,  the general theory is more complex to describe and less intuitive compared to the case in this paper. In \cite{dahmenmultiscale,dahmenman1,dahmenman2} it requires an introduction of a primal collection of finite dimensional approximation spaces, as well as  an associated family of dual approximation spaces. It is precisely this complexity we seek to avoid in this  overview of how these methods can be applied to Koopman theory.    In short, we have elected to limit the discussion to such orthonormal bases primarily for pedogogic reasons. We have chosen to stick to the theory that is simpler to state, to understand it in an intuitive sense, and to note where the more general case could be brought to bear when the result is fairly direct. Such is the case when we discuss approximation of measures in Section \ref{sec:equiv_metric}. In this section we include a discussion of approximations of some measures that are understood in terms of duality to the more general class of linear approximation spaces $A^{r,q}(X)$ for a Banach space $X$. Pragmatic implementations are certainly possible in a the  more general context.  The   multiscale analysis defined on a variety of  manifolds such as in  \cite{dahmenman1,dahmenman2} could play an important role in Koopman theory, for those who are not faint of heart. 

Second, we have elected to employ orthonormal wavelets and multiwavelets that exactly reproduce some common families of  piecewise polynomials, finite element spaces, or  splines in   all of our examples. Such choices do generate practical algorithms, of arbitrarily high approximation order, depending on the smoothness of the basis.   It should be observed  that the  approximation rates in spaces such as  $A^{r,2}(L_\mu^2(\Omega))$  can be used to obtain the best rates achievable by linear approximation methods based on  piecewise polynomial, finite element,  or spline functions that are  contained in the span of the selected wavelets or multiwavelets. 
In fact, when the measure $\mu$ is just Lebesgue measure, there are Daubechies orthonormal wavelets, ``Coiflets'',  and orthonormal multiwavelets that span a large selection of polynomial, piecewise polynomial, and spline spaces. The approach in this paper therefore gives the best possible convergence rates for linear approximations built from all of these more traditional, and perhaps more familiar,  approximant  spaces that do not enforce  any type of orthonormality among the basis functions. This interpretation of the best possible approximation rates for piecewise polynomial spaces is discussed in more detail in  Examples \ref{ex:ON_wave} and \ref{ex:ON_multi}.

%
%

\section{Overview of Primary Results}
\label{sec:overview}  
We have categorized a number of  problems that arise naturally in the study of approximations of the Perron-Frobenius and Koopman  operators  in Section \ref{sec:types}.  Here  we specifically summarize the contributions of this paper to each of the categories when 1) the problem data $p,w, \mu,\Omega$ are known, 2) the measure $\mu$ is unknown but $p,w,\Omega$ are known, and 3) the problem data $p,w,\mu,\Omega$ are unknown or our knowledge of them is uncertain, but we are given a collection of observations of the input-output behavior of the dynamical system.  Roughly speaking, the presentation in this paper proceeds from rather general results to those that are most specific to applications of Koopman theory. In this sense the paper progesses from a well-known framework to investigate its implications for Koopman theory.  

The simplest case is presented  first, when the problem data is known. This situation must be studied to address more difficult problems of interest to dynamical systems and Koopman theory.  We follow this discussion with a summary of the last two categories, where some of the problem data is uncertain or unknown. 

\noindent 
\subsection*{\noindent The case when $p,w,\mu,\Omega$ are known}
\label{sec:mu_p_known}

When the problem data $p,w,\mu,\Omega$  are known, this paper strengthens several existing results that study the  convergence of approximations of the Koopman or Perron-Frobenius operators. For the most part these results are direct applications of the theory of linear approximation spaces to Koopman theory. Perhaps the most unique or novel insight in this section are the results that tie the approximation of the dual (or adjoint) Perron-Frobenius and Koopman operators to linear approximation spaces defined in terms of warped wavelets.  Thi section begins with introduction of  the spectral approximation spaces $A^{r,2}_\lambda(U)$ in Equation \ref{eq:spectral_space} with $U:=L^2_\mu(\Omega)$. The spaces $A^{r,2}_\lambda(U)$ depend on a self-adjoint  compact operator $T$ that has eigenpairs $\left \{
( \lambda_i,u_i),
\right \}_{ i\in \mathbb{N}}$ with $\lambda_i$ the eigenvalue corresponding to the eigenvector $u_i$. 
From spectral theory we know that the eigenvalues are nonincreasing, and they can only accumulate at zero. 
The Koopman or Perron-Frobenius operators are assumed to be contained in a family of operators $\mathbb{A}^{r,2}_\lambda(U)$ acting on  the Hilbert space $U$ that are analogous to the spaces $A^{r,2}_\lambda(U)$.   The principal results in this context are Theorems \ref{th:spectral_approx}  and \ref{th:spectral_approx_ns} that  construct approximations $\mathcal{P}_j$ of $\mathcal{P}$, or through duality $\mathcal{U}_j$ of $\mathcal{U}$, from the approximant spaces $A_j=\text{span}_{i<j}\{u_i\}$  that converge at a rate $O(\lambda_j^{r/2})$ whenever $\mathcal{P}f\in A^{r,2}_\lambda(U)$.  Essentially, the membership of a function in the space $A_\lambda^{r,2}(U)$ guarantees that its  projection error decays like $O(\lambda_j^{r/2}$).  This analysis can be used to obtain rates of convergence in numerous  approaches that currently only guarantee convergence.

The final result for the spectral approximation spaces rewrites the error bounds by grouping the basis functions into blocks of length $n_j$. 
This grouping  is important when $\{n_j\}_{j\in \mathbb{N}_0}$ is a quasigeometric sequence of integers defined in Section \ref{sec:notation}.  
If the eigenvalues $\left \{\lambda_{2^j}\right \}_{j\in \mathbb{N}}$ are quasigeometric in the sense that they satisfy 
\begin{equation}
\lambda_{n_j}^{1/2} \approx n_j^{-1} 
\label{eq:quasi}
\end{equation}
for such a sequence of integers $\{n_j\}_{j\in \mathbb{N}_0}$, 
then we have the error estimate
$$
|\mathcal{P}-\mathcal{P}_{n_j}|_{\mathbb{A}^{r,2}_\lambda(U)} \lesssim n_j^{-(s-r)} 
$$
for all $s>r>0$.  
Perhaps the most common choice of the sequence selects $n_j:=2^{dj}$ for $U:=L^2_\mu(\Omega)$ and $\Omega\subset \mathbb{R}^d$.  This result for the spectral approximation spaces  is written this way to emphasize its resemblance  to the more general error estimates derived in Section \ref{sec:approx_spaces_Arq}. Specifically, we have $n_j\approx 2^{j}$ when $d=1$ in many applications using wavelets,  multiwavelets, or dyadic splines  to define the approximant spaces $A_j$. 

For the most part, this paper estimates  the Perron-Frobenius and Koopman operators  in a more general framework using  the approximation spaces $A^{r,q}(U)$ for $r>0$ and $1\leq q \leq \infty$ that are contained in the Hilbert space $U$. These are introduced in  Section \ref{sec:approx_spaces_Arq}.  Approximations   are constructed from  families of orthonormal bases that determine the  approximant space $A_j$ that has dimension $n_j=\#(A_j)$ for $j\in \mathbb{N}$. Theorem \ref{th:th_equiv} shows that the spectral approximation spaces $A^{r,2}_\lambda(U)$ can be viewed as special cases of the approximation spaces $A^{r,q}(U)$ in  the important case that the eigenvalues are quasigeometric as in Equation \ref{eq:quasi}. The  approximations in $A^{r,2}(U)$ are cast in this paper in terms of compactly supported wavelets and multiwavelets that reproduce families of piecewise polynomials and splines. It is important to see that these basis functions are readily available: they have been constructed  from first principles in a host of references. Perhaps the most well-known are the Daubechies compactly supported orthonormal wavelets described in \cite{daubechies88,daubechiesbook}. Another well-known family of compactly supported $L^2(\mathbb{R})-$orthonormal wavelets are the ``Coiflets'' described in \cite{coifman1,coifman2}. We outline the use of compactly supported orthonormal multiwavelets of \cite{dgh, dgh18}. Unlike the Daubechies wavelets these functions are piecewise polynomials. The wavelets and multiwavelets above are not constructed from eigenfunctions, and therefore they avoid the challenges associated with computation of the bases  from  nontrivial Koopman or Perron-Frobenius operators.  Approximations $\mathcal{P}_j$ are constructed, and when $\mathcal{P}f\in A^{r,2}(U)$, their   error is bounded by  $\|(\mathcal{P}-\mathcal{P}_j)f\|_U\approx O(2^{-rj})$  with $j$ being the resolution level of a uniform dyadic grid over $\Omega \subset \mathbb{R}^d$ composed of cells having side length $2^{-j}$. If $d=1$, we have $n_j:= \#(A_j) \approx 2^j$ when using wavelets or multiwavelet families, which again yields the approximation rate $O(n_j^{-r})$ analogous to that for the spectral spaces   with quasigeometric eigenvalues. 

An additional theoretical result of this approach follows, since the wavelet and multiwavelet bases we use reproduce certain spline spaces. The study of rates of convergence of approximants constructed from the wavelet and multiwavelets determine upper bounds for the  the best possible (linear) convergence rates   of some   common approaches that are  expressed in terms of algebraic polynomials,  piecewise algebraic polynomials, finite element spaces,   or splines.  Some of the (generalized) Galerkin projection methods fall into this category, which are discussed in \cite{klus2016}, for instance. 

Usually  in this paper we employ approximation spaces that are contained in a Hilbert space $U$.  When the problem data $\mu,p,\Omega$ and $w$ are known, some results for the construction of approximations of signed measures are possible using the Hilbert spaces $A^{r,2}(U)$ contained in the Hilbert space $U$.  A bit more generality can be  advantageous when we study approximation of signed and probability measures in Section \ref{sec:equiv_metric}. 
 Approximation methods that  are more widely applicable are  obtained with  introduction  of the  spaces $A^{r,q}(X)$ for a Banach space $X$.  These  priors  enable the derivation of rates of convergence for approximations of signed  measures or probability measures  in a weak$^*$ sense. This is the one section of the paper in which we do not restrict attention to approximations in $A^{r,2}(U)$ with $U$ a Hilbert space.  Approximation of signed  measures is achieved using a  uniformly bounded family of dual operators $\tilde{\Pi}'_j$ of the linear projection operators $\tilde{\Pi}_j:X\rightarrow A_j$ that are onto the approximant spaces $A_j$.  When approximation of signed measures is carried out in $A^{r,2}(U)$, for instance, it is shown that  then we have
$$
\left |  \left <(I-{\Pi}'_j)\nu , f \right > _{C^*(\Omega) \times C(\Omega)} \right | 
\approx O(2^{-jr}) \approx O(n_j^{-r}) 
$$
for $r>0$,  $d=1$,  $f\in A^{r,2}(U)\subset C(\Omega) \subset U$, and $\nu\in U^* \subset  C^*(\Omega)$ with 
$\Pi'_j:C^*(\Omega)\subset A_j^* \rightarrow U^*\subset C^*(\Omega)$ the transpose of the  orthogonal projection operator $\Pi_j:U\rightarrow A_j$  
in Theorem \ref{th:approx_meas}. 
By construction, this example yields approximations that are signed measures that converge in $C^*(\Omega)$, but the approximations  are not guaranteed to be probability measures even if $\nu$ is. 
An algorithmic approach based on this result that generates approximations that are  probability measures  is  discussed in Section \ref{sec:approx_PM}.  
 
Again, to the best of our  knowledge, the  approach exploiting priors has not been  studied in a systematic way for Koopman theory. Several examples in Section \ref{sec:approx_measures1} illustrate the approximation of signed and probability measures, and likewise demonstrates how they can be used in the  study of Perron-Frobenius or Koopman operators.  Applications  of this strategy are given in Examples \ref{ex:meas0}, \ref{ex:meas1}, and \ref{ex:meas2}.
We show in our analysis that  it is possible to base approximations $\mathcal{P}_j$, and therefore by duality approximations $\mathcal{U}_j$, on approximations $\mu_j$ of the measure $\mu$. 
We  show that in some cases this notion of convergence, which relies on the definition of an approximation space, implies convergence in the bounded Lipschitz metric $d_{BL}$ on probability measures.

\subsection*{The case when $p,w,\Omega$ are known but $\mu$ is unknown}
\label{sec:mu_unknown_p_known}
\noindent 
The approaches above enable the determination of rates of convergence for several important problems in Koopman theory, but require full knowledge of the data $\mu,p,w,\Omega$ of the dynamic system under study.
Also, if the eigenfunctions are to be used to construct approximations, their calculation must be tractable. If these bases are not known or cannot be computed, the theorems and results above  do not give a realizable algorithm to approximate the Perron-Frobenius operator $\mathcal{P}$ or Koopman operator $U$. Generally, the popularity of Koopman theory over the past few years can be attributed in large degree to its success in deriving approximations from samples when some or much of  the problem data is unknown.
In this case it is popular to construct approximations of $\mathcal{P}$ or $\mathcal{U}$ from samples $z_n:=\{(x_n,y_n)\} \in  \Omega\times \Omega$ of the input state $x_n$ and output state $y_n$ of the dynamic system. In this paper we study two different scenarios for how observations of the system are collected. In the first, we simply make observations of the state process over time for a fixed initial condition $x_0\in \Omega$, and we have $\{x_n,y_n\}_{n\leq m}$ for time steps $n=0,1,\ldots,n$. In this case the observations are along a {\em  sample path } starting at $x_0$. There is another important case where observations { are not } collected over a sample path of the system.  We consider the case in which a family of initial conditions   $\{x_i\}_{i\leq m}$ are determined either deterministically or stochastically, and each $\{z_i\}_{i\leq m}=\{(x_i,y_i)\}_{i\leq m}$ records  the single step response of the system for each selected $x_i$. This way of collecting samples can  be thought of as a set of input-output responses  for a number  of test cases. Here, the samples $z_i$ are indexed by  the initial conditions or test case, not time step. Of course, there also are hybrids of the two above realizations of experiments: it is possible that  initial conditions are selected  randomly according to some measure $\mu$ on $\Omega$ and subsequently  observations along the sample path of each are made, and so forth.  It remains an exciting area of research to explore rates of convergence of  approximations for these various cases, with the problem of fusing approximations obtained over different sample paths being one important subproblem.

As discussed in more detail in Section \ref{sec:chains_operators} and Appendix \ref{sec:app_processes}, the assumptions about the statistics of the samples to a large extent determines how the overall rates of convergence can be derived. It is almost always  easier to derive approximation rates when the samples  are independent and identically distributed (IID). Such is the case if we randomly choose an initial condition $x_i$ independently  according to some fixed probability measure $\mu$ on $\Omega$ and then measure the single step response $y_i$ that is generated by this $x_i$. This means that for the deterministic system the states $z_i:=(x_i,y_i)$ are IID according to the probability distribution
$\nu(dx,dy):=\delta_{w(x)}(dy)\mu(dx)$. More generally, the canonical  discrete dynamical systems such as in Equations \ref{eq:det_Phi} and \ref{eq:stoch_Phi} are all examples of a Markov chain.  For  any Markov chain that has transition probability kernel $\mathbb{P}(dy,x)$, as discussed in Section \ref{sec:processes}, the  samples of single step input-output response  $z_i:=(x_i,y_i)$ are IID according to the probability measure $\nu:=\mathbb{P}(dy,x)\mu(dx)$. Here we want to emphasize that the state process is still a  Markov chain, as in Equations \ref{eq:det_Phi} and \ref{eq:stoch_Phi},   and it generates successive  samples $\{x_n\}_{n\in \mathbb{N}_0}$   over time that are dependent. It is the collection of single step  responses $\{(x_i,y_i)\}_{i\leq m}$ indexed by initial condition or test case that are IID. In this paper, the study of a particular problem is first carried out under the IID assumption: the samples correspond to a collection of test cases or different selections of initial conditions. Then, based on insights gained from the IID scenario, the case when samples are along a sample path is considered. It is important to note that this latter case  seems to be the one studied most frequently in the literature on Koopman theory, while the former is closely related to methods of nonlinear regression and statistical learning theory. \cite{gyorfy, kerkyacharianwarped, devore2006approxmethodsuperlearn, cs2002,vapnik}

Suppose that the family of samples $\{z_i\}_{i\leq m}:=\{x_i,y_i\}_{i\leq m}$ are IID, indexed by initial condition, as discussed above. We define an approximation $\mathcal{P}_{j,z}$ in Section \ref{sec:mu_unknown_p_known} that depends on the $m$ samples and approximant space $A_j$, and we show that the error decays like
$$
\|\mathcal{P}f-\mathcal{P}_{j,z}f\|_U\lesssim \lambda_{j}^{r/2} + \epsilon
$$
when the analysis is carried out in a spectral approximation space such as  $A^{r,2}_\lambda(U)$, 
or 
$$
\|\mathcal{P}f-\mathcal{P}_{j,z}f\|_U\lesssim 2^{-rj} + \epsilon
$$
where $j$ is the level of resolution of the grid used in an approximation from $A^{r,2}(U)$. Again, these estimates reflect the same rate for $d=1$ if it so happens that the eigenvalues $\lambda_{2^j}^{1/2}\approx n_j:=2^j$. 
In both cases $\epsilon$ is the error due to  stochastic contributions from the samples $z$  to the error.
The error bounds derived here  therefore depend on the samples and the selected finite dimensional bases. Theorem \ref{th:th3}  expresses a novel convergence rate for Koopman theory that illustrates the classical bias versus variance tradeoff for probabilistic error estimates.  The rate of convergence is cast in terms of an accuracy confidence function $AC(\epsilon,j):=AC(\epsilon,j;f,\mathcal{P})$ that describes the measure of the set of ``bad samples''  where the probabilistic error is large.   The error $\epsilon$ is small except for a set of samples that has exponentially low probability. The size of this set of bad samples is measured by the accuracy confidence function. See Theorem \ref{th:th3} for a detailed discussion. 

This result is generalized and considers observations collected along the sample path of certain types of  Markov chain in Theorem  \ref{th:mixing}.  In this case, the Markov chains are assumed to be exponentially strongly mixing, an assumption closely related to  the ergodicity assumptions  of many studies in Koopman theory.

\subsection*{The case when $p,w,\mu,\Omega$ are unknown}
Finally, this paper also derives new estimates of the rate of convergence that can be achieved when the problem data $p,w,\mu,\Omega$ is  uncertain or unknown.  Specifically, we use a combination of the above error analyses to derive rates of convergence for  one version of the  data-driven Extended Dynamic Mode Decomposition (EDMD) algorithm that is used in the study of some  discrete dynamic systems.   Again, this novel error rate is derived based on the  tradeoff between contributions of  deterministic errors from estimates in the approximation space $A^{r,2}(U)$ and the probabilistic errors that arise from dependence on the samples $z$. 
 It  assumes that the approximant space $A_j$ is the span of the characteristic functions $1_{\square_{j,k}}$ of dyadic cubes
$$
\square_{j,k}:=\left \{ x\in \mathbb{R}^d \ | \ k2^{-j} \leq x_i < (k+1)2^{-j},\  i=1,\ldots, d  \right \}
$$ 
that define a partition of $\Omega$. We show in this case that the EDMD algorithm can be understood in terms of certain estimates developed via techniques of empirical risk minimization (ERM) in distribution-free learning theory.  We denote by $\mathcal{U}^{edmd}_{j,z}f$ and $\mathcal{U}_{j,z}^{erm}f$ the  estimates generated by the EDMD and ERM approaches 
when  the estimates are constructed from the approximant space $A_j$ and $z$ denotes the dependence  on a family of observations $z$.  When $A_j$ is the span of  piecewise constant functions over a dyadic partition of $\Omega\subset \mathbb{R}$,  
we have 
$$
\mathcal{U}^{edmd}_{j,z}\equiv \mathcal{U}^{erm}_{j,z}\Pi_j
$$
with $\Pi_j$ the orthonormal projection onto $A_j$. 
In fact  it follows that 
if $f\in A^{s,2}(U)\subset C(\Omega)$  with $U:=L_\mu^p(\Omega)$ we have 
$$
\|\mathcal{U}_{j,z}^{edmd}f-\mathcal{U}^{erm}_{j,z}f\|_U \lesssim 2^{-(s-r)j}\|f\|_{A^{s,2}(U)}
$$
for $s>r>0$. 
Also, if  the samples $\{z_i\}_{i\leq m}=\{(x_i,y_i)\}_{i\leq m}$ are  a collection of observations along the sample path of a Markov chain, and the function $f\in A_j$, we also show that the expected value over  $m$ samples of the $L_\mu^2(\Omega)$-error satisfies 
$$
\mathbb{E}_{\mathbb{P}^m_{\{x\}}} \left (
\|\mathcal{U}f- \mathcal{U}_{j,z}^{edmd}f\|_{L_\mu^2(\Omega)}\right ) \lesssim
\left (  \frac{\text{log}(e(m))}{e(m)}\right )^{2r/(2r+1)}
$$
for certain exponentially strongly mixing Markov chains with $e(m)$ the effective number of samples. In the above equation 
$\mathbb{P}^m_{\{ x\}}$ denotes the probability distribution of the first $m$ steps of the Markov chain as summarized in Section \ref{sec:processes} or in reference \cite{meyn}. 
This new result is expressed when $\Omega\subset \mathbb{R}^d, \ d=1$ in Theorem \ref{eq:error_EDMD}.

%
%

\subsection{Notation and Basic  Definitions}
\label{sec:notation}
In this paper we denote by $\mathbb{Z}, \mathbb{N}$, and $\mathbb{N}_0$, the integers, integers greater than zero, and nonnegative integers, respectively. We write $a\approx b$ if there are two positive constants $c_1,c_2$ such that  $c_1 a \leq b \leq c_2 a$. The notation $a\lesssim b$ implies that there is a constant $c>0$ such that $a \leq c b$, and the relation $\gtrsim$ is defined similarly. 
We use $\#(S)$ to denote the cardinality of any set $S$.
The Banach spaces of $p$-summable sequences $\ell^p$ have the usual norms $\|a\|_{\ell^p}^p:=\sum_{i}|a_i|^p$ for $1\leq p<\infty$ and $\|a\|_{\ell^\infty}:=\sup_{i}|a_i|$.  A sequence of integers $\{n_j\}_{j\in \mathbb{J}}$ is quasigeometric provided there exists two positive constants $c_1,c_2$ such that 
$1<c_1 \leq n_j/n_{j+1} \leq c_2$ for all $j\in \mathbb{J}\subset \mathbb{N}$. We say that a sequence of real numbers $\{\lambda_j\}_{j\in \mathbb{J}}$, such as eigenvalues of an operator, is quasigeometric whenever there exists positive constants $c_1,c_2$ with $1<c_1 \leq \lambda_{n_j}/\lambda_{n_{j+1}} \leq c_2$ for all $j\in \mathbb{J}$. 
A domain $\Omega$ in this paper is usually either $\Omega:=\mathbb{R}^d$, or  a compact set $\Omega\subset \mathbb{R}^d$. Some constructions of bases for instance are carried out over $\Omega=\mathbb{R}^d$, and subsequently the basis set  is modified so that their support is  a compact set $\Omega\subset \mathbb{R}^d$. 
We denote by $L^p_\mu(\Omega)$ the Banach space of functions $f:\Omega \rightarrow \mathbb{R}$ that are  $p$-integrable with respect to a measure $\mu$ on $\Omega\subseteq \mathbb{R}^d$.  We have $\|f\|_{L^p_\mu(\Omega)}^p :=\int |f(\xi)|^p\mu(d\xi)$ for $1\leq p<\infty$ and $\|f\|_{L^\infty_\mu(\Omega)}|=\sup\left \{ |f(\xi)| \ :  \text{ $\mu-$a.e. } \xi\in  \Omega \right \}$. 
We overload our notation and also denote  the 
vector-valued Lebesgue space $(L_\mu^p(\Omega))^d$ as $L^p_\mu(\Omega)$. Correspondingly, we define  
\begin{align*}
&\|f\|^p_{L^p_\mu(\Omega)}:=\int_\Omega \|f(\xi)\|^p_{\mathbb{R}^d} \mu(d\xi),\\
&\|f\|_{L^\infty_\mu(\Omega)} := \underset{\text{ $\mu$-a.e. } x\in \Omega}{ess\sup} \|f(\xi)\|_{\mathbb{R}^d}, 
\end{align*}
for $1\leq p<\infty$ and functions $f:\Omega\to \mathbb{R}^d$. 
We define $C(\Omega)$ to be the Banach space of all continuous functions on $\Omega$ with $\|f\|_{C(\Omega)}:=\sup_{x\in \Omega}\|f(x)\|_{\mathbb{R}^d}$.  By $UC(\Omega)$ we denote the subspace of $C(\Omega)$ that consists of uniformly continuous functions endowed with the norm it inherits. If $\Omega$ is compact, these two spaces are identical, or course.  

We will also have occasion to use certain spaces of Lipschitz functions in this paper. 
The most general of these are the generalized Lipschitz spaces $\text{Lip}^*(\alpha, L^p(\Omega))$ for $\alpha>0$ and $1\leq p \leq \infty$. Define the $r^{th}$ order difference operator
$$
\Delta^r_{h}(f,x):= \sum_{k=0}^r \left ( \begin{array}{c} r\\k \end{array}\right )(-1)^{r-k}f(x+kh)
$$
and the $r^{th}$ modulus of smoothness
$$
\omega_r(f,t)_p:=\sup_{0<h\leq t}\|\Delta^r_h(f,\cdot)\|_{L^p(\Omega)}.
$$
The $r^{th}$ modulus of smoothness is well-defined for $f\in L^p(\Omega)$ with $1\leq p<\infty$, for $f\in C(\Omega)$ when $\Omega$ is compact, or for $f\in UC(\Omega)$ if $p=\infty$ and $\Omega$ is not compact. The seminorm on the generalized Lipschitz space is given by 
$$
|f|_{\text{Lip}^*(\alpha,L^p(\Omega))}:=\sup_{t>0} \left ( t^{-\alpha} w_r(f,t)_p \right ).
$$
While this definition of a Lipschitz space might seem rather abstract, we introduce it here since it is quite useful in relating linear approximation spaces to Lipschitz spaces in Theorem \ref{th:approx_besov} in Appendix \ref{sec:approx_spaces}. Fortunately, for a restricted range of $\alpha$, these spaces reduce to some more familiar definitions. 

For instance, we say that a function $f\in C(\Omega)$  satisfies an  $r-$Lipschitz inequality   if there is a constant $L>0$ such that 
$$
\|f(x) - f(y)\|_{\mathbb{R}^d} \leq L \|x-y\|^r_{\mathbb{R}^d}
$$
for all $x,y\in \Omega$. We define the seminorm 
$$
|f|_{\text{Lip}(r,C(\Omega))} := \sup_{x,y\in \Omega} \frac{\|f(x)-f(y)\|_{\mathbb{R}^d}}{\|x-y\|_{\mathbb{R}^d}^r}.
$$
Some references simply refer to this space as $\text{Lip}\ r$ \cite{devore1993constapprox}.
 With the norm  $$\|f\|_{\text{Lip}(r,C(\Omega))}:=\|f\|_{C(\Omega)}+|f|_{\text{Lip}(r,C(\Omega))},$$ the set of  functions   
$$\text{Lip}(r,C(\Omega))\subset UC(\Omega) \subset C(\Omega)$$ 
is a Banach space for $0<r\leq 1$.  When $0<r<1$, functions in $\text{Lip}(r,C(\Omega))$ coincide with functions $\text{Lip}^*(r,L^\infty(\Omega))$. 

We also will have occasion to employ the space of Lipschitz functions contained in $L^p(\Omega)$ for $1\leq p\leq \infty$. 
 We define the family $\text{Lip}(L,r,L^p(\Omega))\subset L^p(\Omega)$ for a fixed constant $L\geq 0$ as those functions for which 
$$
\|f(\cdot + h)-f \|_{L^p(\tilde{\Omega})} \leq L \|h\|^r
$$
for every $\|h\|>0$, with  the integration above  over the set $\tilde{\Omega}:=\left \{ x \in \Omega \ | \ x,x+h\in\Omega\right \}$. 
 For $0<r\leq 1$ the Banach space $\text{Lip}(r,L^p(\Omega))$ is given by 
\begin{equation}
\label{eq:lip_Lp}
\text{Lip}(r,L^p(\Omega)):=\bigcup_{L>0} \left \{ f\in \text{Lip}(L,r,L^p(\Omega))\right \}, 
\end{equation}
and the seminorm $|f|_{\text{Lip}(r,L^p(\Omega))}$ is the smallest constant $L$ for which the integral Lipschitz inequality holds for $f$.
When $0<r<1$, we have $\text{Lip}(r,L^p(\Omega))\approx \text{Lip}^*(r,L^p(\Omega))$. 

We have restricted the range of $0<r<1$ for $\text{Lip}(r,C(\Omega))$ and $\text{Lip}(r,L^p(\Omega))$ simply because for this range they are equivalent to the  generalized Lipschitz spaces $\text{Lip}^*(r,\Omega)$, the latter of which are defined for $r>0$. There are other ways to extend the definitions  $\text{Lip}(r,C(\Omega))$ and $\text{Lip}(r,L^p(\Omega))$ for greater values of $r$, but then their relationship to the generalized spaces becomes a delicate matter. See the detailed discussion in \cite{devore1993constapprox}, or the summary in \cite{runst}, for further nuances. We have included all three in this paper since we feel that the definition of $\text{Lip}(r,C(\Omega))$ and $\text{Lip}(r,L^p(\Omega))$ are more convenient for calculations in examples, while  the  generalized Lipschitz spaces $\text{Lip}^*(r,L^p(\Omega))$ are more readily to related to approximation spaces. The latter topic can be found in   \cite{devore1993constapprox}, page 358, Theorem 2.4,  or as as discussed  in the Appendix in Section \ref{sec:approx_splines}.  

For any Banach space $X$ we denote its topological dual $X^*$, and the duality pairing $\langle x^*,x\rangle_{X^*\times X}:=x^*(x)$ for all $x\in X$, $x^*\in X^*$. Let ${B}_1,{B}_2$ be Banach spaces. The dual or transpose operator $L'$ of a  bounded linear operator $L:B_1\rightarrow B_2$ is the unique bounded linear operator $L':B^*_2 \rightarrow B^*_1$ that satisfies $\left < g^*,L f\right >_{B_2^*\times B_2}=\left < L' g^*,f\right>_{B_1^*\times B_1}$ for all $f\in B_1$ and $g^*\in B_2^*$. If $L:H_1\rightarrow H_2$ is a bounded linear operator acting between the Hilbert spaces $H_1,H_2$, the adjoint operator $L^*:H_2 \rightarrow H_1$ is the unique bounded linear operator that satisfies $(Lf,g)_{H_2}=(f,L^*g)_{H_1}$ for all $f\in H_1$ and $g\in H_2$. The Riesz map $R_H:H^*\rightarrow H$ associated with the Hilbert space $H$ is the isometric isomomorphism $R_H:u^*\mapsto u$ defined by 
$$
<h^*,g>_{H^*\times H}:=(R_Hh^*,g)_H:=(h,g)_H
$$
for all $g\in H$. 

For a compact set $\Omega\subset \mathbb{R}^d$, the topological dual space $C^*(\Omega)$ is identified with the finite, regular, countably additive set functions $rca(\Omega)$, which are also known as the regular signed measures. We denote by $\mathbb{M}^{+}(\Omega)$ the positive measures, and by $\mathbb{M}^{+,1}(\Omega)$ the probability measures, contained in $rca(\Omega)$. For any measure $\mu$ on $\Omega$, $\mu^m$ denotes the product measure on $\Omega^m$. The symbol $\mathbb{P}^m_{\{z\}}$ is the  probability distribution of the first $m$ steps of the discrete stochastic process $\{z_i\}_{i\in \mathbb{N}}$. A brief discussion of discrete stochastic processes, and in particular Markov chains,  is given in the Appendix in Section \ref{sec:processes}.

%
%

\section{Markov Chains, Perron-Frobenius,  and Koopman Operators \cite{klus2016}, \cite{lasota}}  
\label{sec:chains_operators}
 
 In this section we define  the class of dynamical systems,  and their associated Perron-Frobenius and Koopman operators, that are studied in the paper. Overall, the definitions of the dynamical systems and  of  their associated  operators differ widely depending on the reference. The articles { \cite{Mezic2017ddsa,mezic2017kosda}} define the operators for the study of continuous and discrete deterministic flows. Reference {\cite{Nitin2018}} defines  Koopman operators on $L^2_\mu(\Omega)$ for periodic approximations of discrete deterministic evolutions. Other recent studies define the operators for deterministic or stochastic flows in {\cite{klus2016}}, and {\cite{klusTO}} define them as  transfer operators having probability density kernels. The popular text  \cite{lasota} introduces various definitions in the event flows evolve in continuous (Chapter 7) or discrete (Chapter 3) time, and for deterministic (Chapters 3-5) or stochastic (Chapter 10) systems. In fact reference \cite{lasota} further subcategorizes the Perron-Frobenius and Koopman operators into subclasses such as the Foias and Barnesley operators. The latter arises in some examples in this paper when we discuss certain approximations of  evolution supported on fractals. Reference \cite{ergodicnagel} gives an in depth account of the analysis of flows and semiflows based on the Koopman operator $\mathcal{U}:f\mapsto f\circ w$ for some fixed mapping $w$. 
 
 As we will see shortly, the operators associated with {\em all of the  discrete flows} above can be understood in terms of the action on either measures or functions of the transition probability kernel of a Markov chain. \cite{meyn,saperstone} The traditional technique for inducing a deterministic flow on signed or probability measures from a Markov chain \cite{saperstone} has been known for some time, so we  follow this convention  from the outset in the paper. We begin by reviewing a few of the basic definitions for the processes we study in  Section \ref{sec:processes}, then discuss realizations of observations in Section \ref{sec:measurements}, and  subsequently define the classes of operators in Section  \ref{sec:operators}.
 
\subsection{The Class of Discrete Dynamical Systems}
\label{sec:processes}
 In the most straightforward case studied in this paper  a discrete evolution is defined on a configuration space $\Omega\subset \mathbb{R}^d$ by    the  recursion 
\begin{equation}
x_{n+1}=w(x_n) \label{eq:flow}
\end{equation}
with  $w:\Omega\rightarrow \Omega$ a $\mu-$measurable map. Here and below the mapping $w$ will generally be nonlinear. Trajectories or sample paths starting at some initial condition $x_0\in \Omega$ of the dynamical system are just the sequence of iterates $\{x_n\}_{n\in \mathbb{N}_0}:= \{w^n(x_0)\}_{n\in \mathbb{N}_0}$. In many problems, the dynamics governed by Equation \ref{eq:flow} may not seem realistic enough since  actual experiments are subject to noise, or the model might be uncertain, etc.  Common modifications of the above deterministic equation yield  the stochastic recursions such as 
\begin{align}
 x_{n+1}&=w(x_n)+\xi_n, 
 \label{eq:MC_ex1}  \\
 x_{n+1}&=w(x_n, \lambda_n), \label{eq:MC_ex2}
\end{align}
with the sequences $\{\xi_n \}_{n\in \mathbb{N}_0}$ and $\{\lambda_n\}_{n\in \mathbb{N}_0}$  a collection of independent  and identically distributed (IID) random variables taking values in $\Omega$ and  the finite symbol space $\Lambda$, respectively.  In Equation \ref{eq:MC_ex2} above the function $w:\Omega \times \Lambda\rightarrow \Omega$. While dynamical systems governed by Equation \ref{eq:MC_ex1} include many prosaic  physical systems, these equations also define some quite   abstract dynamical systems.   Equation \ref{eq:MC_ex2} is the form of governing equation for stochastic dynamical systems on fractals \cite{barnsley}, for instance. We include an analysis of the approximation of Perron-Frobenius and Koopman operators for this system in Example \ref{ex:meas2}. Of course, many other forms of these stochastic equations are also  possible. 

All three of the above examples are examples of Markov chains, which is the family  of dynamical systems we study in this paper for the representation of state evolution.  A brief account of the theoretical foundations of Markov chains is given in Appendix \ref{sec:processes}. A detailed study of the theory over general state spaces, which we employ in this paper,  is given in \cite{meyn}.  Suppose that $(\Omega,\Sigma(\Omega))$ is a measurable  space with $\Sigma(\Omega)$ a sigma-algebra of subsets of $\Omega$. 
  A Markov chain  is a stochastic process that is defined in terms of a transition probability kernel $\mathbb{P}:\Sigma(\Omega) \times \Omega \rightarrow [0,1]$. The quantity $\mathbb{P}(A,x)$  is the  probability of a transition from the current state $x$ to the measurable set $A\in \Sigma(\Omega)$ in the next step of the discrete stochastic process.  The transition probability kernel for the deterministic  flow in Equation \ref{eq:flow} is given by $\mathbb{P}(A,x)\equiv \delta_{w(x)}(A)$ with $\delta_{w(x)}$ the Dirac measure concentrated at $w(x)\in \Omega$. 
  The   transition probability kernel of the chain in Equation \ref{eq:MC_ex1} is given by $\mathbb{P}(A,x)=\mu(A-w(x))$.   See  \cite{lasota} for a discussion of Equation \ref{eq:MC_ex2}, or  for other examples that underly different types of Perron-Frobenius or Koopman operators. 
  
  \subsection{Realizations of Observations}
  \label{sec:measurements}
In view of the above summary, the  most general class of the dynamical systems studied in this paper are those for which {\em the states $x_n$} evolve according to Markov chains having transition kernels $\mathbb{P}(dy,x)$.    Up until this point we have not concerned ourselves about how we model observations of a particular dynamic system. The stochastic process that represents the observations of the state equation can have quite different statistical properties, depending on the definition  or construction of an experiment.  Because there are several ways to model how measurements of a dynamical system are realized in experiments, we briefly review models of a few common setups. 

\subsection*{Deterministic Input-Output Samples}
In one possible scenario, we assume that $m$ input-output samples 
$$z=\{(x_i,y_i)\}_{i\leq m}\subset \Omega \times \Omega$$ for the simple deterministic system are  generated by fixing collection of initial conditions $\{x_i\}_{i\leq m}\subset \Omega$ and measuring the single step output for each initial condition 
\begin{equation}
y_i:=w(x_i). \label{eq:qq}
\end{equation}
We choose the index $i$ to denote that the samples here are indexed by the initial condition or test case.
 The input-output samples in this scenario are exact {\em single step} observations of a noise-free system. We can then ask how the rates of convergence of approximations of Koopman or Frobenius-Perron operators depend on  the collection of test cases. The goal here might be to determine rates of convergence in terms of number of test cases and coverage of the test cases over $\Omega$. For compact domains $\Omega$ we can construct nested grids of initial conditions and analyze convergence rates as the mesh parameter of the grids approaches zero. 

\subsection*{Independent Input-Output Samples}
In a slight modification of the deterministic  scenario  we assume we have a fixed probability distribution $\mu$ on $\Omega$, the initial conditions $\{x_i\}_{i\leq m}$ are drawn independently according to $\mu$, and the single step outputs $y_i$ are generated exactly according to Equation \ref{eq:qq}. The initial conditions are said to be independent and identically distributed (IID) with respect to the measure $\mu$. The sequence of samples $\{z_i\}_{i\leq m}$ that are generated this way are IID with respect to the probability measure $\nu(dx,dy):=\delta_{w(x)}(dy)\mu(dx)$ for  $z:=(x,y)\in \Omega\times \Omega$ in the deterministic, noise-free case. We can easily modify this case somewhat to allow for noisy observations of the output state. Again, we suppose that the initial states $\{ x_{n} \}_{n\leq m}$ are drawn independently according to the fixed probability distribution $\mu$. The single step output states $y_i$ are assumed to be generated by a Markov chain having transition kernel $\mathbb{P}(dy,x)$.  In this case the single step samples $z:=\{(x_i,y_i)\}_{i\leq m}$ are IID on $\Omega\times \Omega$ with distribution ${\nu}{(dz)}:=\mathbb{P}(dy,x)\mu(dx)$ for    $z=(x,y)\in \Omega$.  It should be noted that this manner of collecting observations underlies many strategies for constructing approximations in   publications on nonlinear regression or statistical learning theory. \cite{devore2006approxmethodsuperlearn, kerkyacharianwarped, cs2002, vapnik,smale2009, smale2007,smale2009Online,gyorfy} 
However, the standing assumption in these approaches is that rates of convergence for the  approximation of a typical function is desired, not approximations of an operator such as $\mathcal{U}$ or $\mathcal{P}$. This is a subtle distinction between nonlinear regression, learning theory, and Koopman theory.  Some approaches  for nonlinear regression,  statistical learning theory, or  empirical process estimation  study  processes that are not IID. It is safe to say, however, that these techiques are not as  widely applicable nor as mature  as the results based on IID samples. 
 
\subsection*{Dependent Input-Output Samples}
In the application of Koopman theory, the assumption that samples are IID  is sometimes made.   However, it is also frequently the case that observations are  measured over multiple time  steps  for a single initial condition $x_0\in \Omega$,  instead of over just one time step.  In other words the input-output responses $\{z_n\}_{n\in \mathbb{N}_0}:=\{(x_n,y_n)\}_{n\in \mathbb{N}_0}$ are collected along the sample path of the Markov chain that starts at $x_0$. This case can arise in ergodic approximations in Koopman theory. \cite{ergodicnagel} We use the index $n$, the same time index as in the recursions above, for the measurements in this case to emphasize that observations are indexed in terms of the time step. In this case the 
samples $\{z_n\}_{n\in \mathbb{N}_0}$ constitute a dependent stochastic process. In the noise-free case we have $\{x_n\}_{n\leq m}:=\{w^{n}(x_0)\}_{n\in \mathbb{N}},$ while  for the stochastic case the observations are along a sample path of the Markov chain having transition probability $\mathbb{P}(dy,x)$.

Of course, it is also possible that hybrid collections of measurements are made that combine aspects of the above realizations of observations. We could choose initial conditions randomly  according to some fixed probability distribution, and then measure the response over time along  each sample path for a certain number of time steps. To the authors' knowledge error rates for such methods have not figured prominently in the literature on Koopman theory.

\subsection{Perron-Frobenius and Koopman Operators}
\label{sec:operators}
Koopman and Perron-Frobenius operators $\mathcal{U}$ and $\mathcal{P}$ are defined in terms of, or associated to, specific dynamical systems.  They have many uses including understanding the stability properties of a flow,  studying  the  convergence and rates of convergence of flows to equilibria or attracting sets,  or constructing predictors of observations for flows. 
In view of the conventions for  Markov chains \cite{meyn,saperstone},  we define the Koopman operator $\mathcal{U}$ and Perron-Frobenius operator $ \mathcal{P}$, respectively, in terms of the transition probability $\mathbb{P}(A,x)$ as 
\begin{align}
(\mathcal{P}\nu)(dy)&:=\int_\Omega \mathbb{P}(dy,x)\nu(dx),  \label{eq:genU}\\
(\mathcal{U}f)(x)&:=\int_\Omega  \mathbb{P}(dy,x)f(y), \label{eq:genP}
\end{align}
for a measure $\nu$ on $\Omega$ and function $f:\Omega \rightarrow \mathbb{R}$. We take these expressions as the most general form of the definitions for $\mathcal{U}$ and $\mathcal{P}$ in this paper. 
We say that a probability measure $\nu$  is invariant for the Markov chain having a transition probability kernel $\mathbb{P}(dy,x)$ whenever 
$$
(\mathcal{P}\nu)(A):=\int_\Omega \mathbb{P}(A,x)\nu(dx)=\nu
$$
for all measurable subsets $A\subseteq \Omega$. This definition of invariance of measures of a Markov chain \cite{meyn} takes a familiar form if the chain happens to be the simple deterministic evolution law in Equation \ref{eq:flow}. In that case the transition kernel is $\mathcal{P}(dy,x):=\delta_{w(x)}(dy)$, and we have 
\begin{align*}
   \nu(A)&= \int_\Omega \mathbb{P}(A,x)\nu(dx)  = \int_\Omega \delta_{w(x)}(A) \nu(dx) \\
    &=\int_\Omega \delta_x(w^{-1}(A)) \nu(dx) = \int_\Omega 1_{w^{-1}(A)}(x)\nu (dx) 
    =\nu(w^{-1}(A))
\end{align*}
for all measurable $A\subseteq \Omega$. Thus, for the deterministic case we say that the measure $\nu$ is invariant with respect to the mapping $w:\Omega \rightarrow \Omega$ provided $\nu(A)\equiv \nu(w^{-1}(A))$ for all measurable sets $A\subseteq \Omega$. This is the definition of invariance common in ergodic systems or operator theory. \cite{ergodicnagel}

Several specialized definitions of the Koopman and Perronn-Frobenius operators can be constructed from this general form, depending on a duality structure. We  summarize some of these below.   

\subsubsection{The Dual Pairing  $C^*(\Omega)\times C(\Omega)$}
The development of a theory for approximation of the Koopman or Perron-Frobenius operators in this paper makes assumptions regarding the  regularity or smoothness of  these operators.  
In our case these regularity conditions will be expressed in terms of specific duality structures associated with the Perron-Frobenius and Koopman operators.
One important case studied in this paper regards $\mathcal{U}$ as  a bounded linear operator on the continuous functions $C(\Omega)$ on $\Omega$, and $\mathcal{P}:C^*(\Omega)\rightarrow C^*(\Omega)$.  
 It is often the case  in our analysis that $\Omega$ is a compact subset of $\mathbb{R}^d$, which simplifies some of the duality arguments.  See \cite{meyn} for a discussion of the operators $\mathcal{P}$ and $\mathcal{U}$  when the domain $\Omega$ is not compact. 
When $\Omega$ is compact,  the normed dual $C^*(\Omega)$ is just the family of  regular countably additive set functions, or regular signed measures, denoted $C^*(\Omega)\equiv rca(\Omega)$. \cite{dunford,roubicek} The Koopman operator $\mathcal{U}$ and Perron-Frobenius operator $\mathcal{P}$ are then related by the duality expression  
$$
\left < \mathcal{U} \nu, f\right >_{C^*(\Omega) \times C(\Omega)}= 
\left < \nu , \mathcal{P}f\right >_{C^*(\Omega) \times C(\Omega)}
$$
for all $f\in C(\Omega)$ and $\nu \in C^*(\Omega):=rca(\Omega)$. 
This identity means that $\mathcal{U}=\mathcal{P}'$, that is, $\mathcal{U}$ is the topological transpose or dual operator of $\mathcal{P}$ relative to the pairing $<\cdot,\cdot>_{C^*(\Omega)\times C(\Omega)}$. When we apply this condition for the discrete dynamical flow, which has the transition probability kernel $\mathbb{P}(dy,x):=\delta_{w(x)}(dy)$, we find that \begin{align}
    (\mathcal{U}f)(x)&=(f\circ w)(x), \label{eq:koop_deter}\\
    (\mathcal{P}\nu)(A)&= \nu(w^{-1}(A)). \label{eq:PF_deter}
\end{align}

 \subsubsection{The Dual Pairing  $L^\infty_\mu(\Omega)\times L^1_\mu(\Omega)$}
We also study transition kernels $\mathbb{P}(dy,x)$ that  are given in terms of a transition probability density function $p:\Omega\times \Omega \rightarrow \mathbb{R}$ as in  
$$
\mathbb{P}(dy,x):=p(y,x)\mu(dy),
$$
for some  probability measure $\mu$ on $\Omega$. If we further suppose that $\nu(dx):=m(x)\mu(dx)$ for some $m\in L^1_\mu(\Omega)$,
we then have 
\begin{align*}
(\mathcal{P}\nu)(dy)&=\int \mathbb{P}(dy,x)\nu(dx)\\
&=\int_\Omega p(y,x) m(x)  \mu(dx) \cdot\mu(dy) 
=   (\hat{\mathcal{P}}m)(y) \mu(dy).
\end{align*}
With suitable restrictions on the density $p$, this last expression leads  to an  alternate definition of the Perron-Frobenius operator $\hat{\mathcal{P}}:L^1_\mu(\Omega)\rightarrow L^1_\mu(\Omega)$ with 
\begin{equation}
( \hat{\mathcal{P}} m )(y) = \int_\Omega p(y,x) m(x) \mu(dx). \label{eq:PF_density}
\end{equation}
 In this setup 
 the Koopman operator $\hat{\mathcal{U}}:L^\infty_\mu(\Omega) \rightarrow L^\infty_\mu(\Omega)$ is defined with respect to  the dual pairing $\left < \cdot,\cdot \right >_{L^\infty_\mu(\Omega)\times L^1_\mu(\Omega)}$, since $(L_\mu^1(\Omega))^*=L^\infty_\mu(\Omega).$
 That is, we define the Koopman operator $\hat{\mathcal{U}}$ from the relation
$$
<\hat{\mathcal{U}}g,f>_{L^\infty_\mu(\Omega) \times L^1_\mu(\Omega)} = (g,\hat{\mathcal{P}}f)_{L^\infty_\mu(\Omega) \times L^1_\mu(\Omega)}
$$
for $g\in L^\infty_\mu(\Omega)$ and $f\in L^1_\mu(\Omega)$.
 \cite{lasota}   
 It follows that the Koopman operator $\hat{\mathcal{U}}$ is then induced by the dual 
 kernel in 
\begin{equation}
(\hat{\mathcal{U}}g)(x):=\int_\Omega p(y,x)g(y)\mu(dy). \label{eq:PF_density1}
\end{equation}

\subsubsection{Adjoint Operators $\mathcal{U}, \mathcal{P}$  on a Hilbert Space}
We note one last definition of these operators that is found frequently in the literature. 
When the measure $\mu$ is finite, that is it satisfies  $\mu(\Omega)<\infty$,  and the set $\Omega$ is compact, we have the embeddings of the primal spaces 
$$
C(\Omega) \subseteq L^\infty_\mu(\Omega) 
\subseteq \cdots \subseteq L^2_\mu(\Omega) \subseteq L^1_\mu(\Omega), 
$$
and of the dual spaces 
$$
L^\infty_\mu(\Omega)\equiv (L^1_\mu(\Omega))^*
\subseteq (L^2_\mu(\Omega))^* \cdots \subseteq C^*(\Omega)\equiv rca(\Omega). 
$$
A familiar duality structure can be extracted from the above by  identifying $L^2_\mu(\Omega)$ with itself via the Riesz mapping,
$$
 C(\Omega)\subset L^2_\mu(\Omega)\approx (L^2(\mu))^* \subset (C(\Omega))^* \equiv rca(\Omega).
$$
This is a specific  example of a Gelfand triple, a mathematical structure we discuss in some detail in our analysis of the approximation of measures
in Section \ref{sec:approx_measures1}.

Not surprisingly, it is quite common  to  encounter  a definition of  $\tilde{\mathcal{U}}$ and  $\tilde{\mathcal{P}}$ as adjoints written in terms of  the inner product
\begin{equation}
(\tilde{\mathcal{P}}f,g)_{L^2_\mu(\Omega)} =
(f,\tilde{\mathcal{U}}g)_{L^2_\mu(\Omega)}, \label{eq:PU_adjoint}
\end{equation}
for all $f,g\in L^2_\mu(\Omega)$. From the definition of the Riesz map   $R_{L^2_\mu(\Omega)}:(L_\mu^2(\Omega))^*\rightarrow L^2_\mu(\Omega)$, this identity means that $\tilde{\mathcal{P}}=R_{L^2_\mu(\Omega)}\hat{\mathcal{P}}$ since 
$$
\left< \hat{\mathcal{P}}f,g\right>_{(L^2_\mu(\Omega))^*\times L^2_\mu(\Omega)} = 
\left < R^{-1}_{L^2_\mu(\Omega)}\tilde{\mathcal{P}}f,g\right >_{(L^2_\mu(\Omega))^*\times L^2_\mu(\Omega)}=(\tilde{\mathcal{P}}f,g)_{L^2_\mu(\Omega)}.
$$

We consider a slight generalization of this setup in some examples in our paper. Above, the Koopman and Perron-Frobenius operators are defined as adjoint  operators in the same Hilbert space. For example, if we consider the deterministic system with $w:\Omega \rightarrow \Omega$  an onto mapping, then this may be a fruitful strategy.   In many of our examples we consider operators induced by a mapping  $w:\Omega\rightarrow \tilde{\Omega}\subseteq \mathbb{R}^d$ and admit the possibility that $\Omega$ and $\tilde{\Omega}:=w(\Omega)$ do not coincide. Then it may be advantageous to define for some measure $\tilde{\mu}$ on $\tilde{\Omega}$ 
\begin{align*}
    \tilde{\mathcal{U}} &:=L^2_\mu (\Omega)\rightarrow L^2_{\tilde{\mu}}(\tilde{\Omega}), \\
    \tilde{\mathcal{P}}&:= L^2_{\tilde{\mu}}(\tilde{\Omega}) \rightarrow L^2_\mu(\Omega),  
\end{align*}
with $\tilde{\mathcal{U}}$ and $\tilde{\mathcal{P}}$ adjoints as in Equation \ref{eq:PU_adjoint}.

In the remainder of this paper, we use the  common notation $\mathcal{P}$ and $\mathcal{U}$ for any of the  definitions of $(\mathcal{P},\mathcal{U})$, $(\hat{\mathcal{P}},\hat{\mathcal{U}})$, or
 $(\tilde{\mathcal{P}},\tilde{\mathcal{U}})$  given above. Whether $\mathcal{P}$ acts on measures or functions, for example, will be  clear from context in each application or example. 
\section{Reproducing Kernel Hilbert Spaces}
\label{sec:rkhs}
In this section we summarize  the theory of reproducing kernel Hilbert spaces (RKHS) $H$ that will enable the formulation of one family of  approximations,  and the determination of  their rates of convergence. We suppose that the evolution 
law is such  that the discrete state remains in the  compact set $\Omega\subseteq \mathbb{R}^d$.  
In fact, 
later in the paper, we assume that we are given field observations $\left \{ x_1,  x_2, \cdots, x_N  \right \} \subseteq \Omega$ that are generated as random samples that are distributed in terms of 
the probability measure $\mu$ on $\Omega$. The  measure $\mu$ describes how the samples are concentrated in $\Omega$. We then  are interested in constructing approximations in Koopman theory that somehow reflect  the structure of the measure $\mu$. This is accomplished in this section by introducing a RKHS $V\subset C(\Omega) \subset  U:=L^2_\mu(\Omega)$ that depends on the measure $\mu$.

The construction  begins with a continuous, symmetric, positive definite  kernel $K: \Omega \times \Omega \rightarrow \mathbb{R}$  that is  assumed to generate a reproducing kernel Hilbert space $(V,(\cdot,\cdot)_V)$  over $\Omega$. \cite{steinwart2012mercertheorgenerdomain}  The reproducing property of the kernel guarantees that 
\begin{equation}
(K_x,f)_V = f(x)
\end{equation}
for all $f\in V$ and $x\in \Omega$ with  the function $K_x(\cdot):=K(x,\cdot)$. Alternatively, it is known that if all the evaluation functionals acting on a Hilbert space $V$ are bounded, then $V$ is a RKHS. This means that for each $x\in \Omega$, there is a constant $c_x$ such that $f(x)\leq c_x \|f\|_V$.  We further assume that the kernel $K$ is sufficiently regular to continuously embed $V$ in $U:=L^2_\mu(\Omega)$.
In other words the linear injection $I_K: V \rightarrow U$ 
\begin{align}
I_K & : f \mapsto I_K f = f 
\end{align}
is bounded, and  we have $\| f\|_{U} \leq \| I_K \|  \|f\|_V$ for all $f \in V$.  This fact can be guaranteed if we know that the kernel $K$ satisfies $\sup_{x\in \Omega} K(x,x) <\infty$ as shown by Smale and Zhou in  \cite{smale2007,smale2009}.  It then also follows that $V$ is separable and compactly embedded in $C(\Omega)$. \cite{rosascoint}

The adjoint operator $I_K^*: U \rightarrow V$ is given by
\begin{align*}
(I_K K_q, g)_{U} &= (K_q, I_K^* g)_V   
 = (I_K^*g)(q) = \int_\Omega K_q(y) g(y) \mu(dy).
\end{align*}
We  define the operator $T_K: U  \rightarrow U$ as $T_K:= I_K I_K^* $, and we see that 
\begin{align*}
T_K g :&= I_K I_K^* g = I_K \underbrace{\int_\Omega K(\cdot, r) g(r) \mu(dr)}_{\in V} 
 = \underbrace{\int_\Omega K(\cdot, r) g(r) \mu(dr)}_{\in U}.
\end{align*}
Analogously, we set $T_\mu:=I_K^*I_K$ so that $T_\mu:V \rightarrow V$.  
 \begin{figure}
 \begin{center}
 \begin{tikzcd}
& U \arrow[rd,"I^*_K"] \arrow[rr,"T_K"]&     & U&  \\
V \arrow[ru,"I_K"] \arrow[rr,"T_\mu"] &                           &  V \arrow[ru,"I_K"]&  &
\end{tikzcd}
 \caption{Commutative diagram defining operators $T_K$ and $T_\mu$ in terms of $I_K, I^*_K$}
 \label{fig:ch1_1}
\end{center}
\end{figure}
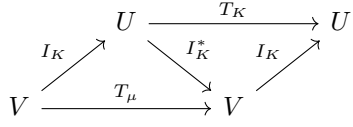
 Since $I^*_K$ is a linear compact operator, both $T_\mu$ and $T_K$ are compact and self-adjoint.  The relationship among the operators $T_K,T_\mu,I_K$, and $I_K^*$ is depicted in Figure \ref{fig:ch1_1}
 
 The operators $T_K$ and $T_\mu$ have convenient representations that are a consequence of spectral theory. The spectral theory for compact, self-adjoint operators 
 is reviewed in Appendix \ref{app:spectral}. More extensive summaries can be found   in \cite{piestch,weidmann1980linearoperathilberspaces}.  The eigenvalues of 
 the operators $T_K$ and $T_\mu$ are identical and are arranged in an extended enumeration, including multiplicities, in nonincreasing order 
 \begin{equation*}
 \lambda_1 \geq \lambda_2 \geq \ldots \geq 0.
 \end{equation*}
 Each eigenspace corresponding to a nonzero eigenvalue is finite 
 dimensional,  and the 
 only possible accumulation point of this infinite sequence is zero. We denote by $\seq{v_i}_{i\in \mathbb{N}}\subseteq V$ and $\seq{u_i}_{i\in \mathbb{N}}\subseteq U$  orthonormal  eigenvectors of $T_\mu$ and $T_K$, respectively, associated with the 
 eigenvalues $\seq{\lambda_i}_{i\in \mathbb{N}}$.
  The spectral theory for compact, self-adjoint operators guarantees that the following expansions are norm-convergent,
 \begin{align}
 T_\mu g & = \sum_{i\in \mathbb{N}} \lambda_i (g,v_i)_V v_i  & &\text{ in $V$} , \label{eq:ev_1} \\
 T_K f & =  \sum_{i\in \mathbb{N}} \lambda_i (f,u_i)_{U} u_i  & & \text{ in $U$} , \label{eq:ev_2}\\
 I^*_K f & = \sum_{i\in \mathbb{N}} \sigma_i (f,u_i)_{U} v_i & & \text{ in $V$}, \label{eq:ev_3}
 \end{align}
 for each $f\in U$ and $g\in V$. By convention these summations are carried out only over the nonzero eigenvalues. The  families $\{u_i\}_{i\in \mathbb{N}}$ and $\{v_i\}_{i\in \mathbb{N}}$ associated with nonzero eigenvalues  are an orthonormal basis for $ N(T_K)^\perp\subset U$ and $N(T_\mu)^\perp\subseteq V$, respectively. \cite{rosascoint} In these equations $\sigma_i:=\sqrt{\lambda_i}$ is the $i^{th}$ singular value of the operator $I^*_K$.  When the eigenvalues are non-increasing, the decompositions in Equations \ref{eq:ev_1},\ref{eq:ev_2}, and \ref{eq:ev_3} are also known as the unique monotonic Schmidt decompositions of the compact operators $T_\mu,T_K,$ and $I_K^*$, respectively. \cite{piestch}

Note that since each $u_i\in U:=L^2_\mu(\Omega)$, it is not defined for all $x\in \Omega$, but only for $ \mu-$a.e.   $x\in \Omega$. On the other hand, $v_i\in V\subset C(\Omega)$ is defined for each $x\in \Omega$.  It is always possible to extend each $u_i$ to a continuous  function $\tilde{u}_i(x):=(T_K u_i)(x)/{\lambda_i}$ for all $x\in \Omega$. 
That is, the function $\tilde{u}_i$ is a continuous representative of the equivalence class $u_i$.  
 This is  the Nystrom extension \cite{rosascoint}, and 
 it is known that $v_i= \sqrt{{\lambda}_i} \tilde{u}_i$. 
 In the following we suppress the extension notation $\tilde{(\cdot)}$, but it must be kept in mind when expressing  $v_i$ in terms of  $u_i$.

We will use several probabilistic error bounds later in this paper that are readily cast in terms of spaces of Hilbert-Schmidt operators, a type of operator of the Schatten class.  The Schatten class of operators $S^p(U)$ on $U$  of order $1\leq p<\infty$ is the Banach space 
$$
S^p(U):=\left \{ T:U \rightarrow U \ | \ \text{ $T$ is compact and } \|T\|_{S^p} <\infty\right \}
$$  
with the norm given in the above definition by 
$$
\|T\|_{S^p}:=\left ( \sum_{i\in \mathbb{N}} \sigma_i^p(T)\right )^{1/p}.
$$
Following convention, we define  $S^\infty(U):=\left\{T\in \mathcal{L}(U)\ |\  \text{$T$ is compact} \right \}$. We then have 
$$
\|T\|_{S^{p}} < \infty \quad \Longrightarrow \quad 
\|T\|_{S^{p+1}}<\infty,
$$
for $1\leq p \leq \infty$, and therefore 
$$
S^1(U)\subset S^2(U)\cdots \subset S^\infty(U) \subset \mathcal{L}(U).
$$ 
The Schatten class $S^p(V)$ is defined similarly. The Hilbert-Schmidt operators are obtained by choosing  $p=2$, while trace class operators correspond to $p=1$. 

Because $V\subseteq U$ is an RKHS, more can be said about the relationship of the series expansions in Equations \ref{eq:ev_1} and \ref{eq:ev_2} by exploiting properties of the Hilbert-Schmidt operators $S^2(V)$.   We know that $g(x)=(K_x,f)_V$ for all $x\in \Omega$ and $f\in V$. In this case it is possible to represent the operator $T_\mu$ in terms of the Bochner integral $T_\mu:= \int K_x\otimes K_x \mu(dx)$ with  the tensor product $(K_x\otimes K_x) g=K_x (K_x,g)_V$. \cite{rosasco} 
For any $g\in V$ we have 
\begin{align*}
\|g\|^2_{U}&= \int g^2(x)\mu(dx) 
= \int (K_x,g)_V^2 \mu(dx)=\int (K_x\otimes K_x g, g)_V \mu(dx) \\
&= \left ( \int K_x \otimes K_x \mu(dx) g,g \right)_V = (T_\mu g,g)_V.
\end{align*}
 This sequence of steps can be used to show that $(g,v_j)_V=(g,u_j)_{U}/\sqrt{\lambda_j}$, from which we conclude 
$$
T_\mu g = \sum_{i\in\mathbb{N}} \lambda_i^{1/2} (g,u_i)_{U}v_i.
$$
We thereby can directly compare the norms  in terms of their action on the basis $\left \{u_i \right \}_{i\in\mathbb{N}}$, 
\begin{align*}
\|T_K f\|_{U}^2&:\sum_{i\in \mathbb{N}} \lambda_i^2 |(f,u_i)_{U}|^2  , \\
\|T_\mu g\|_{V}^2&:\sum_{i\in \mathbb{N}} \lambda_i |(g,u_i)_{U}|^2 ,
\end{align*}
for $f\in U$ and $g\in V$. 
If $T_K$ is infinite dimensional,  $\lambda_j^2\leq \lambda_j$ for all $j$ large enough. It  is evident that $D(T_\mu)\subseteq D(T_K)$. This means that the generalized Fourier coefficients of functions $f$ in the domain $T_\mu$ decay faster than those in the domain of $T_K$. This idea can be formulated systematically by introducing the spectral approximation spaces, discussed next.

\section{Spectral Approximation Spaces $A^{r,2}_\lambda(U)$} 
\label{sec:spectral_spaces}
This section introduces {\em spectral approximation spaces} that are defined in terms of a fixed compact,  self-adjoint  operator $T:U\rightarrow U$. Specifically, the eigenvalues $\lambda_i:=\lambda_i(T)$  and $U$-orthonormalized eigenfunctions $\left \{u_i\right \}_{i\in \mathbb{N}}$ are used to construct $A^{r,2}_\lambda(U)$. Typically, we choose $T:=T_K$ as described in the last section, although other choices are also possible.  
If  we happen to have a RKHS $V\subset U$ that satisfies the assumptions of the last section, we find that $U$ and $V$ are two particular spaces in a scale of spectral approximation spaces $A^{r,2}(\Omega)$.

For a  self-adjoint and  compact operator $T$,  with  non-zero eigenvalues and associated  eigenfunctions $\left \{ \right (\lambda_i,u_i)\}_{i\in \mathbb{N}},$  we define  spectral  approximation spaces $A^{r,2}_\lambda(U)$ for $r\geq 0$ via the formula
\begin{equation}
A^{r,2}_\lambda(U):= A^{r,2}_{\lambda(T)}(U):=
\left \{ f \in U \ \left | \ |f|^2_{A_\lambda^{r,2}(U)} :=\sum_{i\in \mathbb{N}}  (\lambda_i^{-r/2} |(u_i,f)|_U)^2 < \infty \right .
\right \}. \label{eq:spectral_space}
\end{equation}
 Note that $A^{0,2}_\lambda(U)\equiv U$ in this definition. 
Intuitively, these spaces have a simple interpretation: a function $f\in A^{r,2}_\lambda(U)$ provided that  the generalized Fourier coefficients $(u_i,f)_H$ decay at a rate that is controlled by the speed that the inverse $\lambda_i^{-r/2}$ of the eigenvalues grow. The next theorem summarizes some standard properties of the spectral spaces. 
\begin{theorem}
\label{th:spec_sp_approx}
The spectral approximation spaces $A^{r,2}_\lambda(U)$ are nested, 
$$
A^{s,2}_\lambda(U)\subset A^{r,2}_\lambda(U) 
$$
for all $s>r.$
Let $\Pi_n$ be the $U$-orthogonal projection onto the finite dimensional space of approximants 
$$
A_n:=\text{span}
\left \{ u_i\ | \  i \in \mathbb{N}, i\leq n-1
\right \}. 
$$
If $f\in A^{r,2}_\lambda(U)$, we have the error estimate 
\begin{align*}
\|(I-\Pi_n)f\|_U & \lesssim 
\lambda_n^{r/2} \|f\|_{A^{r,2}_\lambda(U)}.
\end{align*}
\end{theorem}
\begin{proof}
Nestedness follows since 
\begin{align*}
    |f|^2_{A^{r,2}_\lambda(U)} &=
    \sum_{i\in \mathbb{N}}\lambda_i^{-r}|(f,u_i)_U|^2 
    \leq 
    \sum_{i\in\mathbb{N}}\lambda_i^{s-r}\lambda_i^{-s}|(f,u_i)_U|^2\\
    &\lesssim  \sum_{i\in \mathbb{N}} \lambda_i^{-s}|(f,u_i)_U|^2=
    |f|^2_{A^{s,2}_\lambda(U)},
\end{align*}
provided that $s>r$. The error in approximation induced by the $U-$orthonormal projection $\Pi_n$  is shown similarly.
\begin{align*}
    \|(I-\Pi_n)f\|^2_U &= 
    \sum_{i\geq n}|(f,u_i)_U|^2
    = \sum_{i\geq n} \lambda_i^r \lambda_i^{-r}|(f,u_i)_U|^2 \\
    &\leq \lambda_n^r \sum_{i\geq n} \lambda_i^{-r} |(f,u_i)_U|^2 \leq \lambda_n^r \|f\|^2_{A^{r,2}_\lambda(U)}.
\end{align*}
\end{proof}

\noindent It is worth noting that the proof of the  error bound above can be easily modified to derive 
$$
\|(I-\Pi_n)f\|_{A^{r,2}_\lambda(U)} \leq \lambda_n^{(s-r)/2} \|f\|_{A^{s,2}_\lambda(U)}
$$
whenever $s>r>0$ and $f\in A^{s,2}_\lambda(U)$.  The bound in the theorem can be understood as the limiting case of the above when $r=0$. 

 We next  see how the  approximation spaces $A_\lambda^{r,2}(U)$ and $A_\lambda^{r,2}(V)$ are related when $V$ is a RKHS, and describe some simple mapping properties of the operators $T_\mu$ and $T_K$ when we choose $T=T_K$ in the definition of $A^{r,2}_\lambda(U)$.  We assume that the general setup discussed in Section \ref{sec:rkhs} holds. 
\begin{theorem}
\label{th:smoothing}
 If $V\subseteq U$ is a RKHS and the imbedding $i_K:V\rightarrow U$ is compact and continuous, it follows that 
 \noindent \begin{itemize}
 \item[1)]
   ${A}^{r+1,2}_\lambda(U)\approx {A}^{r,2}_\lambda(V)$ for $r>0$, and 
   \item[2)]
 the operators $T_K$ and $T_\mu$ are smoothing in the sense that \begin{align*}
T_K&: A^{r,2}_\lambda(U) \rightarrow A^{r+2,2}_\lambda(U),\\ 
T_\mu&:A^{r,2}_\lambda(V) \rightarrow A_\lambda^{r+2,2}(V).
\end{align*}
\end{itemize}
\end{theorem}
\begin{proof}
The proof of (1) follows directly from the calculation 
\begin{align*}
|f|^2_{A^{r,2}_\lambda(V)}
&= \sum_{i\in \mathbb{N}} \lambda_i^{-r}|(f,v_i)_V|^2, \\
&= \sum_{i\in \mathbb{N}} \lambda_i^{-r}|\lambda_i^{-1/2}(f,u_i)_U|^2, \\ 
&= \sum_{i\in \mathbb{N}} \lambda_i^{-(r+1)}|(f,u_i)_U|^2 = |f|^2_{A_\lambda^{r+1,2}(U)}.
\end{align*}
Conclusion (2) in the above theorem holds  because 
\begin{align*}
    |T_Kf|^2_{A^{r+2,2}_\lambda(U)}&= \sum_{i\in \mathbb{N}} \lambda_{i}^{-(r+2)}|(T_Kf,u_i)_U|^2, \\
    &= \sum_{i\in \mathbb{N}} \lambda_{i}^{-(r+2)}\left |\left (\sum_{m\in \mathbb{N}}
    \lambda_m (f,u_m)_U u_m,u_i\right )_U\right |^2, \\
    &= \sum_{i\in \mathbb{N}} \lambda_i^{-r}|(f,u_i)_U|^2 =  |f|^2_{A^{r,2}_\lambda(U)}.
\end{align*}
The result for $T_\mu$ in (3)  follows similarly using eigenfunction expansions in $V$ in terms of  $\left \{v_i \right \}_{i\in \mathbb{N}}.$
\end{proof}

The mapping properties described above for $T_K$ or $T_\mu$ can also be understood in terms of the operators $\sqrt{T_K}$ and $\sqrt{T_\mu}$. We have 
\begin{align*}
    \sqrt{T_K}&:A^{r,2}_\lambda(U) \rightarrow A^{r+1,2}_\lambda(U) \\
     \sqrt{T_\mu}&:A^{r,2}_\lambda(V) \rightarrow A^{r+1,2}_\lambda(V) \\
\end{align*}
This means that we can interpret the square root operators as the (increasing) shift operator on the scale of spaces $A^{r,2}_\lambda(U)$ and $A^{r,2}_\lambda(V)$. \cite{dahmenmultiscale}
 
\subsection{The Compact, Self-Adjoint  Operators in $\mathbb{A}^{r,2}_\lambda(U)$}
\label{sec:SA_family}
We define a family of admissible operators $\mathbb{A}^{r,2}_\lambda(U)$ that are convenient for studying rates of convergence of approximations in terms of the    spectral spaces   ${A}^{r,2}_\lambda(U)$. 
We define for $r>0$ the   family of self-adjoint, compact  operators   $\mathbb{A}^{r,2}_\lambda(U)$ via  
\begin{align}
    \mathbb{A}^{r,2}_\lambda(U):=
    \left \{\mathcal{P}=\sum_{i\in \mathbb{N}} p_i u_i\otimes u_i \in S^\infty(U) 
     \ \biggl | \ 
     |\mathcal{P}|^2_{A_\lambda^{r,2}(U)}:=
    \sum_{i\in \mathbb{N}} \lambda_i^{-r}|p_i|^2 < \infty
    \right \}
    \label{eq:A_fam_1}
\end{align}
where the  formula for $\mathcal{P}$ above is a Schatten expansion of $\mathcal{P}$ in   terms of the eigenvector basis $\left \{ u_i\right \}_{i\in\mathbb{N}}$ of the operator $T$ used to define $A^{r,2}_\lambda(U)$, and $\left \{ p_k\right \}_{i\in\mathbb{N}}\subset \mathbb{R}$. Admittedly, the family of operators $\mathbb{A}^{r,2}_\lambda(U)$ contains operators that are highly structured. Each  operator $\mathcal{P}\in \mathbb{A}^{r,2}_\lambda(U)$ is self-adjoint. We emphasize, however, that the operator $\mathcal{P}\in \mathbb{A}^{r,2}_\lambda(U)$ is not  diagonal with respect to some arbitrary  orthonormal basis of $U$; it is diagonalized in terms of the basis $\{u_j\}_{j\in \mathbb{N}}$ generated from $T$.  Suppose that $\{a_j\}_{j\in \mathbb{N}}$ is another orthonormal basis for $U$.  Since, as we show below in Theorem \ref{th:spec_embedding} that $\mathbb{A}^{r,2}_\lambda(U)\subset S^2(U)$, the Hilbert-Schmidt operator $\mathcal{P}\in \mathbb{A}^{r,2}_\lambda(U)$ is guaranteed to have  the representation 
$$
\mathcal{P}:=\sum_{i,j\in \mathbb{N}} p_{ij}a_i\otimes a_j
$$
with $p_{ij}=(\mathcal{P}a_i,a_j)_U$ and  $p_{ij}=p_{ji}$.  Of course, the value of the norm in $S^2(U)$ does not depend on the choice of orthonormal basis, so we must have 
$$
\|\mathcal{P}\|_{S^2(U)}^2:=
\sum_{i\in\mathbb{N}} (\mathcal{P}u_i,u_i)^2_U=\sum_{i\in \mathbb{N}} p_i^2 =\sum_{i,j\in \mathbb{N}} p_{ij}^2 = \sum_{i,j\in \mathbb{N}} (\mathcal{P}a_i,a_j)^2_U.
$$

We have introduced  this definition so that proofs of the rates of convergence of the operators $\mathcal{P}$ and $\mathcal{U}$ are particularly simple and illustrative. We will see that this definition can be generalized easily to certain classes of  operators defined in Section \ref{sec:NSA_family} that can contain operators that are not self-adjoint. In fact, essentially all of the error bounds derived in  
 Theorem \ref{th:spectral_approx} for the family defined in Equation \ref{eq:A_fam_1} hold for the more general class of admissible operators introduced in Section \ref{sec:NSA_family} that contains non-self-adjoint operators too.

The family of operators $\mathbb{A}^{r,2}_\lambda(U)$ can  again be understood intuitively like the definition of the spectral spaces $A^{r,2}_\lambda(U)$. A feasible Perron-Frobenius operator $\mathcal{P}\in \mathbb{A}^{r,2}_\lambda(U)$ has a Schatten class representation whose coefficients decay at a rate that is inversely proportional to the rate at which the eigenvalues $\lambda_i(T)$ converge to zero for some fixed compact, self-adjoint operator $T$. In this sense, the fixed operator $T$, by virtue of its eigenstructure, defines rates of convergence in $\mathbb{A}^{r,2}_\lambda(U)$.  When $V\subseteq U$ is a RKHS, the operators $T_K$  or $T_\mu$  that are  induced by a  symmetric kernel $K:\Omega\times \Omega \rightarrow \Omega$ are a  natural choice for the definition of $\mathbb{A}^{r,2}_\lambda(U)$ or $\mathbb{A}^{r,2}_\lambda(V)$, respectively. However, the definition above need not be restricted to this case.  We summarize a few of the easy properties of the operators in  $\mathbb{A}^{r,2}_\lambda(U)$. 
\begin{theorem}
\label{th:spec_embedding}
For each $r>0$ we have 
$$
\mathbb{A}^{r,2}_\lambda(U)\subset S^2(U).
$$
The family of operators $\mathbb{A}^{r,2}_\lambda(U)$ are nested, 
$$
\mathbb{A}^{s,2}_\lambda(U) \subset \mathbb{A}^{r,2}_\lambda(U),
$$
whenever $s>r>0$.
\end{theorem}
\begin{proof}
Each $\mathbb{A}^{r,2}_\lambda(U)\subset S^2(U)$ since 
$$
\|\mathcal{P}\|^2_{S^2(U)}:=\sum_{i\in \mathbb{N}} \sigma_i^2(P)= \sum_{i\in\mathbb{N}} p_i^2 \lesssim \sum_{i\in \mathbb{N}} (p_i \lambda_i^{-r/2})^2 = \|\mathcal{P}\|^2_{\mathbb{A}^{r,2}_\lambda(U)}.
$$
The second assertion follows from 
\begin{align*}
    |\mathcal{P}|^2_{A_\lambda^{r,2}(U)}&= \sum_{i\in \mathbb{N}} \lambda_i^{-r}|p_i|^2 = \sum_{i\in \mathbb{N}} \lambda_i^{s-r}\lambda_i^{-s}|p_i|^2,\\
    &\leq \sum_{i\in \mathbb{N}} \lambda_i^{-s}|p_i|^2 = 
    |\mathcal{P}|^2_{A_\lambda^{s,2}(U)},
\end{align*}
as long as $s>r>0$. 
\end{proof}
Note carefully that the larger the approximation index $r>0$, the smaller the space $A^{r,2}_\lambda(U)$. A similar inclusion holds for the operators in $\mathbb{A}^{r,2}_\lambda(U)$. The norm inequality above implies an  imbedding of the scale of operators $\mathbb{A}^{r,2}_\lambda(U)$ that resembles that for the spectral spaces $A^{r,2}_\lambda(U)$ in the sense that  
$$
\cdots \subset \mathbb{A}^{r+1,2}_\lambda(U)\subset  \mathbb{A}^{r,2}_\lambda(U) \subset 
\mathbb{A}^{r-1,2}_\lambda(U) \subset 
\cdots \subset S^2(U)\subset \cdots  \subset S^\infty(U) \subset \mathcal{L}(U).
$$


%
%
Having defined the spaces $A^{r,2}_\lambda(U)$ and the family of admissible operators  $\mathbb{A}^{r,2}_\lambda(U)$, 
we begin with a rather straightforward result. Although it is nearly self-evident, it is often a building block  for more complex error bounds derived later.
Specifically, we derive an  approximation rate that holds for the family of operators $\mathbb{A}_{\lambda}^{r,2}(U)$ and the approximation spaces $A^{r,2}_\lambda(U)$.
%
%
\begin{theorem}
\label{th:spectral_approx}
Suppose that $r>0$ and $\mathcal{P}$ has the monotonic Schmidt decomposition $\mathcal{P}:=\sum_{i\in \mathbb{N}} p_iu_i\otimes u_i\in {S}^\infty$ 
with respect to the $U$-orthonormal basis of eigenfunctions $\left \{u_i \right \}_{i\in \mathbb{N}}$ of the compact, self-adjoint  operator $T$. Define the associated approximation space $A^{r,2}_\lambda(U)$ in terms of the eigenstructure of $T$, and denote by  $\mathcal{P}_n$  the approximation obtained when the Schatten class representation is truncated to   
$
    \mathcal{P}_n:=\sum_{1\leq i\leq n-1} p_i u_i \otimes u_i.
$
If $\mathcal{P}f\in A_{\lambda}^{r,2}(U)$,  we have the error bound  
\begin{equation}
\|
(\mathcal{P}-\mathcal{P}_n)f\|_{U} \lesssim  \lambda_{n}^{r/2} |\mathcal{P}f|_{A^{r,2}_\lambda(U)}.
\label{eq:spec_rate2}
\end{equation}
This bound holds in particular for the  two important cases when 1) $\mathcal{P}\in \mathcal{L}(U)$ and $f\in A^{r,2}_\lambda(U)$ or when 2)  $\mathcal{P}\in \mathbb{A}^{r,2}_\lambda(U)$ and $f\in U$. 
Suppose that the eigevalues are quasigeometric in that there are two constants $c_1,c_2$ with 
$$
1<c_1 \leq \frac{\lambda_{n_{j-1}}}
{\lambda_{n_j}} \leq c_2
$$
for all $j\in \mathbb{N}$ with $\{n_j\}_{j\in \mathbb{N}_0}$ a quasigeometric sequence of integers.
In this case for  $s>r>0$  we have 
\begin{equation}
\left | \mathcal{P}-\mathcal{P}_{n_j}\right |_{\mathbb{A}^{r,2}_\lambda(U)}\leq \lambda^{(s-r)/2}_{n_j}|\mathcal{P}|_{\mathbb{A}^{s,2}_\lambda(U)}. 
\label{eq:spec_rate1}
\end{equation}
\end{theorem}
\begin{proof}
First, we know we have   $f=\sum_{i\in \mathbb{N}} (f,u_i)_U u_i$ since
$A^{r,2}_\lambda(U)\subset U$. 
When $\mathcal{P}f\in A^{r,2}_\lambda(U)$, we  compute the error 
\begin{align*}
\left \|(\mathcal{P}-\mathcal{P}_n \right)f \|_{U}^2 &=
\left \| 
\left (\sum_{i\ge n}p_i u_i\otimes u_i \right )\left ( \sum_{k\in \mathbb{N}}(f,u_k)_U u_k \right ) \right \|_U^2,\\
&= \sum_{i\geq n}p_i^2|(f,u_i)_U|^2 
    \leq \sum_{i\geq n} p_i^2 \lambda_i^{r}\lambda_i^{-r} |(f,u_i)_U|^2,\\
    &\leq \lambda_{n}^{r}
    \sum_{i\in \mathbb{N}} \lambda_i^{-r}|p_i(f,u_i)_U|^2 = \lambda_n^r|\mathcal{P}f|^2_{A^{r,2}(U)}.
\end{align*}
Thus Equation \ref{eq:spec_rate2} holds.
When $\mathcal{P}\in \mathcal{L}(U)$ and $f\in A^{r,2}_\lambda(U)$, we have 
\begin{align*}
\|\mathcal{P}f\|^2_{A^{r,2}_\lambda(U)}&= \sum_{i\in \mathbb{N}} \lambda_i^{-r} |p_i(f,u_i)_U|^2 = \sum_{i\in \mathbb{N}} p_i^2\lambda_i^{-r}|(f,u_i)|_U^2\\
&\leq  \|\mathcal{P}\|_{\mathcal{L}(U)}^2 \|f\|_{A^{r,2}_\lambda(U)}^2<\infty,
\end{align*}
which shows that the bound above holds for case (1). 
Now, at the other extreme, if we only know that $f\in U$, but $\mathcal{P}\in \mathbb{A}^{r,2}_\lambda(U)$, we see that 
\begin{align*}
    \|\mathcal{P}f\|^2_{A^{r,2}_\lambda(U)} &
    =\sum_{i\in \mathbb{N}}  (\lambda_i^{-r}p_i^2) |(f,u_i)|^2_U \leq |\mathcal{P}|^2_{\mathbb{A}^{r,2}_\lambda(U)} \|f\|_U^2<\infty, 
\end{align*}
which gives the same rate in case (2). 
We now show that Equation \ref{eq:spec_rate1} is true.  We prove the result in the theorem for $n_j=2^j$ since this case resembles many of the rates derived later in the paper. See \cite{piestch1981approxspaces,piestch1982tensorproducsequenfunctoperat} for the details associated with a general quasigeometric sequence.   The proof for a general quasigeometric sequence follows similarly. We can write
\begin{align*}
    \sum_{i\geq n} \lambda_i^{-r}|p_i|^2 &=
    \sum_{j\ge n}
    \sum_{k=2^j}^{2^{j+1}-1} \lambda_k^{-r}|p_k|^2 
    \leq \sum_{j\geq n} \lambda_{2^{(j+1)}}^{-r}
    \sum_{k=2^j}^{2^{j+1}-1} |p_k|^2 \\
    &\leq \sum_{j\ge n} \lambda_j^s\lambda_{2^{(j+1)}}^{-(r+s)}
    \sum_{k=2^j}^{2^{j+1}-1} |p_k|^2
    \lesssim \lambda_n^s \sum_{j\in \mathbb{N}_0} \lambda_{2^{j+1}}^{-(r+s)}
    \sum_{k=2^j}^{{2^{j+1}-1}} |p_k|^2 \\
    &\lesssim 
     \lambda_n^s \sum_{j\in \mathbb{N}_0} \lambda_{2^{j}}^{-(r+s)}
    \sum_{k=2^j}^{{2^{j+1}-1}} |p_k|^2
    \lesssim 
    \lambda_n^s \sum_{j\in \mathbb{N}_0} 
    \sum_{k=2^j}^{{2^{j+1}-1}}
    \lambda_{k}^{-(r+s)}
    |p_k|^2
    \\
    &=\lambda_n^s |\mathcal{P}|^2_{\mathbb{A}^{
    r+s}_\lambda(U)}
\end{align*}
The error bound now follows by defining $\hat{s}:=r+s$ and rewriting the above as 
$$
|\mathcal{P}-\mathcal{P}_j|^2_{\mathbb{A}^{r,2}_\lambda(U)} 
\leq \lambda_n^{\hat{s}-r}|\mathcal{P}|^2_{\mathbb{A}^{\hat{s}}_\lambda(U)}.
$$

\end{proof}

The next example describes an overall process by which the preceding analysis is applied. Initially, the  operator $T:U\rightarrow U$ is selected, and its eigenvalues and $U-$orthonormal eigenfunctions are used to define the approximation spaces $A^{r,2}_\lambda(U)$. Then, the approximation rates of operators $\mathcal{P}_j$ are  studied,  for example, when  $\mathcal{P}f\in \mathbb{A}_{\lambda}^{r,2}(U)$  or $\mathcal{P}\in \mathbb{A}^{r,2}_\lambda(V)$.  This case considers a linear dynamical system for purposes of illustration, but as is clear from several examples that follow, the same general process is applicable to nonlinear systems. 

\medskip 

\begin{framed}
\begin{example}[Discrete Approximation of the Heat Equation]
\label{ex:heat_eq}
\noindent 
\subsection*{\underline{Defining  $T$, $A^{r,2}_\lambda(U)$, and $\mathbb{A}^{r,2}_\lambda(U)$}}
Let $\mathbb{T}^1\subset\mathbb{R}^2$ be the unit circle, $L^2(\mathbb{T}^1)$ be the periodic square integrable functions over $\mathbb{T}^1$, and let $\tilde{T}$ be the second order differential operator   $\tilde{T}(\cdot):=-d^2(\cdot)/dx^2$. It is straightforward to check that the eigenvalue problem that seeks a nontrivial solution of 
$$
\tilde{T}f=\tilde{\lambda} f
$$
subject to the periodic boundary conditions \begin{align*}
f(0)&=f(2\pi), \\
\frac{df}{dx}(0)&=\frac{df}{dx}(2\pi),
\end{align*}
generates the orthonormal eigenpairs 
$$
\left \{\left (\tilde{\lambda}_m,\psi_m\right ) \right \}_{m\in\mathbb{Z}}=\left \{\left (m^2,\frac{1}{\sqrt{2\pi}}e^{\hat{j} m x}\right ) \right \}_{m\in\mathbb{Z}}.
$$ 
Orthonormality is defined with respect to the inner product 
$ (f,g)_{L^2(\Omega)}:=\int_\Omega f(\xi)\overline{g(\xi)}d\xi$ on the $L^2(\mathbb{T}^1)$ space of complex functions.
\begin{framed}
A quick check. We know that $\frac{d \psi_m}{dx}= \frac{\hat{j}m}{\sqrt{2\pi}} e^{\hat{j} m x}$. We then have 
$$
(\tilde{L}\psi_m)(x)=-\frac{d^2\psi_m}{dx^2}(x)=-\left ( \frac{-m^2}{\sqrt{2\pi}} e^{\hat{j}mx} \right )=m^2 \psi_m(x)=\tilde{\lambda}_m \psi_m(x).
$$
Also, 
\begin{align*}
\psi_m(0)&=\frac{1}{\sqrt{2\pi}}e^0= \frac{1}{\sqrt{2\pi}}e^{\hat{j}2\pi m} =\psi_m(2\pi), \\
\psi_m'(0)&=\frac{\hat{j}m}{{\sqrt{2\pi}}}e^0 = \frac{\hat{j}m}{{\sqrt{2\pi}}}e^{\hat{j}2\pi m}=\psi'_m(2\pi).
\end{align*}
\end{framed}

\noindent 
The operator $\tilde{T}$ is a differential operator that generates a Sturm-Liouville system. The differential operator $\tilde{T}$ can be used define an   associated inverse operator ${T}:=(\tilde{T} |_{N(\tilde{L})^\perp})^{-1}$ that is in fact a linear, self-adjoint,  compact integral operator on  $L^2(\mathbb{T}^1)$. \cite{naylorsell} 
Any function $f\in L^2(\mathbb{T}^1)$ consequently  has the Fourier series   representation
\begin{align}
f=\sum_{k\in \mathbb{Z}} (f,\psi_k)_{L^2(\mathbb{T}^1)} \psi_k. 
\label{eq:complex_fourier}
\end{align}
Although our theory in Section \ref{sec:rkhs} above studies real-valued functions, only a slight reindexing is needed to modify the definitions to make sense for the complex functions $\psi_m$. 
The spectral approximation space generated by this complex orthonormal basis is given by 
{\small 
$$
A^{r,2}_{\lambda}(L^2(\mathbb{T}^1)):=
\left \{ 
f\in L^2(\mathbb{T}^1)  \biggl |  |f|^2_{A^{r,2}_\lambda(L^2(\mathbb{T}^1))}
:= \sum_{m\in \mathbb{Z}-\{0\}} \lambda_m^{-r} |(f,\psi_m)_{L^2(\mathbb{T}^1)}|^2 \leq \infty
\right \}.
$$
}
In this case we have 
$$
|f|^2_{A^{r,2}_\lambda(L^2(\mathbb{T}^1))}= \sum_{m\in \mathbb{Z}-\{0\}} \lambda_m^{-r} |(f,\psi_k)_{L^2(\mathbb{T}^1)}|^2 \approx  \sum_{m\in \mathbb{Z}-\{0\}} m^{2r} 
|(f,\psi_k)_{L^2(\mathbb{T}^1)}|^2
$$
It is known that the rightmost series above is in fact equivalent to the seminorm on the Sobolev space $W^{r,2}(\mathbb{T}^1)$. \cite{devorenonlinear} We also see that the RKHS space $V=A^{1,2}_\lambda(U)\approx W^{1,2}(\mathbb{T}^1)$. 
This choice of the operator $T$, as a compact self-adjoint integral operator, is consistent with the assumptions of Section \ref{sec:rkhs} when $d=1$. The Sobolev Embedding Theorem states that if $m>d/2$, then $W^{m,2}(\mathbb{T}^1) \hookrightarrow C(\mathbb{T})$, and we have   
$$
V\approx W^{1,2}(\mathbb{T}^1)\subset 
C(\mathbb{T}^1) \subset L^2(\mathbb{T}) = U.
$$
\cite{devore1993constapprox}

The eigenfunctions above are elements in the  $L^2(\mathbb{T}^1)$ space of complex functions. We can further study the kernels that induce the operator $L$ in terms of the complex eigenfunctions. However, for the form of the real-valued RKHS spaces presented in Section \ref{sec:rkhs},  it is more convenient to cast the analysis in terms of the  $L^2(\Omega)$ space of real functions. 
The   eigenvectors $u_{k,i}$ of  ${T}$, viewed as an operator on the  $L^2(\Omega)$ space of real functions,  are given by 
\begin{align*}
    u_{k,i}(x):= \left \{ 
    \begin{array}{ccc}
         \frac{1}{\sqrt{2\pi}}&  &i=1,k=0, \\
        \frac{\cos kx}{\sqrt{\pi}} &  &i=1,\ k\geq 1, \\
        \frac{\sin kx}{\sqrt{\pi}} &  &i=2,\ k\geq 1,
    \end{array}
    \right .
\end{align*}
and the corresponding eigenvalues are 
\begin{align*}
    \lambda_{k,i}:= \left \{ 
    \begin{array}{ccc}
         0&  &i=1,k=0,  \\
        1/k^2 &  &i=1,\ k\geq 1, \\
       1/k^2 &  &i=2,\ k\geq 1. 
    \end{array}
    \right .
\end{align*}
Note that $u_{0,2}$ is not defined in this numbering  convention. It is easy to check that the functions $u_{k,i}$ are orthonormal with respect to the real inner product $(f,g)_{L^2(\mathbb{T}^1)}=\int_\Omega f(\xi)g(\xi)d\xi$. 
For any $f$ in the real $L^2(\Omega)$ space, we have
$$
f=\sum_{k\geq 1} \sum_{i\subseteq 1,2)} (f,u_{k,i})_{L^2(\mathbb{T}^1)} u_{k,i},
$$
which yields the same result as in the complex expansion in Equation  \ref{eq:complex_fourier} when the function $f$ is real-valued. We define the real Hilbert  space $U:=L^2(\mathbb{T}^1)$, and in the notation of Sections \ref{sec:rkhs} and \ref{sec:spectral_spaces},  we have 
$$
T_Kf:=Tf=\sum_{k\geq 1} \sum_{i=1,2}\lambda_{k,i} (u_{k,i},f)_{U} u_{k,i},
$$
{\small 
$$
A^{r,2}_{\lambda}(L^2(\mathbb{T}^1)):=
\left \{ 
f\in L^2(\mathbb{T}^1)  \biggl |  |f|^2_{A^{r,2}_\lambda(L^2(\mathbb{T}^1))}
:= \sum_{k\geq 1}\sum_{i=1,2} \lambda_{k,i}^{-r} |(f,\psi_{k,i})_{L^2(\mathbb{T}^1)}|^2 < \infty
\right \},
$$
}
and
$$
\mathbb{A}^{r,2}_\lambda(L^2(\mathbb{T}^1)) = 
\left \{
\mathcal{P}=\sum_{k\geq 1,i=1,2}p_{k,i}u_{k,i}\otimes u_{k,i} \ \biggl | \ \sum_{k\geq 1,i=1,2} (p_{k,i}\lambda_{k,i}^{-r/2})^2 < \infty 
\right \}.
$$
Note that the kernel of the operator $T$ is not included in the definition of $\mathbb{A}_\lambda^{r,2}(L^2(\mathbb{T}^1))$ above. 
\subsection*{\underline{Analysis of  the Operator $\mathcal{U}$, $\mathcal{U}_j$} }
We next illustrate how the approximation spaces  can be used to estimate the Koopman or Perron-Frobenius operators for an example of an  evolution equation.  Consider the model for  the time evolution of the temperature $\tau$ in a heat conduction problem over a ring where the thermal conductivity is normalized to one. The governing equation for the temperature $\tau$ takes the form  
\begin{align*}
\frac{d\tau}{dt}(t,x) & = \frac{\partial^2 \tau}{\partial x^2}(t,x) \quad \quad (t,x)\in [0,\infty)\times [0,2\pi],
\end{align*}
subject to the boundary conditions 
\begin{align*}
    \tau(t,0)&=\tau(t,2\pi), \\
    \frac{\partial \tau}{\partial x}(t,0)&=
     \frac{\partial \tau}{\partial x}(t,2\pi),
\end{align*}
and to the initial condition 
$$
\tau(0,x)=\tau_0(x) \quad \quad x\in [0,2\pi].
$$
It can be shown that the solution (modulo constant functions) of this evolution equation is given by 
$$
\tau(t):=\tau(t,\cdot):=\sum_{k\geq 1,i=1,2} e^{-k^2 t} (\tau_0,u_{k,i})_{L^2(\mathbb{T}^1)} u_{k,i}(\cdot),
$$
or $\tau(t)=S(t)\tau_0$ with $S(t)$ a linear  $C^0$-semigroup of operators. By sampling,  this  continuous flow induces a discrete flow $\mathcal{P}:L^2(\Omega) \rightarrow L^2(\Omega)$ with 
$$
\tau_{i+1}:=\tau(t_{i+1})=S(t_{i+1}-t_i)\tau(t_i)=\mathcal{P}\tau_i, 
$$
with $\left \{\tau_i\right \}_{i\in \mathbb{N}_0}\subset L^2(\Omega)$.
The discrete evolution law is induced by a kernel $p:=p_h(\cdot,\cdot):[0,2\pi] \times [0,2\pi] \rightarrow \mathbb{R}$ that depends parametrically on the step size $h>0$, with 
$$
\tau_{i+1}(x):=
(\mathcal{P}\tau_i)(x):=\int_{\Omega} p(x,y) \tau_i(y)dy 
$$
where  
$$
p(x,y):=\sum_{k\geq 1, i=1,2 }e^{-hk^2 } u_{k,i}(x) u_{k,i}(y).
$$
In other words with $p_{k,i}:=e^{-hk^2}$ we have 
$$
\mathcal{P}:=\sum_{k\geq 1, i=1,2}p_{k,i} u_{k,i} \otimes u_{k,i}, 
$$
which is the form of a kernel in $\mathbb{A}^{r,2}(U)$. 
It turns out that the operator  $\mathcal{P}$   above is very smooth. 
Now consider the sum
\begin{align*}
    \sum_{k\geq 1} \sum_{i=1,2} \lambda_{k,i}^{-r}p_{k,i}^2 & =  \sum_{k\geq 1} \sum_{i=1,2} k^{2r} e^{-2hk^2}.
\end{align*}
The function $x^{2r}e^{-2hx^2}$ is monotonically decreasing for all $x>\sqrt{r/h}$ and approaches zero as $x \rightarrow  \infty$.
This means that 
$$
\sum_{k\geq k_0} \sum_{i=1,2} k^{2r}e^{-2hk^2}
\leq \int_{0}^\infty x^{2r}e^{-2hk^2}dx
$$
for some any integer $k_0\geq \sqrt{r/h}$. 
Since the integral on the right is finite for every positive $r$ and $h$, we  conclude that $\mathcal{P} \in \mathbb{A}^{r,2}_\lambda(U)$ for every $r>0$, and the results of Theorem \ref{th:spectral_approx} hold. 
\end{example}
\end{framed}

\begin{framed}
 \begin{example}[Spaces Adapted to a Specific $\mathcal{P}$]
 \label{ex:p_and_alambda}
 In this example we explore a bit more how the results of Theorem 
 \ref{th:spectral_approx} can be further refined. From the definition in Equation  \ref{eq:A_fam_1} we known that the the operator 
 $$
 \mathcal{P}:=\sum_{i\in \mathbb{N}} p_i v_i \otimes v_i.
 $$
Suppose  the coefficients  $\{p_i\}_{i\in \mathbb{N}}$ are nonincreasing and can only accumulate at zero. It is therefore possible to define the approximation space 
 $$
 A_p^{r,2}(V):= \left \{ f\in V \ \biggl | \  |f|^2_{A^{r,2}_p}(V):=\sum_{i\in \mathbb{N}} |p_i|^{-r} |(f,v_i)_V|^2\right \}.
 $$
 It is immediate that 
 \begin{align*}
     |f|^2_{A^{r,2}_\lambda}&:= \sum_{i\in \mathbb{N}} |\lambda_i|^{-r}|(f,v_i)_V|^2 = 
     \sum_{i\in \mathbb{N}} |\lambda_i|^{-r}|p_i|^2 |p_i|^{-2}|(f,v_i)_V|^2 \lesssim |f|^2_{A_p^{2,2}(V)}, 
 \end{align*}
 and therefore 
 $$
 A_p^{2,2}(V)\subseteq A_\lambda^{r,2}(V).
 $$
 The space $A_p^{r,2}(V)$ can be endowed with the inner product 
 $$
 (f,g)_{A^{r,2}_p(V)}:= \sum_{i\in \mathbb{N}} \left ( p_i^{-r/2}f_i, p_i^{-r/2}g_i\right )_V
 $$
 where $f:=\sum f_i v_i$ and $g=\sum g_i v_i$. It is clear from this definition that $f\in A_p^{r,2}(V)$ if and only if 
 $$
 \|f\|_{A^{r,2}_p(V)}^2 := {(f,f)_{A^{r,2}_p(V)}} < \infty.
 $$
 It is also immediate that the family of functions  $\{v_{p,i}^r\}_{i\in \mathbb{N}}:=\{p_i^{r/2}v_i\}_{i\in \mathbb{N}}$ is an orthonormal basis for $A^{r,2}_p(V)$. 
 
 Now suppose that we have a  stronger condition that relates the kernels $K$ and $p$: we assume the equivalence of the  sequences
 $$
\{ p_i \}_{i\in \mathbb{N}}  \approx \{
\lambda_i^{r/2}
\}_{i\in \mathbb{N}}.
$$
That is, there exist two constants $c_1,c_2>0$
such that
$$
c_1 \lambda^{r/2}_i \leq p_i \leq c_2  \lambda^{r/2}_i
$$
for all $i\in \mathbb{N}$.
With this assumption we see that $A^{2,2}_p(V)\approx A^{r,2}_\lambda(V)$. 
 Since the kernel $K$ is positive, 
 the kernel of $\mathcal{P}$  is positive, symmetric, and continuous, and therefore it is possible to define a RKHS $V_p \subset L^2_\mu(\Omega)$ in terms of $p(x,y)$. Then we know that 
 $$
 \mathcal{P}:=\int_\Omega p_x\otimes p_x \mu(dx)=\sum_{i\in \mathbb{N}} p_i v_i \otimes v_i.
 $$
 \end{example}
 From Theorem 11 of \cite{rkhs} we can compute the reproducing kernel on $A^{r,2}_p(V)\approx A^{r,2}_\lambda(V)$ from the kernel $K$ on $V$. 
\end{framed}

\subsection{Compact and Non-Self-Adjoint Operators in  $\mathbb{A}^{r,2}_{\lambda}(U)$}
\label{sec:NSA_family}
In the last section we presented a definition of feasible operators $\mathbb{A}^{r,2}_{\lambda}(U)$ that facilitated the study of convergence rates relative to the spectral approximation spaces $A^{r,2}_\lambda(U)$. In a typical application, Theorem \ref{th:spec_sp_approx} is used by first fixing some compact self-adjoint operator $T$ whose eigenvalues and eigenfunctions $\left \{ \lambda_i, u_i \}_{i\in \mathbb{N}}\right \}$ are used in the construction of the spectral space $A^{r,2}_\lambda(U)$. The set of admissible operators consists of compact, self-adjoint operators that have Schatten decompositions that decay faster that the inverse of the eigenvalues used to build $A_\lambda^{r,2}(U)$. If we construct a RKHS $V\subset U$, then  the integral operator $T_K$ given by 
\begin{equation}
(T_Kf)(x):=\int_\Omega K(x,y)f(y)\mu(dy)
\label{eq:int_eq_ex}
\end{equation}
is one logical choice for the operator $T$ that defines $A^{r,2}_\lambda(U)$. Here, the kernel $K(x,y)$ symmetric.  It is a requirement that ensures the symmetry of the inner product on the real RKHS $V\subset U$. 

However, it is important to observe  that there is no compelling requirement to restrict attention to Perron-Frobenius, or Koopman operators,   that are induced by a symmetric kernel. It is a simple matter to define deterministic or stochastic discrete evolutions for which the associated operators are non-self-adjoint.  We now discuss how the setting of the last section can be extended to allow for some operators that 
are not self-adjoint.

The theory of integral operators and their mapping properties has been studied extensively over the years, and many approaches exist to study them. Comprehensive accounts can be found in \cite{piestch,piestchhistory}. One choice that is  canonical and serves as an exemplar for other approaches is the case when it is assumed that $p(x,y)$ is a kernel in  $L^2_{\mu\times \mu}(\Omega\times \Omega)$ that induces an operator $\mathcal{P}:L^2_\mu(\Omega) \rightarrow L^2_\mu(\Omega)$. In fact, an operator $\mathcal{P}:L^2_\mu(\Omega)\rightarrow L^2_\mu(\Omega)$ has the form in Equation \ref{eq:int_eq_ex} if and only if it is a Hilbert-Schmidt operator. In this case, for any orthonormal basis $\left \{\psi_i \right \}_{i\in \mathbb{N}}$
of $L^2_\mu(\Omega)$, such an operator $\mathcal{P}$ is induced by the kernel 
$$
p(x,y)=\sum_{m\in \mathbb{N}} \sum_{n\in\mathbb{N}} (\mathcal{P}\psi_m,\psi_n)_U\psi_m(x) \psi_n(y),
$$
and we also know that 
$$
\|\mathcal{P}\|^2_{S^2(U)}=
\sum_{i\in\mathbb{N}} \sigma^2_i := \sum_{i\in \mathbb{N}} \|\mathcal{P}\psi_i\|_U^2 = \sum_{m\in \mathbb{N}} \sum_{n\in \mathbb{N}} |(\mathcal{P}\psi_m,\psi_n)|^2.
$$

The study of kernels of this type is facilitated by exploiting the equivalence of these norm expressions to the norms of certain infinite matrix operators that act on $\ell^2(\mathbb{N})$. We define the infinite matrices  
$[p_{m,n}]:=[(\mathcal{P}\psi_m,\psi_n)]$
    and 
$
[D_{\lambda_m}^{s}]:=\text{diag}(\lambda_m^s)
$
for $s\in \mathbb{R}$. The induced matrix  operator norm $\|p_{m,n} \|_{M}$  is defined in the usual way
$$
\|[p_{m,n}]\|_{M}:= \sup_{
\begin{array}{c}
\{z_n\}\in \ell^2(\mathbb{N})\\ \{z_n\}\not = \{0\}\end{array}
} \frac{\|[p_{m,n}]\{z_n\}\|_{\ell^2(\mathbb{N})}}{\|\{z_n\}\|_{\ell^2(\mathbb{N})}}.
$$
With this notation we have $\|f\|_U=\|\{(f,\psi_m)_U\}\|_{\ell^2(\mathbb{N})}$, and it is straightforward to show that 
$$
\|\mathcal{P}\|_{\mathcal{L}(U)}=\|[p_{m,n}]\|_M.
$$
It also follows that $f\in A^{r,2}_\lambda(U)$ if and only if 
$\|[D_{\lambda_n}^{-r/2}]\{f_n\}\|_{\ell^2(\mathbb{N})}<\infty$. We overload the definition of $D_{\lambda_\bullet}^{s}$ and also interpret it as an operator on functions by associating its representation 
$$
D^{s}_{\lambda_\bullet}f \quad  \sim  \quad [D^{s}_{\lambda_m}] \{ f_m\} 
$$
with $f_m:=(f,\psi_m)_U$  for  $m\in \mathbb{N}$.

We then can generalize the definition of the families of feasible spectral operators and write 
{\scriptsize 
\begin{align}
    \mathbb{A}^{r,2}_\lambda(U):=
    \left \{ \left . 
    P=\sum_{m,n\in\mathbb{N}}p_{m,n}u_{m} \otimes u_{n} \subset S^\infty(U)\  \right | 
    \ 
    |P|^2_{\mathbb{A}^{r,2}_\lambda(U)}:=\sum_{m\in \mathbb{N}} 
    \left ( \sum_{n\in \mathbb{N}}  p_{m,n}\lambda_n^{-r/2}\right )^2
    <\infty
    \right \}.
    \label{eq:A_def_NSA}
\end{align}
}

\noindent The expression above reduces to that in the Definition in Equation \ref{eq:A_fam_1} when the operator $\mathcal{P}$ is diagonalized by the eigenfunctions of the operator $T$ used to define $\mathbb{A}_\lambda^{r,2}(U)$. 
The seminorm $|\mathcal{P}|_{\mathbb{A}^{r,2}_\lambda(U)}$ is easily shown to have the alternative representations below:

\begin{align*}
    |\mathcal{P}|^2_{\mathbb{A}^{r,2}_\lambda(U)} &:= \text{trace}\left ((D^{-r/2}_{\lambda_\bullet} \mathcal{P})^*D^{-r/2}_{\lambda_\bullet} \mathcal{P} \right )=\left (
    D^{-r/2}_{\lambda_\bullet} \mathcal{P}, D^{-r/2}_{\lambda_\bullet} \mathcal{P}
    \right)_{S^2(U)}, \\
    &:= \sum_{\ell\in \mathbb{N}}
    \left (D^{-r/2}_{\lambda_\bullet} \mathcal{P}\psi_\ell, D^{-r/2}_{\lambda_\bullet } \mathcal{P}\psi_\ell \right)_U,\\
    &=\text{trace}\left ( \left [[p_{\ell,i}][D_{\lambda_i}^{-r/2}]\right]^T\left [[p_{\ell,i}][D_{\lambda_i}^{-r/2}]\right] \right )\\
    &=  \text{trace} \left(  [p_{i,\ell}][p_{\ell,i}] \left [D_{\lambda_i}^{-r/2}\right ]^2 \right)\\
    &=  \text{trace} \left ( \left [[p_{\ell,i}][D_{\lambda_i}^{-r/2}]\right ]^2 \right)\\
    &= \sum_{\ell \in \mathbb{N}} 
    \left ( \sum_{i}  p_{\ell,i}\lambda_i^{-r/2}\right )^2.
\end{align*}
Other equivalent forms of the last line follow from the identities $\text{trace}(A^*B)=\text{trace}(B^*A)$ for any operators $A,B\in S^2(U)$.

\noindent We have the following approximation bounds in terms of this updated   definition of $\mathbb{A}^{r,2}_\lambda(U)$. 
\begin{theorem}
\label{th:spectral_approx_ns}
The results of Theorem \ref{th:spectral_approx} hold with the definition of $\mathbb{A}^{r,2}_\lambda(U)$ in Equation \ref{eq:A_def_NSA}. 
\end{theorem}
\begin{proof}
If we define $f_n:=(f,u_n)_U$ for each $n\in \mathbb{N}$, we  can bound the error by writing 
\begin{align*}
    \|(\mathcal{P}-\mathcal{P}_j)f\|_U^2&= \sum_{m\geq j} \left (\sum_{n\geq j} p_{m,n}f_n \right )^2 =\sum_{m\geq j} \left (\sum_{n\geq j} \lambda_n^{r/2}\lambda_n^{-r/2}p_{m,n}f_n \right )^2, \\
    &\leq \lambda_j^{r} \sum_{m\geq j} \left (\sum_{n\geq j} p_{m,n} \lambda_n^{-r/2}f_n \right )^2 \leq \lambda_j^r\left \| [p_{m,n}][D_{\lambda_n}^{-r/2}]\{f_n\} \right \|_{\ell^2(\mathbb{N})}^2, \\
    &\leq \lambda^r_j \left \| [p_{m,n}] \right \|_{M}^2 \|[D_{\lambda_n}^{-r/2}] \{f_n\}\|^2_{\ell^{2}(\mathbb{N})}\leq \lambda_j^r \|\mathcal{P}\|_{\mathcal{L}(U)}^2 |f|^2_{A^{r,2}_\lambda(U)},
\end{align*}
when $f\in A^{r,2}_\lambda(U)$. It is clear from  above that when $f\in U$ only, we get  
\begin{align*}
 \|(\mathcal{P}-\mathcal{P}_j)f\|_U^2 &\leq      \lambda^r_j
 \|[D_{\lambda_m}^{-r/2}][p_{m,n}]\{f_n\}\|^2_{\ell^2} \\
 &
 \leq      \lambda^r_j
 \|[D_{\lambda_m}^{-r/2}][p_{m,n}]\|^2_M\|\{f_n\}\|^2_{\ell^2} 
 \\
 &\leq \lambda_j^r |\mathcal{P}|_{\mathbb{A}^{r,2}_\lambda(U)}^2 |f|^2_{U}
\end{align*}
provided that 
 $\mathcal{P}\in \mathbb{A}^{r,2}_\lambda(U)$. We also see that 
 $$
 \|(\mathcal{P}-\mathcal{P}_j)f\|^2_{U}
 \leq \lambda_j^r|\mathcal{P}f|^2_{\mathbb{A}^{r,2}(U)}
 $$
 when $\mathcal{P}f\in A^{r,2}_\lambda(U)$.
\end{proof}

\section{The Approximation Spaces $A^{r,q}(U)$}
\label{sec:approx_spaces_Arq}

In this section we introduce the more general linear approximation spaces $A^{r,q}(U)$ for rates   $r>0$  and $1\leq q\leq \infty$.  
These  will be important to characterize the regularity of a wider variety of   Perron-Frobenius  or Koopman operators  and to determine the rates of convergence of the approximations. 

Spectral approximations feature prominently in many of the recent papers that construct approximations in Koopman theory.   There are several reasons for introducing the  spaces $A^{r,q}(U)$,  although they do not seem to have been used extensively, or at all, in the study of Koopman theory for  dynamical systems.
We have chosen to present spectral approximation spaces $A^{r,2}_\lambda(U)$ first as a means of building insight about   the more  abstract  spaces $A^{r,q}(U)$. We will see that in some important cases the spectral approximation spaces are special cases  of the  spaces $A^{r,q}(U)$. 
Essentially,  the equivalences result from  making assumptions on the rates of convergence of the eigenvalues $\left \{\lambda_k\right \}_{k\in \mathbb{N}}$ to zero in the operator expansions in Equations \ref{eq:ev_1}, \ref{eq:ev_2}, and \ref{eq:ev_3}.   So approximation spaces provide a reasonable framework to cast approximations in terms of eigenvalues and eigenfunctions of $T_K$ or $T_\mu$, provided their eigenvalues converge   to zero at a compatible rate. However, they are much more general and their construction is not given in terms of the eigenstructure of some fixed self-adjoint, compact operator $T$.

One other important reason for the introduction  of the scale $A^{r,q}(U)$ is simply pragmatics:  approximation spaces can be   defined in terms of  a specific choice of   bases used for realizing approximations, although it is also possible to define them in a coordinate free manner just using projections. \cite{dahmenmultiscale}   While the axiomatic foundations of the theory can be abstract, applications of the  theory can therefore be a direct source of  realizable algorithms. As we have summarized in the introduction, and  discuss more fully below, this approach can be used to deduce rates of convergence for a wide variety  of bases.  These bases include  trigonometric polynomials, algebraic polynomials, piecewise algebraic  polynomials, splines, and wavelets.  This means that in the event that the calculation of eigenfunctions is infeasible, it is still possible to construct estimates and study  rates of convergence of approximations.

There is  yet  another technical reason for pursuing an  approximation framework that is not spectral in nature. 
We have emphasized that the definition of the spectral spaces are essentially tied to a fixed operator $T$, and the approximation rates are stated  in terms of  eigenvalues of the kernel of a RKHS.
 In fact,   the spectral spaces $A_{\lambda}^{r,2}(U)$ defined in this way are  an example of what is called a ``native space'' associated with a RKHS. It is a rule of thumb that error rates are easily derived for functions in the native space generated by a particular kernel, but it is often not easy to describe how such estimates can be applied to other more common spaces.  Naturally, approximation theorists seek to understand how these approximation rates in the native space relate to approximation rates in general, ``standard'' spaces such as Lipschitz,  Sobolev,  or   Besov spaces. A number of such relationships  have been derived, see  \cite{wendland}. Researchers who study variants of  approximations from RKHS have also referred to this problem as ``escaping the native space.''  \cite{narcowich05}.   As a rule, error estimates are more valuable when they apply to a broad family of spaces, not just to ones that are specifically tied to a problem, or operator $T$ at hand. 

One other important feature of the approximation spaces $A^{r,q}(U)$ is that there is a rich and systematic theory  that establishes the equivalence approximation spaces to   interpolation  between more common spaces. We will use a few of the fruits of this analysis in our paper, but the interested reader should consult standard references  for the details. 
A discussion of the  theory for quite general  approximation spaces can be found in classical references such as \cite{piestch1981approxspaces}, \cite{piestch},  \cite{devore1993constapprox},\cite{sharpley}. 
 See Section \ref{sec:approx_spaces} for a very brief overview of the general theory.    We will introduce the spaces $A^{r,q}(U)$ here with $U=L^2_\mu(\Omega)$, but the definitions of the approximation spaces $A^{r,q}(V)$ are entirely analogous.

Let $\left \{A_j\right \}_{j\in \mathbb{N}_0}$   be a collection of approximant  subspaces of a Banach space $U$ that satisfy the properties summarized in the Appendix \ref{sec:approx_spaces}. These assumptions include the fact that the approximant spaces are nested and that their closed  linear span  is dense in $U$.    The approximation error $E_{n}(f,U)$ of $f\in U$ over $A_n$ is given by
$$
E_n(f,U):=\inf_{a\in A_n} \|f-a\|_U, 
$$
and the   approximation space $A^{r,q}(U)$  for $r>0$ is the Banach space
$$
A^{r,q}(U):=\left \{ f\in U \ \biggl | \ \|f\|_{A^{r,q}(U)} < \infty\right \}
$$
with 
\begin{align*}
    \|g\|_{A^{r,q}(U)}:=\left \{
    \begin{array}{ccc}
      \left (
        \sum_{n=1}^\infty 
           \left [
               n^rE_{n-1}(f,U)
           \right ]^q
         \frac{1}{n}  
         \right )^{1/q} & \quad & 1\leq q <\infty\\
       \sup_{n\geq 1}\left [ 
            n^{r} E_{n-1}(f,U) 
           \right ] & & q=\infty.
     \end{array}
    \right .
\end{align*}
There is another useful way to express the norm on  $A^{r,2}(U)$ when $U$ is a Hilbert space. 
Denote  by $\Pi_j:U \rightarrow A_j$ the $U$-orthogonal projection of $U$ onto $A_j$ for each $j\in \mathbb{N}_0$.
We let $Q_j:=\Pi_{j}-\Pi_{{j-1}}$ for $j\in \mathbb{N}_0$ and $\Pi_{-1}=0$ and set  
\begin{align*}
    |f|_{A^{r,q}(U)}:=\left \{
    \begin{array}{ccc}
      \left (
        \sum_{j=0}^\infty 
            [2^{rj} \|Q_jf\|_U]^q \right)^{1/q}
          & \quad & 1\leq q <\infty,\\
       \sup_{j\geq 0}\left [ 
            2^{jr} E_{2^j}(f,U) 
           \right ] & & q=\infty.
     \end{array}
    \right .
\end{align*}
In this definition $|f|_{A^{r,q}}(U)$ is only a seminorm: the norm is defined as  $\|g\|_{A^{r,q}(U)}:=\|g\|_U + |g|_{A^{r,q}(U)}$.
For $q=2$ and $U$ a Hilbert space,  the space  $A^{r,2}(U)$ is a Hilbert space with  the inner product given by 
$$
(f,g)_{A^{r,2}(U)}:=(f,g)_{U} +  \sum_{j\in \mathbb{N}}
2^{2jr}(Q_jf,Q_jg)_{U} = \sum_{j\in \mathbb{N}_0} 
2^{2jr}(Q_jf,Q_jg)_{U}
$$
The following theorem is analogous to Theorem \ref{th:spec_sp_approx} for the spectral approximation spaces. 
\begin{theorem}
\label{th:approx_simple_U}
The approximation spaces are nested,  
$$
{A}^{s,2}(U) \subset A^{r,2}(U)
$$
for $s>r>0$. Let $\Pi_j$ be the $U-$orthogonal projection onto the $A_j$ in the definition of $A^{r,2}(U):=A^{r,2}(U,\{A_j\}_{j\in \mathbb{N}_0})$.
If $f\in A^{r,2}(U)$, we have the error estimate
$$
\|(I-\Pi_j)f\|_{U} \leq 2^{-rj}|f|_{A^{r,2}(U)}.
$$
\end{theorem}
\begin{proof}
The proof of this theorem resembles the spirit of the proof in Theorem \ref{th:spec_sp_approx} when we identify  $\lambda_j^{1/2}\sim 2^{-j}$.
Nestedness as stated in this theorem is trivial to establish, and in fact there are other nestedness conditions for the general case as summarized in the full theory. \cite{devore1993constapprox}.
 Since $f\in U$,  the telescoping series 
$$
f=\sum_{j\in \mathbb{N}_0} Q_j f
$$
converges in $U$. Each  $Q_jf$ is perpendicular to all $Q_jf$ for $j\not=i$ by definition, and we have the orthogonal sum 
\begin{align*}
    \|f-\Pi_jf\|_U^2&:= \sum_{i\geq j} \|Q_if\|^2_U \leq \sum_{i\geq j}2^{2ri}2^{-2ri} \|Q_if\|^2_U\\
    &\leq 2^{-2rj} \sum_{i\geq j}2^{2ri} \|Q_if\|^2_U \leq 2^{-2rj}|f|_{A^{r,2}(U)}^2.
\end{align*}
\end{proof}

\noindent This bound, just like in our analysis of spectral approximation spaces in Theorem \ref{th:spec_sp_approx}, can be alternatively written as 
$$
\|(I-\Pi_j)f\|_{A^{r,2}(U)}\leq 2^{-(s-r)j}\|f\|_{A^{s,2}(U)}
$$
for $s>r>0$ when $f\in A^{s,2}(U)$.

\subsection{The Multiscale Structure, Wavelets, and  Multiwavelets}
\label{sec:wavelets}

Earlier we reviewed how $L^2_\mu(\Omega)$-orthonormal bases can be constructed from the eigenfunctions of linear, compact, self-adjoint operators. In that construction $U:=L^2_\mu(\Omega)$ is separable, and so is   the  reproducing kernel Hilbert space (RKHS) $V$ over a domain $\Omega\subseteq \mathbb{R}^d$. The bases $\{u_i\}_{i\in \mathbb{N}}$ and $\{v_i\}_{i\in \mathbb{N}}$ are natural choices for the spectral approximation spaces $A^{r,2}_\lambda(U)$ and $A^{r,2}_\lambda(V)$, respectively.  These bases are generally  supported globally on $\Omega$ and have no multiscale structure.

In this section we describe some particular bases that, unlike the bases of eigenfunctions, exhibit a multiscale structure. That is, the bases are defined by introducing a family of nested grids over which the basis functions are defined. 
Here, the bases are selected to be well-known examples of orthonormal wavelets or multiwavelets for the separable Hilbert space  $U$.  
We begin with a rather general overview  of the structure of multilevel decompositions induced by such multilevel bases in Section  \ref{sec:multiscale_structure}. We then discuss  examples  that are applicable to the spaces $U:=L^2(\Omega)$.  These bases enable  development of the  theory of  approximation spaces $A^{r,2}(L^2(\Omega))$ in a transparent fashion. Such spaces will be applicable  to Koopman theory if we know that the measure of interest has the form $\mu(dx):=m(x)dx$ for an integrable function that satisfies $c_1 \leq m(x) \leq c_2$ almost everywhere in $\Omega$ for two positive constants $c_1,c_2$.  We finally discuss approximation spaces defined in terms of warped wavelet and warped multiwavelet bases, which enable the treatment of a wider class of  measures $\mu$.

\subsection{An Overview of The Multiscale Structure}
\label{sec:multiscale_structure}

This section summarizes the structure and indexing of multiscale bases, especially those that arise from wavelets and multiwavelets. The notation is fairly standard in the study of multigrid methods  or multiresolution analyses.  Our discussion is necessarily brief, see \cite{strang, daubechies88, devorenonlinear,dahmenmultiscale} for more examples of this common notation. 

Suppose  that    we are given  an orthonormal wavelet  basis  $\{\psi_{j,\bm{k}}\ | \ j\in \mathbb{N}, \bm{k}\in \Gamma_j^\psi\} \subset U$ with   $\Gamma_{j}^\psi$  the family of admissible  indices for each fixed $j$. Roughly speaking, the  integer  $j$ denotes the mesh resolution level and $\bm{k}\in \Gamma_j^\psi$ ranges over all the functions that are defined on that mesh level.  We mostly only consider tensor product (multi)wavelets for $\Omega\subseteq \mathbb{R}^d$ when  $d>1$ to keep our discussion elementary. Example \ref{ex:haar_multi_tri} provides one easy example of a domain that is not a tensor product and is included to demonstrate the numbering and indices for such a multiwavelet basis.
Such domains, ones that are  tilings of a few master subdomains, are one class over which Haar multiwavelets are often straightforward to construct. This class could be important to approximations in Koopman theory since many examples study dynamics over partitions of some domain. 
Multiscale structures over more general domains can be quite complicated, see \cite{dahmenman1,dahmenman2} for examples. We take the  mesh on level $j$ to be finite union of some collection of dyadic cubes $\square_{j,\bm{m}}\subseteq \Omega \subseteq \mathbb{R}^d$ that have the form 
$$
\square_{j,\bm{m}}:=\left \{ x\in \mathbb{R}^d\ \biggl | \ 2^{-j}m\leq x_i \leq 2^{-j}(m+1), 1\leq i\leq d \right \}
$$
for $\bm{m}\in \mathbb{Z}^d$.  The  wavelet spaces 
consist of all the wavelet functions $\psi_{j,\bm{k}}$ that are defined on mesh level $j$,
\begin{equation}
\label{eq:wave_basis}
W_j:=\text{span}\left \{ \psi_{j,\bm{k}}\ | \  \bm{k} \in \Gamma_j^\psi  \right \}.
\end{equation}
As explained in more detail in the Examples \ref{ex:ON_wave}  and \ref{ex:ON_multi}, the indexing of the functions $\psi_{j,\bm{k}}$ in  the spaces $W_j$ uses an overloaded definition that accounts for a variety of possibilities. This notation inherently is designed to reflect as multiscale structure in the analysis that follows. 
In most of our  examples  these bases will initially be defined over $\mathbb{R}^d$, and subsequently the  set of basis functions is modified  so that that they are a basis for functions  over  a compact set $\Omega$. Adapting a given set of wavelets defined over $\mathbb{R}^d$ to a general set $\Omega$ is a delicate, difficult, and lengthy process in general. \cite{Dahmen1995MultiscaleDO,dahmenman1,dahmenman2} 
All of our  examples are carried out  when $\Omega$ is  $[0,1]^d$ or the $d-$dimensional torus $\mathbb{T}^d$. Modifications of a global basis on  $\mathbb{R}^d$  to life on the torus $\mathbb{T}^d$  are particularly  simple since it is only necessary to periodize a finite set of functions on each grid resolution level. Restriction of Haar bases to $[0,1]^d$ is trivial. The modification of the multiscale orthonormal bases of \cite{dgh,dgh18} to $[0,1]^d$ for Neumann or Dirichlet boundary conditions  is also relatively easy and discussed in these papers.  With these considerations in mind, we  suppress the bold font style in the remainder of this section  for entries $\bm{k}\in \Gamma^\psi_j$ and just use $k$ to denote  in the set of admissible functions on mesh level $j$.

 We define the  approximant spaces $A_j:=\cup_{0\leq m\leq j-1}W_m$ in terms of the wavelet spaces $W_j$, and they    satisfy $A_{j+1}=A_j \oplus W_j$ for each $j\in \mathbb{N}_0$. \cite{daubechies88,strang} The  spaces $A_j$ are  also often referred to  as the space of scaling functions on level $j$ in the literature on multiresolution analysis. 
 We denote by $\Pi_j$ and $Q_j$ the orthogonal projections on $A_j$ and $W_j$, respectively. Common wavelet bases  $\left \{ \psi_{j,k}\right \}_{j\in \mathbb{N},k\in \Gamma_j^\psi}$ are constructed so that  any $f\in U$ can be written in the  so-called {\em multiscale expansion} given by the   $U$-convergent  series 
\begin{equation}
\label{eq:Wspace}
f=\sum_{j\in \mathbb{N}_0} \sum_{k\in \Gamma^\psi_j} (f,\psi_{j,k})_U \psi_{j,k}. 
\end{equation}
The outer summation over $j$ runs over the grids each having uniform mesh parameter $2^{-j}$, and the inner summation spans all the possible functions on a particular grid.  
In many cases,  the summation above in $j$ is expressed over   all $0\leq j_0\leq j\in \mathbb{N}_0$, where $j_0$ denotes the resolution level of  the coarsest grid in the summation. 
Note that here  we have followed the usual convention,  discussed more carefully in Section \ref{sec:wavelets},  of subsuming the scaling functions in the coarsest  wavelets on level $j=0$ or $j=j_0$ in this multiscale representation. \cite{dahmenmultiscale,devore2006approxmethodsuperlearn}. 

We next discuss how multiscale bases of this type can be realized. 

\subsection{$L^2(\Omega)-$Orthonormal Wavelets and Multiwavelets}
\label{sec:on_wave_mwave}
Among the tens of thousands of wavelet and multiwavelet papers that have been published over the past three decades, we have elected to express the theory in our paper using only orthonormal wavelets and multiwavelets. As mentioned in the introduction, this assumption is hardly necessary, but it keeps the treatment intuitive. 
There  is a well-documented collection of $L^2(\Omega)-$orthonormal wavelet systems that can be used to realize the multilevel  setup in the last section  \cite{daubechies88,dgh,dgh18} when the  measure $\mu$ is just Lebesgue measure. Much has been said about orthonormal wavelets and multiwavelets in the literature for this case, as well as their generalizations that yield the  biorthogonal wavelet families. To review them specifically and in detail here would distract from our primary goal, the determination of convergence rates of approximations of operators in Koopman theory. 
For those who have not seen these constructions before, we give a  detailed summary of the Daubechies compactly supported orthonormal wavelets and some compactly supported orthonormal multiwavelets  in Examples  \ref{ex:ON_wave} and \ref{ex:ON_multi}, respectively, in the Appendix in Section \ref{sec:app_A}.   

Here we discuss  just  two concrete examples of the low approximation order orthonormal wavelets and multiwavelets. These are of course the Haar wavelets and a  generalization of them that results in  some simple multiwavelets on a triangular domain. We include these since they give clear definitions of the scaling functions $\phi_{j,k}$, wavelets $\psi_{j,k}$, as well as the indexing sets $\Gamma_j^\phi$ and  $\Gamma_j^\psi$ in specific  cases. The multiscale decomposition is easy to understand in these canonical examples. 

\medskip

\begin{framed}
\begin{example}[Haar Wavelets for $d=1$]
\label{ex:haar_basis}
A discussion of the Haar wavelet system appears in many references 
 on wavelets as a beginning example. 
 Despite the simplicity of this basis choice, it should be of interest to researchers that study Koopman theory since so many examples deal with approximations by piecewise constants over partitions of the domain. 
In one dimension the Haar scaling function $\phi:\mathbb{R}\rightarrow \mathbb{R}$ is a constant function, and the Haar wavelet $\psi:\mathbb{R}\rightarrow \mathbb{R}$  is a  piecewise constant function,  over the domain $\Omega:=[0,1)$. They are written as 
\begin{align*}
    \phi(x)&:=\left \{ \begin{array}{ccc}
    1 & \quad & x \in [0,1)\\
    0 &  & \text{otherwise}
    \end{array}
    \right . , \\
    \psi(x)&:= \phi(2x)-\phi(2x-1).
\end{align*}
We define a family of  nested dyadic grids on $\Omega$ formed from dyadic cubes $\square_{j,k}$ 
having sidelength $2^{-j}$ that are  given by 
$$
\square_{j,k}:=\left \{x \in \Omega \ | \  2^{-j}k \leq x < 2^{-j}(k+1)\right \}
$$
for $k=0,\ldots,2^{j}-1$. The grid that consists of a union of such cubes is said to have resolution level $j$. 
\begin{center}
	\includegraphics[width=.7\textwidth]{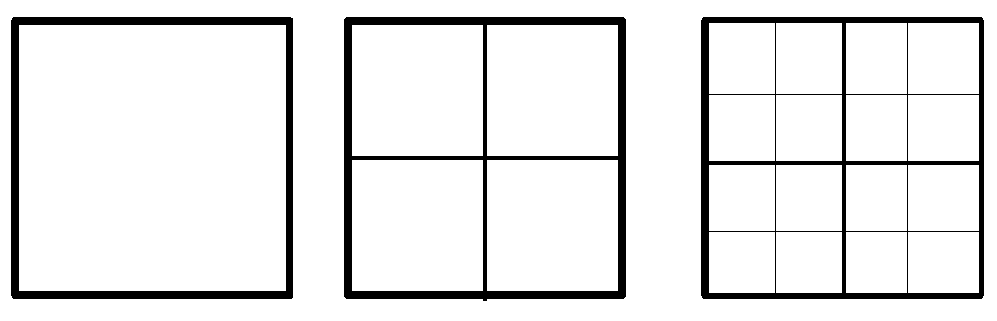}
	\captionof{figure}{Nested Grids and Dyadic Cubes $\square_{j,k} \in \mathbb{R}^2$}
	\label{fig:dyadic_cubes}
\end{center}

Based on this enumeration we define the index sets  $\Gamma_j^\phi:=\Gamma_j^\psi:=\{k: 0\leq k \leq 2^{j}-1\}$,  
and  the admissible scaling functions  $\phi_{j,k}:=2^{j/2}\phi(2^jx-k)$ for $k\in \Gamma_{j}^\phi$ and wavelets $\psi_{j,k}:=2^{j/2}\psi(2^jx-k)$  for $k\in \Gamma_{j}^\psi$.  It  can be shown that $\{\phi_{j,k}\}_{k\in \Gamma_j^\phi}$ are $L^2(\Omega)-$orthonormal for each fixed grid having resolution level $j$. Also, the family $\{\psi_{j,k}\}_{j\in \mathbb{N}_0,k\in \Gamma_j^\psi}$ are $L^2(\Omega)$-orthonormal. 
In a typical approximation problem in $\Omega\subset \mathbb{R}^1$, we set 
\begin{align*}
    A_j&:=\text{span}\{ \phi_{j,k} \ | \ k\in \Gamma_j^\phi \}=\text{span} \\
    &= \{ 1_{\square_{j,k}} \ | \ k\in \Gamma_{j}^\phi \}\\
    &=\text{span}\{
    \psi_{i,k} \ | \  i\leq j, \  k\in \Gamma_{i}^\psi \},
\end{align*}
so that $n_j:=\#(A_j)=2^j$. As noted  above, we can alternatively expand the family of wavelets to subsume the scaling functions in the coarsest scale. One way to do this is by setting 
$
\psi_{-1,0}:=\phi
$
and redefining the index set for the wavelets as 
$$
\Gamma^\psi_{j}:=\left \{ 
\begin{array}{ccc}
\{0, \ldots, 2^j-1\} & \quad & j\geq 0 \\
0 & & j=-1.
\end{array}
\right .
$$
Then we also can write 
$$
A_j:=\text{span}\left \{ 
\psi_{i,k} \ | \ i\leq j, \  k\in \Gamma_i^\psi \right \}.
$$
Of course we still have $n_j=2^j$. Any $f\in L^2(\Omega)$ has the multiscale decomposition
$$
f= \sum_{j\in \mathbb{N}_{-1}} \sum_{k\in \Gamma_{j}^\psi} (f,\psi_{j,k})_{L^2(\Omega)} \psi_{j,k},
$$
and its 
orthogonal projection $f_j\in A_j$ can be written 
$$
f_j:= \Pi_j f:=\sum_{k\in \Gamma_j^\phi} (f,\phi_{j,k})_{L^2(\Omega)} \phi_{j,k} = \sum_{i\leq j} \sum_{k\in \Gamma_i^\psi} (f,\psi_{i,k})_{L^2(\Omega)} \psi_{i,k}.
$$
This construction is based on the assumption that $d=1$, but the extension to $\Omega=[0,1)^d$ via tensor products of the wavelets is routine. An explicit discussion of the process for defining tensor products  is given in Examples \ref{ex:ON_wave} and \ref{ex:ON_multi} in the Appendix.  From Theorem \ref{th:approx_besov} in the Appendix, we know that the approximation space generated by the Haar wavelets $A^{r,\infty}(L^p(\Omega);\{A_j\}_{j\in \mathbb{N}_0})$ is equivalent to the generalized Lipschitz space $\text{Lip}^*(\alpha,L^p(\Omega))$ for all $0<\alpha <r-1+1/p$. 
This means that we have 
$$
\|(I-\Pi_j)f\|_{L^2(\Omega)} \lesssim 2^{-\alpha j} |f|_{\text{Lip}^*(\alpha,L^2(\Omega))}
$$
for all functions $f \in \text{Lip}^*(r,L^2(\Omega))$ over the range $0<\alpha<1/2$, for example. 
When $p=\infty$ and $f\in UC(\Omega)$, we also have 
$$
\|(I-\Pi_j)f\|_{L^\infty(\Omega)} \lesssim 2^{-\alpha j} |f|_{\text{Lip}(\alpha,UC(\Omega))}
$$
over the same range.

As a final observation,  if one wants to generate a partition of the domain $\Omega:=[0,1]^d$, the argument above proceeds the same but the definition of the last scaling function on each dyadic level in each coordinate direction is modified. The wavelets on each dyadic level are correspondingly modified and a multiscale decomposition over  $\Omega:=[0,1]^d$ is defined. 
\end{example}
\end{framed}

\begin{framed}
\begin{example}[Haar Multiwavelets over Triangles]
\label{ex:haar_multi_tri}
In this example we consider a slight generalization of the Haar wavelets to define a system of  multiwavelets over a different domain, one that is not a product of compact sets. 
It is an example of a domain that is a self-similar tiling of itself. This contruction can be carried out form many similar self-similar tilings, ones that seem to arise often in Koopman theory.  The point of this example is to illustrate the form of the numbering schemes and indices for similar multiwavelet systems.  The numbering here differs  only slightly from the preceding case, but illustrates the situation when there are two indices $(i,\bm{k})$, the multiscaling function or multiwavelet  number $i$ and the translation $\bm{k}$.

\includegraphics[width=.9\textwidth]{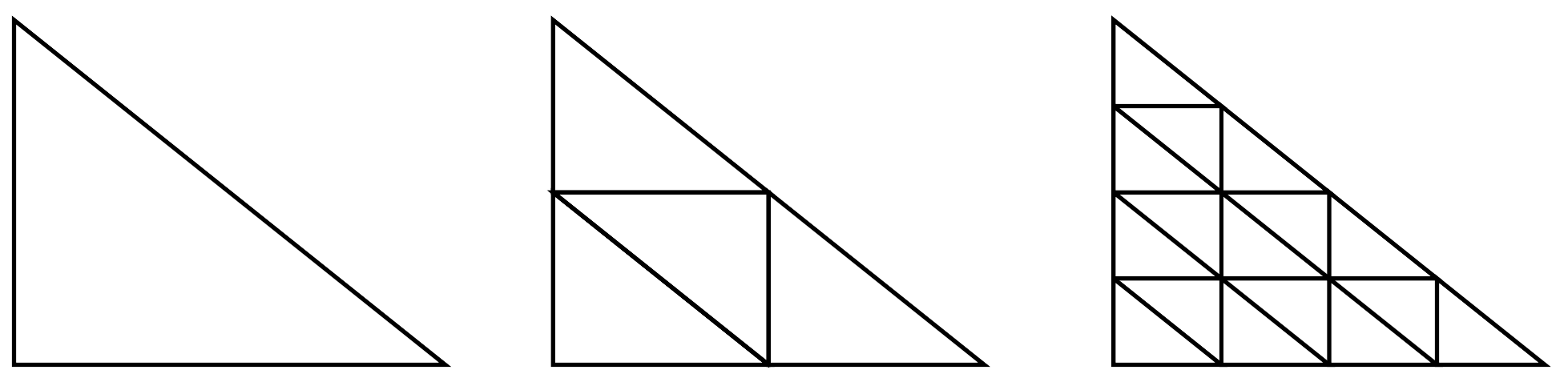}
\captionof{figure}{Triangular Domain and Nested Grids}
\label{fig:triangle_grid}

{\centering 
\begin{tabular}{ccc}
\includegraphics[width=.45\textwidth]{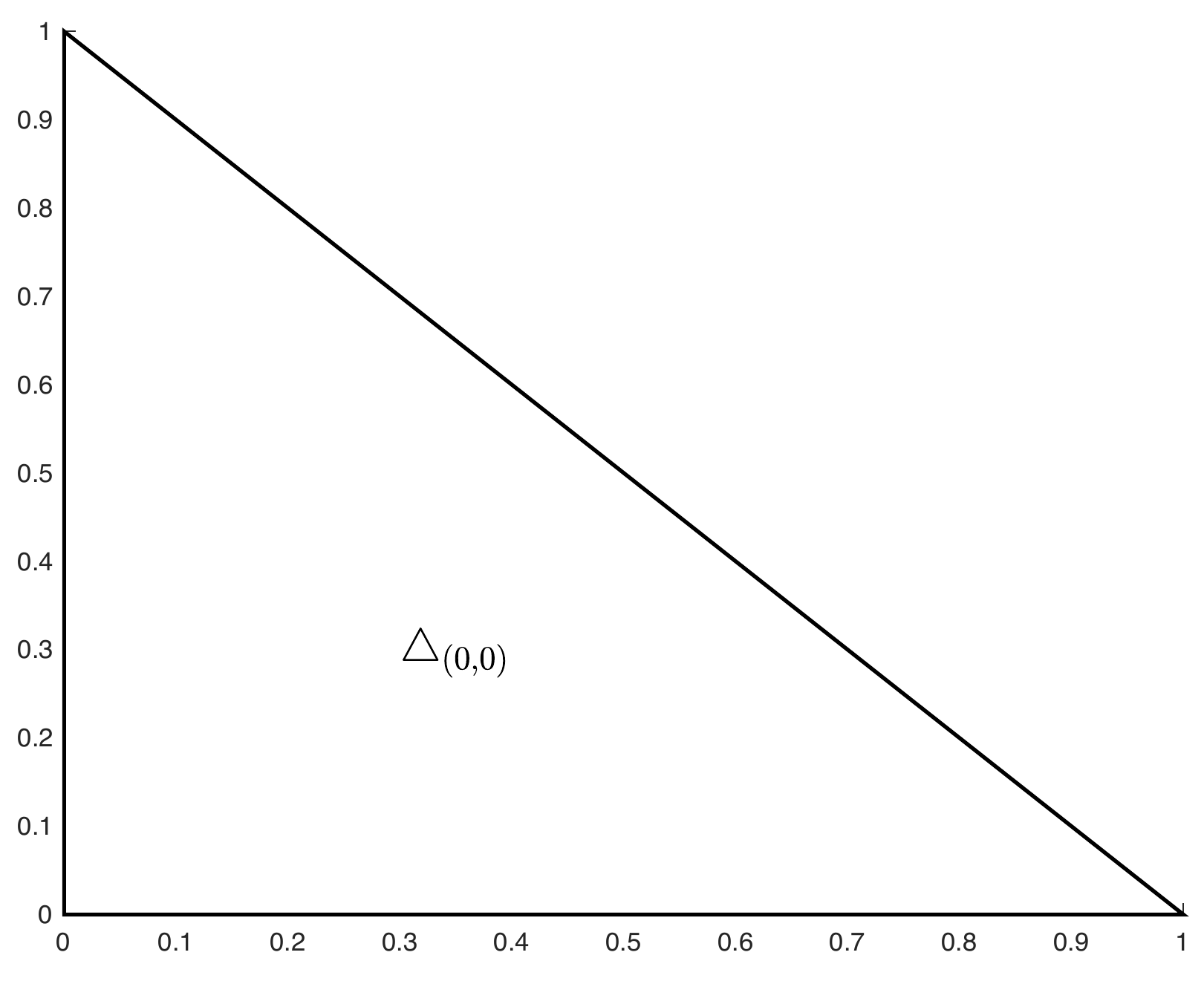}& 
\includegraphics[width=.45\textwidth]{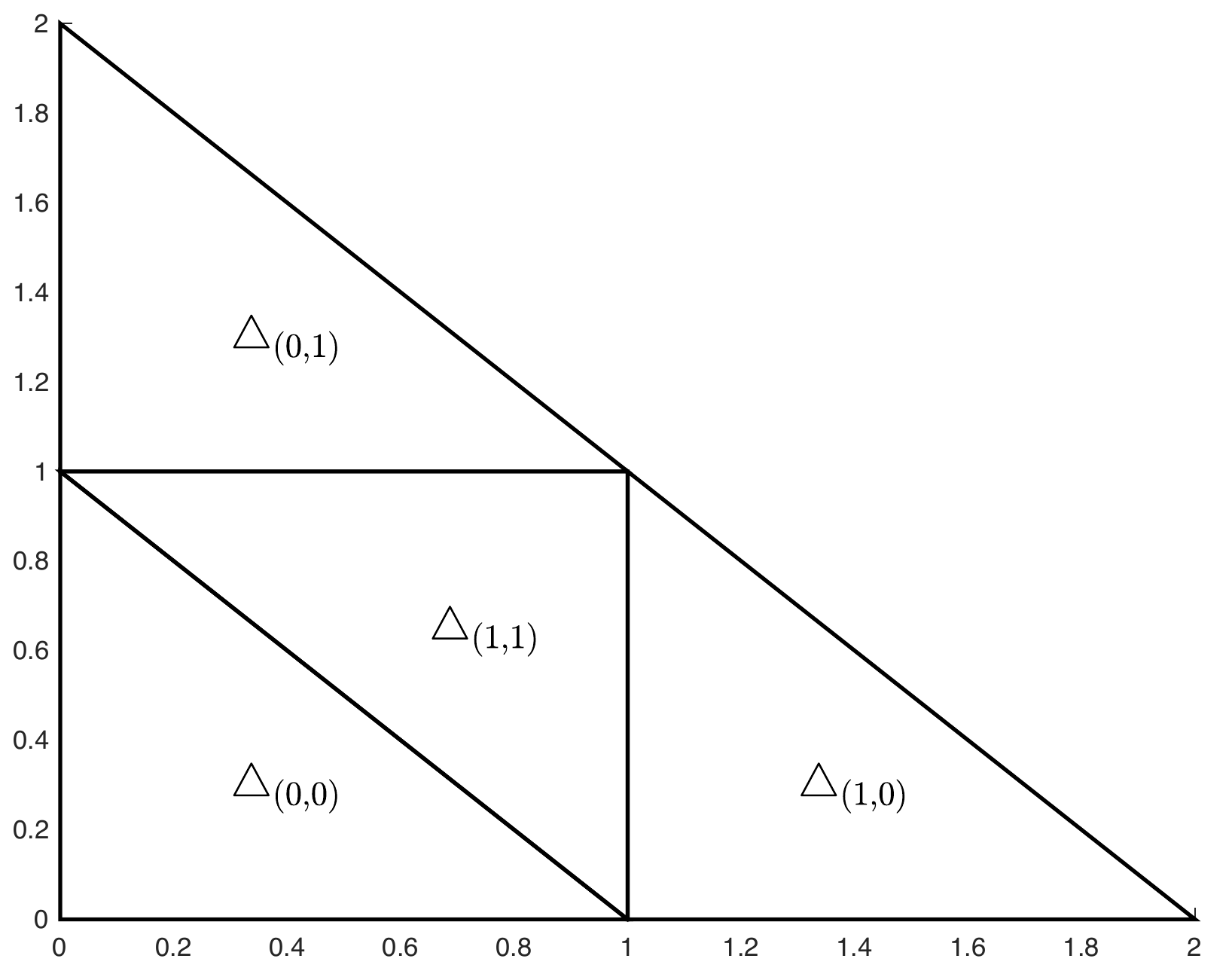}
\end{tabular}
\captionof{figure}{Labelling of Triangles on Grid Level $j=0$}
\label{fig:level_0}
}

We define the triangles $\triangle:=\triangle_{0,0}$ and $\triangle:=\triangle_{1,1}$ on grid level $j=0$  as 
\begin{align*}
\triangle:=\triangle_{0,0}&:=\{(x,y) \ | \ y\geq 0,\  x\geq 0,\  y+x< 1\}, \\ 
\triangle_{1,1}&:=\{(x,y) \ | \ y<1,\  x< 1,\ y+x\geq  1\},  
\end{align*}
and set the translations $\triangle_{1,0}$ and $\triangle_{0,1}$ of $\triangle_{0,0}$ as depicted in Figure \ref{fig:level_0}. The vertical and horizontal side  of each triangle $\triangle_{0,0},\triangle_{1,0},\triangle_{0,1}$  on grid level $j=0$ is $1$, while $\triangle_{1,1}$ is equal to $1$ over its hypotenuse.  By recursion, the  grid of level $j$ is obtained by the uniform subdivision illustrated in Figure \ref{fig:triangle_grid} where each cell is a scaled and translated copy of $\triangle_{0,0}$ or $\triangle_{1,1}$. A triangle $\triangle_{j,\bm{k}}$ on grid level $j$ is defined as 
$$
\triangle_{j,\bm{k}}:=\left \{ x\in \mathbb{R}^2 \ | \ 2^jx-\bm{k} \in \triangle_{0,0} \text{ or } \triangle_{1,1}\  \right \}
$$
for some  $\bm{k}\in \mathbb{Z}^2$. 
We define the index set $\Gamma_{j}^\triangle$ for grid level $j$ to be the indices of the triangles $\triangle_{j,\bm{k}}$ that meet $\triangle_{0,0}$. In view of the complimentrary boundary conditions of $\triangle_{0,0}$ and $\triangle_{1,1}$, the triangles $\{\triangle_{j,k}\}_{k\in \Gamma_{j}^\triangle}$ form a partition of $\triangle_{0,0}$. 
We set the multiscaling function 
$$
\bm{\phi}(x):=\begin{Bmatrix} \phi_1(x) \\ \phi_2(x)\\ \phi_3(x) \\ \phi_4(x) \end{Bmatrix}
:=\begin{Bmatrix}
1_{\triangle_{0,0}}(x) \\
1_{\triangle_{1,0}}(x) \\
1_{\triangle_{0,1}}(x) \\
1_{\triangle_{1,1}}(x) 
\end{Bmatrix}.
$$
The $L^2(\Omega)-$orthonormal scaling functions on grid level $j$ are given by 
$$
\phi_{j,(i,\bm{k})}(x):=\frac{1}{|\triangle_{j,\bm{k}}|^{1/2}} \phi_i(2^j x -\bm{k})
$$
for $i\in \{1,2,3,4\}$ and  some $k\in \mathbb{Z}^2$. For each triangle $\triangle_{j,\ell}\subset \triangle_{0,0}$ there is a unique function $\phi_{j,(i,k)}$ that is supported on $\triangle_{j,\ell}$. We define the collection of admissible indices $\Gamma_j^\phi$ for the multiscaling functions as 
$$
\Gamma_j^\phi:=\left \{ (i,\bm{k}) \ | \ 
i\in \{1,\ldots,4\}, \ (i,\bm{k})\text{ corresponds to } \triangle_{j,\bm{\ell}} \text{ for some } \ell\in \Gamma^\triangle_j \right \},
$$
so that the approximant space on level $j$ is then 
$$
A_j:=\text{span}\left \{ 
\phi_{j,(i,\bm{k})} \ | \ j\in \mathbb{N}_0, \ 
(i,\bm{k})\in \Gamma_j^\phi
\right \},
$$
We have $n_j:=\#(A_j)=2^{dj}$ with $d=2$. 
There are a number of ways to define wavelets for this system.  We  select the multiwavelet $\bm{\psi}:=\{\psi_1,\psi_2,\psi_3\}$ to be the  functions 
\begin{align*}
    \psi_1(x)&:=(\phi_1(x)-\phi_4(x))/\sqrt{2}, \\ 
    \psi_2(x)&:=(\phi_2(x)-\phi_3(x))/\sqrt{2},\\
    \psi_3(x)&:=[(\phi_1(x) + \phi_4(x))-(\phi_2(x)+\phi_3(x))]/2, 
\end{align*}
and then the scaled and translated  wavelets are 
$$
\psi_{j,(i,\bm{k})}:=\frac{1}{|\triangle_{j,\bm{k}}|^{1/2}} \psi_i(2^jx-\bm{k})
$$
for some $\bm{k}\in \mathbb{Z}^2.$ We define the index set of admissible wavelets on level $j$ to be 
$\Gamma_j^\psi$. 
The $L^2(\triangle)-$orthonormal projection of any $f\in L^2(\Omega)$ onto $A_j$ is given by the single scale expansion 
$$
\Pi_j f:= \sum_{(i,\bm{k})\in \Gamma_j^\phi} (f,\phi_{j,(i,\bm{k})})_{L^2(\Omega)} \phi_{j,(i,\bm{k})},
$$
and it has the multiscale representation
$$
f=
(f,\phi_{0,(1,0)})_{L^2(\Omega)} \phi_{0,(1,0)}
+ 
\sum_{j\in \mathbb{N}_0} \sum_{(i,\bm{k})\in \Gamma_{j}^\psi} (f,\psi_{j,(i,\bm{k})})_{L^2(\Omega)} \psi_{j,(i,\bm{k})}.
$$
As above, or as in Examples \ref{ex:ON_wave} and \ref{ex:ON_multi}, the above multiscale expression can be simplified somewhat by subsuming the scaling function into the collection of wavelets on the coarsest level. We set 
$ \psi_{-1,0}:=\phi_{0,(1,0)}$ and modify the index set $\Gamma_{j}^\psi$ accordingly to get 
$$
f=
\sum_{j\in \mathbb{N}_{-1}} \sum_{(i,\bm{k})\in \Gamma_{j}^\psi} (f,\psi_{j,(i,\bm{k})})_{L^2(\Omega)} \psi_{j,(i,\bm{k})}.
$$
The approximation of functions $f\in C(\Omega)\subset L^2(\triangle)$ can be constructed in terms of the family of operators $\tilde{\Pi}_j:C(\Omega) \rightarrow A_j$ with 
$$
(\tilde{\Pi})_jf(x):= \sum_{\ell \in \Gamma_j^\triangle} f(\xi_{j,\ell} ) 1_{\triangle_{j,\ell}}(x), 
$$
with $\{ \xi_{j,\ell}\}_{\ell \in \Gamma_j^\triangle}$ the centroids of the triangles $\triangle_{j,\ell}$.  Following essentially the same steps as in Example \ref{ex:two_cases}, we find that 
$$
\| (I-\tilde{\Pi}_j)f\|_{C(\Omega)} \lesssim 2^{-rj} |f|_{\text{Lip}(r,C(\Omega))}.
$$
This situation is a particular example of the more general bound 
$$
\|(I-\tilde{\Pi}_j)f\|_{L^p(\Omega)} \lesssim  2^{-rj}|f|_{\text{Lip}(r,L^p(\Omega))}
$$
in Equation 6.16 of \cite{devorenonlinear}. 

As in the last section, the analysis above can be carried out ove the compact set $\Omega=\overline{\triangle_{0,0}}$ by a simple modification of all scaling functions and multiwavelets that meet the hypotenuse of $\triangle_{0,0}$. 
\end{example}
\end{framed}

\subsubsection*{Using $L^2(\Omega)$ Wavelets for Approximations in $L^2_\mu(\Omega)$}
The last two examples show that there are a large family of multiscale orthonormal wavelets and multiwavelets that generate a basis for $L^2(\Omega):=L^2_\mu(\Omega)$ with $\mu$ simple Lebesgue measure. 
If  $\mu$ is not Lebesgue measure, warped wavelets  \cite{kerkyacharianwarped} based on modifications $L^2(\Omega)$ orthonormal wavelets can be shown to be $L^2_\mu(\Omega)$-orthonormal for a reasonable class of measures $\mu$.  A discussion of warped wavelets is presented in Section \ref{sec:warp_wave_multi}.
However, the approximation of functions in $L^p_\mu(\Omega)$ for certain types of measures can also be constructed directly with the bases on $L^2(\Omega)$.  Suppose $\Omega\subset \mathbb{R}^d$ is compact,  
$
\mu(dx):=m(x)dx,
\label{eq:meas_con_1}
$
and that there  are two  constants  $c_1,c_2$ such that 
$
0<c_1 \leq m(x) \leq c_2 <\infty 
\label{eq:meas_con_2}
$
for $x \text{ a.e. in }\Omega$. It is immediate that 
$$
\|f\|_{L^2(\Omega)} \approx \|f\|_{L^2_\mu(\Omega)},
$$
that is, the two norms are equivalent. 
The sets $L^2(\Omega)$ and $L^2_\mu(\Omega)$ contain the same collection of functions,  and the topology on the two spaces is the same. Any $f\in L^2_\mu(\Omega)$ can therefore be approximated in terms of the basis  $\left \{ \psi_{j,k}\ | \ j_0\leq j \in \mathbb{N}_0, k\in \Gamma_j^\psi \right \}$ for $L^2(\Omega)$, and the  approximation spaces $A^{r,q}(L^2(\Omega))$ can be used to approximate  functions in $L^2_\mu(\Omega)$.  In fact, we have 
$$
A^{r,2}(L^2(\Omega),\{A_j\}_{j\in \mathbb{N}_0})
\equiv
A^{r,2}(L^2_\mu(\Omega),\{A_j\}_{j\in \mathbb{N}_0})
$$
with $A_j:=\text{span}\{ \phi_{j,k} \ | \ k\in \Gamma_{j}^\phi \}$, and the norms on these spaces are equivalent. 
%
%
%

\subsection{Warped Wavelets and Multiwavelets}
\label{sec:warp_wave_multi}
In this section we discuss  the method of warped wavelets \cite{kerkyacharianwarped} and multiwavelets. This approach  will be useful for cases when the Koopman or Perron-Frobenius operators are defined over a domain of interest $\tilde{\Omega}$ that is related to a master domain  $\Omega$ by a change of variables. In some cases we can think of $\Omega$ as the original domain over which initial conditions are defined, and $\tilde{\Omega}$ as the image of the initial domain under a change of variables. While we are careful to distinguish the initial and image domains in the theory discussed here, the approach is of course applicable to the common examples when the change of variables maps the initial domain onto itself. In both scenarios the measures $\mu$ and $\tilde{\mu}$ are  defined over $\Omega$ and $\tilde{\Omega}$, respectively. 
\begin{figure}[h!]
    \centering
    \includegraphics[scale=0.4]{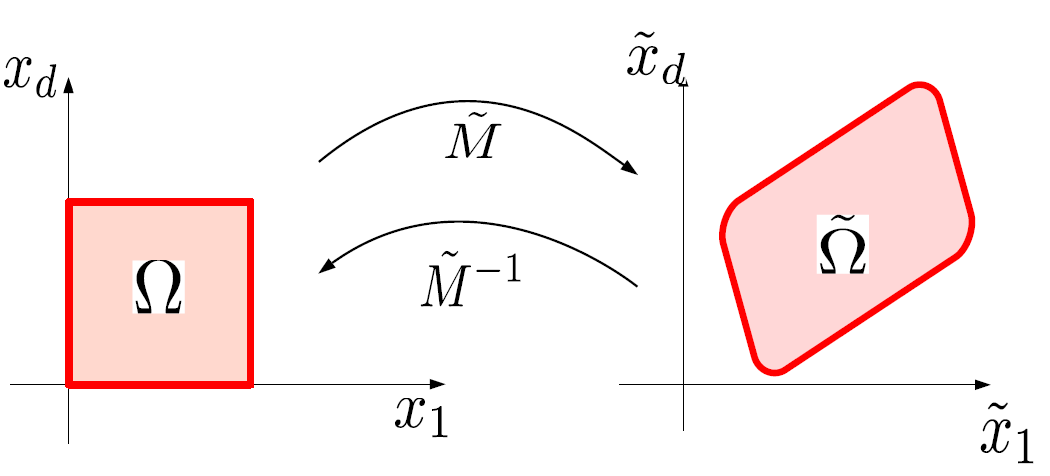}
    \caption{The mapping of $\tilde{\Omega}$ to $\Omega$}
    \label{fig:domain_map}
\end{figure}
\vspace{-2.5mm}
Figure  \ref{fig:domain_map} depicts a domain of interest $\tilde{\Omega} \subset\mathbb{R}^d$ that is the image of the master domain $\Omega:=[0,1]^d$ under a suitably smooth change of coordinates $\Omega=\tilde{M}(\tilde{\Omega}).$
We write $x:=\tilde{M}(\tilde{x})$ for each $x\in\Omega$ and  $\tilde{x}\in \tilde{\Omega}$, and  the determinant of the Jacobian matrix  is given by 
$$
m(\tilde{x}):= \left | \frac{\partial x}{\partial \tilde{x}}\right | :=\left | \frac{\partial \tilde{M}}{\partial \tilde{x}}\right | .
$$
Let $\left \{\psi_{j,k}\ | \ j\in \mathbb{N}_0, k\in \Gamma_j^\psi \right \}$ be any of the orthonormal bases of $L^2(\Omega)$ constructed above from wavelets or multiwavelets over the master domain $\Omega$. The bases discussed in Examples \ref{ex:ON_wave} and \ref{ex:ON_multi} are just two possibilities that can be employed to construct the warped wavelets here. We define  an associated collection of bases over the domain $\tilde{\Omega}$ 
from the identity
$$
\tilde{\psi}_{j,k}(\tilde{x}):=
\psi_{j,k}(\tilde{M}(\tilde{x}))
$$
for $j\in \mathbb{N}_0$ and $k\in \Gamma_j^\psi$. The family $\{\tilde{\psi}_{i,k}\}_{i\in \mathbb{N}_0,k\in \Gamma_{i}^\psi}$ are the warped wavelets generated by $\{ \psi_{j,k} \}$. We have the integration formula that results from the change of variables, 
\begin{align*}
\delta_{(j,k),(\ell,m)}&=
\int_\Omega \psi_{j,k}(x) \psi_{\ell,m}(x) dx
=\int_{\tilde{\Omega}} {{\psi}}_{j,k}(\tilde{M}(\tilde{x}))
{{\psi}}_{\ell,m}(\tilde{M}(\tilde{x}))
\left | \frac{\partial x}{\partial \tilde{x}} \right |d\tilde{x}
\\
&= \int_{\tilde{\Omega}} {\tilde{\psi}}_{j,k}(\tilde{x})
{\tilde{\psi}}_{\ell,m}(\tilde{x})m(\tilde{x})d\tilde{x}
:=
\left( {\tilde{\psi}}_{j,k},
{\tilde{\psi}}_{\ell,m} \right )_{L^2_{\tilde{\mu}}(\tilde{\Omega})} .\\
\end{align*}
The orthonormality of the basis $\left \{\psi_{j,k}\ | \ j\in \mathbb{N}_0, k\in \Gamma_j^\psi \right \}$ in the $L^2(\Omega)$ inner product on $\Omega$  implies that the basis 
$\left \{\tilde{\psi}_{j,k}\ | \ j\in \mathbb{N}_0, k\in \Gamma_j^\psi \right \}$ is orthonormal in the usual $\tilde{\mu}-$weighted inner product  
$$
(\tilde{f},\tilde{g})_{L^2_{\tilde{\mu}}(\tilde{\Omega})}
:= \int_{\tilde{\Omega}} \tilde{f}(\tilde{x}) \tilde{g}(\tilde{x})\tilde{\mu}(d\tilde{x})
$$
over $\tilde{\Omega}$.
We finally {\em define} the Hilbert space $\tilde{U}$ as the completion of the set of orthonormal functions $\left \{\tilde{\psi}_{j,k}\ | \ j\in \mathbb{N}_0, k\in \Gamma_j^\psi \right \}$ defined over $\tilde{\Omega}$ in the above $\tilde{\mu}-$ weighted  inner product.
 The approximation spaces ${A}^{r,q}(\tilde{U})$ are then defined in the usual manner, according the definitions   in Section \ref{sec:approx_spaces_Arq} or Appendix \ref{sec:approx_spaces}. When the measure $\tilde{\mu}$ is known,  the approximation spaces  can be used to build and measure rates of  approximations of the Koopman and Perron-Frobenius operators for functions over $\tilde\Omega$. The error rates are identical to that in Theorem \ref{th:approx_simple_U} with $U$ replaced by $\tilde{U}$.  Examples in Section \ref{sec:approx_mu_u_p_k}  illustrate that such warped  bases are also important in deriving data-dependent approximations from samples.

Note carefully that the discussion above refers to (at least) three different spaces of square integrable functions. We have the space $L^2(\Omega)$ that is the the usual space of real-valued, Lebesgue square-integrable functions over $\Omega$. We also refer to the usual $\tilde{\mu}$-weighted space of square-integrable functions $L^2_{\tilde{\mu}}(\tilde{\Omega})$ and its inner product over $\tilde{\Omega}$.  Finally, we define $\tilde{U}$ as  the Hilbert space of functions defined over $\tilde{\Omega}$ as  the completion of the finite linear span of warped wavelets in the $\tilde{\mu}-$weighted inner product on $L^2_{\tilde{\mu}}(\tilde{\Omega})$. Since  each of the wavelets $\tilde{\psi}_{j,k}$ is contained in $L^2_{\tilde{\mu}}(\tilde{\Omega})$, we know that 
$$
\tilde{U}\subseteq L^2_{\tilde{\mu}}(\tilde{\Omega}),
$$
In general the study of the properties of weighted approximation spaces can be delicate and the reader should see \cite{kerkyacharianwarped} for a presentation of the theory, or  consult \cite{binevI,binevII} to see pragmatic adaptive estimators that rely on these spaces.

We discuss in Examples \ref{ex:warp1}, \ref{ex:warp2}, \ref{ex:warp3} the use of warped wavelets in the approximation of Koopman and Perron-Frobenius operators. 

%
%
%

\subsection{Approximation of  $\mathcal{P}, \mathcal{U}$ over ${A}^{r,2}(U)$}

Recall that the approximation space  $A^{r,2}(U)$ of order $r\geq 0$ is defined in terms of the seminorm   $|\cdot|_{A^{r,2}(U)}$. When the basis $\{\psi_{j,k}\ | \ j\in \mathbb{N}_0,k\in \Gamma_j^\psi\}$ is orthonormal, it can be expressed  as 
\begin{align}
\label{eq:A_S}
A^{r,2}(U)&:=
\left \{
f\in U \ \left |  \ |f|^2_{A^{r,2}(U)}:=
\sum_{j\in \mathbb{N}_0} 2^{2rj} \| Q_jf\|^2_{U}
\right . \right \},\\
&=\left \{
f\in U \ \left |  \ |f|^2_{A^{r,2}(U)}:=
\sum_{j\in \mathbb{N}_0} 2^{2rj} \sum_{k\in \Gamma_j^\psi} |(f,\psi_{j,k})_U|^2
\right . \right \}.
\end{align}
Here $Q_j$ the orthogonal projection operator from $U$ onto $W_j$, and $Q_jf$ is simply the sum of the generalized Fourier coefficients for $k\in \Gamma_{j,k}^\psi$ of the function $f$. We now  construct approximations of either Koopman or Perron-Frobenius operators using these spaces. In this section we will state the  primary results in terms of the Perron-Frobenius operator $\mathcal{P}$, but approximations of the Koopman operator then follow by duality.  The discussion that follows is structured like our presentation of the spectral
Corresponding to the family of {\em self-adjoint} operators in  $\mathbb{A}^{r,2}_\lambda(U)$, we  define the family of admissible  self-adjoint Perron-Frobenius operators $\mathbb{A}^{r,2}(U)$ to be 
\begin{equation}
\mathbb{A}^{r,2}(U):= 
\left \{ \left . 
\mathcal{P}:=\sum_{j\in \mathbb{N}_0,\\ k\in \Gamma_j^\psi }p_{j,k}\psi_{j,k}\otimes \psi_{j,k} \in S^\infty(U) \ 
\right | \ 
|\mathcal{P}|_{\mathbb{A}^{r,2}(U)}<\infty
\right \}
\label{eq:Ar2_SA}
\end{equation}
with the seminorm in the above expression given by 
$$
|\mathcal{P}|^2_{\mathbb{A}^{r,2}(U)}:=\sum_{j\in \mathbb{N}_0}
2^{2jr}\sum_{k\in \Gamma_j^\psi}
|p_{j,k}|^2.
$$
These operators $\mathbb{A}^{r,2}(U)$ have the same intuitive interpretation as the spectral spaces $\mathbb{A}^{r,2}_\lambda(U)$. The higher the rate $r>0$, the faster the coefficients $\{p_{j,k}\}$ must converge to zero. Membership in $\mathbb{A}^{r,2}(U)$ is equivalent to the requirement that the sequence $\{2^{rj}p_{j,k}| j\in \mathbb{N}_0, k\in \Gamma_j^\psi \}$ is in $\ell^2$. 
We follow the same general procedure used in the study of the spectral spaces  $\mathbb{A}^{r,2}(U)$. We initially define $\mathbb{A}^{r,2}(U)$ to consist only of certain self-adjoint operators. This makes  the derivation of the error bounds in Theorem \ref{th:th1} proceed easily. Then, in Section \ref{sec:Ar2_NSA_family} and Theorem \ref{th:Ar2_NSA_a}, we show that the same error bounds that are derived for the self-adjoint case apply when $\mathbb{A}^{r,2}(U)$ is defined as in Equation \ref{eq:Ar2_SA}.

The following theorem, that gives the fundamental error estimates for the  operators when the priors are expressed in terms of the spaces $A^{r,2}(U)$ or the self-adjoint operators $\mathbb{A}^{r,2}(U)$, is the analog of the results for the spectral spaces described  in Theorem \ref{th:spec_sp_approx}. 
 \begin{theorem}
\label{th:th1}
Suppose that  $\left \{\psi_{j,k}\right \}_{j\in \mathbb{N}_0, k\in \Gamma_j}$ is any orthonormal basis for $U$ and 
$\mathcal{P}$ has the representation
$$
\mathcal{P}:=\sum_{j\in \mathbb{N}_0} \sum_{k\in \Gamma_j} p_{j,k}\psi_{j,k}\otimes \psi_{j,k}\in S^{\infty}(U).
$$
When we define  $\mathcal{P}_j:=\sum_{0\leq m<j} \sum_{k\in \Gamma_m^\psi} p_{m,k}\psi_{m,k} \otimes \psi_{m,k}$,
 we have 
$$
\|(\mathcal{P}-\mathcal{P}_j)f\|_{U} \lesssim  2^{-rj}|\mathcal{P}f|_{{A}^{r,2}(U)}
$$
whenever $\mathcal{P}f\in A^{r,2}(U)$.
Whenever we have $\mathcal{P}\in \mathbb{A}^{r,2}(U)$ and  $s>r$, we also have 
\begin{align*}
    |\mathcal{P}-\mathcal{P}_j|_{A^{r,2}(U)}\lesssim 2^{-j(s-r)}|\mathcal{P}|_{A^{r,2}(U)}.
\end{align*}
\end{theorem}
\begin{proof}
The proof of these results follows, with slight modifications, the analysis for the spectral approximation spaces in Theorem \ref{th:spec_sp_approx}. 
We express the error as 
\begin{align*}
    \|(\mathcal{P}-\mathcal{P}_j)f\|_U^2 &= \sum_{\ell\ge j} \sum_{k\in \Gamma_\ell^\psi} |p_{\ell,k}^2| |(f,\psi_{\ell,m})_U|^2 \\
    &\leq \sum_{\ell\ge j} \sum_{k\in \Gamma_\ell^\psi} |p_{\ell,k}^2|2^{2\ell r}2^{-2\ell r} |(f,\psi_{\ell,m})_U|^2 \\
    &\leq 2^{-2jr} \sum_{\ell\in \mathbb{N}_0} \sum_{k\in \Gamma_\ell^\psi} 2^{2\ell r} |p_{\ell,m} (f,\psi_{\ell,m})_U|^2\leq 2^{-2jr}|\mathcal{P}f|_{A^{r,2}(U)}^2.
\end{align*}
We also have 
\begin{align*}
|\mathcal{P}-\mathcal{P}_j|^2_{\mathbb{A}^{r,2}(U)} & =\sum_{\ell\geq j} 2^{2\ell r} \sum_{k\in \Gamma_\ell^\psi}|p_{\ell,k}|^2 \leq \sum_{\ell\geq j} 2^{2\ell r} 2^{-2\ell s} 2^{2\ell s} \sum_{k\in \Gamma_\ell^\psi}|p_{\ell,k}|^2\\
&\leq 2^{-2j(s-r)} \sum_{\ell\in \mathbb{N}_0} 2^{2\ell s} \sum_{k\in \Gamma_\ell^\psi} |p_{\ell,k}|^2 = 2^{-2j(s-r)} |\mathcal{P}|_{\mathbb{A}^{s,2}}^2.
\end{align*}
\end{proof}

The error bounds for the spectral approximation spaces in Theorem \ref{th:spec_sp_approx} should be carefully compared to those in the above theorem. It is important to note in this comparison that the rates in the spectral approximation spaces are referenced to the eigenvalues of a fixed self-adjoint operator, but those in the above theorem are independent of such operator  dependence.

Also, before discussing some applications of this theorem, we discuss the relationship between the mesh resolution level $j$, the projector $\Pi_j:U\rightarrow A_j$, and the dimension $n_j:=\# A_j$ of the finite dimensional approximation spaces $A_j$. 
In one spatial dimension, on a compact domain such as $\Omega=[0,1]$ or $\Omega=\mathbb{T}^1$ for instance, we frequently have 
$$
n_j:=\#A_j \approx 2^j.
$$
This order of dimension holds for any of the Daubechies wavelets, Coiflets, or orthonormal multiwavelets in one dimension. 
Now suppose we estimate a function $f:\Omega \rightarrow \mathbb{R}^d$ for a compact set  $\Omega\subset \mathbb{R}^d$.   In our examples that follow that use tensor products we have  
$$
n_j:= \#A_j\approx d\cdot 2^{dj} \approx 2^{dj},
$$
where the rightmost equivalence reflects only the asymptotics in the mesh resolution level $j$. This means that  an  alternative form of the error bounds can be written as 
\begin{align*}
    \|(I-\Pi_j)f\|_U & \lesssim n_j^{-r/d}|f|_{A^{r,2}}, \\
    \|(\mathcal{P}-\mathcal{P}_j)f\|_U & \lesssim n_j^{-r/d} |\mathcal{P}f|_{A^{r,2}(U)},
\end{align*}
with the multiplying constant a function of $d$, like  $Cd^{-r/d}$ with $C$ independent of $d,j$.

\begin{framed}
\begin{example}[Haar Wavelet Approximation of $\mathcal{P}:L^2(\Omega)\rightarrow L^2(\Omega)$]
\label{ex:ex1}
In this first example we study a problem when  the Perron-Frobenius operator is induced by a kernel $p$ as in Equation \ref{eq:simpleP}. The  problem is motivated by examples  in Sections 2 and 3 of the publication  \cite{Nitin2018} that has appeared in 2018. This paper studies  a class of discrete dynamical systems and approximations of the Koopman operator in terms of permutation operators on measurable partitions. As a fundamental step, the approach employs estimates expressed in terms of finite dimensional spaces of characteristic functions over the partitions. Here we study a  related problem to see  how rates of convergence can be derived in this setting. 

We assume  the domain  $\Omega=[0,1]^d$ and  that the measure $\mu(dx):=m(x)dx$ for a some function $m$.  We further suppose that there exists a pair of constants $c_1,c_2>0$ such that $c_1 \leq m(x) \leq c_2$ for all $x\in \Omega$. We know then that $m\in L^\infty(\Omega)\subset L^2(\Omega)$
since the domain is compact and that  $L^2_\mu(\Omega) \approx L^2(\Omega)$ in this case. We express the domain $\Omega$ as the union of dyadic cubes ${\square_{j,k}}\subset \mathbb{R}^d$ following the definition described in Example \ref{ex:haar_basis} for $d=1$. 
We extend the definition of the Haar scaling functions for $d=1$ given in Example \ref{ex:haar_basis} to $d>1$ using the strategy described in Examples \ref{ex:ON_wave} or \ref{ex:ON_multi} for tensor products, and  define the  Haar scaling functions $\phi:\mathbb{R}^d \rightarrow \mathbb{R}$. From these we define the  scaled and translated Haar scaling functions  $\phi_{j,k}:=2^{jd/2}\phi(2^jx-k)$  and their associated tensor product wavelets $\psi^e_{j,k}:=2^{jd/2}\psi^e(2^jx-k)$ for $e\in \{0,1\}^d$.   These are used to construct the  orthonormal basis  for $A_j\subset U$.  Finally we  construct the approximation spaces $A^{r,2}(L^2(\Omega))$ in Equation \ref{eq:A_S} in terms of the vector-valued,  tensor product basis $\{\psi^e_{i,k}\ | \ e\in \{0,1\}^d, \ i\leq j,  k\in \Gamma_{j}^\psi \}$. With a bit of bookkeeping, we find that the construction of the tensor product scaling functions and wavelets yields $n_j:=\#(A_j)\approx O(2^{dj})$. When the  Perron-Frobenius operator $\mathcal{P}\in \mathbb{A}^{r,2}(L^2(\Omega))$, we are guaranteed from Theorem \ref{th:th1} that 
$\|(\mathcal{P}-\mathcal{P}_j)f\|_U\approx O(n_j^{-r/d})$. 

In addition, in this example problem, we have $\mathcal{P}^*=\mathcal{P}=\mathcal{U}$.  When a  kernel $p(x,y)$ induces $\mathcal{P}$,  then $\mathcal{P}^*$ is induced by the kernel that results when $x,y$ are interchanged. We have $p(x,y)=\sum_{j,k}p_{j,k} \psi_{j,k}(x)\psi_{j,k}(y)$ in this problem, which is clearly symmetric in $x,y$.  Therefore, 
$$
\|(\mathcal{P}-\mathcal{P}_j)f\|_{U} = 
\|(\mathcal{P}-\mathcal{P}_j)^*f\|_{U} 
=
\|(\mathcal{U}-\mathcal{U}_j)f\|_{U}\lesssim n_j^{-r/d}\|\mathcal{P}\|_{\mathbb{A}^{r,2}}\|f\|_{U}
$$
when we observe that  $\mathcal{U}_j:=\mathcal{P}_j^*=\mathcal{P}_j$. Later in this paper we consider the case when the kernel $p:\Omega \times \Omega \rightarrow \mathbb{R}$ may not be symmetric, when it has a representation of the form 
$$
p(x,y):=\sum_{j,\ell \in \mathbb{N}_0} 
\sum_{k\in \Gamma_{j}^\psi,\\ s\in \Gamma_\ell^\psi}
p_{(j,k),(\ell,s)} \psi_{j,k}(x)\psi_{\ell,s}(y).
$$
In this highly structured example then, as in later examples, the determination of a convergence rate of approximations of  one operator implies a rate of convergence for its dual or adjoint. 

At this point we have a description of convergence that relies on the fact that the operator $\mathcal{U}f$ or $\mathcal{P}f$ lies in the  abstract approximation space $A^{r,2}(U)$. It is an important issue to understand or interpret this condition in terms of well-known or standard spaces of functions. Here we use one special case of the much more general analysis  summarized in \cite{devorenonlinear}. The error $E_j(f,U)$ of best approximation from $A_j$  introduced in 
Section \ref{sec:approx_spaces_Arq}  for  $U:=L^2(\Omega)$  takes the form  
$$
E_j(f,L^2(\Omega)):=\inf_{a\in A_j} \|f-a\|_{L^2(\Omega)}:=\|\Pi_{j+1} f\|_{L^2(\Omega)}.
$$
It is known   that for linear approximation methods over the piecewise constant functions on a uniform grid we have 
$$
 f\in \text{Lip}(r,L^2(\Omega)) \quad \quad \Longrightarrow \quad \quad E_j(f,L^2(\Omega))\approx O(2^{-jr}) \approx O(n_j^{-r/d}) 
$$
for the range $0<r<1/2$.  \cite{devore1993constapprox,devorenonlinear} {In fact, we have that $\text{Lip}(r,L^2(\Omega))$ is equivalent to the {\em linear} approximation space $A^{r,\infty}(L^2(\Omega))$ for this range of the rate $r$ for nice domains.} 
We also know that the approximation space $A^{r,2}(L^2(\Omega))$ is equivalent to the Sobolev space $W^{r}(L^2(\Omega))$ over the same range. Both of these results can be deduced from Theorem \ref{th:approx_besov} in the Appendix in Section \ref{sec:approx_splines}.  This restriction in the range of the rate $r$ is tied to our choice to use the Haar basis, which is discontinuous. If some other, smoother wavelet system is used, this range would increase up to a value determined by the smoothness of the basis and the domain boundary.  Such a case is presented in Example \ref{ex:ex2}. 

This analysis has an important implication for those who may feel that the approximation space $A^{r,2}(L^2(\Omega))$ is, perhaps,  too abstract.
Suppose that we  develop an alternative algorithm that generates a different solution $\tilde{\mathcal{P}}_jf\in A_j$ for estimating the function $\mathcal{P}f$. In other words the solution $\tilde{P}f$ is still constructed in terms  of the span of the characteristic functions over the dyadic cubes, but it is different from our approximation.  We know that the best we can ever do by such a linear method for functions $\mathcal{P}f\in \text{Lip}(r,L^2(\Omega))$ is to achieve $O(2^{-rj})$ when $0<r<1/2$. 

We can make this observation more clear perhaps by thinking about Galerkin and generalized (or Petrov-)Galerkin  methods of approximation.  First, it is easy to see that the choice of approximation in Theorem \ref{th:th1} above 
$$
\mathcal{P}_j:=\sum_{0\leq i\leq j}  \sum_{k\in \Gamma_i^\psi} p_{i,k} \psi_{i,k}\otimes \psi_{i,k}
$$
is nothing more than the solution $\mathcal{P}_jf:=a$ of the Galerkin equations where we seek an $a\in A_j$ such that 
$$
(\mathcal{P}f-a,g)_U=0
$$
for all $g\in A_j$.  In the generalized Galerkin approximation, we seek a solution  
 $\tilde{P}_jf:=a$ of the equations 
$$
(\mathcal{P}f-a,g)_U=0
$$
for all $g\in \tilde{A}_j$. Here the set of test functions  $\tilde{A}_j$ is not necessarily the same as the approximation space   $A_j$, and these methods will 
in general yield different solutions $\mathcal{P}_jf\not = \tilde{\mathcal{P}}_jf$, both contained in  $A_j$. The comments above apply then to such solutions obtained by the generalized Galerkin approximations. 
See \cite{klus2016} Section 3 for a discussion of the generalized Galerkin method.

\end{example}
\end{framed}
\begin{framed}
\begin{example}[Example Rate of Convergence of EDMD]
\label{ex:ex2}
The EDMD method is a popular method for solving several approximation problems related to Koopman operators. Theorem 3 in the recent paper by Korda and Mezic \cite{Korda2017Convergence} shows that 
$$
\lim_{n\rightarrow \infty} \|(\mathcal{U}-{\Pi}_n \mathcal{U} {\Pi}_n)f \|_{L^2_\mu(\Omega)} =0.
$$
for ${\Pi}_n$ the orthogonal projection onto the $n-$dimensional space of approximations $A_n:=\text{span}\{e_i \ | \ i\leq n \}$ with $\{e_i\}_{i\leq n}$ some collection of basis functions.  Following convention in  approximation theory, we denote by  $\Pi_j$  the orthogonal projection onto the first $n_j$ basis functions, $n_j:=\#(A_j)$.
With the choice of priors in this paper, Theorem \ref{th:th1} gives a rate of convergence  of $\|(\mathcal{U}-\mathcal{U}_j\|_{L^2(\Omega)} \approx O(2^{-rj})$ when $\mathcal{U}f$ belongs to ${A}^{r,2}(U)$. For example, if $n_j\approx 2^{dj}$, as is common when using tensor product bases over a domain $\Omega\subset \mathbb{R}^d$, then we have the rate of convergence $O(n_j^{-r/d})$. 
This rate of convergence will hold for a range of $r$ that depends on the smoothness of the  particular choice of wavelet or multiwavelet basis. 

Here we can be specific when we choose the Daubechies orthonormal wavelets discussed in detail in Example \ref{ex:ON_wave} in the appendix.  Define 
$$
A_j:=\left \{\phi_{j,k}\ | \ {k\in \Gamma^\phi_j}\right \} =
 \left \{
\psi_{i,k}\ | \  0\leq i \leq j , \  k\in \Gamma^\phi_j
\right \}
$$
in terms of the Daubechies scaling function and wavelet $\phi:={}^N\phi$ and $\psi:={}^N\psi$ of order $N$. The Daubechies wavelet system is discussed in detail in Example \ref{ex:ON_wave}. 
From Theorem \ref{th:th1}  we conclude that the error in approximation $\|(\mathcal{U}-\mathcal{U}_j)f\|_{L^2(\Omega)} \approx O(2^{-rj})$ when 
$$
\mathcal{U}f\in {A}^{r,2}(L^2(\Omega)):=A^{r,2}\left (L^2(\Omega); \{A_j\}_{j\in \mathbb{N}_0} \right ).
$$ 
At this point, the convergence rate is achieved in a space $A^{r,2}(L^2(\Omega)$ that is defined in terms of the Daubechies wavelets, and it may not be satisfying that  this rate is not (yet) expressed in terms of  common or well-known spaces. 
But more can be said that connects this rather abstract approximation space to other function spaces. 
The translates of the Daubechies scaling function ${}^N\phi$   of order $N$ reproduce polynomials of order $N$, that is, degree $N-1$.  This implies that the dyadic spline spaces $\mathcal{S}_{r}(\Delta_j)$ \ref{sec:approx_splines}  of order $r=N$ having dyadic knots $\Delta_j\subset \Omega$ are contained in the approximant spaces $A_j:=\text{span}\{\phi_{j,k} \ | \ k\in \Gamma_j^\phi\}$ for the Daubechies wavelet system of order $N$.  It is easily shown that the approximation space defined in terms of linear approximation by splines $\mathcal{S}_r(\Delta_j)$ is contained in the approximation space generated by the Daubechies wavelet system of order $N$, 
$$
A^{r,2}(L^2(\Omega);\{\mathcal{S}_r(\Delta_j)\}_{j\in \mathbb{N}_0}) \subset A^{r,2}(L^2(\Omega);\{ A_j\}_{j\in \mathbb{N}_0}).
$$
We conclude that for any $\mathcal{U}f\in W^{r}(L^2(\Omega))$ with $0<r<N-3/2$, we obtain a rate of convergence that is $O(2^{-rj}):=O(2^{-Nj})$ in $L^2(\Omega)$.  
Theorem \ref{th:approx_besov} shows that $A^{r,2}(L^2(\Omega);\{\mathcal{S}_r(\Delta_j)\}_{j\in \mathbb{N}_0})$ is equivalent to the Sobolev spaces $W^r(L^2(\Omega))$ for $0<r <s=N-3/2$.  Since $C^{r}(\Omega)\subset W^r(L^2(\Omega))$, we conclude this that this approximation rate holds for common $r-$times continuously differentiable functions over this range of $r$. 
\end{example}

\end{framed}

\begin{framed}
\begin{example}[Haar Multiwavelet Approximation of $\mathcal{P}:L^2(\triangle)
\rightarrow L^2(\triangle)$]
 Example \ref{ex:ex1} is motivated by work in \cite{Nitin2018} that constructs  approximations of discrete dynamical systems in terms of permutation  operators defined over measurable partitions of $\Omega$. The error  analysis in Example \ref{ex:ex1} is carried out when $\Omega$ is a product domain $[0,1]^d$,  the partition is realized by dyadic cubes $\square_{j,k}$, and tensor products of classical Haar wavelets and scaling functions are used as bases for approximations.  In this example we suppose that we are interested in approximations of the Koopman or Perron-Frobenius operators for permutation operators on measurable partitions over the triangle $\triangle:=\triangle_{0,0}$ introduced in Example \ref{ex:haar_multi_tri}. We suppose that the measure $\mu(dx_1,dx_2):=m(x_1,x_2)dx_1dx_2$ for some $m:\triangle\subset \mathbb{R}^2 \rightarrow \mathbb{R}$, and as in Example \ref{ex:haar_multi_tri} there are constants $c_1,c_2>0$ with 
 $$
 c_1 \leq m(x_1,x_2)\leq c_2
 $$
 for all $(x_1,x_2)\in \triangle$. We then define the multiscaling functions $\phi_{j,(i,\bm{k})}$,  multiwavelets $\psi_{j,(i,\bm{k})}$, and index sets $\Gamma_j^\phi$, $\Gamma_j^\psi$ as described in Example \ref{ex:haar_multi_tri}. We then apply Theorem \ref{th:th1} with the choice $U:=L^2(\triangle)$ and find that 
$$
\| \mathcal{P}-\mathcal{P}_j \|_{\mathbb{A}^{r,2}(L^2(\triangle))} \lesssim 2^{-j(s-r)} |\mathcal{P}|_{\mathbb{A}^{s,2}(L^2(\triangle)}
$$
for $s>r\geq 0$. 
 For this case $n_j=\#(A_j)=O(2^{dj})$ with $d=2$, and we can alternatively express this bound in the form 
$$
| \mathcal{P}-\mathcal{P}_j |_{\mathbb{A}^{r,2}(L^2(\triangle))} \lesssim n_j^{-(s-r)/d} |\mathcal{P}|_{\mathbb{A}^{s,2}(L^2(\triangle)}. 
$$
This bound is of the same form as that derived in Example \ref{ex:ex1} for the domain $\Omega=[0,1]^d$ when we choose $r=0$ above. 
\end{example}
\end{framed}
\begin{framed}
\begin{example}[Warped Wavelet Approximation of $\tilde{\mathcal{P}}:L^2_{\tilde{\mu}}(\tilde{\Omega}) \rightarrow L^2_{\tilde{\mu}}(\tilde{\Omega})$ ]
\label{ex:warp1}
Before we summarize our next category of  results, we return to Example \ref{ex:ex1} and discuss how the approach of the last example can be modified for another class of systems.  We assume that we must approximate the Perron-Frobenius operator $\tilde{\mathcal{P}}$ where for each $\tilde{y}\in \tilde{\Omega}$ 
$$
(\tilde{\mathcal{P}}f)(\tilde{y}):=\int_{\tilde{\Omega}}\tilde{p}(\tilde{y},\tilde{x})\tilde{f}(\tilde{x})\tilde{m}(\tilde{x})d\tilde{x}
$$  
with $\tilde{m}\in L^1(\tilde{\Omega})$, $\tilde{\mu}(d\tilde{x})=\tilde{m}(\tilde{x})d\tilde{x}$, and the domain $\tilde{\Omega}$ is  related to a master domain such as $\Omega:=[0,1]^d$ by the change of variables $\Omega=\tilde{M}(\tilde{\Omega})$ and  $x=\tilde{M}(\tilde{x})$, $|\tilde{M}'(\tilde{x})|=\tilde{m}(\tilde{x})$ for $\tilde{x}\in \tilde{\Omega}$ and $x\in \Omega$.
We need to construct an orthonormal basis $\tilde{\psi}_{j,k}$ over $\tilde{\Omega}$ to carry out an error analysis of $\tilde{\mathcal{P}}$ in the spirit of Theorem \ref{th:th1}. 
Let $\psi_{j,k}(x)$ be any nice orthonormal wavelet or multiwavelet basis for the Lebesgue space  $L^2(\Omega):=L^2([0,1]^d)$. Then $\tilde{\psi}_{j,k}(\tilde{x}):=\psi_{j,k}(\tilde{M}(\tilde{x}))$ is easily seen to be a family of $L^2_{{\tilde{\mu}}}$-orthonormal wavelets over $\tilde{\Omega}$. We define the space $\tilde{U}$ to be the completion in the norm $L^2_{\tilde{\mu}}(\tilde{\Omega})$ of the warped wavelets.  Now we   {\em define the approximation spaces $A^{r,2}(\tilde{U})$} in terms of the warped wavelets $\tilde{\psi}_{j,k}$ following the philosophy of \cite{kerkyacharianwarped}. Note that here  there is no guarantee  that we have an equivalence  $L^2(\Omega) \approx \tilde{U}$. In fact, we do not even know if the usual $\tilde{\mu}-$Lebesgue space $L^2_{\tilde{\mu}}(\tilde{\Omega})$  over $\tilde{\Omega}$ is the same as $\tilde{U}$, although we do know $\tilde{U}\subseteq L^2_{\tilde{\mu}}(\tilde{\Omega})$.  Since  in this section we assume that the domain $\tilde{\Omega}$ and measure  $\tilde{m}$ are known, this approach can be used to construct approximations as 
$$
\tilde{\mathcal{P}}_j:=\sum_{i<j} \sum_{k\in \Gamma_i^\psi} \tilde{p}_{i,k} \tilde{\psi}_{i,k} \otimes \tilde{\psi}_{i,k},
$$
and error rates will have the form $O(2^{-rj})$ {\em in $A^{r,2}(\tilde{U})$. } It  must be kept in mind that in this case the relation of the approximation spaces to the conventional Lebesgue space $L^2(\tilde{\Omega})$, Sobolev spaces $W^{s}(L^2(\tilde{\Omega}))$, or Lipschitz space $\text{Lip}(r,L^2(\tilde{\Omega}))$ may not be easy or even feasible to establish. It does, however, suggest a path to feasible algorithms with demonstrable rates of convergence. This strategy is studied more  closely  in Example \ref{ex:warp1}, so we leave the details to the reader.  The next example gives the details of the use of warped wavelets in a slightly different context. 
\end{example}
\end{framed}

\begin{framed}
\begin{example}[Warped Wavelet Approximation $\mathcal{U}:\tilde{U}\rightarrow U:=L^2(\Omega)$]
\label{ex:warp2}
The last few examples illustrated cases where the Perron-Frobenius or Koopman operator is induced by a kernel and has a representation in terms  of orthonormal basis functions.  A basic assumption in these examples is that $\mathcal{P}$ is an integral operator induced by the kernel $p:\Omega\times \Omega \rightarrow \mathbb{R}$.  However,  for some problems a specific form for the operator suggests itself, and it may not have this convenient form.   The current example shows the benefit of employing the approximation space framework above even if the form of the operator is not exactly  that in Theorem  \ref{th:th1}.  Let $\Omega:=[0,1]^d$ and suppose that $w:\Omega \rightarrow \Omega$. We consider the discrete iteration 
$$
x_{n+1}=w(x_n).
$$
As noted in the introduction, this deterministic system can always be interpreted as a Markov chain. It has the transition probability kernel $\mathbb{P}(A,x):=\delta_{w(x)}(A)$ for any measurable set $S\subseteq \Omega$, and  
it follows that 
$$
(\mathcal{U}f)(x):=f(w(x)).
$$
We set the  change of variables $\tilde{x}:=w(x)$ and  denote  $\tilde{\Omega}:=w(\Omega)$. 
We define  the mapping $\tilde{M}$ as in Example \ref{ex:warp1} in terms of the transition mapping $\tilde{M}(\tilde{x}):=w^{-1}(\tilde{x})$, define the measure $\tilde{\mu}(d\tilde{x}):=\tilde{m}(\tilde{x})dx$ with $\tilde{m}(\tilde{x}):=|\partial x/\partial \tilde{x}|$. From the identity 
$$
(\mathcal{U}f,g)_{L^2(\Omega)} = (f,\mathcal{P}g))_{L^2_{\tilde{\mu}}(\tilde{\Omega})}, 
$$
we also find that 
$$ (
\mathcal{P}g)(\tilde{x}) = g(w^{-1}(\tilde{x}))
$$
for all $\tilde{x}:=w(x)\in \tilde{\Omega}.$

 Let $\psi_{j,k}$ be an orthonormal basis for $U:=L^2(\Omega)$, and define the warped wavelets $\tilde{\psi}_{j,k}(\tilde{x}):=\psi_{j,k}(w^{-1}(\tilde{x}))$ on $\tilde{\Omega}$.  We have the integration rule 
\begin{align*}
\int_\Omega \psi_{j,k}(x)\psi_{\ell,m}(x)dx
&= \int_{\tilde{\Omega}} \psi_{j,k}(w^{-1}(\tilde{x})) \psi_{\ell,m}(w^{-1}(\tilde{x})) \left | 
\frac{\partial x}{\partial\tilde{x}}
\right |d\tilde{x} \\
&=\int_{\tilde{\Omega}} \tilde{\psi}_{j,k}(\tilde{x}) \tilde{\psi}_{\ell,m}(\tilde{x})  \tilde{\mu}(d\tilde{x})
\end{align*}
for $\tilde{x}\in \tilde{\Omega}:=w(\Omega)$. As in the last example we define $\tilde{U}$ as the closed finite span of the set of warped wavelets in $L^2_{\tilde{\mu}}(\tilde{\Omega})$.  
Suppose that $f\circ w \in L^2(\Omega)$. Then we have 
\begin{align}
\mathcal{U}f:=f\circ w &= \sum_{j\in \mathbb{N}_0} \sum_{k\in \Gamma^\psi_j} (f\circ w, \psi_{j,k})_{L^2(\Omega)} \psi_{j,k}, \notag \\ 
& =
\sum_{j\in \mathbb{N}_0}
\sum_{k\in \Gamma^\psi_j}
\left (f, \tilde{\psi}_{j,k}\right )_{\tilde{U}} \psi_{j,k}. \label{eq:dec_primal}
\end{align}
This representation implies that 
$$
\mathcal{U}f:=f\circ w \in A^{r,2}(U)
$$ 
if and only if $ f \in A^{r,2}(\tilde{U})$. This equivalence depends on the mapping $w$ that is an intrinsic part of the definition of the approximation space $A^{r,2}(\tilde{U})$. 
Select the approximation $\mathcal{U}_j$ to be 
$$
\mathcal{U}_jf:=\sum_{ i< j}
\sum_{k\in \Gamma^\psi_i}
\left (f, \tilde{\psi}_{i,k}\right )_{\tilde{U}} \psi_{i,k}.
$$
If $f\in A^{r,2}(\tilde{U})$, then we have 
$$
\| (\mathcal{U}- \mathcal{U}_j)f\|_U \lesssim 2^{-jr}\|f\|_{A^{r,2}(\tilde{U})}, 
$$
which is the same rate of convergence as that in Theorem \ref{th:th1}. 
On the other hand, by duality, we have the representation   
\begin{align}
\mathcal{P}g:=g\circ (w^{-1})&= 
\sum_{j\in \mathbb{N}_0}
\sum_{k\in \Gamma^\psi_j}
\left (g, {\psi}_{j,k}\right )_{U} \tilde{\psi}_{j,k}. \label{eq:dec_dual}
\end{align} 
We define the estimate $\mathcal{P}_j$ by truncating this expression, just as in the definition of $\mathcal{U}_j$. 
In this case we see that $g\in A^{r,2}(L^2(\Omega))$ if and only if $\mathcal{P}g\in A^{r,2}(\tilde{U})$. 
Carefully note that dual structure in Equations \ref{eq:dec_primal} and \ref{eq:dec_dual}. 
In summary, we have 
\begin{align*}
    \mathcal{U}&:\tilde{U}\rightarrow U ,\\
     \mathcal{P}&: U\rightarrow \tilde{U}.
\end{align*}
For $g\in A^{r,2}(U)$, we have 
$$
\| (\mathcal{P}-\mathcal{P}_j)g\|_{\tilde{U}} \lesssim 2^{-rj}|\mathcal{P}g|_{A^{r,2}(\tilde{U})}, 
$$
and for $f\in A^{r,2}(\tilde{U})$  it follows that 
$$
\| (\mathcal{U}-\mathcal{U}_j)f\|_{U}  \lesssim 2^{-rj}|\mathcal{U}f|_{A^{r,2}(U)}. 
$$
We return to this Example in \ref{ex:meas1} when we discuss the approximation of measures where the role of this dual structure again emerges prominently. 
\end{example}
\end{framed}

\subsection{Compact, Non-Self-Adjoint   Operators and $\mathbb{A}^{r,2}(U)$}
\label{sec:Ar2_NSA_family}
Just as in our analysis of the spectral spaces, our initial definition of $\mathbb{A}^{r,2}(U)$ in Equation \ref{eq:Ar2_SA} includes only self-adjoint operators. We fully expect that it will be necessary to consider discrete evolutions that are characterized by non-self-adjoint Perron-Frobenius and Koopman operators. 
Here we assume that $\mathcal{P}:L^2_\mu(\Omega)\rightarrow L^2_\mu(\Omega)$, and it is induced by the kernel 
\begin{align*}
    p(x,y)&:= 
    \mathop{\sum}_{{\tiny 
\begin{array}{l}
j\in \mathbb{N}_0, \\ k\in \Gamma_j^\psi
\end{array}
}} 
    \mathop{\sum}_{{\tiny
\begin{array}{l}
\ell \in \mathbb{N}_0, \\ m\in \Gamma_\ell^\psi
\end{array}
}}
p_{(j,k),(\ell,m)}
    \psi_{j,k}(x) \psi_{\ell,m}(y)
\end{align*}
with $ p_{(j,k),(\ell,m)}:=\left ( \mathcal{P}\psi_{j,k}, \psi_{\ell,m} \right)_{U}$.
We therefore have 
$
\|\mathcal{P}\|^2_{\mathcal{L}(U)}= \|p\|^2_{L^2_{\mu\times \mu}(\Omega\times \Omega)}
=\sum_{(j,k)}
\sum_{(\ell,m)} p^2_{(j,k),(\ell,m)}.$ We define the seminorm 
$$
|\mathcal{P}|^2_{\mathbb{A}^{r,2}(U)}:= \sum_{j\in \mathbb{N}_0} \sum_{k\in \Gamma_j^\psi} \left (
\sum_{\ell \in \mathbb{N}_0}  2^{r\ell} \sum_{m\in \Gamma_j^\psi}
p_{(j,k),(\ell,m)}
\right )^2, 
$$
and the family of possibly non-self-adjoint feasible operators 
\begin{equation}
\mathbb{A}^{r,2}(U):=
\left \{
\mathcal{P}:= \sum_{(j,k),(\ell,m)} p_{(j,k),(\ell,m)} \psi_{j,k}\otimes \psi_{\ell,m} \in S^\infty(U) \ \biggl | \ |p|_{\mathbb{A}^{r,2}(U)}<\infty
\right \}.
\label{eq:Ar2_fam_2}
\end{equation}

Note that the seminorm above reduces to that in Equation \ref{eq:Ar2_SA}. Also, we have several equivalent expressions for the seminorm $|\cdot|_{\mathbb{A}^{r,2}(U)}$. We define the infinite dimensional matrix
$
    [p_{(j,k),(\ell,m)}]:= [(\mathcal{P}\psi_{j,k},\psi_{\ell,m})_U] 
$
and introduce the operator $D^s_{2^{-2\bullet}}:=\sum_{j\in \mathbb{N}_0} \sum_{k\in \Gamma^\psi_j}2^{-2sj}\psi_{j,k}\otimes \psi_{j,k}$ and its associated diagonal matrix representation $[2^{sj}]$. The signs selected in the definition $D^s_{2^{-2\bullet}}$ are chosen so as to be analogous to the definition of the operator  $D^s_{\lambda_\bullet}$ on the spectral spaces. 
We now can write 
\begin{align}
    |\mathcal{P}|^2_{\mathbb{A}^{r,2}(U)}&:= \left (D^{-r/2}_{2^{-2\bullet}}\mathcal{P},D^{-r/2}_{2^{-2\bullet}}\mathcal{P}\right )_{S^2(U)},\notag \\
    &= \sum_{j\in \mathbb{N}_0} \sum_{k\in \Gamma_j^\psi}\left (D^{-r/2}_{2^{-2\bullet}}\mathcal{P}\psi_{j,k},D^{-r/2}_{2^{-2\bullet}}\mathcal{P}\psi_{j,k}\right )_U. \label{eq:P_norm_equiv}
\end{align}
\begin{proof}
We briefly show that this expression is equivalent to the above.
{\scriptsize
\begin{align*}
D^{-r/2}_{2^{-2\bullet}}&\mathcal{P} \psi_{j,k} \\
&=\left( \sum_{
\begin{array}{c}
\ell \in \mathbb{N}_0\\
m\in \Gamma_{\ell}^\psi
\end{array}
} 2^{r\ell}\psi_{\ell,m}\otimes \psi_{\ell,m} \right )
\left (
\sum_{\begin{array}{c}s\in \mathbb{N}_0\\
t\in \Gamma_{s}^\psi
\end{array}}
\sum_{
\begin{array}{c}
u\in \mathbb{N}_0\\
v\in \Gamma_{u}^\psi
\end{array}
}p_{(s,t),(u,v)} \psi_{s,t}\otimes \psi_{u,v} \right )\psi_{j,k} \\
&=
\left( \sum_{
\begin{array}{c}
\ell \in \mathbb{N}_0\\
m\in \Gamma_{\ell}^\psi
\end{array}
} 2^{r\ell}\psi_{\ell,m}\otimes \psi_{\ell,m} \right )
\left (
\sum_{
\begin{array}{c}
u\in \mathbb{N}_0\\
v\in \Gamma_{u}^\psi
\end{array}
}p_{(j,k),(u,v)}  \psi_{u,v} \right )\\
&=
\sum_{
\begin{array}{c}
\ell\in \mathbb{N}_0\\
m\in \Gamma_{\ell}^\psi
\end{array}
}2^{r\ell}p_{(j,k),(l,m)}  \psi_{l,m}
= \sum_{\ell \in \mathbb{N}_0} 2^{r\ell} \sum_{m\in\Gamma_\ell^\psi} p_{(j,k),(\ell,m)}\psi_{\ell,m}.
\end{align*}
}

\noindent The definition of $|\cdot|_{\mathbb{A}^{r,2}(U)}$ is then equivalent to the square of the norm of this vector and the equality in Equation \ref{eq:P_norm_equiv} holds.

\end{proof}
\begin{theorem}
\label{th:Ar2_NSA_a}
Let $\mathcal{P}:L^2_\mu(\Omega) \rightarrow L^2_\mu(\Omega)$ be given by 
$$
\mathcal{P}:=
\mathop{\sum}_{{\tiny 
\begin{array}{l}
i\in \mathbb{N}_0, \\ k\in \Gamma_i^\psi
\end{array}
}}
\mathop{\sum}_{{\tiny 
\begin{array}{l}
\ell \in \mathbb{N}_0, \\ m\in \Gamma_\ell^\psi
\end{array}
}} p_{(i,k),(\ell,m)} \psi_{i,k}\otimes \psi_{\ell,m}
$$
with $p_{(i,k),(\ell,m)}:=(\mathcal{P}\psi_{i,k},\psi_{\ell,m})_U$ and define the approximation $\mathcal{P}_j$ as  
$$
\mathcal{P}_j:=
\mathop{\sum}_{{\tiny 
\begin{array}{l}
i<j, \\ k\in \Gamma_i^\psi
\end{array}
}}
\mathop{\sum}_{{\tiny 
\begin{array}{l}
\ell<j , \\ m\in \Gamma_\ell^\psi
\end{array}
}} p_{(i,k),(\ell,m)} \psi_{i,k}\otimes \psi_{\ell,m}
$$
Then the error bounds  of Theorem \ref{th:th1} hold with   the definition of $\mathbb{A}^{r,2}_\lambda(U)$ in Equation \ref{eq:Ar2_fam_2}. 
\end{theorem}
\begin{proof}
We follow a similar pattern to the proof of Theorem \ref{th:spec_sp_approx}. 
\begin{align*}
    \|(\mathcal{P}-\mathcal{P}_j)f\|_U^2 & \leq 
    \left \| 
    \sum_{i<j}\sum_{k\in \Gamma_i^\psi} \sum_{\ell<j}\sum_{m\in \Gamma_\ell^\psi} p_{(i,k),(\ell,m)} f_{(i,k)} \psi_{\ell,m}
    \right \|^2_U \\
    &\leq 2^{-2rj} 
      \left \| 
    \sum_{i\in \mathbb{N}_0}\sum_{k\in \Gamma_i^\psi} \sum_{\ell\in \mathbb{N}_0}\sum_{m\in \Gamma_\ell^\psi} p_{(i,k),(\ell,m)}2^{ri} f_{(i,k)} \psi_{\ell,m}
    \right \|^2_U \\
    &\leq 2^{-2rj} 
    \sum_{\ell\in \mathbb{N}_0}\sum_{m\in \Gamma_\ell^\psi}
    \left (
    \sum_{i\in \mathbb{N}_0}
    2^{ri}\sum_{k\in \Gamma_i^\psi}
    p_{(i,k),(\ell,m)} f_{(i,k)}
    \right )^2\\
    &=2^{-2rj}\left \|
    [D^{-r/2}_{2^{-2i}}][p_{(i,k),\ell,m)}]\{f_{(i,k)}\}
    \right \|_{\ell^2}^2 
    \end{align*}
If we know that $f\in U$ and $\mathcal{P}\in \mathbb{A}^{r,2}(U)$, we then see that 
    \begin{align*}
    \|(\mathcal{P}-\mathcal{P}_j)f\|^2_U 
    &\leq 2^{-2rj} \|  [D^{-r/2}_{2^{-2i}}][p_{(i,k),\ell,m)}]\|_M^2 \|\{f_{i,k}\}\|^2_{\ell^2} \leq 2^{-2rj}\|D^{-r/2}_{2^{-2\bullet}} \mathcal{P}\|_{\mathcal{L}(U)}^2 \|f\|_U^2\\
    &\leq 
    2^{-2rj}\|D^{-r/2}_{2^{-2\bullet}} \mathcal{P}\|_{S^2(U)}^2 \|f\|_U^2 
    = 2^{-2rj}|\mathcal{P}|^2_{\mathbb{A}^{r,2}(U)}\|f\|^2_U.
\end{align*}
On the other hand, if $\mathcal{P}\in S^2(U)$ and $f\in A^{r,2}(U)$, we have 
 \begin{align*}
    \|(\mathcal{P}-\mathcal{P}_j)f\|^2_U
    &\leq 2^{-2rj} \|  [D^{-r/2}_{2^{-2i}}][p_{(i,k),\ell,m)}]\{f_{i,k}\}\|^2_{\ell^2}\\
    &\leq 2^{-2rj}\|  [p_{(i,k),\ell,m)}][D^{-r/2}_{2^{-2i}}]\{f_{i,k}\}\|^2_{\ell^2}\\
    &\leq 2^{-2rj}\|[p_{(i,k),(\ell,m)}]\|_M^2 \| [D^{-r/2}_{2^{-2i}}]\{f_{i,k}\} \|^2_{\ell^2} \\
    &\leq 2^{-2rj}\|\mathcal{P}\|_{\mathcal{L}(U)}^2 \|f\|_{A^{r,2}(U)}^2
    \leq 2^{-2rj}\|\mathcal{P}\|_{S^2(U)}^2\|f\|_{A^{r,2}(U)}^2.
\end{align*}
\end{proof}

%
%

\subsection{$A^{r,2}_\lambda(U)$ as a special case of $A^{r,2}(U)$} 
Before we conclude this section, we relate the spectral approximation spaces of Section \ref{sec:spectral_spaces}  with the more general approximation spaces in this section. For any compact, self-ajoint operator $T$, 
there is no loss of generality by  renumbering the eigenfunctions  to follow the conventions of the multiscale framework. For instance, when $d=1$, we set $\Gamma_j^\psi:=\left \{0,1,\dots, 2^{j}-1 \right \}$. We subsequently   define $\psi_{j,k}:=u_m$ for $m=2^j+k$ with  $k\in \Gamma_j$  for $j\in \mathbb{N}_0$. This numbering is easily modified for cases when $d>1$. 
\begin{theorem}[Equivalence of $A^{r,2}_\lambda(U)$ and $A^{r,2}(U)$]
\label{th:th_equiv}
Suppose that the self-adjoint, compact operator $T:U\rightarrow U$ has the Schatten class representation 
$$
T=\sum_{i\in \mathbb{N}} \lambda_i u_i\otimes u_i, 
$$
so that $T\in  A^{r,2}_\lambda(U) $. Denote by $\left \{ \psi_{j,k} \right \}_{j\in \mathbb{N}_0,k\in \Gamma_j^\psi}\sim \left \{ u_i\right \}_{i\in \mathbb{N}}$,  consistent with the multilevel structure as discussed above. 
 If the eigenvalues satisfy  $\lambda_{2^j}\approx 2^{-2j}$, then $A^{r,2}_\lambda(U) \approx A^{r,2}(U)$. 
\end{theorem}
\begin{proof}
Let $f\in A^{r,2}_\lambda(U)$. We have 
\begin{align*}
    |f|^2_{A^{r,2}_\lambda(U)}&=
    \sum_{i\in \mathbb{N}} \lambda_i^{-r} |(f,u_i)_U|^2 =\sum_{j\in \mathbb{N}_0} \sum_{k\in \Gamma_j^\psi} \lambda_{j,k}^{-r} |(f,\psi_{j,k})_U|^2, \\
    &\lesssim \sum_{j\in \mathbb{N}_0} 
    \sum_{k\in \Gamma_j^\psi} \lambda_{2^j+k}^{-r}
    \|(f,\psi_{j,k})_U|^2,  \\
    &\leq \sum_{j\in \mathbb{N}_0} \lambda_{2^j}^{-r} \sum_{k\in \Gamma^\psi_j} 
    |(f,\psi_{j,k})_U|^2
    \lesssim \sum_{j\in \mathbb{N}_0} 
    2^{2jr} \sum_{k\in \Gamma_j^\psi} |(f,\psi_{j,k})_U|^2.
\end{align*}
On the other hand, we can write
\begin{align*}
      |f|^2_{A^{r,2}_\lambda(U)}&=
    \sum_{i\in \mathbb{N}} \lambda_i^{-r} |(f,u_i)_U|^2 =\sum_{j\in \mathbb{N}_0} \sum_{k\in \Gamma_j^\psi} \lambda_{j,k}^{-r} |(f,\psi_{j,k})_U|^2,\\
    &\geq \sum_{j\in \mathbb{N}_0} 
    \lambda_{2^{j+1}}^{-r} 
    \sum_{k\in \Gamma_j^\psi}  |(f,\psi_{j,k})_U|^2
    \geq  \sum_{j\in \mathbb{N}_0} 
    2^{2(j+1)r}  
    \sum_{k\in \Gamma_j^\psi}  |(f,\psi_{j,k})_U|^2,\\
    &\gtrsim   \sum_{j\in \mathbb{N}_0} 
    2^{2jr}  
    \sum_{k\in \Gamma_j^\psi}  |(f,\psi_{j,k})_U|^2.
\end{align*}
This concludes the proof. 
\end{proof}
%
%
\section{Approximation of Signed and Probability Measures}
\label{sec:approx_measures1}

  In this section we study  the approximation  of signed and probability measures as they arise in the theory of Koopman or Perron-Frobenius operators.  The approximation of measures is an important topic in Koopman theory for several reasons. When we study the Perron-Frobenius operator, for instance, one that is associated with a stochastic flow, it is an operator that maps measures into measures. Approximations $\mathcal{P}_j$  of such a  $\mathcal{P}$ are naturally expressed in terms of finite dimensional subspaces of measures. In analogy to our study of approximations of functions, we are interested in defining certain classes of priors that contain measures, and the priors   make it possible to determine rates of convergence of approximations. These convergence rates are of interest in their own right, but there are additional benefits of this analysis. Since   both $\mathcal{P}$ and $\mathcal{U}$ are sometimes  defined explicitly in terms of a kernel and the measure $\mu$, we will see that methods for approximating $\mu$ can be used to construct approximations of these  operators.  
  
   We begin our presentation with a review of basic definitions in Section \ref{sec:basic_measure} and the introduction of Gelfand triples in Section \ref{sec:gelfand}. Section \ref{sec:approx_measures2} introduces methods to estimate signed measures with priors defined by duality relative to $A^{r,2}(U)$, $U$ a Hilbert space. The primary result of this section is encapsulated in Theorem \ref{th:approx_meas}. In many applications we seek to construct approximations of probability measures, and such estimates are the topic of Section \ref{sec:equiv_metric}.  

\subsection{Measures, Duality, and Weak$^*$ Convergence}
\label{sec:basic_measure}
We assume in this section that $\Omega$ is a compact subset of $\mathbb{R}^d$.  In this case it is known that the topological dual $C^*(\Omega)$ is the family of regular countably additive set functions, or regular signed measures, $rca(\Omega)$ on the set $\Omega$. For any signed measure $\mu$ on $\Omega$, there exist mutually singular positive measures $\mu^+$ and $\mu^-$ such that 
$
\mu:=\mu^+-\mu^-.
$
 The total variation norm for  signed measures over the set $\Omega$ is given by
$$
\|\mu\|_{TV} := |\mu|(\Omega):=\mu^+(\Omega)+\mu^{-}(\Omega),
$$
and with this norm $rca(\Omega)$ is a Banach space \cite{dunford}. 
The total variation norm of the difference of two signed measures  also has the convenient representation $\|\mu-\nu\|_{TV}=2\sup_{S\subset \Omega}|\mu(S)-\nu(S)|$. \cite{thorisson}

The total variation norm often induces a topology that is too fine for applications in the study of dynamical systems. When we seek to study the  discrete deterministic dynamics in Equation \ref{eq:det_Phi}, the Koopman operator is expressed in terms of the Dirac measure $\delta_x$ for  $x\in \Omega$.  The Perron-Frobenius operator is also  expressed in terms of this Dirac measure. Suppose we have  sequence of points generated by a discrete dynamical system, $\left \{ x_k \right \}_{k\in \mathbb{N}}\subset\Omega $, that converges to ${x}$, but $x_k \not = x$ for all $k\in \mathbb{N}$. We easily calculate that  $\|\delta_{{x}}-\delta_{x_k}\|_{TV}=2$ for all $k\in \mathbb{N}$, so we conclude that $\delta_{x_k} \not \rightarrow \delta_{\bar{x}}$ in the total variation norm. Intuitively, in the applications to the study of discrete dynamical systems, we would like to have the convergence in some sense  of the probability measures $\delta_{x_k}\rightarrow \delta_x$ for this trajectory. One such topology for which $x_k\rightarrow x$ implies $\delta_{x_k}\rightarrow \delta_x$ is the weak$^*$ topology on  signed measures. 

 Recall that if $X^*$ is the dual space of a Banach space $X$, the weak$^*$ topology on $X^*$ is the weak topology that it inherits from $i_{X,X^{**}}(X)\subseteq X^{**}$, with $i_{X,X^{**}}:X\rightarrow X^{**}$ the  canonical injection of $X$ into its second dual space $X^{**}$. For any set $S\subset X$, we denote by $\text{weak}^*(X^*,S)$ the weak topology on $X^*$ induced by $i_{X,X^{**}}(S)\subset X^{**}$.  When we unwrap this definition for $X:=C(\Omega)$ and $X^*:=rca(\Omega)$, we find that a net of measures $\left \{ \mu_\gamma \right \}_{\gamma\in \Gamma}$ with $\Gamma$ a directed set converges to $\mu$ in the weak$^*$ topology if and only if 
$$
\left < \mu_\gamma , f \right >_{X^* \times X} \rightarrow 
\left <\mu, f \right >_{X^* \times X}
$$
for all $f\in C(\Omega)$.  For computations, this duality statement is just 
$$
\int_\Omega f(\xi) \mu_{\gamma}{(d\xi)} \rightarrow f(\xi) \mu(d\xi)
$$
for all $f\in C(\Omega)$. Since here the domain $\Omega$ is a compact metric space, the set of probability measures $\mathbb{M}^{+,1}(\Omega)\subset rca(\Omega)$  is  compact and metrizable.   Since we always have $\Omega\subset \mathbb{R}^d$ in this paper,  it also  suffices to characterize  weak$^*$ convergence in terms of sequences in this paper.  \cite{parthasarathy}

\subsection{The Gelfand Triple}
\label{sec:gelfand}
The approximation of measures in this section and the next will be facilitated by the use of a Gelfand triple. The Gelfand triple is a standard construct used in the study of partial differential equations, \cite{wloka} and we use it here to relate dual pairings  and inner products.  Let $X$ be a Banach space that is dense in the Hilbert space $U$. We suppose that the injection $i_{X,U}:X\rightarrow U$, which is simply the linear map $f\in X \mapsto i_{X,U}f=f \in H$, is continuous.  This embedding is represented  symbolically as $X \hookrightarrow U$ and implies that $\|f\|_U=\|i_{X,U}f\|_U \lesssim \|f\|_X$. Using the setup described in Section \ref{sec:approx_spaces}, we assume that we are given a dense family of approximant subspaces $\left \{ A_j \right \}_{j\in \mathbb{N}}\subset U$ and associated $U-$orthogonal  projection  operators ${\Pi}_j:U\rightarrow A_j$ that are onto $A_j$. The approximation space $A^{r,q}(U):=A^{r,q}(U,\{A_j\}_{j\in \mathbb{N}})$ is defined in the usual way.
Finally, we assume that for each $j$ the space $A_j\subset X$. This implies that we have the scale of inclusions 
\begin{align*}
A_j \subset 
\underset{i_{X,U}}{X \hookrightarrow U}
&\approx U^* 
\underset{i_{X,U}'}{\hookrightarrow} X^* \subset A^*_j, \\
A^{r,2}(U) \subset
\underset{i_{X,U}}{X \hookrightarrow U}
&\approx U^* 
\underset{i_{X,U}'}{\hookrightarrow} X^* \subset (A^{r,2}(U))^*.
\end{align*}
By duality every $u^*\in U^*\approx U$ defines an element of the dual space $i'_{X,U}u^*\in X^*$.
We write $U\approx U^*$ to denote the isometric isomorphism of the Hilbert space onto its dual space  that is given by the Riesz map $R_U$. In the few instances when we want to be explicit about the role of the Riesz map, we use the definition that  $u^*=R_Uu$ provided 
$$
\left <u^*,v \right >_{U^*\times U}:=\left < R_U u, v \right  >_{U^*\times U}=( u, v)_U
$$
for all $v\in U$. 
Since the projection  operators ${\Pi}_j$ map $U$ to the finite dimensional space $A_j$,  the dual approximation operators $\Pi'_j$ map $A_j^*$ to $U^*$.
Symbolically, we depict this relationship as 
$$
A_j
\underset{{\Pi}_j}{ \longleftarrow }U 
\approx U^* 
\underset{{\Pi}'_j}{\longleftarrow}  A^*_j.
$$

Our primary reason for using the Gelfand structure is that it gives a clear and  useful expression for the relationship between the duality pairing $\langle \cdot, \cdot \rangle_{X^*\times X}$ and the inner product $(\cdot,\cdot)_U$  on $U$. As derived in \cite{wloka}, the inner product on $U$ can be extended by continuity to represent the duality pairing on $X$ in the sense that we have 
$$
\left <g^*,f\right >_{X^*\times X} =\left < R_Ug ,f \right >_{X^*\times X}=(g , f)_U
$$
for all $f\in X$,  $g\in U$, and $g^*:=R_Ug\in U^*\subset X^*$.  As a consequence we see that the dual operators $\Pi'_j$ satisfy the relationship 
\begin{equation}
\left < {\Pi}_j' g^*, f  \right >_{X^*\times X}
=(g,{\Pi}_jf)_U \label{eq:gelfand_dual}
\end{equation}
whenever $g\in U$,  $g^*:=R_U g\in U^*$ and $f\in X$.

\subsection{Approximation with $\Pi_j'$, Priors Dual to  $A^{r,2}(U)$}
\label{sec:approx_measures2} 
With the Gelfand triple discussed in Section \ref{sec:gelfand}, we can introduce priors that describe the regularity of signed measures. The notion of regularity of signed measures is determined by duality to the  approximation spaces $A^{r,2}(U)$ in this section.   This choice is seen to imply a type of convergence of dual forms with respect to the  weak$^*$ topology on the measures. The rate of convergence of the approximations of the measures is then established by using the rates of convergence of approximations of functions in   the space $A^{r,2}(U)$.    

\begin{theorem}
\label{th:approx_meas}
Suppose $X$ is  a   Banach space, $U$  is a Hilbert space, and the pair $X\hookrightarrow U$  forms  a Gelfand triple. 
 Let the following conditions hold:
 \begin{enumerate}
 \item $A^{r,2}(U):=A^{r,2}(U,\{A_j\}_{j\in \mathbb{N}_0})\subseteq X$ is  an  approximation space for  $r>0$, 
 \item  $u^*=R_Uu$ for some $u\in U$ with $R_U:U\rightarrow U^*$ the Riesz map, and 
 \item $A_j\subset X$ for all $j\in \mathbb{N}_0$, and the  operator $\Pi_j$ is the $U$-orthogonal projection onto $A_j$.
 \end{enumerate}
 Then the sequence $\left \{{\Pi}'_j u^* \right\}_{j\in \mathbb{N}}$ converges  to $u^*\in U^*\subset X^*$ in the $\text{weak}^*(X^*,A^{r,2}(U))$  topology with the rate 
$$
|\langle (I-{\Pi}'_j)u^*,f\rangle_{X^*\times X}|\lesssim 2^{-rj}\|u\|_{U} \|f\|_{A^{r,2}(U)}
$$
for all $f\in A^{r,2}(U)$. 
\end{theorem}
\begin{proof}
We directly expand and bound the duality pairing using the property in Equation \ref{eq:gelfand_dual} for the Gelfand triple:
\begin{align*}
\left |\left <(I-{{\Pi}}_j')u^*,f \right >_{X^* \times X}\right | &= 
\left |\left <u^*, (I-{{\Pi}}_j)f \right >_{X^* \times X}\right | \\
& =\left | \left <R_U u, (I-{\Pi}_j)f \right >_{X^* \times X} \right |
= \left |  \left (u,(I-{\Pi}_j)f \right )_U\right | \\
 &\leq \|u\|_{U} \|(I-{{\Pi}}_j)f\|_{U}  \leq 2^{-rj} \|u\|_{U} \|f\|_{A^{r,2}(U)}.
\end{align*}
Note that the last line above follows from the approximation rate of functions in $A^{r,2}(U)$ described in Theorem \ref{th:th1}.
\end{proof}

Example \ref{ex:meas0} that follows  presents the rather straightforward case  when the measure $\nu$ to be approximated is given as $\nu(dx):=m(x)dx$ for a function $m\in L^2(\Omega)$ and $dx$ denoting ordinary Lebesgue measure.
Here we  assume additionally that  the domain $\Omega$ is compact,  and the approximation space $A^{r,2}(U)$ is densely and continuously  contained in $X:=C(\Omega)$, which in turn is densely and continuously embedded in $U:=L^2(\Omega)$. 
In summary then, the spaces that are used in this analysis define a  duality structure where 
\begin{align}
A_j \subset  C(\Omega) \hookrightarrow  L^2(\Omega) & \approx (L^2(\Omega))^* \hookrightarrow C^*(\Omega) \hookrightarrow   A_j^*, \notag \\
A^{r,2}(U) \hookrightarrow  C(\Omega) \hookrightarrow  L^2(\Omega)  &\approx (L^2(\Omega))^* \hookrightarrow C^*(\Omega) \hookrightarrow  (A^{r,2}(U))^* .
\label{eq:duality_1}
\end{align}
The requirement that $A^{r,2}(U)\subset C(\Omega)$  can be relaxed in this example, but when it holds it makes the analysis particularly straightforward.

\begin{framed}
\begin{example}[Approximation of Signed Measures,  Duality to $A^{r,2}(U)$]
\label{ex:meas0}
In Example \ref{ex:ex1} approximations of functions are defined in terms of piecewise constants. In this example, we consider a somewhat different  situation and construct   the finite dimensional approximations of functions from the spaces  $A_j:=\text{span}\left \{ \psi_{i,k}\ | \ 0\leq i\leq j, k\in \Gamma_i^\psi\right \}$ with $\psi_{i,k}$ the  dilates and translates of any $L^2(\Omega)-$orthonormal wavelets or multiwavelets that are smooth enough that they are contained in $C(\Omega)$. There are many such choices. We could choose the Daubechies compactly supported wavelets $^N\psi$ with { $N\geq 2$} \cite{daubechiesbook,daubechies88}, the Coiflets in \cite{coifman1,coifman2},  or any of the orthonormal, compactly supported multiwavelets described in of \cite{dgh} that are continuous.

To build our approximation of measures, we now carefully discuss how dual measures $\psi^{j,k}(dx)\in C^*(\Omega)$ are defined relative to the orthonormal functions $\psi_{j,k}\in U$. The connection between these bases is clear in this example owing to the orthonormality of the basis functions $\psi_{j,k}$. We summarize this argument in some detail  as it serves as the prototype of arguments for  more general  approximations of measures. 

Since $A_j$ is finite dimensional, its dual space $A_j^*$ is finite dimensional.  There exists a unique  basis $\left\{ \psi^{j,k}\ | \ 0\leq i \leq j, k\in \Gamma_i^\psi \right \}$ for $A_j^*\subset U^*$ that is dual to the  basis $\left \{ \psi_{j,k} \ | \ 0\leq i\leq j, k\in \Gamma_i^\psi \right \}\subset U^*$ with respect to the pairing $<\cdot,\cdot>_{U^*\times U}$. In fact we just have $\psi^{j,k}:=R_U\psi_{j,k}$ in this example since the dual basis is unique and  
$$
\delta_{(\ell,m),(j,k)}=(\psi_{\ell,m},\psi_{j,k})_U=\left < R_U\psi_{\ell,m},\psi_{j,k}\right>_{U^*\times U} :=\left < \psi^{\ell,m},\psi_{j,k}\right>_{U^*\times U}.
$$ 
By hypothesis we also know that $A^{r,2}(U)  \subseteq C(\Omega) \subseteq U:=L^2(\Omega)$ in this example.   Since $C(\Omega)\subset L^2(\Omega)\approx (L^2(\Omega))^* \subset rca(\Omega)=C^*(\Omega)$, any function in $L_\mu^2(\Omega)$ can be viewed as a signed  measure on $\Omega$. 
It is immediate that the biorthogonal basis function  $\psi^{j,k}\in U^*$
 defines a  measure
$$
\psi^{j,k}(d\xi):=\psi_{j,k}(\xi) d\xi \in C^*(\Omega):=rca(\Omega),
$$
and they  satisfy
$$
\langle \psi^{j,k}, \psi_{\ell,m}\rangle_{X^* \times X} = 
\int_\Omega \psi_{j,k}(\xi) \psi_{\ell,m}(\xi)d\xi = \delta_{(j,k),(\ell,m)}
$$
with $X:=C(\Omega)$. 
We then define the approximation operators $\Pi'_j$ by duality, 
\begin{align*}
\langle\Pi'_j\nu,f\rangle_{X^*\times X} &= \langle\nu,\Pi_jf\rangle_{X^*\times X} = \int_\Omega \sum_{0\leq i\leq j} \sum_{k\in \Gamma_i^\psi} (f,\psi_{i,k})_U \psi_{i,k}(y) \nu(dy) \\
&=\left < 
\sum_{0\leq i\leq j} \sum_{k\in \Gamma_i^\psi} \psi_{i,k} \int \psi_{i,k}(y)\nu(dy),f
\right>_{X^*\times X}.
\end{align*}
This means that 
\begin{align*}
(\Pi_j' \nu)(d\xi)&= \sum_{0\leq i\leq j} \sum_{k\in \Gamma_i^\psi} \int \psi_{i,k}(y)\nu(dy) \psi_{i,k}(\xi)d\xi \\
&=  \sum_{0\leq i\leq j} \sum_{k\in \Gamma_i^\psi} \int \psi_{i,k}(y)\nu(dy) \psi^{i,k}(d\xi)\\
&=\sum_{0\leq i\leq j} \sum_{k\in \Gamma_i^\psi} \nu_{i,k} \psi^{i,k}(d\xi).
\end{align*}
In other words the operators $\Pi'_j: X^* \subset  A_j^* \rightarrow U^* \subset X^* = rca(\Omega).$ Furthermore, we have from Theorem \ref{th:approx_meas} that 
$$
|\left < (I-\Pi'_j)\nu, f \right >_{X^*\times X}|
=|\left <\nu,(I-\Pi_j)f\right >_{X^*\times X}|\lesssim 2^{-rj} \|v\|_{U}\|f\|_{A^{r,2}(U)}
$$
for $\nu=R_U v\in U^*\subset X^*$, $v\in U$, and  all $f\in A^{r,2}(U)$.  This last inequality follows again from Theorem \ref{th:approx_meas}.
\end{example} 
\end{framed}

We must emphasize that  Theorem \ref{th:approx_meas} and Example \ref{ex:meas0} study the approximation of { \em signed measures}. Since we have that the probability measures $\mathbb{M}^{+,1}(\Omega)\subset C^*(\Omega)$, any probability measure can be approximated in this way. However, such an approximation $\nu_j:=\Pi_j'\nu$ of a probability measure $\nu$ is not guaranteed {\em a priori} to be a probability measure. 
\begin{framed}
\begin{example}[An  Approximation that is Not a Probability Measure]
\label{ex:not_PM}
Consider again the situation in Example \ref{ex:heat_eq}.  We define $U:=L^2(\mathbb{T}^1)$ and set 
\begin{align*}
A_j &:= \text{span} \left \{
u_{k,i} \ \biggl | \ i=1,2, \  k\leq j
\right \}, \\
A_j^*&:= \text{span}\left \{
u^{k,i}\ \biggl | \ u^{k,i}(\xi)d\xi
\right \}, \\ 
\Pi_jf&:= \sum_{i=1,2} \sum_{k\leq j} (f,u_{k,i})_U u_{k,i} ,\\
A^{(r,2)}(U)&:=A^{r,2}(L^2(\mathbb{T}^1); \{A_j\}_{j\in \mathbb{N}} ).
\end{align*}
Here the eigenfunctions are $u_{k,1}(x):=\cos(kx)/\sqrt(\pi)$, $u_{k,2}(x):=\sin(kx)/\sqrt{\pi}$.  By definition we know that 
$$
u^{k,i}(S)=\int_S u_{k,i}(\xi)d\xi.
$$
 Choose $\nu(dx)=\frac{1}{{\pi}}1_{[\pi,2\pi]}(x)(dx)$ and $S=[0,\pi]$. We then have 
 \begin{align*}
 (\Pi'_1\nu)(S)&=\sum_{k\leq 1} \sum_{i=1,2} \int_{[0,2\pi]} u_{k,i}(d\xi)\nu(d\xi) u^{k,i}(S),\\
 &=\frac{1}{\pi^{2}}\left \{
 \int_{[\pi,2\pi]}\cos x dx  \int_{[0,\pi]}\cos x dx +  \int_{[\pi,2\pi]}\sin x dx \int_{[0,\pi]}\sin x dx
  \right \} ,\\
  &=\frac{1}{\pi^{2}} \left \{ \left (\sin x\biggl |_{[\pi,2\pi]}\right ) \left ( \sin x\biggl |_{[0,\pi]} \right ) +  \left (\cos x\biggl |_{[\pi,2\pi]}\right ) \left ( \cos x\biggl |_{[0,\pi]} \right )\right \},\\
  &=-\frac{4}{\pi^{2}}.
 \end{align*}
We conclude that while  $\nu$ is a probability measure, $\Pi_j'\nu \not \in \mathbb{M}^{+}$ and is not a probability measure. 
\end{example}
\end{framed}

\noindent 
In view of this example, additional work is needed to construct direct estimators of probability measures that generate probability measures.  One strategy is  to construct approximations that yield probability measures by  solving constrained optimization problems as in Reference \cite{kordameasure}. 
We return to this issue in Section \ref{sec:approx_PM} and describe an alternative approach that is amenable to the derivation of associated rates of approximation.

In Example \ref{ex:meas0} it is  assumed that the measure to be approximated  has  a density so that $\nu(dx):=m(x)dx$ for some function $m\in L^2_\mu(\Omega)$ and $dx$ denoting Lebesgue. In this Gelfand triple structure this corresponds to the statement that $\nu=R_u m$ for some $m\in U$. In the next example  we consider the case when we only know that $\nu\in C^*(\Omega)$.  Specifically, we study the choice $\nu(dx):=\delta_{w(x)}$ for a sufficiently smooth function $w:\Omega \rightarrow \Omega$.

%
%
\medskip
\begin{framed}
\begin{example}[Approximations of Signed Measures and  $\mathcal{U}_j,\mathcal{P}_j$]
\label{ex:meas1}
We now return to Example \ref{ex:warp2} and consider the approximation of the Perron-Frobenius operator or Koopman operator,  and associated approximations of measures for the canonical case when $(\mathcal{U}f)(x)=f(w(x))$.  We begin by reviewing the relationship between the Perron-Frobenius operator, the Koopman operator, and the probability measure $\delta_{w(x)}$ in this case.   We then derive rates of convergence for  approximations of the measure $\delta_{w(x)}$, and subsequently we construct the finite dimensional operator approximations 
\begin{align*}
\mathcal{P}_j&:= \Pi_j' \mathcal{P}\Pi_j': X^* \rightarrow X^*, \\
\mathcal{U}_j&:= \Pi_j \mathcal{U} \Pi_j:X\rightarrow A_j \subset X.
\end{align*}
We will see that these operators coincide with one that is constructed in terms of the approximations $\Pi_j'\delta_{w(x)}$ of the Dirac measure $\delta_{w(x)}.$

\medskip 
\noindent{\textbf{Duality expressions induced by $w:\Omega \rightarrow \Omega$:}}\newline 
First we review expressions for the dual pairing. By duality we know that 
\begin{align*}
\langle \mathcal{P}\nu , f  \rangle_{X^* \times X}&= \langle \nu , \mathcal{U}f \rangle_{X^* \times X} = \int_\Omega (f\circ w)(\xi) \nu(d\xi).
\end{align*}
The Koopman operator in this problem can be re-written as 
\begin{align*}
(\mathcal{U}f)(x):=f(w(x)) = \int_\Omega \delta_{w(x)}(d\xi)f(\xi),
\end{align*}
and it follows from  
\begin{align*}
\left < \nu, \mathcal{U}f\right >_{X^*\times X}&= \int_{\Omega}\nu(dx) \int_\Omega \delta_{w(x)}(d\xi)f(\xi)=\int_\Omega\left (  \int_\Omega \delta_{w(x)}(d\xi)\nu(dx) \right) f(\xi)\\
&=\left <\mathcal{P}\nu,f\right >_{X^*\times X}
\end{align*}
that 
$$
(\mathcal{P}\nu)(d\xi):=\int_\Omega \nu(dx) \delta_{w(x)}(d\xi).
$$
We next construct approximations of the measure $\delta_{w(x)}$, which is subsequently  used to construct the approximate operators $\mathcal{U}_j$ and $\mathcal{P}_j$. 

\medskip
\noindent {\bf Approximation of the measure $\delta_{w(x)}$:} \newline 

 We assume, again, that the duality structure takes the form in Equation \ref{eq:duality_1}. We choose $U:=L^2(\Omega)$ and $X:=C(\Omega)$.
By assuming that $f\circ w \in C(\Omega)\subset L^2(\Omega)$,  we compute the expansion 
$$
\left < \mathcal{P} \nu , f\right >_{X^*\times X} = \int_\Omega (f\circ w)(\xi)\nu(dx) 
$$
in terms of the warped wavelets $\tilde{\psi}_{j,k}$ as discussed in Example \ref{ex:warp2}. We obtain  
\begin{align*}
\int_\Omega (f\circ w)(\xi) \nu(d\xi)&=\int_\Omega \left ( 
\sum_{j\in \mathbb{N}_0}
\sum_{k\in \Gamma^\psi_j}
\left (f, \tilde{\psi}_{j,k}\right )_{L^2_{\tilde{\mu}}(\tilde{\Omega})} \psi_{j,k}(\xi)
\right )\nu(d\xi), \\
&=  
\sum_{j\in \mathbb{N}_0}
\sum_{k\in \Gamma^\psi_j}
\left (f, \tilde{\psi}_{j,k}\right )_{L^2_{\tilde{\mu}}(\tilde{\Omega})} \int_\Omega \psi_{j,k}(\xi)
\nu(d\xi),\\
&=\sum_{j\in \mathbb{N}_0} \sum_{k\in \Gamma^\psi_j} ({f\circ w})_{j,k} \nu_{j,k}=
\sum_{j\in \mathbb{N}_0} \sum_{k\in \Gamma^\psi_j} \tilde{f}_{j,k} \nu_{j,k},
\end{align*}
with $\tilde{f}_{j,k}:=\left (f,\tilde{\psi}_{j,k}\right)_{L^2_{\tilde{\mu}}(\tilde{\Omega})}$ and $\nu_{j,k}:=\int \psi_{j,k}(\xi)\nu(d\xi)$. 
We can generate an approximation of the measure $\delta_{w(x)}$ via the operator $\Pi_j'$ and obtain 
\begin{align*}
(\Pi_j'\delta_{w(x)})(d\eta)&:=\sum_{0\leq i \leq j}\sum_{k\in \Gamma_i^\psi} \int_\Omega \psi_{i,k}(y) \delta_{w(x)}(dy) \psi_{i,k}(\eta)d\eta\\
&=\sum_{0\leq i \leq j}\sum_{k\in \Gamma_i^\psi}
\psi_{i,k}(w(x))\psi^{i,k}(d\eta)=\sum_{0\leq i \leq j}\sum_{k\in \Gamma_i^\psi} \tilde{d}_{i,k}(x)\psi^{i,k}(d\eta)
\end{align*}
with $\tilde{d}_{j,k}(x):=\psi_{j,k}(w(x))$. 
We therefore obtain the dual expansions
\begin{align*}
\left < \Pi'_j \delta_{w(x)}, f \right >_{X^*\times X}
&=\left < \delta_{w(x)},\Pi_j f\right>_{X^*\times X}= \sum_{0\leq i \leq J} \sum_{k\in \Gamma_i^\psi}
\tilde{d}_{i,k}(x)f_{i,k}.
\end{align*}
The weak$^*$ convergence rate for the approximation of $\delta_{w(x)}$ is expressed in terms of the duality pairing as  
\begin{align*}
\left | \left<(I-\Pi'_j)\delta_{w(x)},f\right >_{X^*\times X}\right |
&=
\left | 
\left<\delta_{w(x)},(I-\Pi_j)f\right >_{X^*\times X}
\right |, \\
&\leq  \|\delta_{w(x)}\|_{X^*} \|(I-\Pi_j)f\|_{X}, \\
& \lesssim \|(I-\Pi_j)f\|_{A^{r,2}(U)}     
\lesssim 2^{-(s-r)j}\|f\|_{A^{s,2}(L^2(\Omega))},
\end{align*}
for all $f\in A^{s,2}(L^2(\Omega))$ with $s>r>0$. 
Note the difference between the rate of convergence that depends on $2^{-(s-r)}$ in this example and that in 
Theorem \ref{th:approx_meas} which is bounded by $2^{-rj}$. Theorem \ref{th:approx_meas} relies on the fact that $\nu\in U^*$, which enables the use of the Gelfand triple to derive that the error rate is $O(2^{-rj})$. In this case, the measure $\nu:=\delta_{w(x)}\not \in {U}^*$. In other words there is no $L^2(\Omega)$ function $m$ such that $\delta_{w(x)}:=R_U m$.   We only know that $\delta_{w(x)}\in C^*(\Omega):=X$, and as measured by the duality structure, the measures in $C^*(\Omega)$ are less regular than those in $U^*$. We expect the convergence rate of approximations of measures   $\nu \in C^*(\Omega)$ to be lower than those $\nu\in U^*\subset C^*(\Omega)$.  

\medskip
\noindent {\textbf{The approximate operators $\mathcal{P}_j$ and $\mathcal{U}_j$}:} \newline 
\noindent With the expression for the approximations of the measure $\delta_{w(x)}$, we define the approximation $\mathcal{P}_j$ as 
$$
(\mathcal{P}_j \nu)(d\xi) = \int \left (\Pi'_j \delta_{w(x)} \right)(d\xi) \left ( \Pi_j'\nu \right)(dx),
$$
and the operator $\mathcal{U}_j$ is defined by duality. In fact, we find that $\mathcal{P}_j:=\Pi'_j \mathcal{P} \Pi_j'$ and $\mathcal{U}_j:=\Pi_j \mathcal{U} \Pi_j$. 
To see why this is so, we can rewrite the approximate Perron-Frobenius operator 
\begin{align*}
\left < \mathcal{P}_j \nu,f \right >_{X^*\times X} &=
\int \left ( \Pi'_j \nu\right )(dx) \left ( \int \left ( \Pi'_j \delta_{w(x)}\right )(d\xi)f(\xi) \right ) \\
&= \int \left (\Pi'_j \nu  \right )(dx) \left <\delta_{w(x)},\Pi_j f \right >_{X^*\times X} \\
&= \left < \Pi_j' \nu, \left < \delta_{w(\cdot)}, \Pi_jf\right >_{X^*\times X}\right >_{X^*\times X} \\
&=\left <\nu, \Pi_j \mathcal{U} \Pi_j f \right >_{X^*\times X} = \left <\Pi'_j\mathcal{P} \Pi'_j \nu,f \right >_{X^*\times X}
\end{align*}
The approximate Koopman operator $\mathcal{U}_j$ and approximate Perron-Frobenius operator $\mathcal{P}_j$ are dual operators since 
\begin{align*}
\left <\mathcal{P}_j \nu, f \right >_{X^*\times X} &= 
\left <\Pi_j' \mathcal{P} \Pi_j' \nu, f \right >_{X^*\times X}= \left < \nu, \Pi_j \mathcal{P}' \Pi_j \right >_{X^*\times X} \\
& =\left <\nu, \Pi_j \mathcal{U} \Pi_j \right >_{X^*\times X} = \left < \nu, \mathcal{U}_j f\right >_{X^*\times X}. 
\end{align*}
We will see that the derivation of a rate of convergence for  $\mathcal{U}_j$ in a strong sense will induce a similar rate of convergence in a weak$^*$ sense for $\mathcal{P}_j$. 
If  $\mathcal{U}\in \mathcal{L}(A^{r,2}(L^2(\Omega)))$, we have the bound 
\begin{align*}
\left \|(\mathcal{U}-\mathcal{U}_j)f \right \|_{L^2(\Omega)}
&\le \|(I- \Pi_j)\mathcal{U}f  \|_{L^2(\Omega)} + \left \| \Pi_j \mathcal{U}(I-\Pi_j)f\right \|_{L^2(\Omega)}, \\
&\lesssim 2^{-rj}\|\mathcal{U}f\|_{A^{r,2}(L^2(\Omega))}+ \|\Pi_j\| \|\mathcal{U}\|_{0} \|(I-\Pi_j)f\|_{L^2(\Omega)},\\
&\lesssim 2^{-rj}\left (\|\mathcal{U}\|_{r} + \|\mathcal{U}\|_0 \right )\|f\|_{A^{r,2}(L^2(\Omega))}, 
\end{align*}
for any $f\in A^{r,2}(L^2(\Omega))$. 
We obtain the  weak$^*$ rate of convergence 
\begin{align*}
\left |\left < (\mathcal{P}-\mathcal{P}_j)\nu,f \right >_{X^*\times X} \right | &= \left |\left < \nu, (\mathcal{U}-\mathcal{U}_j)f \right >_{X^*\times X} \right | \\
& \lesssim  2^{-rj} \|\nu\|_{X^*} \|f\|_{A^{r,2}(L^2(\Omega))}
\end{align*}
for all $f\in A^{r,2}(L^2(\Omega))$ and $\nu \in X^*$. 
\end{example}
\end{framed}
\begin{framed}
\begin{example}[Approximation of $\delta_{w_\alpha(x)}$ with $w_\alpha(x):=x^\alpha$]
Here we return to the problem studied in Example \ref{ex:warp3}, but now approximate the measure $\delta_{w_\alpha(x)}$.  
We choose the sequence of finite dimensional approximation spaces as 
\begin{align*}
A_j&:=\text{span}
\left \{ 1_{\square_{j,k}}
\ \biggl | \ k\in  \Gamma_{j}^\phi \right \}=\text{span}\left \{ \phi_{j,k} \ \biggl | \ k\in \Gamma_j^\phi \right
\},\\
&=
\text{span}\left\{
\psi_{i,k}\ \biggl | \ 0\leq i \leq j-1, k\in \Gamma_{i}^\psi\right \},
\end{align*}
with $\phi_{j,k}$ and $\psi_{j,k}$ the Haar scaling functions  and wavelets for $d=1$ introduced in Example \ref{ex:haar_basis}.  
We then have 
\begin{align*}
(\Pi'_j \delta_{x^\alpha})(d\xi) &= \sum_{0\leq i\leq j-1} \sum_{k\in \Gamma_i^\psi} \int_\Omega \psi_{j,k}(\eta) \delta_{x^\alpha}
(d\eta) \psi^{j,k}(d\xi), \\
&= \sum_{0\leq j-1} \sum_{k\in \Gamma_{i^\psi}}
\psi_{j,k}(x^\alpha) \psi_{j,k}(\xi)d\xi.
\end{align*}
The action on any $f\in A^{r,2}(L^2(\Omega))$ is given by 
\begin{align*}
\left < \Pi'_j \delta_{x^\alpha},f \right >_{C^*(\Omega)\times C(\Omega)} & = 
\left <  \delta_{x^\alpha}, \Pi_jf \right >_{C^*(\Omega)\times C(\Omega)}\\
&=
\sum_{0\leq i<j} \sum_{k\in \Gamma_i^\psi}
\psi_{j,k}(x^\alpha)
\int_{\Omega} \psi_{j,k}(\xi)f(\xi)d\xi \\
&= 
\sum_{0\leq i<j} \sum_{k\in \Gamma_i^\psi}
\psi_{j,k}(x^\alpha)
(f,\psi_{j,k})_{L^2(\Omega)}. 
\end{align*}
We then compute the error bound 
\begin{align*}
\left | \left <(I-\Pi'_j)\delta_{x^\alpha},f \right >_{C^*(\Omega)\times C(\Omega)}\right |&= \left |
\left <\delta_{x^\alpha},(I-\Pi_j)f \right >_{C^*(\Omega)\times C(\Omega)}\right |\\
&\leq 2^{-(s-r)j} 
 \|f\|_{A^{r,2}(L^2(
\Omega))}.
\end{align*}

As a second related example, suppose that $m(x)=x^\alpha$ and $\mu(dx):=m(x)dx$. In this case we have 
\begin{align}
(\Pi_J'\mu)(d\xi)&= \sum_{0\leq i<j} \sum_{k\in \Gamma_i^\psi} \left (  x^\alpha, \psi_{j,k}  \right )_{L^2(\Omega)} \  \psi_{j,k}(\xi)d\xi,
\label{eq:x_alpha_psi}
\end{align}
and the action on a function $f$ is given by 
\begin{align*}
\left | \left <(I-\Pi'_j)\mu,f \right >_{C^*(\Omega)\times C(\Omega)}\right |
= \left |
\left <\mu,(I-\Pi_j)f \right >_{C^*(\Omega)\times C(\Omega)}\right |
\\
\leq n_j^{-r}
\| \mu \|_{TV} \|f\|_{A^{r,2}(L^2(\Omega))}
\end{align*}
for all $f\in A^{r,2}(L^2(\Omega))$.

\end{example}
\end{framed}

%
%

\subsection{Approximation with Dual Operators  $\tilde{\Pi}_j'$, Priors in $A^{r,q}(X)$}
\label{sec:equiv_metric}

Theorem \ref{th:approx_meas} characterizes rates   of convergence of approximations $\Pi_j \nu$ of a signed measure $\nu$ in terms of duality statements relative to  priors  in $A^{r,2}(U)$. In  Theorem \ref{th:dual_conv} we generalize this analysis and study dual convergence relative to priors in the  space $A^{r,q}(X)$ with  $X$ a Banach space.
Also, Theorem \ref{th:dual_conv} can be used to guarantee rates of convergence of a some approximations of probability measures in the bounded Lipschitz metric. 
For many applications, it is required that approximation of a probability measure  in fact generates a probability measure, not just a signed or positive measure. 
When $\nu$ is a probability measure, Theorem \ref{th:dual_conv}   is used to develop a general strategy  to define approximations $\tilde{\Pi}_j\nu$ that are  themselves a  probability measure.  

\subsubsection{Approximations of Signed Measures}
\label{sec:approx_SM}
There are a large number of expressions for a metric  on the probability measures  $\mathbb{M}^{+,1}(\Omega)$.   These include the Wasserstein $W^p$,  Kolmolgorov,  Levy, Kantorovich, Ky Fan, and bounded Lipschitz metrics, among others. \cite{dudley,rachev}
Some of these metrize the  induced  weak$^*$ topology that $\mathbb{M}^{+,1}(\Omega)$ inherits as a subset of the signed measures   $C^*(\Omega)$. In this  section, we show that Theorem \ref{th:approx_meas} implies convergence in 
the bounded Lipschitz metric $d_{BL}$. This   is a popular metric  that induces the  weak$^*$ topology on $\mathbb{M}^{+,1}(\Omega)$.   We assume throughout this section that the domain $\Omega$ is compact. 
The metric $d_{BL}$ is defined in terms of the bounded  Lipschitz functions by the expression 
$$
d_{BL}(\mu,\nu)=\sup_{\|f\|_{BL}\leq 1} |\left <\mu-\nu,f \right >_{X^*\times X}|
$$
with again $X:=C(\Omega).$ 
In references that discuss probability measures, the bounded Lipschitz norm is usually  designated $\|f\|_{BL}$, and it is synonymous with our notation $\|f\|_{BL} \approx  \|f\|_{\text{Lip}(1,C(\Omega))}:=\|f\|_{C(\Omega)}+|f|_{\text{Lip}(1,C(\Omega))}$. We use the  notation $\|f\|_{BL}$  in this section to adhere to the more common notation in measure and probability theory. \cite{dudley} The generalization of the   Theorem \ref{th:approx_meas},  one that employs duality with respect to an arbitrary approximation space  $A^{r,q}(X)$ for a Banach space $X$,  is obtained by introducing   a family of near-best approximation operators $\{\tilde{\Pi}_j\}_{j\in \mathbb{N}}$.
\begin{theorem}[Dual Convergence of Measures]
\label{th:dual_conv}
Let $X$ be a Banach space,  suppose that a  family of    subspaces $\{A_j\}_{j\in \mathbb{N}}\subset X$  define the approximation space  $A^{r,q}(X)$   with $r>0$ and $1\leq q\leq \infty$ as in Section \ref{sec:approx_spaces}, and set  $n_j:=\#A_j$. Suppose that  $\{\tilde{\Pi}_j\}_{j\in \mathbb{N}}$ be a family of uniformly bounded linear projection  operators  $\tilde{\Pi}_j:X\rightarrow A_j$ that are onto $A_j$ for each $j\in \mathbb{N}$.  When $\{n_j\}_{j\in \mathbb{N}}$  is a quasigeometric series, we  then   have 
\begin{equation}
\|(I-\tilde{\Pi}_{n_j})f\|_{X}\lesssim n_j^{-r} \|f\|_{A^{r,q}(X)}
\label{eq:dc_1}
\end{equation}
and 
\begin{equation}
\left |\left <(I-\tilde{\Pi}'_{n_j})\nu , f \right >_{X^*\times X}\right |\lesssim n_j^{-r} \|\nu\|_{X^*} \|f\|_{A^{r,q}(X)}
\label{eq:dc_2}
\end{equation}
for all $j\in \mathbb{N}$,  $\nu \in X^*$, and $f\in A^{r,q}(X)$.    
\end{theorem}

\begin{proof}
The approximation error bound in the Banach space $X$ in Equation \ref{eq:dc_1} above is well known and can be found in many places including \cite{piestch1981approxspaces,piestch1982tensorproducsequenfunctoperat,devorenonlinear}. For completeness we derive the result here in our proof of Equation \ref{eq:dc_2}. 
To begin, we can bound the duality pairing 
\begin{align*}
\left |\left < (I - \tilde{\Pi}_j')\nu,f\right>_{X^*\times X}\right |&\lesssim \left |\left < \nu, (I - \tilde{\Pi}_j)f\right>_{X^*\times X}\right | \\
& \leq \|v\|_{X^*} \|(I-\tilde{\Pi}_j)f\|_{X}.
\end{align*}
When $f\in A^{r,q}(X)$, we have 
\begin{align*}
\|f-\tilde{\Pi}_jf\|_X&=\|f-\tilde{\Pi}_j a + \tilde{\Pi}_j a - \tilde{\Pi}_j f \| \leq \|f-a\|_X + \| \tilde{\Pi}_j\| \|f-a\| \\
&\leq (1+C)\|f-a\|_X
\end{align*}
for all $a\in A_j$, 
and we can conclude that $\|f-\tilde{\Pi}_jf\|_X \lesssim E_{j}(f,X)$. 
The  representation Lemma on page 336 of   \cite{piestch1982tensorproducsequenfunctoperat}  states that we have the equivalent norms 
$
A^{r,q}(X) \approx \left \| \{n_j^r a_{n_j}(\cdot,X) \}\right \|_{\ell^q}
$
with $a_j(f,X):=E_{j-1}(f,X)$ the $j^{th}$ approximation number of the function $f$ in the Banach space $X$. This means that 
$$
\sum_{j\in \mathbb{N}_0} (n_j^r a_{n_j}(f,X) )^q \approx \|f\|_{A^{r,q}(X)}^q,
$$
and therefore $a_{n_j}(f,X) \lesssim n_{j}^{-r}\|f\|_{A^{r,q}(X)}.$
Since the approximant spaces are nested,  we see that 
$$
E_{n_j}(f,X)\leq E_{n_j-1}(f,X):= a_{n_j}(f,X) \lesssim  n_j^{-r} \|f\|_{A^{r,q}(X)},
$$
and the Theorem holds.   
\end{proof}
In summary then, Theorem \ref{th:dual_conv} establishes that the sequence $\{ \tilde{\Pi}'_j\nu\}_{j\in \mathbb{N}_0}$ converges in $\text{weak}^*(X^*,A^{r,q}(X))$, and the weak${^*}$ rate is $O(n_j^{r})$. 

For the Hilbert space $U$, the norm of the error in the orthogonal projection  $\|(I-\Pi_{n_j})f\|_U$ is estimated  as in Theorem \ref{th:th1}, while in a  Banach space $X$ the approximation error $\|(I-\tilde{\Pi}_{n_j})f\|_X$ is studied in Theorem  \ref{th:dual_conv}.  
Before proceeding with our discussion of approximation of measures, we illustrate the efficacy of the approximation error bound for functions given in  Equation \ref{eq:dc_1} in an example application.   

\medskip 

\begin{framed}
\begin{example}[Near Best Approximations in $A^{r,q}(L^p(\Omega))$ and $A^{r,q}(C(\Omega))$]
\label{ex:two_cases}
In this example, we examine a case where  the first conclusion in Example \ref{th:dual_conv} in Equation \ref{eq:dc_1} holds for  approximant spaces $\{A_j\}_{j\in \mathbb{N}}$ in $A^{r,q}(L^p(\Omega))$. 
 We define the approximant 
 spaces 
 $$
A_j:=\text{span}\left \{1_{\square_{j,k}} \ | \ k\in \Gamma_{j}^\phi \right \},  
$$
which are the same as those built from classical Haar scaling functions and wavelets in Example \ref{ex:haar_basis}. 
 When we want to use the approximant spaces $\{A_j\}_{j\in \mathbb{N}}$  to construct approximations in  $A^{r,q}(L^p(\Omega))$, we define the family of approximation operators 
$$
(\tilde{\Pi}_jf)(x):= \sum_{k\leq n_j} \frac{1}{\int_\Omega 1_{\square_{j,k}}(\xi)d\xi}\int_\Omega f(\xi) 1_{\square_{j,k}}(\xi) d\xi \cdot 1_{\square{j,k}}(x),
$$
so that $\tilde{\Pi}_jf|_{\square_{j,k}}$ is the average of $f$ over  $\square_{j,k}$. Clearly each $\tilde{\Pi}_j$ is a linear projection onto $A_j$. From Equation 3.12 in \cite{devorenonlinear} we have 
$$
\|(I- \tilde{\Pi}_j)f\|_{L^p(\Omega)} \lesssim 
n_j^{-r}|f|_{\text{Lip}(r,L^p(\Omega))}
$$
for $1\leq p\leq \infty$. 
{ This result again follows directly from  Theorem \ref{th:dual_conv} and  Theorem \ref{th:approx_besov} using the fact that for $0<r<1/2$ we have 
$$
A^{r,\infty}(L^p(\Omega))=\text{Lip}(r,L^p(\Omega)).
$$
}

Before we close this example, we carry out an  error analysis using a different family of approximation operators $\tilde{\Pi}_j$. 
Denote the center of each $\square_{j,k}$  by $\xi_{j,k}$. We introduce the family of associated projection operators 
$$
(\tilde{\Pi}_jf)(x) \sum_{k\leq n_j} f(\xi_{j,k}) 1_{\square_{j,k}}(x).
$$
Each  $\tilde{\Pi}_j:C(\Omega)
\rightarrow A_j$ is onto $A_j$. 
The family of approximation operators $\{\tilde{\Pi}_j\}_{j\in \mathbb{N}}$ is uniformly bounded
as maps from $C(\Omega)\rightarrow L^\infty(\Omega)$ 
 since  we have  
\begin{align*}
\|\tilde{\Pi}_j f\|_{L^\infty(\Omega)} =\sup_{ x \text{ a.e. } \in \Omega} \left | \sum_{k\leq n^j} f(\xi_{j,k}) 1_{\square_{j,k}}(x)  \right |\leq \|f\|_{C(\Omega)}.
\end{align*}
From first principles we can conclude  that 
\begin{align}
\|(I-\tilde{\Pi}_j)f\|_{L^\infty(\Omega)} & = \sup_{x \text{ a.e. } \in \Omega} \left |f(x)- \sum_{k\leq n^j} f(\xi_{j,k}) 1_{\square_{j,k}}(x)\right |  \notag \\
&\leq \max_{k\leq n_j} \sup_{x \text{a.e. } \in \square_{j,k}}\left|f(x)-f(\xi_{j,k})\right |\notag \\
&\leq \left ( 
\frac{1}{2} (2^{-j})\right)^{\alpha} \|f\|_{\text{Lip}(r,C(\Omega))} \lesssim 2^{-jr} \|f\|_{\text{Lip}(r,C(\Omega))}. \label{eq:err_UB}
\end{align}
This  final inequality yields the same rate of approximation $O(2^{-rj})\approx O(n_j^{-r})$  for a continuous function in $f\in C(\Omega)$. 

\end{example}
\end{framed}

\begin{framed}
\begin{example}[Example of $(\mathcal{U}f)(x)
:=(f\circ w_\alpha)(x)=x^\alpha$ ]
\label{ex:warp3}
In this example we give straightforward, concrete case of  Examples \ref{ex:warp1} and \ref{ex:warp2}.
We study the case that we are given a deterministic dynamics on the set $[0,1]$ that is generated by the recursion
$$
x_{n+1}=w_\alpha(x_n):=x_n^\alpha
$$
over $[0,1]$ for a fixed $0<\alpha\leq 1$.  This example is selected since it provides a good study of how the rate of approximation of the Koopman operator depends on its action on approximation spaces.  We know that the Koopman operator satisfies 
$$
(\mathcal{U}f)(x)=f(w_\alpha(x))=f(x^\alpha) = \int_{\Omega}\delta_{w_\alpha(x)}(d\xi) f(\xi) = \int_\Omega \delta_{x^\alpha} (d\xi) f(\xi).
$$
The family of functions $w_\alpha$ 
for $0<\alpha\leq 1$ is shown in Figure \ref{fig:x_alpha}.
Intuitively it seems natural to say that the function   $w_\alpha$  exhibits a singularity at $x=0$. 

\medskip
\begin{center}
\includegraphics[width=.5\textwidth]{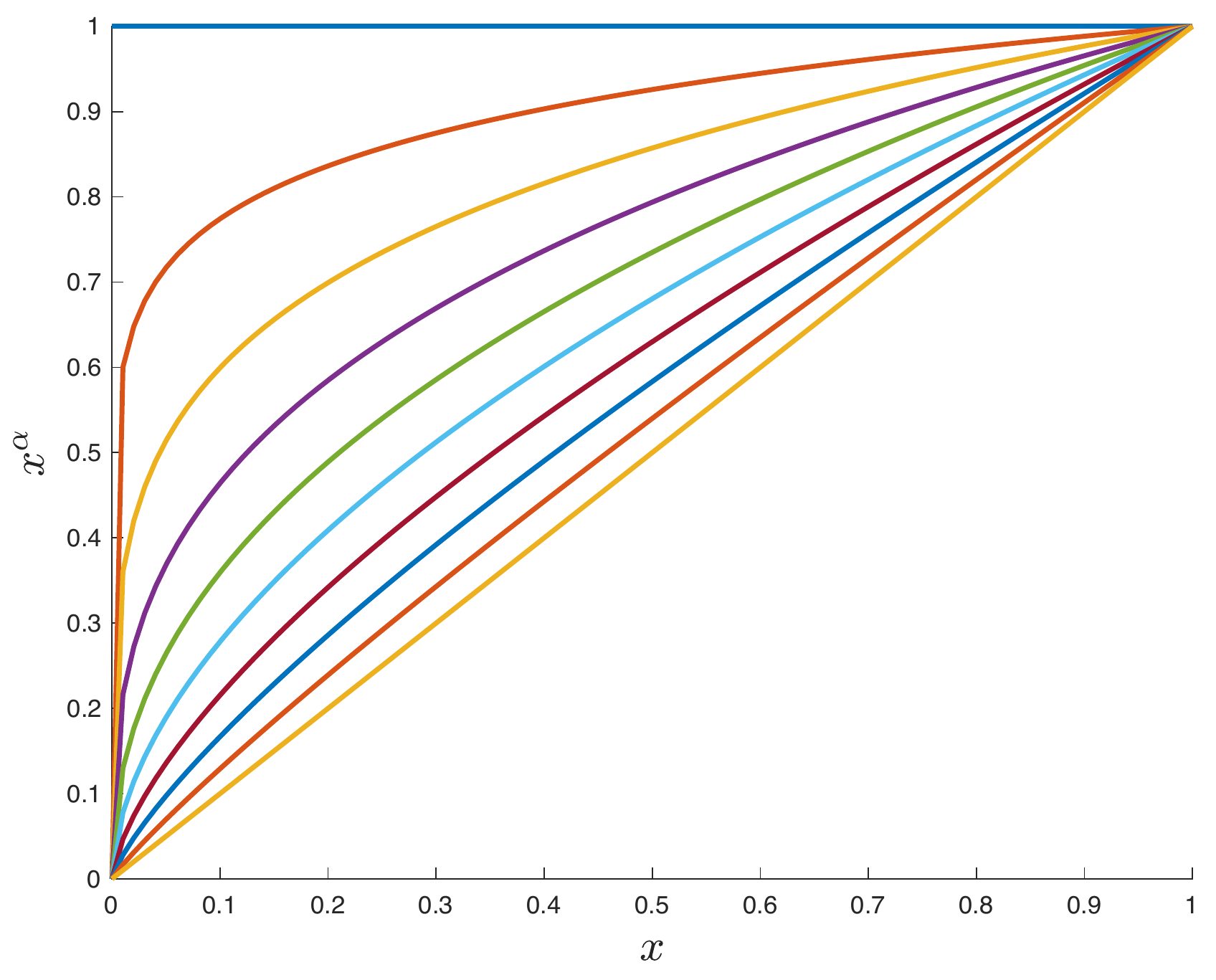}
\captionof{figure}{Plots of $x^\alpha$ for $0<\alpha\leq 1$.}
\end{center}
\label{fig:x_alpha}

\noindent 
We expect that the approximation of the Koopman operator will depend on the severity of the singularity. We show that the definition of priors in terms of approximation spaces enables a rigorous and specific description of this dependence. 

\subsection*{Smoothness in $\text{Lip}(s,C(\Omega))$}
Initially, we might choose to model the dynamics by choosing a Lipschitz space $Lip(r,C(\Omega))$, since this  smoothness space is one of the simpler to understand.  Even though this choice is not of much practical use for studying this deterministic recursion, it is illustrative of how smoothness of the operator $\mathcal{U}$, in terms of its action on approximation spaces, plays a critical role in building approximations.  The derivative $dw/dx$ is uniformly bounded on every subset of the form $[\epsilon,1]:=\Omega_\epsilon$ for $0<\epsilon<1$, but unbounded at $x=0$. We see that $f\not \in \text{Lip}(1,C(\Omega))$.  However,  $w_\alpha \in \text{Lip}(1,C(\Omega_\epsilon))$ for each $\epsilon\in (0,1)$.  Now,  for any $f\in \text{Lip}(r,C(\Omega))$ with $0<r\leq 1$, we have 
\begin{align*}
    |(\mathcal{U}f)(x) - (\mathcal{U}f)(y)|&= 
   | (f\circ w_\alpha)(x) - (f\circ w)(y)|, \\
   &\leq |f|_{\text{Lip}(s,C(\Omega))}
   |w_\alpha(x)-w_\alpha(y)|^r, \\
   &\leq |f|_{\text{Lip}(s,C(\Omega))}
   \left (|w_\alpha|_{\text{Lip}(1,C(\Omega_\epsilon))}|x-y|\right )^r, \\
   & \lesssim |x-y|^r
\end{align*}
for all $x,y\in \Omega_\epsilon$. 
This inequality shows that $
\mathcal{U}f\in {\text{Lip}(r,C(\Omega_\epsilon))}$ for any $0<\epsilon<1$ and $0<r\leq 1$. The operator $\mathcal{U}$ maps between the spaces 
$$
\mathcal{U}:\text{Lip}(r,C(\Omega)) \rightarrow \text{Lip}(r,C(\Omega_\epsilon)).
$$
 As discussed in \cite{devorenonlinear}, it is known that $Lip(r,C(\Omega_\epsilon))$ for $0<r<1$ is the {\em linear } the interpolation space between the spaces $C(\Omega_\epsilon)$ and $\text{Lip}(1,C(\Omega_\epsilon))$. In Example \ref{ex:two_cases} we have shown from first principles that $f\in \text{Lip}(r,C(\Omega_\epsilon))$ implies that the operators $\tilde{\Pi}_j$ onto the piecewise constants defined by point evaluation yield errors like $\|(I-\tilde{\Pi}_j)f\|_{C(\Omega_\epsilon)} \approx O(2^{-rj})$. Following essentially the same steps as in Example \ref{ex:two_cases} for the difference $(\mathcal{U}-\mathcal{U}_j)f$ yields an error bound $O(2^{-rj})$.  Reference \cite{devorenonlinear} shows that in fact a uniformly continuous  function is approximated via a linear method from the family of piecewise constants with a rate $n_j^{-r}$  if and only if   $f\in \text{Lip}(r,L^\infty(\Omega_\epsilon))$. Since in this example $\Omega_\epsilon$ is compact, the result applies here,  of course,  with $n_j\approx 2^j$. 
We also note that the choice of the space $\text{Lip}(r,C(\Omega))$ does prove useful for  approximations of some related dynamical systems, including stochastic evolutions induced by  Markov chains as  discussed in Example \ref{ex:meas2}. 

\subsection*{Smoothness in $\text{Lip}(s,L^p(\Omega))$}
The analysis above omits the origin since $w$ exhibits a singularity there, so it can not be used to analyze discrete evolutions over $[0,1]$. We could further expand our analysis above using the smoothness space $\text{Lip}(r,C(\Omega_\epsilon))$, but it can be fruitful to  
demonstrate another metric for smoothness in this case. 
We now study this  case using the Lipschitz space $\text{Lip}(r,L^p(\Omega))$ defined in terms of the $L^p$-integrated Lipschitz inequalities.  In particular, from \cite{devorenonlinear} page 66 we know that 
$w_\alpha\in \text{Lip}(\alpha,L^\infty(\Omega))$, but for no higher index.  In this sense, membership in the space $\text{Lip}(\alpha,L^\infty(\Omega))$ describes the strength of the singularity at the origin of $w_\alpha(x):=x^\alpha$. 

Now suppose the $f$ is any function in $\text{Lip}(s,C(\Omega))$ for $0<s\leq 1$. We see that  \begin{align*}
    \|(\mathcal{U}f)(\cdot+h)-\mathcal{U}f\|_{L^p(\Omega)}^p &= \int 
    \left |(f\circ w)(x+h)-f\circ w(x) \right |^p dx ,\\
    &\leq |f|^p_{\text{Lip}(s,C(\Omega))}\int \left |w(x+h)-w(x) \right|^{sp}dx, \\
    &\leq |f|^p_{\text{Lip}(s,C(\Omega))}\left (  |w|_{\text{Lip}(\alpha,L^{sp}(\Omega))} \cdot h^\alpha \right )^{sp}, \\
    & \leq |f|^p_{\text{Lip}(s,L^p(\Omega))}
    |w|^{sp}_{\text{Lip}(\alpha,L^{sp}(\Omega))} \cdot h^{\alpha s p}.
\end{align*}
This series of inequalities demonstrates that  
$$
\mathcal{U}:\text{Lip}(s,C(\Omega)) \rightarrow \text{Lip}(\alpha s, L^p(\Omega)).
$$
Again we can construct estimates and derive rates of convergence based on the fact that $\text{Lip}(\alpha s,L^p(\Omega))$ is equivalent to certain linear approximation spaces. 

Having chosen a measure of smoothness for the study of the dynamical system and Koopman operator, we must choose a basis. Any of the orthonormal wavelets or multiwavelets, such as those described in Examples  \ref{ex:ON_wave} or \ref{ex:ON_multi} could be used here.  But for purposes of illustration we  choose, yet again,  the Haar scaling functions $\phi_{j,k}$ and wavelets $\psi_{j,k}$ for $j\in \mathbb{N}_0$ and $k\in \Gamma_j^{\phi}$ and $ \Gamma_{j}^{\psi}$, respectively, introduced in Example \ref{ex:haar_basis}. 
We define the spaces of approximants to be the  collections of wavelets 
\begin{align*}
A_j:=\text{span}\{\psi_{j,k}\} \quad \text{and} \quad \tilde{A}_j:=\text{span}\{\tilde{\psi}_{j,k}\}, 
\end{align*}
as well as the associated  approximation spaces 
\begin{align*}
    A^{r,p}(U)&:= A^{r,p}(U;\{A_j\}_{j\in \mathbb{N}_0}), \\
    A^{r,p}(\tilde{U})&:= A^{r,p}(\tilde{U};\{\tilde{A}_j\}_{j\in \mathbb{N}_0}),
\end{align*}
with $\tilde{\psi}_{j,k}$ the warped wavelet generated from $\psi_{j,k}$ in terms of the mapping $w_\alpha$. 
We know that if $f\in \text{Lip}(s,C(\Omega))$, then  $\mathcal{U}f\in \text{Lip}(\alpha s ,L^p(\Omega))$. But this space is equivalent to the approximation space $A^{r,\infty}(L^p(\Omega))$ with $r=\alpha s$  that is obtained by interpolating between the spaces $L^p(\Omega)$ and $\text{Lip}(1,L^p(\Omega))$.

We know from page 131 of \cite{devorenonlinear} or \cite{devore1993constapprox} that approximations by piecewise constants converge to the function $f$ with rate  $O(n_j^{-r})$ for $0<r<1/2$ if  and only if $f\in \text{Lip}(s,L^2(\Omega))$.  Alternatively, this specific approximation rate follows from Theorem \ref{th:approx_besov} upon recognizing that the span of the Haar scaling functions in $A_j$ coincides with the Schoenberg spline space $\mathcal{S}_1(\triangle_j)$ of order $1$, and that $\text{Lip}(r,L^2(\Omega)) \approx \text{Lip}^*(r,L^2(\Omega))$ for $0<r<1$. 

\end{example}
\end{framed}

The following corollary is a simple consequence of Theorem \ref{th:dual_conv}. 
\begin{corollary}
Let the hypotheses of Theorem \ref{th:dual_conv} hold with a family of approximant spaces $\{A_j\}_{j\in\mathbb{N}_0}$ that satisfy 
$$
\|f\|_{A^{1,\infty}(C(\Omega))} \lesssim  \|f\|_{\text{Lip}(1,C(\Omega)) }.
$$
Then we have 
$$
d_{BL}(\Pi_j'\nu,\nu)\lesssim n_j^{-1} 
$$
for all $j\in \mathbb{N}$ and 
$\nu \in \mathbb{M}^{+,1}(\Omega)$.
\end{corollary}
\begin{proof}
The corollary is a direct consequence of the Theorem choosing $r=1,q=\infty$.
We have for $X:=C(\Omega)$ the bounds 
\begin{align*}
\left < (I-\tilde{\Pi}'_j)\nu,f\right >_{X^*\times X}&= \left < \nu,(I-\tilde{\Pi}_j)f\right >_{X^*\times X} \\
&\leq \|\nu\|_{X^*} \|(I-\tilde{\Pi}_j)f\|_X  \lesssim n_j^{-1}\|f\|_{A^{1,\infty}(X)} \\
&\leq n_j^{-1}
\|f\|_{\text{Lip}(1,C(\Omega))}.
\end{align*}
We now take the supremum of both sides of this inequality over the set of $f$ such that $\|f\|_{\text{Lip}(1,C(\Omega))}=1$, and the theorem follows. 
\end{proof}

\subsubsection{Approximation of Probability Measures} 
\label{sec:approx_PM}
As emphasized above, by definition we  know that $\Pi_j': A_j^*  \rightarrow C^*(\Omega)$, but we are not guaranteed that it maps from  $\mathbb{M}^{+,1}(\Omega)\subset A_j^*$ into $\mathbb{M}^{+,1}(\Omega)$. In many papers and example problems, it is an important requirement to be able to define approximations of probability measures that are in fact probability measures. This problem has been tackled in a more general sense of duality for (generalized) Young's  measures in \cite{roubicek,matache}.
A Young's measure $\nu$ is also known as parameterized probability measure, that is, it is a probability measure $\nu_t$ on $\Omega$ for each $t\in \mathcal{T}$  of some index set $\mathcal{T}$.  We might think of $\mathcal{T}$ as a set of possible times, and $\nu_t$ as a probability measure in space at time $t\in \mathcal{T}$. The approach  in these references  differs from the strategy here in two important   ways. First, they make no explicit use of approximation spaces, although the approximation error hypotheses they rely on  can often be inferred from the definition of an approximation space.  Secondly, the analysis of error in these references makes systematic use of a duality of a more general  nature: the duality between a  (generalized)  Young's measure and certain Cartheodory kernels.  In the simplest of the cases treated in \cite{roubicek},  this duality is given by 
$(L^1(\mathcal{T}, C(\Omega)))^*:=L^\infty_w(\mathcal{T},rca(\Omega)))$. Here $L^1(\mathcal{T},C(\Omega))$ is the space of integrable Banach space-valued functions  taking values in $C(\Omega)$, and $L^\infty_w(\mathcal{T},rca(S))$ is the collection of weakly measurable functions $t\mapsto \nu_t\in rca(S)$.  We will not pursue this line of thought in this paper, which employs a simpler duality structure.


\subsection{ Approximations of Probability Measures, Duality in $L^\infty_\mu(\Omega) \times L^1_\mu(\Omega)$}
We motivate the approach in this section by considering  again the simple deterministic recursion  
$$
x_{n+1}=w(x_n).
$$

\medskip

\begin{framed}
\begin{example}
From our earlier discussions, the  recursion above defines the sample path of a Markov chain with transition probability kernel $\mathbb{P}(dy,x):=\delta_{w(x)}(dy)$, which means that the Koopman operator is just 
$$
(\mathcal{U}g)(x):=\int_\Omega \delta_{w(x)}(dy)g(y) = g(w(x)).
$$
Now we suppose that $\mu$ is an probability measure, $\mu\in \mathbb{M}^{+,1}$. We can then define the Perron-Frobenius operator by the duality condition
$$
<g,\mathcal{P}f>_{L^\infty_\mu(\Omega)\times L^1_\mu(\Omega)} = <\mathcal{U}g,f>_{L^\infty_\mu(\Omega)\times L^1_\mu(\Omega)} = \int_\Omega g(w(x))f(x)\mu(dx)
$$
for any $g\in L^\infty_\mu(\Omega)$ and $f\in L^1_\mu(\Omega)$. When we choose $g(x)=1_A(x)$
we obtain an alternate  representation in Remark 3.2.2 from \cite{lasota}  with 
\begin{equation}
\int_A (\mathcal{P}f)(x) \mu(dx) =
\int_{w^{-1}(A)} f(x) \mu(dx). \label{eq:alt_P}
\end{equation}
As pointed out in Section 3.2 of \cite{lasota}, a useful alternative representation of $\mathcal{P}$ can be derived {\it when the measure $\mu$ is Lebesgue measure} on the real line. 
In this strategy we  can use the expression above to derive the equation  
$$
\mathcal{P}f(x)=f(w^{-1}(x)) \frac{d}{dx}(w^{-1}(x))
$$
following  \cite{lasota} on page 43, and subsequently build approximations of this expression.  Such approximations could be fashioned, for example, by constructing finite dimensional estimates of $f\circ w^{-1}$ and just multiplying the result by $d(w^{-1})/dx$, which is assumed known.  The approximations of $f\circ(w^{-1})$ would follow quite similarly to our approach for approximating $\mathcal{U}f:= f\circ w$ in Example \ref{ex:warp3}. 

However, if the measure $\mu$ is not Lebesgue measure the process above is of no immediate help.
An alternative could be to use the identity in Equation \ref{eq:alt_P}, substituting an approximation $\mu_j$ for $\mu$.  In this context, since $\mu$ is assumed to be a probability measure, it seems vital to derive estimates of $\mu$ that are themselves probability measures. We discuss this problem in the remainder of this section. 
\end{example}
\end{framed}

 Denote the bases for the approximant spaces as 
$$
A_j:=\text{span}\left \{ a_{j,k} \ | \ k\leq n_j \right \} \subset L^1_\mu(\Omega).
$$
We have the duality structure  
$$
A_j \subset L^1_\mu(\Omega) \quad \text{and} \quad (L^1_\mu(\Omega))^*=  L^\infty_\mu(\Omega) \subset A_j^* 
$$
with $A_j^*$ 
the topological dual of  $A_j$. That is, $A_j^*$ is the topological dual space of   $A_j$ when $A_j$ is endowed with its inherited  $L^1_\mu(\Omega)$ norm. There is a unique dual  basis $\{a^{j,k}\}_{k\leq n_j}\subset L^\infty_\mu(\Omega)$ for the dual space $A^*_j$ 
and we have 
$$
A_j^*:=\text{span}\left \{ a^{j,k}\ | \ k\leq n_j\right \}.
$$
We have the representations 
\begin{align*}
f&=\sum_{i\leq n_j} \left <a^{j,k},f \right >_{L^\infty_\mu(\Omega) \times L^1_\mu(\Omega)} a_{j,k}:=\sum_{k\leq n_j} a^{j,k}(f) a_{j,k},  \\
h&=\sum_{i\leq n_j} \left <h,a_{j,k} \right >_{L^\infty_\mu(\Omega) \times L^1_\mu(\Omega)} a^{j,k}:=\sum_{k\leq n_j} a_{j,k}(h) a^{j,k} , 
\end{align*}
for each $f\in A_j$ and $h\in A_j^*.$ 
Define the approximation 
operators  $\tilde{\Pi}_j:L^1_\mu(\Omega)\rightarrow A_j$ by the expression 
$$
\left ( \tilde{\Pi}_j f\right)(x)
=\sum_{k\leq n_j} 
\frac{a^{j,k}(f)}{
\|a^{j,k}\|_{L^1_{\mu}(\Omega)}}  a_{j,k}(x)
$$
for any $f\in A_j$. It can be directly shown \cite{matache} that the dual operators  $\tilde{\Pi}'_j:A_j'\rightarrow (L^1(\Omega))^*:=L^\infty_\mu
(\Omega)$ are given by 
$$
\left (\tilde{\Pi}_j'h\right )(x):=\sum_{k\leq n_j}  \frac{a_{j,k}(h)}{\|a^{j,k}\|_{L^1_\mu(\Omega})} a^{j,k}(x).
$$
Now, suppose that $a_{j,k},a^{j,k}$ are positive functions.  Then 
it is clear that $\tilde{\Pi}_jh\geq 0$ for all  $h\geq 0$, and $\tilde{\Pi}_j'$ is a positive operator. Moreover, we can calculate directly that 
$$
\|\tilde{\Pi}_j'h \|_{L^1_{\mu}(\Omega)}= \sum_{k\leq n_j}
 \frac{a_{j,k}(h) \|a^{j,k}\|_{
L^1_\mu(\Omega)}}{\|a^{j,k}\|_{L^1_\mu(\Omega)}}
= \sum_{k\leq j} a_{j,k}(h).
$$
The $L^1_\mu-$norm of $h$ is computed to be 
\begin{align*}
\|h\|_{L^1_\mu(\Omega)}= \int_\Omega \left |
\sum_{k\leq n_j} a_{j,k}(h)a^{j,k}(\xi)\right |\mu(d\xi) = \sum_{k\leq n_j}  a_{j,k}(h) \|a^{j,k}\|_{L^1_\mu(\Omega)}.
\end{align*}
Suppose that the dual basis is normalized so that $\|a^{j,k}\|_{L^1_\mu(\Omega)}=1$.  
We see then  that if $\|h\|_{L^1_\mu}=1$, so that $h$ is the density of a probability measure, then $\|\Pi_j'h\|_{L^1(\Omega)}=\|h\|_{L^1(\Omega)}=1$. The function $\Pi_j'h$ is consequently  the density of a probability measure. 

\begin{framed}
\begin{example}[Approximation of Probability Measures by Piecewise Constants]
We define the bases $\{a_{j,k}\}_{j\in \mathbb{N}}$ and approximant spaces 
$$
A_j:=\text{span}\left \{a_{j,k}\right \}_{k\leq n_j}:= \text{span}\left \{1_{\square_{j,k}} \right \}_ { k\in \Gamma_{j}^\phi }. 
$$
as in Example \ref{ex:ex1}. The dual basis is 
readily computed to be 
\begin{align*}
a^{j,k}(x)&:=\frac{1}{\mu(\square_{j,k})} 1_{\square_{j,k}}.
\end{align*}
In fact, we have 
$$
\|a^{j,k}\|_{L^1_\mu(\Omega)}=1.
$$
The operator $\tilde{\Pi}_j'$ defined in terms of these bases maps probability density functions into probability density functions. 
For purposes of illustration, suppose that $\mu$ is normalized Lebesgue $\mu(dx):=dx/2\pi$.  In this case the definition of $\tilde{\Pi}_j$ and $\tilde{\Pi}'_j$ coincide with their definitions in \ref{ex:two_cases}. We can conclude that 
$$
\|(I-\tilde{\Pi}_j')f\|_{L^p(\Omega)} \lesssim n_j^{-r}|f|_{ \text{Lip}^*(r,L^p(\Omega))} 
$$
for $0<r<1/2$ and $1\leq p\leq \infty$. 

We can then define the approximations of the Koopman operators 
$$
\left ( \tilde{\mathcal{U}}_jf \right )(x) =
\left (\tilde{\Pi}_j(f\circ w) \right )(x), 
$$
and calculate 
\begin{align*}
\| (\mathcal{U} - \tilde{\mathcal{U}}_j)f\|_{L^p(\Omega)} &=
\| (I-\tilde{\Pi}_j)(f\circ w)\|_{L^p(\Omega)} \\
& \lesssim n_j^{-r} \|f\circ w\|_{A^{r,\infty}(L^p(\Omega))}\\
& \approx n_j^{-r} \|f\circ w\|_{\text{Lip}^*(r,L^p(\Omega))}
\end{align*}
over this range of $r$ with $1\leq p\leq \infty$. If  we do not know the probabability measure $\mu$ is Lebesgue measure,   we obtain only 
\begin{align*}
\| (\mathcal{U} - \tilde{\mathcal{U}}_j)f\|_{L^p_\mu(\Omega)} &=
\| (I-\tilde{\Pi}_j)(f\circ w)\|_{L^p_\mu(\Omega)} \\
& \lesssim n_j^{-r} \|f\circ w\|_{A^{r,\infty}(L^p_\mu(\Omega))},
\end{align*}
since  Theorem \ref{th:approx_besov} applies only for the case that $\mu$ is Lebesgue measure. It can not be applied directly to establish the equivalence of the Lipschitz space $\text{Lip}(r,L^p_\mu(\Omega))$ to the approximation space $A^{r,\infty}(L^p_\mu(\Omega))$, for instance.  
\end{example}
\end{framed}

%
%
%

In this last example we apply the results of this section to study the rate of approximation of the Koopman operator for a non-trivial Markov chain, one that corresponds to an iterated function system (IFS). 
\begin{framed}
\begin{example}
\label{ex:meas2}
In this example we describe  a well-known  semidynamical system that is amenable to the estimation of measures as introduced in this section. We briefly review the notion of iterated function systems, the semidynamical systems associated with them, and the construction of fractals. \cite{lasota,barnsley} We let $\Omega:=[0,1]^d$, and suppose we are given a finite family of Lipschitz maps  $w_\lambda\in\text{Lip}(1,C(\Omega))$ and $|w_\Lambda|_{\text{Lip}(1,C(\Omega))}=L_\lambda<1$.  Associated with this family we define the set valued map 
$$
W(S):=\cup_{\lambda\in \Lambda} w_\lambda(S).
$$
It is known that the iteration 
\begin{align}
S_{n+1}=W(S_n) \label{eq:det_IFS}
\end{align}
is a semidynamical system on the family $\mathcal{H}$ consisting of  compact subsets of $\Omega$ when it is endowed with the Hausdorff metric $d_\mathcal{H}$. In fact, since each mapping $w_\lambda$ is a contraction, it follows that $W:\mathcal{H} \rightarrow \mathcal{H}$ is a contraction. Using Banach's fixed point theorem, we know that there is a unique solution to the fixed point equation $W(S^*)=S^*$, and it  is given by 
$$
S^*:=\lim_{n\rightarrow \infty}W^{\circ n}(S_0) \in \mathcal{H}
$$
for any initial compact set $S_0\subseteq \Omega$. This theorem has been used as a constructive way to define a variety of fractals \cite{barnsley}.  

There is a second probabalistic interpretation of an iterated function system.  Intuitively, it proceeds as follows.  Let $n=0$. Given some initial condition $x_0\in \Omega$, we roll a dice that has $\#(\Lambda)$ sides that are weighted with probabilities $p_\lambda$,  $\sum_{\lambda\in\Lambda} p_\lambda=1$. If the result is $\lambda_n\in \Lambda$, we set 
\begin{align}
x_{n+1}=w_{\lambda_n}(x_{n}) \label{eq:pr_IFS1}
\end{align}
and repeat this process. A typical trajectory then has the form 
\begin{align}
\ldots w_{\lambda_n}\circ w_{\lambda_{n-1}} \circ \cdots \circ w_{\lambda_0}(x_0). \label{eq:pr_IFS2}
\end{align}
The relationship between the attractor of the deterministic attractor in Equation \ref{eq:det_IFS}, the stochastic system in Equation \ref{eq:pr_IFS1}, and the trajectories  in \ref{eq:pr_IFS2} has been studied in great detail. It provides a model for understanding the interplay of symbolic dynamics, the support of discrete dynamics, and dynamics of probability measures. \cite{barnsley}
We have included this  example  since its associated Frobenius-Perron and Koopman operators are key to understanding the asymptotic behavior of  this dynamical system.     

We can describe the stochastic recursion in Equation \ref{eq:pr_IFS1} in terms of the transition probability kernel
$$
\mathbb{P}(A,x)=\sum_{\lambda \in \Lambda} p_\lambda \delta_{w_\lambda(x)}(A).
$$
that gives the probability that the   next step of the chain is in $A$ given that the current state is $x$. The Koopman and Perron-Frobenius operators are calculated directly from the definitions in Equations \ref{eq:genU} and \ref{eq:genP} to obtain 
\begin{align*}
    (\mathcal{U}f)(x)&= \sum_{\lambda \in \Lambda}p_\lambda f(w_\lambda(x)), \\
    (\mathcal{P}\nu)(A)&= \sum_{\lambda\in \Lambda} p_\lambda \nu(w^{-1}_\lambda(A))
\end{align*}
for all measurable $A$ and $x\in \Omega$. 
It is easy to see that the Perron-Frobenius operator $\mathcal{P}$ above maps a probabability measure $\nu$ into a probability measure since 
$$
(\mathcal{P}\nu)(\Omega)= \sum_{\lambda\in \Lambda}p_\lambda \nu(w_\lambda^{-1}(\Omega))= \sum_{\lambda \in \Lambda} p_\lambda \cdot 1 = 1.
$$

An intuitive explanation of what the Perron-Frobenius operator represents is useful here. Suppose that we are given a probability measure $\mu_0$ that is interpreted as a  distribution of initial conditions over the configuration space for the stochastic evolution, and we want to understand the probability  that a set $A\subseteq \Omega$ is visited during the first step of the recursion. This probability is given by a measure  $\mu_1$ computed from 
$$
\mu_1(A)= \int_\Omega \mathbb{P}(A,x)\mu_0(dx)=\mathcal{P}\mu_0.
$$
This process can be repeated for any finite number of steps $n$, and $\mu_n(A)=(\mathcal{P}^n\mu_0)(A)$  describes the probability that the $n^{th}$ step of the recursion lands in $A$ when the initial conditions are distributed according to $\mu_0$. We see that if $\mathcal{P}^n\mu_0$ converges to some probability measure $\mu^*$, then $\mu^*$ describes where the  iterations eventually concentrate. 

This argument can be made precise by endowing the set of probability measures $\mathbb{M}^{1,+}(\Omega)\subset rca(\Omega)$ with a metric. The most popular choice is the bounded Lipschitz metric $d_{BL}$ introduced in Section \ref{sec:approx_SM}. It is known that this metric induces  the weak* topology that $\mathbb{M}^{+,1}(\Omega)$ inherits from $C^*(\Omega)=rca(\Omega)$.  The metric space $\mathbb{M}^{+,1}(\Omega)$ is in fact a closed, compact subset of $rca(\Omega)$. \cite{parthasarathy} 
With this choice, $(\mathcal{P},(\mathbb{M}^{+,1}(\Omega),d_{BL}))$ is a  ($\text{weak}*-$) continuous semigroup in discrete time.  Generally, the $\text{weak}^*$  continuity of the semigroup is equivalent to the Feller property, which holds if  $\mathcal{U}:C(\Omega) \rightarrow C(\Omega)$. In the present case,  since the functions  $w_\lambda$ are Lipschitz continuous,  the Feller property is immediate.  

The primary reason for the choice of the bounded Lipschitz metric in the study of the dynamical system 
$(\mathcal{P},\mathbb{M}^{+,1}(\Omega))$ is that simple criteria can be developed that ensure that the Perron-Frobenius operator is a contraction. We have 
\begin{align}
    \left < \mathcal{P}(\mu-\nu),f\right>_{X^*\times X}
  & =  \left < \mu-\nu,\mathcal{U}f\right>_{X^*\times X}, \notag \\
  &= \sum_{\lambda\in \Lambda} p_\lambda \int_\Omega f(w_\lambda(x))(\mu(dx)-\nu(dx)), \notag \\
  &= \sum_{\lambda\in \Lambda} p_\lambda  L_\lambda \int_\Omega \frac{f(w_\lambda(x))}{L_\lambda}(\mu(dx)-\nu(dx)) . \label{eq:IFS_con}
\end{align}
But if $|f|_{BL}\leq 1$, then the function $|f(w_\lambda)/L_\lambda|_{BL}<1$ since  
$$
|f(w_\lambda(x))-f(w_\lambda(y))|\leq |f|_{BL}|w_\lambda|_{BL}\leq L_\lambda.
$$
When we take the supremum of both sides of  the inequality in Equation \ref{eq:IFS_con} over  $\|f\|_{\text{Lip}(1,C(\Omega))}\leq 1$, we conclude that 
$$
d_{BL}(\mathcal{P}\mu,\mathcal{P}\nu)\leq L_\Lambda d_{BL}(\mu,\nu)
$$
with $L_\Lambda:=\sum_{\lambda\in \Lambda} p_\lambda L_\lambda.$  Whenever $L_\Lambda<1$, we have a unique solution $\mu^*$ of the fixed point Equation
$$
\mu^*=\mathcal{P} \mu^*,  
$$
that is given by 
$$
\mu^*=\lim_{k\rightarrow \infty} \mathcal{P}^k\mu_0
$$
for any initial probability
measure $\mu_0$. 
This result is shown in \cite{lasota} Proposition 12.8.1. See \cite{szarek} for generalizations that allow for place dependent probabilities $p_\lambda=p_\lambda(x)$. 
In fact, we know that 
$$
d_{BL}(\mathcal{P}^k\mu_0,\mu^*)\leq L_\Lambda^kd_{BL}(\mu_0,\mu^*). 
$$
By definition, the measure  $\mu^*$  that solves the above fixed point equation is the invariant measure  of the discrete dynamical system $(\mathcal{P},(\mathbb{M}^{+,1}(\Omega),d_{BL}))$. 

We consider the  specific example depicted in Figure \ref{fig:ifs_triangle}. 
\begin{center}
    \includegraphics[width=.7\textwidth]{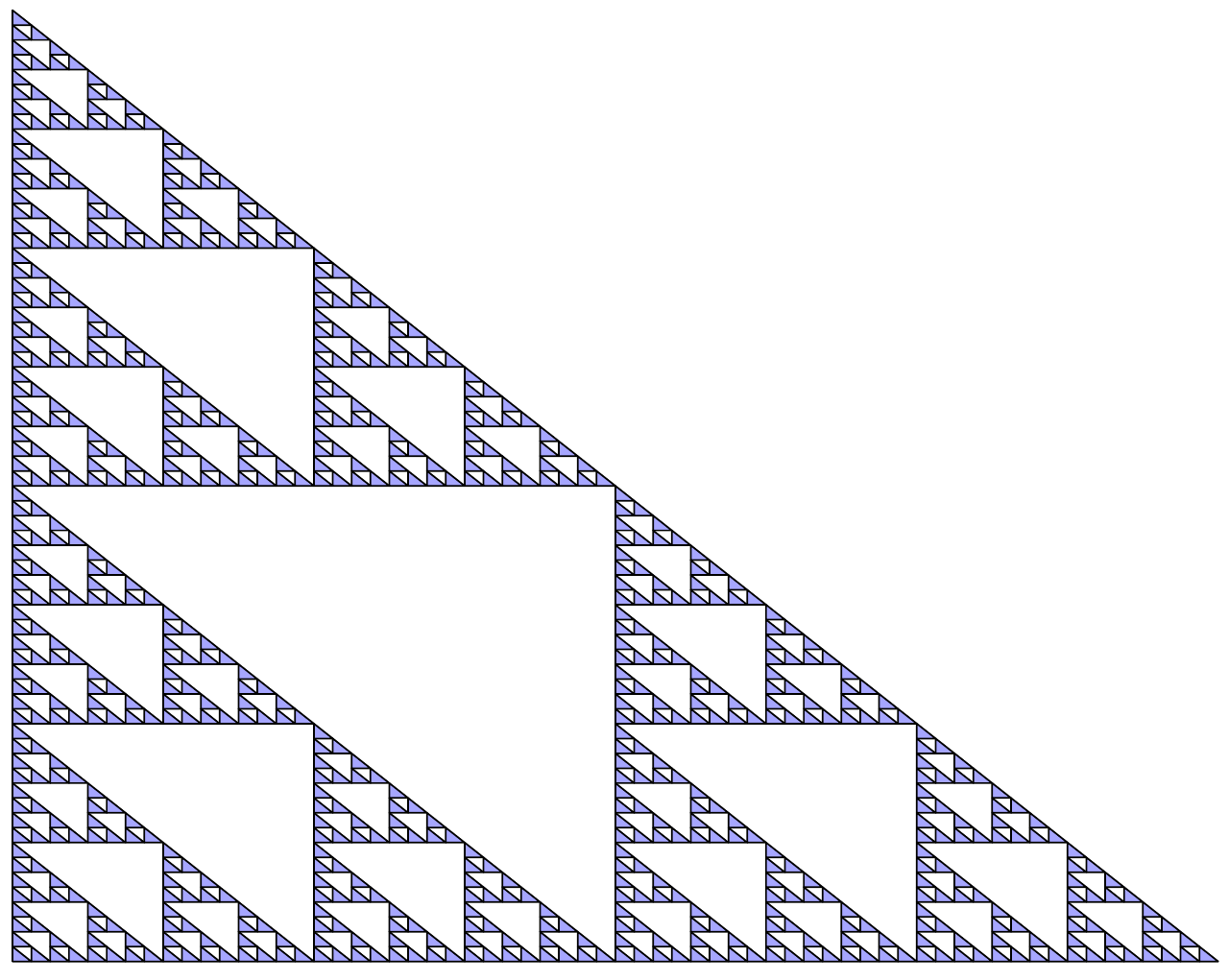}
    \captionof{figure}{Attractor of Deterministic IFS $\approx$  Support of Invariant Measure}
    \label{fig:ifs_triangle}
\end{center}
This iterated function system is defined in terms of  three Lipschitz mappings $w_\lambda$ for $\lambda=1,2,3$ defined as 
\begin{align*}
    w_1\left (\begin{Bmatrix}
    x_1 \\ x_2
\end{Bmatrix}\right )&:=\frac{1}{2}\begin{Bmatrix}
    x_1 \\ x_2
\end{Bmatrix}, \\
    w_2\left (\begin{Bmatrix}
    x_1 \\ x_2
\end{Bmatrix}\right )&:=\frac{1}{2} \left (  \begin{Bmatrix}
    x_1 \\ x_2
\end{Bmatrix} +\begin{Bmatrix}
    1\\0
    \end{Bmatrix} \right ),\\
    w_3\left (\begin{Bmatrix}
    x_1 \\ x_2
\end{Bmatrix}\right)&:=\frac{1}{2}\left (\begin{Bmatrix}
    x_1 \\ x_2
\end{Bmatrix} +  \begin{Bmatrix}
    0\\ 1\end{Bmatrix}
    \right ).
\end{align*}
\end{example}
This is an example of a ``just touching'' IFS, one known as a Sierpinski gasket. We will use the multiwavelet basis defined in Example \ref{ex:haar_multi_tri} for the construction of approximations of the Koopman operator for this IFS. From the definition of the mappings $w_\lambda$ it is evident that 
\begin{align*}
    w_1&: \triangle \rightarrow \triangle_{1,(0,0)}:=\tilde{\triangle}_1, \\
     w_2&: \triangle \rightarrow \triangle_{1,(0,1)}:=\tilde{\triangle}_2,  \\
      w_3&: \triangle \rightarrow \triangle_{1,(
      1,0)}:=\tilde{\triangle}_3, 
\end{align*}
with $\triangle_{1,(0,0)}, \triangle_{1,(1,0)}, \triangle_{1,(0,1)}$ defined in Example \ref{ex:haar_multi_tri}.
We define 
\begin{align*}
    \tilde{\triangle}&:=
    \bigcup_{i\in \Lambda} \tilde{\triangle}_i, \\
    L^2_{\tilde{\mu}}(\tilde{\triangle})&:=
    \underset{i \in \Lambda}{\oplus}L^2_{\tilde{\mu}_i}(\tilde{\triangle}_i),
\end{align*}
with $\Lambda=\{1,2,3\}$,  $\tilde{x}_i=w_i(x)$, $\tilde{\triangle}_i=w_i(\triangle)$, $\tilde{M}_i(\tilde{x}_i)=w_i^{-1}(\tilde{x}_i)$, $\tilde{m}(\tilde{x}_i)=|\partial \tilde{M}_i/\partial \tilde{x}_i|$, and $\tilde{\mu}_i(d\tilde{x}_i):=\tilde{m}_i(\tilde{x}_i)d\tilde{x}_i$. 
For any $f$ such that $f\circ w_\lambda\in L^2(\triangle)$ for all $\lambda\in \Lambda$, we have the expansion 
\begin{align*}
    \mathcal{U}f&=\sum_{\lambda\in \Lambda}
    p_\lambda 
    \sum_{i\in \mathbb{N}_0} 
    \sum_{\bm{k}\in \Gamma_{i}^\psi }
    \left (
    f\circ w_\lambda, \psi_{j,\bm{k}}
    \right )_{L^2(\triangle)} \psi_{j,\bm{k}}, \\
    &=\sum_{\lambda\in \Lambda}
    p_\lambda 
    \sum_{i\in \mathbb{N}_0} 
    \sum_{\bm{k}\in \Gamma_{i}^\psi }
    \left (
    f, \tilde{\psi}_{\lambda,({i,\bm{k}})}
    \right )_{L^2_{\tilde{\mu}_\lambda}(\tilde{\triangle}_\lambda)} \psi_{i,\bm{k}},   
\end{align*}
and we define the approximation $\mathcal{U}_j$ via 
$$
\mathcal{U}_jf=\sum_{\lambda\in \Lambda}
    p_\lambda 
    \sum_{i\leq j} 
    \sum_{\bm{k}\in \Gamma_{i}^\psi }
    \left (
    f, \tilde{\psi}_{\lambda,({i,\bm{k}})}
    \right )_{L^2_{\tilde{\mu}_\lambda}(\tilde{\triangle}_{\lambda})} \psi_{i,\bm{k}}.
$$ 
In this equation $\tilde{\psi}_{\lambda,(j,\bm{k})}$ are  the warped wavelets over $\tilde{\triangle}_\lambda$ induced by the mapping $w_\lambda$. For any $f\in A^{r,2}(L^2_{\tilde{\mu}}(\tilde{\triangle}))$, we have the error estimate 
$$
\|(\mathcal{U}-\mathcal{U}_j)f\|_{L^2(\triangle)} \lesssim 2^{-rj}|\mathcal{U}f|_{A^{r,2}(L^2(\triangle))}.
$$
By duality, we define the approximation $\mathcal{P}_j$ of the Perron-Frobenius operator $\mathcal{P}$ as
$$
(\mathcal{P}_j\nu)(A):=\sum_{\lambda \in \Lambda} p_\lambda
\sum_{i\leq j} \sum_{\bm{k}\in \Gamma_i^\psi} \int_\Omega \psi_{i,\bm{k}} (x)\nu(dx) \int_A \tilde{\psi}_{\lambda,(i,\bm{k})}(\tilde{x}_\lambda)\tilde{\mu}_\lambda(d\tilde{x}_\lambda).
$$

\end{framed}

\section{Approximations when $\mu$ is unknown and $p$ is known}
\label{sec:approx_mu_u_p_k}
This section studies a case of intermediate complexity, one that helps bridge the gap between the situation when all the problem data is known as in Section \ref{sec:mu_p_known} and cases when the only available information is a family of samples of the input-output states of the  underlying dynamical system in Section \ref{sec:mu_p_unknown}.

We begin the analysis in this section assuming that single step output measurements are generated from a collection of initial states that are chosen randomly according to a probability distribution $\mu$ on $\Omega$. The  IID  assumption is simpler to study and leads to expressions that illustrate the tradeoff between the deterministic and probabalistic contributions to the error.  We subsequently explain  how the argument based on IID initial states  can be modified to obtain error estimates for  dependent samples  generated along the sample path of a Markov chain. 

\subsection{Approximation from IID Initial States}
\label{sec:IID_states}
We let $z:=\{x_1,\ldots,x_m\}\subset \Omega^m$ be a collection of $m$ independent and identically distributed (IID) samples that are distributed according to the measure $\mu$ on $\Omega$. 
With the breadth of technical tools covered in Sections \ref{sec:rkhs}, \ref{sec:spectral_spaces}, and  \ref{sec:approx_spaces_Arq} there are many specific routes to address this problem. 
For purposes of illustration, the next theorem describes a  simple case among all the variants of the problem facing us in this section. Suppose  that the domain  $\Omega \subset \mathbb{R}$ is compact with $|\Omega|=\mu(\Omega)$ and   that  $K$ is a symmetric, positive definite, and bounded kernel on $\Omega\times \Omega$.  
We follow the construction in  Section \ref{sec:rkhs} and  suppose that the spectral approximation space  $A^{r,2}_\lambda(U) \subset C(\Omega) \subset L^2_\mu(\Omega):=U$ defined in terms of the eigenvalues and eigenfunctions $\{(\lambda_i,u_i)\}_{i\in \mathbb{N}}$ of the operator $T_K:U\rightarrow U$. It is always possible find such a spectral  approximation space $A^{r,2}_\lambda(U)$ that is contained in the continuous functions. We can, for instance,  choose $r=1$ since the RKHS space  $A^{1,2}_\lambda(U)\approx V\subset C(\Omega)$ by construction. Any larger $r\geq 1$ will likewise work. 

The Perron-Frobenius operator $\mathcal{P}$ and its approximation 
 $\mathcal{P}_j$, as well as their respective kernels  are   defined as in Theorem \ref{th:spec_sp_approx}, 
    \begin{align}
         \mathcal{P}&:=\sum_{i\in \mathbb{N}} p_i u_i \otimes u_i, \quad \quad 
         p(x,y)=\sum_{i\in \mathbb{N}} p_i u_i(x) u_i(y)
         \label{eq:z0}\\ 
         \mathcal{P}_j&:=\sum_{i\leq j} p_i u_i  \otimes  u_i, \quad \quad 
         p_j(x,y):=\sum_{i\leq j}p_i u_i(x)_i(y).
         \label{eq:z1}
    \end{align}
We define  the sample dependent  operators  $\mathcal{P}_z$ and $\mathcal{P}_{j,z}$  as 
\begin{align}
(\mathcal{P}_zf)(x)&:=\frac{|\Omega|}{m}\sum_{i=1}^m p(x,x_i) f(x_i), \\
(\mathcal{P}_{j,z}f)(x)&:=\frac{|\Omega|}{m}\sum_{i=1}^m p_j(x,x_i) f(x_i), 
\label{eq:sample_ops}
\end{align}
respectively, for any $f\in A^{r,2}_\lambda(U)$. In this last expression 
\begin{equation}
p_j(\cdot,x_i):=\Pi_j(p(\cdot,x_i)) = \sum_{k\leq j} \int_\Omega p(\xi,x_i)u_k(\xi)\mu(d\xi) \ u_k.  \label{eq:z4}
\end{equation}
For these operators to make sense, it must be true that pointwise evaluation of the kernels and function $f$  at $x_i\in \Omega$ makes sense. From the  requirement   that   $A_j\subset A^{r,2}_\lambda(U) \hookrightarrow C(\Omega)\subset L^2_\mu(\Omega)$, all of the required  evaluation operators acting on $f$ make sense.   We further assume that $\bar{p}:\sup_{(x,y)\in \Omega\times \Omega}|p(x,y)| < \infty$ so that pointwise  evaluations of the kernel $p$  makes sense.

We next study an approximation problem when $p$ is known, the value $\mu(\Omega)$ is known,  but the measure $\mu$ is unknown. (The assumption that $\mu(\Omega)$ is known is not essential.) Specifically, we suppose that the spectral expansions in Equations \ref{eq:z0} through \ref{eq:sample_ops} are known, and we are given a closed form expression for the function $f$. However, since the measure $\mu$ is unknown, the Fourier coefficients $(f,u_i)_U$ are not computable. Therefore, we  cannot compute $\mathcal{P}f$ nor $\mathcal{P}_jf$ using Equations \ref{eq:z0} and \ref{eq:z1}. We seek an error analysis of how well the sample generated approximant $\mathcal{P}_{j,z}f$, which is computable, estimates $\mathcal{P}f$. We will see that this analysis is a precursor to the results presented in Section \ref{sec:mu_p_unknown}.
\begin{theorem}
\label{th:th3}
Suppose that the assumptions described above and the definitions in Equations \ref{eq:z1} through \ref{eq:z4} hold.
Let $z:=(x_1,\ldots,x_m)\subset \Omega$ be IID random samples drawn according the probability measure $\mu$ on $\Omega$.  
Then we have the error estimate
\begin{equation}
\|\mathcal{P}f-\mathcal{P}_{j,z}f\|_{U} \lesssim  \left ( \lambda_j^{r/2} + \epsilon \right )\|f\|_{A^{r,2}(U)} 
\label{eq:error_fix}
\end{equation}
for all samples $z\not \in\mathcal{B}(\epsilon,j):= \mathcal{B}(\epsilon, j;
\mathcal{P})$, and the  set of bad samples 
\begin{equation}
\label{eq:Beps}
\mathcal{B}(\epsilon,j) := \left \{ z\in \Omega^m \ \left | \  \| \mathcal{P}_{j}-\mathcal{P}_{j,z}\|_{S^2(U)} > \epsilon \right . \right \} 
\end{equation}
has probability 
\begin{equation}
Pr(\mathcal{B}(\epsilon,j)) <  AC(\epsilon,j):=
\left \{
\begin{array}{ccc}
1 & \text{if} & \epsilon \leq  \alpha \ \sqrt{j/m}\\
2e^{-\beta m \epsilon^2/j} & \text{if} & \epsilon > \alpha\  \sqrt{j/m}
\end{array}
\right.
\label{eq:ACfun}
\end{equation}
with the constants $\alpha:= 4\sqrt{\text{ln}2} \overline{p}$ and $\beta:=1/16\overline{p}^2$. The function $$AC(\epsilon,j):=AC(\epsilon,j;\mathcal{P})$$ is the accuracy confidence function for the fixed  operator $\mathcal{P}$.
\end{theorem}
It is important to note that the  error estimate above in Equation \ref{eq:error_fix} depends on the fixed function $f$ and the operator $\mathcal{P}$. 
Before proving this theorem, we state a corollary that is perhaps the most succinct representation of the bias-variance tradeoff that can arise in the approximation of Koopman or Frobenius-Perron operators from samples.

\begin{corollary}
\label{th:cor1}
Under the assumptions of Theorem \ref{th:th3}, we have the error bound in expectation over the samples $z=\{x_1,\ldots,x_m\}$ given by 
$$
\mathbb{E}_{\mathbb{P}^m}\left ( \|\mathcal{P}f - \mathcal{P}_{j,z}f\|_U^2 \right ) \lesssim (\lambda_j^r + \frac{j}{m})\|f\|^2_{A^{r,2}_\lambda(U)}
$$
with $\mathbb{P}:=\mu/\mu(\Omega)$. The constant in the upper bound above depends on $\bar{p}$, and hence on the fixed operator $\mathcal{P}.$

\end{corollary}
The proof of this corollary follows from integration of the accuracy confidence function introduced in Theorem \ref{th:th3}. The integration strategy is essentially identical to the integration carried out in Example \ref{ex:b_v}, so we omit the proof of the corollary here. 
\begin{proof}
By Proposition 9 of \cite{rosascoint}, $\mathcal{P}_{j,z}$ is a Hilbert-Schmidt operator on $U$.
For any $f\in U$, we know that 
\begin{align*}
    \|(\mathcal{P}-\mathcal{P}_{j,z})f\|_U &\leq \|(\mathcal{P}-\mathcal{P}_j)f\|_U + \|(\mathcal{P}_j - \mathcal{P}_{j,z})f\|_U ,\\
    &\lesssim \left (\lambda_j^{r/2}|\mathcal{P}|_{\mathbb{A}^{r,2}_\lambda(U)} + \|(\mathcal{P}_j-\mathcal{P}_{j,z}\|_{S^2(U)} \right )
    \|f\|_U,
\end{align*}
 where the last line follows from Theorem \ref{th:spec_sp_approx} and the fact that  $\|\cdot\|_{\mathcal{L}(U)}\leq \|\cdot\|_{S^2(U)}$. Note that we have $\mathcal{P}_{j,z}=\Pi_j \mathcal{P}_z =\mathcal{P}_z\Pi_j$ since 
 \begin{align*}
     (\mathcal{P}_{j,z}f)(x)&:= \frac{|\Omega|}{m}\sum_{1\leq i\leq m} p_j(x,x_i)f(x_i), \\
     &=\frac{|\Omega|}{m} \sum_{1\leq i\leq m} \sum_{1\leq k<j} p_ku_k(x) u_k(x_i) f(x_i),\\
     &=\sum_{1\leq k < j} \left ( \frac{|\Omega|}{m}\sum_{1\leq i\leq m} p_k u_k(x_i)f(x_i) \right )
     u_k(x)\\
    & =(\Pi_j\mathcal{P}_zf)(x)=(\mathcal{P}_z\Pi_jf)(x),
 \end{align*}
 for each $x\in \Omega$.
 From the expansion $\mathcal{P}=\sum_{i\in \mathbb{N}} p_k u_k \otimes u_k$ we have 
  $$
 \|\mathcal{P}_j
 -\mathcal{P}_{j,z}\|_{S^2(U)}^2 =  \|\Pi_j(\mathcal{P}-\mathcal{P}_z)\|^2_{S^2(U)}\leq 2j\ \|\mathcal{P}-\mathcal{P}_z\|^2_{S^2(U)}. 
 $$
 From Proposition 10 of \cite{rosascoint} we known that $\|\mathcal{P}-\mathcal{P}_z\|_{S^2(U)}\leq 8\bar{p}^2\delta/m$, 
 and thus, 
 $$
 \|\mathcal{P}_j
 -\mathcal{P}_{j,z}\|_{S^2(U)}^2 \leq 16\bar{p}^2 \frac{j}{m} \delta 
 $$
 for all samples $z\in \Omega^m$ in a set having  $\mu^m-$probability in excess of $1-2e^{-\delta}.$ 
 
 We next show that this probabalistic error estimate can be expressed in terms of the accuracy confidence function $AC(\epsilon,j)$ in the theorem. Define the set of bad samples as in Equation \ref{eq:Beps}, and denote its complement $\mathcal{G}$, the set of good samples. We have 
 $$
 \text{Prob}(\mathcal{B})=
 1-Prob(\mathcal{G}) < 1-(1-2e^{-\delta})=2e^{-\delta}.
 $$
Using the bound $\text{Prob}(\mathcal{B})\leq \min \{1,2e^{-\delta}\},$ we have 
 $$
 \text{Prob}(\mathcal{B})\leq   
 \left \{ \begin{array}{ccc}
 1 &  & \delta \leq \delta_{cr} \\
 2e^{-\delta} &  & \delta> \delta_{cr}
 \end{array}
 \right .
 $$
 with $\delta_{cr}:=\ln{2}.$
 Define $\epsilon^2:=16\bar{p}^2 \frac{j}{m}\delta $. A change of variables shows that 
Equation \ref{eq:ACfun} holds. 
\end{proof}
Note carefully that the  accuracy confidence function $A(\epsilon,j)$ here  is for a fixed choice of the function $f$ and the operator $\mathcal{P}$. In principle, it resembles Theorem A  of  Cucker and Smale in \cite{cs2002}, in contrast to the uniform estimates for {\em $f$ in a compact convex subset of $C(\Omega)$} in  Theorem $\text{C}^*$ in \cite{cs2002} and to those finer uniform estimates in   \cite{devore2006approxmethodsuperlearn}, \cite{kerkyacharian2003entropunivercodinapproxbasesproper}, \cite{temlyakov}. 
This bound has a complicated appearance, but for the problems in which either the measure $\mu$ is unknown, it is a convenient way to express the error. It is one of the ways that error in probability esimates are expressed in nonlinear regression, statistics, approximation theory, and statistical learning theory. \cite{devore2006approxmethodsuperlearn,devito2014learnsetsseparkernel,kerkyacharian2003entropunivercodinapproxbasesproper,temlyakov} 
\begin{figure}[h!]
\centering
\includegraphics[width=.7\textwidth]{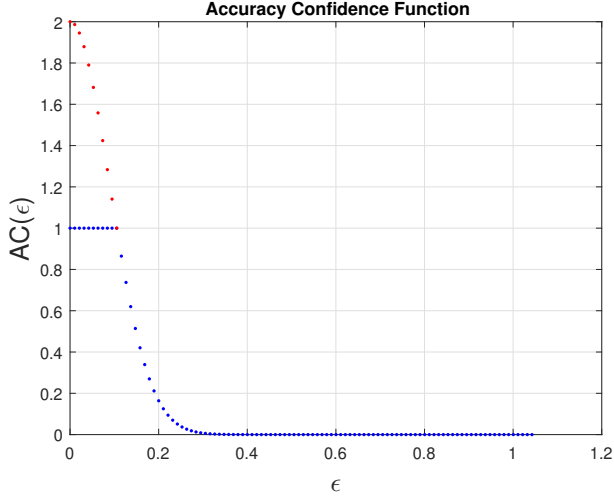}
\caption{Typical plot of accuracy confidence function $AC(\epsilon,f,\mathcal{P})$}
\label{fig:fig2}
\end{figure}
\begin{figure}[h!]
\centering 
\includegraphics[width=.7\textwidth]{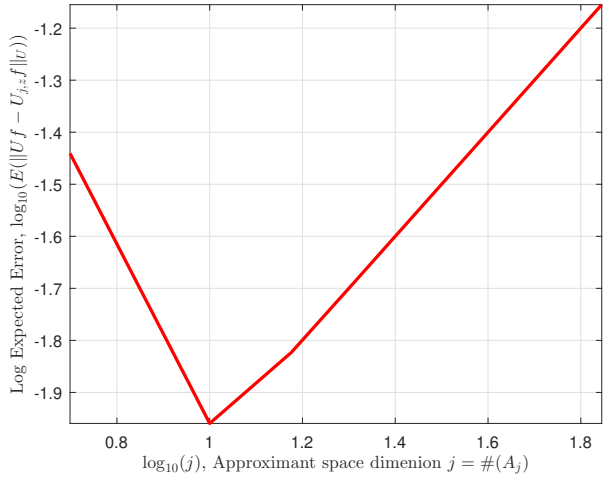}
\caption{Plot of the form of the accuracy confidence function $AC(\epsilon,f,\mathcal{P})$}
\label{fig:fig3}
\end{figure}

To develop an intuitive understanding of this bound, it must be kept in mind that the discrete, sample-based approximation $\mathcal{P}_z$ depends on the random samples $z$. If you are have a particularly bad day, it is possible that the specific estimator $\mathcal{P}_z$  generated by that sample $z$ has a large error. However, the estimate in Theorem \ref{th:th3} guarantees that for a resolution level $j$ the size of   set of bad samples $\mathcal{B}(\epsilon,j)$ decreases exponentially with an increase in the number of samples $m$. Its  size decays like $e^{-\beta m \epsilon^2/j}$.   A plot of a typical accuracy confidence function is shown in Figure \ref{fig:fig2}.  The bound in Theorem \ref{th:th3} is comprised of the approximation space error $O(\lambda_j^{-r})$ and the sample error $O(\epsilon)$. These are also referred to as the bias and probabalistic error, respectively.   Suppose now that the number of samples $m$ is fixed,  we are interested in a sample accuracy $\epsilon=\lambda_j^{r/2}$  that matches the size of the approximation space error,    and we vary the resolution level $j$.  The accuracy confidence function $AC(\lambda_j^{r/2},j)$ is $O(e^{-\beta m\lambda_j^{r}})$ as $j$ increases. The size of the set of bad samples therefore grows as $j$ increases. The plot of the general form of the error in this situation is shown in Figure \ref{fig:fig3} for a case when the number of samples $m$ is fixed and the dimension of the approximation space varies.   The portion of the plot exhibits increasing error as the variance term dominates. The form of this plot is analogous to those that are used to discuss  fundamental results from statistical learning theory. \cite{vapnik} 
\begin{framed}
\begin{example}
\label{ex:b_v}
In this example we study a specific case in which the tradeoff between the approximation error and  sample errors is readily established.  We end up with a worst case upper bound on performance that has the same structure as in Theorem \ref{th:th3}. But since the example is highly structured, the error bound can also be derived from first principles. The approach in this example, while specific, also provides clues  on how to attack the case when the samples $z=\{x_1,\ldots, x_m\}$ are the dependent  observations of the Markov chain having a transition probability kernel $\mathbb{P}(dy,x):=p(y,x)\mu(dy).$

We now return the discrete dynamical system discussed in Example \ref{ex:heat_eq}. Suppose that the  measure $\mu$ is just Lebesgue measure on $\mathbb{T}^1$, but we do not have this information available to us  to use in building approximations. The dynamical system in Example \ref{ex:heat_eq}  gives rise to the operators 
\begin{align*}
    (\mathcal{U}f)(x)&:= \int_\Omega p(y,x)f(y)\mu(dy)= \sum_{k\in \mathbb{N}} \sum_{i=1,2} p_{k,i}(f,u_{k,i})_U u_{k,i}, 
     \end{align*}
    \begin{align*}
    (\mathcal{U}_jf)(x)&:= \sum_{k\leq j} \sum_{i=1,2} (f,u_{k,i})_U u_{k,i}(x), 
    \end{align*}
    \begin{align*}
    (\mathcal{U}_zf)(x)&:=\int_\Omega p(y,x)f(y)\left (\frac{|\Omega|}{m} \sum_{l\leq m} \delta_{x_\ell}(dy) \right ),
\\
 &= \sum_{k\in\mathbb{N}}\sum_{i=1,2} p_{k,i} \left (
 \frac{|\Omega|}{m} \sum_{\ell=1}^m u_{k,i}(x_\ell)f(x_\ell)
 \right)u_{k,i}(x), 
 \end{align*}
 \begin{align*}
 (\mathcal{U}_{j,z})(x)&:= \sum_{k\leq j}\sum_{i=1,2} p_{k,i} \left (
 \frac{|\Omega|}{m} \sum_{\ell=1}^m u_{k,i}(x_\ell)f(x_\ell)
 \right)u_{k,i}(x),
\end{align*}
with $u_{k,i}$ the eigenfunctions defined in Example \ref{ex:heat_eq}. 
We can bound the error using the triangle inequality 
\begin{align*}
    \|\mathcal{U}f-\mathcal{U}_{j,z}f\|_U&\leq 
    \|\mathcal{U}f-\mathcal{U}_j\|_U + \|\mathcal{U}_jf-\mathcal{U}_{j,z}f\|_U \\
    &\lesssim \lambda_j^{r/2} + \|\mathcal{U}_jf-\mathcal{U}_{j,z}\|_U.
\end{align*}
We focus specifically on the sample error and see that 
\begin{align}
     \|\mathcal{U}_jf-\mathcal{U}_{j,z}f\|^2_U&= \sum_{k\leq j}\sum_{i=1,2} p_{k,i}^2 \left (
     \frac{|\Omega|}{m} \sum_{\ell\leq m} u_{k,i}(x_\ell)f(x_\ell) - (f,u_{k,i})_U\right )^2, \notag \\
     &\lesssim \sum_{k\leq j}\sum_{i=1,2}  \left (
     \frac{1}{m} \sum_{\ell\leq m} u_{k,i}(x_\ell)f(x_\ell) - \mathbb{E}_{\mathbb{P}}(u_{k,i}f)\right )^2. \label{eq:zz}
\end{align}
with the probability measure $\mathbb{P}(dx):=\mu(dx)/|\Omega|=dx/\mu(\Omega)$. Again, we emphasize that the constant in the  upper bound in Equation \ref{eq:zz} depends on $p_{1,1}:=\max_{k,i}(p_{k,i})$, so this estimate is again for a fixed operator $\mathcal{P}$. 

The next critical step employs a concentration of measure result in the form of Hoeffding's  inequality.  Suppose that   
$|g(x)-\mathbb{E}_\mathbb{P}(g)|<M$ for $x \text{ a.e. } \in \Omega$. Hoeffding's inequality states that  
$$
\mathbb{P}^m\left ( \left |\frac{1}{m}\sum_{\ell\leq m} g(x_\ell)-\mathbb{E}_{\mathbb{P}}(g)\right |\geq \epsilon  \right ) \leq 2e^{-m\epsilon^2/2M^2}.
$$
We want to use this inequality to bound each terms in the parentheses in Equation \ref{eq:zz}, that is, for $g:=u_{k,i}f$. Since $f\in A^{r,2}_\lambda(U)\subset C(\Omega)$ and $\Omega$ is assumed compact, we know that there is a constant $\tilde{M}_f$ such that $|f|\leq  \tilde{M}_f$. It follows that we have the uniform bound $|u_{k,i}f|\leq M:=M_f:= \tilde{M}_f/\sqrt{\pi}$ for all $k,i$. The dependence of our analysis on the constant $M:=M_f$ is important to note, a topic we discuss on completion of the proof. We conclude that 
\begin{align*}
\mathbb{P}^m\left ( \left |\frac{1}{m}\sum_{\ell\leq m} u_{k,i}(x_\ell)f(x_\ell)-\mathbb{E}_{\mathbb{P}}(u_{k,i}f)\right |\geq \epsilon  \right ) &\leq \min( 1,2e^{-m\epsilon^2/2M^2})
\end{align*}
for all $k\in \mathbb{N},i=1,2$. As in our discussion of Theorem \ref{th:th3}, we define the accuracy confidence function 
$$
AC(\epsilon):=\left \{
\begin{array}{ccl}
1 & \quad &\epsilon \leq \alpha {\sqrt{1/m}} \\
2e^{-m\epsilon^2/2M^2} &  & \epsilon> \alpha \sqrt{1/m}
\end{array}
\right .
$$
with $\alpha= \sqrt{2\text{ln}2}\ M$.  

We will now show the utility of the probabalistic bound in terms of the accuracy confidence function  $AC(\epsilon)$ for the construction of upper bound  on the expectation of the error over $m$ samples. The expected error is given by 
\begin{align*}
    \mathbb{E}_{\mathbb{P}^m}
    &\left ( \left |\frac{1}{m}\sum_{\ell\leq m} g(x_\ell)-\mathbb{E}_{\mathbb{P}}(g)\right |^2\right )   \hspace*{2.0in} \\
    \hspace*{.25in}
    &=\int_0^\infty 
    \mathbb{P}^m\left ( \left |\frac{1}{m}\sum_{\ell\leq m} g(x_\ell)-\mathbb{E}_{\mathbb{P}}(g)\right |^2\geq \epsilon  \right ) d\epsilon, \\
    &=\int_0^\infty 
    \mathbb{P}^m\left ( \left |\frac{1}{m}\sum_{\ell\leq m} g(x_\ell)-\mathbb{E}_{\mathbb{P}}(g)\right |\geq \sqrt{\epsilon}  \right ) d\epsilon, \\
    &\leq \int_0^\infty AC(\sqrt{\epsilon})d\epsilon.
\end{align*}
Now let $\epsilon=\eta^2$ so that $d\epsilon=2\eta d\eta$ and 
\begin{align*}
    \int_0^\infty AC(\sqrt{\epsilon})d\epsilon& = 2 \int_0^\infty \eta AC(\eta) d\eta, \\
    & \leq 2\left ( \int_0^{\alpha\sqrt{1/m}} \eta d\eta + 2 \int_{\alpha\sqrt{1/m}}^\infty \eta  e^{-m\eta^2/2M^2} d\eta \right ) \lesssim \frac{1}{m}
\end{align*}
for a constant  that depends only on the constant  $M:=M_f$. Combining this result with Equation \ref{eq:zz}, we finally obtain 
\begin{equation}
\mathbb{E}_{\mathbb{P}^m}\left ( \|\mathcal{U}f - \mathcal{U}_{j,z}f\|_U^2 \right ) \lesssim \lambda_j^r + \frac{j}{m}. \label{eq:compare_gyorfy}
\end{equation}
This is the same form as upper bound in the conclusions of Theorem \ref{th:th3} and Corollary \ref{th:cor1}. 
\end{example}
\end{framed}
As emphasized in the proof, the coefficient in the upper bound above depends on the constant $M:=M_f$ and $p_{1,1}$, so the error estimate holds for the fixed function $f$ and operator $\mathcal{U}$.  This should be contrasted to the results derived in \cite{cs2002,devore2006approxmethodsuperlearn, kerkyacharianwarped} that give uniform estimates over $f$ in  a compact, convex subset of $C(\Omega)$, a much stronger result. 
References \cite{binevI} and \cite{gyorfy} construct estimates of functions in   nonlinear regression, not estimates of Perron-Frobenius operators, that have a form  analogous to  that in  Equation \ref{eq:compare_gyorfy}. These references allow for  bases that need not be orthogonal functions, and the variance term is of the order $O(j\log j/m)$ in the more general case. 
\subsection{Approximations from  Observations along a Sample Path}
\label{sec:dep_samples}

In this section we answer how the analysis of the last section can be viewed as the foundation for the situation in which the observations $\{x_\ell\}_{\ell\leq m}\subset\Omega$ are collected along the sample path of a Markov chain. We will use the analysis in Example 
\ref{ex:b_v} to guide us. 
When we study the argument in Example \ref{ex:b_v} carefully, we see that the only place that the IID assumption is required is in the hypotheses of Hoeffding's inequality. Since the samples $\{x_\ell\}_{\ell\leq m}$ along the sample path of the Markov chain  are generally dependent, we cannot use Hoeffding's inequality here. 
Hoeffding's inequality states what is known as a concentration of measure formula. It turns out that there is a fairly large collection of stochastic processes, including some dependent ones, that satisfy a similar concentration of measure result. 

A discrete stochastic process $\{x_i\}_{i\in \mathbb{Z}}$ is said to be $k-$dependent if for all $\tau\in \mathbb{Z}$ the joint stochastic variables $\{x_i\}_{i\leq \tau}$ are independent of the joint stochastic variables $\{x_i\}_{i\geq \tau + k + 1}$. This relation can be illustrated  by organizing the blocks as in 
\begin{equation*}
    \underbrace{\left ( \cdots , x_{-1},x_0,x_{1}, \cdots x_\tau\right )}_{
    \text{$
    \begin{array}{c}
    \text{ first block, } \\
    \text{ independent of last block }
    \end{array}
    $
    }
    }
    \ \ \   \biggl |
    \underbrace{\left (x_{\tau+1}, \cdots\cdots  ,  x_{\tau+k}\right )}_{
    \text{$
    \begin{array}{c} \text{dependent  variables} \\ \text{coupled to first, last blocks}
    \end{array}
    $}
    }
    \biggl | 
 \underbrace{\left (  x_{\tau+k+1},\cdots\cdots 
 \right )}_{
 \text{$
 \begin{array}{c}
 \text{last block,} \\
 \text{independent of first block}
 \end{array}
 $}
 }.
\end{equation*}
As is well-known, for a Markov chain the state $x_{\tau+1}$ is only coupled to $x_\tau$, and $x_{\tau+2}$ is only coupled to $x_{\tau+1}$. A Markov chain is therefore always $k-$dependent with $k=1$. It turns out that some $k-$dependent stochastic processes satisfy a concentration inequality quite similar in form to the Hoeffding inequality. We say that a $k-$dependent stochastic process satisfies an {\em $e(m)$-effective concentration inequality}  for a function $f$ if it is true that 
\begin{equation}
\mathbb{P}_{\{x\}}^m \left \{ \left | 
\frac{1}{m}\sum_{\ell\leq m} f(x_\ell) - \mathbb{E}(f(x_1(\cdot))) 
\right |>\epsilon 
\right \} \leq 2e^{-e(m)\epsilon^2/2\sigma^2(f(x_1(\cdot)))}
\label{eq:eff_mix}
\end{equation}
where $m$ is again the number of samples, now along the sample path,  and the function $m\mapsto e(m)$ is defined to be the effective number of samples. One class that satisfies a concentration inequality of this form  is the collection of strongly mixing Markov chains whose $\alpha$-mixing coefficients decay at an exponential rate. \cite{masry}. 
\begin{theorem}
\label{th:mixing}
Let the hypothesis of Theorem \ref{th:th3} and Example \ref{ex:b_v} hold, with the exception that the samples $\{x_\ell\}_{\ell\leq m}$ are collected along the sample path of a Markov chain, and $U=L^2_{\tilde{\mu}}(\Omega)$ with $\tilde{\mu}$ the probability distribution of $x_1$.  Suppose that the  Markov chain is strongly mixing in that there are three constants $a,b,c>0$ such that the  $\alpha$ -mixing coefficients $\alpha(i)$   decay exponentially as 
$$
\alpha(i)\leq ae^{-bi^{c}}
$$
for $i\geq 1$. 
Then  the Markov chain satisfies the $e(m)$-effective concentration inequality in 
Equation \ref{eq:eff_mix} 
for the effective number of samples $$
e(m):= \left  \lfloor m \left \lceil \left \{\frac{8m}{b} \right \}^{1/(c+1} \right \rceil^{-1} \right \rfloor, 
$$
and we have the error bound
$$
\mathbb{E}_{\mathbb{P}^{\{x\}}}\left ( \|\mathcal{U}f - \mathcal{U}_{j,z}f\|_U^2 \right ) \lesssim \lambda_j^r + \frac{j}{e(m)}.
$$
Here $\mathbb{P}^{\{x\}}$ is the probability distribution on $\Omega^m$ of the process $\{x_1, \ldots, x_m\}.$
\end{theorem}
\begin{proof}
The proof of this inequality follows exactly the same steps as that in the proof of Example \ref{ex:b_v}, but now Hoeffding's inequality is replaced with the $e(m)$-effective concentration inequality that is guaranteed to hold as discussed in \cite{masry}.
\end{proof}

\subsection{Case with $\mu$ and $p$ unknown}
\label{sec:mu_p_unknown}
As discussed in the introduction, the popularity of Koopman theory for the study of dynamical systems is due primarily to its utility for systems that are to some degree  uncertain or unknown. In this last section we study the case when the primary evidence about a dynamical system under study is a set of samples of its input-output behavior. We will show how, in some cases, it is possible to relate  the theory of  data-driven algorithms to other popular techniques that synthesize aspects of the theory presented earlier, statistical learning theory, and empirical risk minimization.   We begin our discussion in the next example which  presents the Extended Dynamic Mode Decomposition (EDMD) algorithm. 
\begin{framed}
\begin{example}
\label{ex:edmd}
Of the various techniques to approximate the Koopman operator,   the Dynamic Mode Decomposition (DMD) or the Extended Dynamic Mode Decomposition (EDMD) methods have seen widespread use.  These methods are  a popular approach to building approximations of Koopman, and  through duality, Perron-Frobenius, operators. They have  been studied extensively in applications to the study of evolution of observables of discrete or continuous flows in \cite{Giannikis2017,klus2016,klusdmd,klusTO}.

In this example we discuss in some detail the EDMD algorithm for estimating the Koopman operator associated with  a prototypical discrete dynamical system that evolves on a compact set of $\mathbb{R}$. Observation functions on the configuration space of the discrete flow are assumed to be scalar-valued.  
The general form of the EDMD algorithm applies to  much more  general flows, see references \cite{klusdmd,mezic2017kosda} for a discussion. These assumptions are not critical to  the analysis that follows, but certainly simplify notation. For instance, we can use the same bases for the representations of function $w:\Omega\rightarrow \Omega$ and for the observables $f$ in our analysis below.  In the  general theory these may be different bases. 

To be specific,  we are given the  discrete flows
$$
x_{n+1}=y_n:=w(x_n) \in \Omega \subset \mathbb{R}
$$
on the compact set $\Omega$. 
Suppose we have a measurement process that provides (possible noisy) estimates of the input-output response  $z:=\left \{(x_i,y_i)\right \}_{i=1}^m\subset \Omega \times \Omega:=Z$. 
We let $\mathcal{F}$ denote the family of  observables, or observable functions,  on  the configurations of the dynamical system.   The family $\mathcal{F}$ consists of real-valued functions of a real variable in this example.  For a fixed observable $f\in \mathcal{F}$,  we therefore have 
$$
\tilde{y}_n:=f(x_{n+1}):=f(w(x_n))=(\mathcal{U}f)(x_n) \in \Omega\subset \mathbb{R}.
$$

 We briefly summarize the EDMD  algorithm for the estimation of a Koopman operator in the form it is described in \cite{Korda2017Convergence}. Given the measurements of system input-output response   $\left \{ (x_i,y_i) \right \}_{i=1}^n$,  we form the empirical arrays 
\begin{align*}
\mathbb{X}&:=[x_1,\ldots, x_m] \in \mathbb{R}^{d\times m}, \\
\mathbb{Y}&:=[y_1,\ldots, y_m] \in \mathbb{R}^{d\times m}, 
\end{align*}
for $d=1$. We  choose a family of functions  $A_n:=\left \{ \phi_i \right \}_{i=1}^n$ that will be used as the basis for constructing approximations, $A_n:=\text{span}\left \{ \phi_i \right \}_{i=1}^n$. Each $\phi_i:\Omega \rightarrow \mathbb{R}.$ We define the data matrices
\begin{align*}
\Phi(\mathbb{X})&:=\left [ \phi(x_1), \ldots, \phi(x_m) \right]\in \mathbb{R}^{n\times m}, \\
\Phi(\mathbb{Y})&:=\left [ \phi(y_1), \ldots, \phi(y_m) \right] \in \mathbb{R}^{n\times m}
\end{align*}
with $ \phi := [\phi_1, \ldots , \phi_n]'$
and construct the matrix 
$$
M_{n,m}:=\Phi(\mathbb{Y})\Phi(\mathbb{X})^{\dagger}\in \mathbb{R}^{n\times n}
$$
with $\Phi(\mathbb{X})^{\dagger}:=
\Phi(\mathbb{X})'(\Phi(\mathbb{X})\Phi(\mathbb{X})')^{\dagger}\in \mathbb{R}^{m\times n}$ the right  pseudoinverse of $\Phi(\mathbb{X})$. 
Note carefully the dependence of the  matrix $M_{n,m}$ on the dimension $n$ of the   approximation space $A_n$ and the number of samples $m$. 
As summarized in \cite{Korda2017Convergence,klusdmd}, the final approximation $\mathcal{U}^{edmd}_{n,m}:A_n\rightarrow A_n$ is given by
$$
\left ( \mathcal{U}_{n,m}^{edmd}f \right )(x)
:= \phi'(x) M_{m,n}' c
$$
for any $f=\sum_{i\leq n}\phi_i c_i\in A_n\subset \mathcal{F}$ and $c\in \mathbb{R}^n$. 

Recent references \cite{Korda2017Convergence,klusdmd} derive various quite general results regarding the convergence of the finite dimensional operators $\mathcal{U}_{n,m}^{edmd}$ to the Koopman operator $\mathcal{U}$. For example, \cite{Korda2017Convergence} shows that 
$$
\lim_{m\rightarrow \infty} \|(\mathcal{U}^{edmd}_{n,m}-\mathcal{U}_n)f\|_{A_n}=0
$$
for any $f\in A_n:=\mathcal{F}$ and any norm on the finite dimensional space $A_n$.  We are interested in relating this analysis to methods of distribution-free learning theory and the derivation of stronger convergence rates in a very specific case. This is the subject of the next few  theorems.
\end{example}
\end{framed}

We now derive an alternative form for $\mathcal{U}^{edmd}_{n,m}f$ that holds when the approximant  space  consists of piecewise constants over a uniform dyadic grid. In the example above the subscript $n$  on $A_n$ is equal to dimension of $A_n$, since this is  convention that is prevalent in the Koopman theory literature such as  \cite{Korda2017Convergence,klusdmd}. On the other hand, the convention in approximation space theory is that the subscript $j$ on  $A_j$  is the resolution level  of the grid used to define $A_j$. In our convention we denote  by $n_j:=\#(A_j)$  the dimension of $A_j$. When we make comparisons of the two approximants, we define $\mathcal{U}^{edmd}_{j,m}:=\mathcal{U}^{edmd}_{n_j,m}$ from Example \ref{ex:edmd} to use equivalent indices for both approaches.  
\begin{theorem}
\label{eq:error_EDMD}
Let $U:=L^2(\Omega)$ and  the approximation spaces $A_j$ be given as 
\begin{align}
A_j&:=\text{span}\left \{ \psi_{i,k} \ | \ 0\leq i \leq j, k\in \Gamma_i^\psi \right \}
= 
\text{span}\left \{ \phi_{j,k} \ | \ k\in \Gamma_j^\phi \right \}, \notag \\
& = \text{span}\left \{ 1_{\square_{j,k}} \ | \   k\in \Gamma_j^\phi \right \}, \label{eq:pc_approx}
\end{align}
 the span of piecewise constants on a grid of level $j$. Here 
 $\phi_{j,k}$ and $\psi_{j,k}$  are the Haar scaling functions and wavelets, respectively,  defined on resolution level $j$. The basis functions $1_{\square_{j,k}}$ are the characteristic functions of the cells $\square_{j,k}$ that define the dyadic grid. They are defined so that their supports define a partition of the domain $\Omega$. 
Then the EDMD approximation $\mathcal{U}_{j,m}^{edmd}:=\mathcal{U}^{edmd}_{n_j,m}$  of $\mathcal{U}$ 
can be written  
\begin{align}
(\mathcal{U}^{edmd}_{j,m}f)(x)
&:=  \sum_{k\in \Gamma^\phi_j}p_{j,k}
1_{{\square}_{j,k} }(x)
\end{align}
for each $f:=\sum_{i\leq n}c_i 1_{\square_{j,i}}$ with 
$$
p_{jk}:=
\left \{ 
\begin{array}{ccc}
\frac{ 
\sum_{\ell\leq m} \sum_{i\leq n} 1_{\square_{j,k}}(x_\ell)   1_{\square{j,i}}(y_\ell) c_i
}  
{ \sum_{s\leq m}1_{{\square}_{j,k} }(x_s)}
 &  & \{x_s\}_{s\leq m}\cap \square_{j,k}\not = \emptyset, \\
 &  & \\ 
0 & & \{x_s\}_{s\leq m}\cap \square_{j,k} = \emptyset .
\end{array} 
\right .
$$
\end{theorem}
\begin{proof}
In terms of the discussion in Example \ref{ex:edmd}, we have 
$$
\mathcal{U}^{edmd}_{j,n}f:=\phi'(x) 
\left (\Phi(\mathbb{X})\Phi'(\mathbb{X})\right )^{\dagger} \Phi(\mathbb{X})\Phi'(\mathbb{Y})c
$$
for $f=\sum_{i\leq n}c_i 1_{\square_{j,k}}$ and $c:=\{c_1,\ldots,c_n\}'\in \mathbb{R}^n$, 
\begin{align*}
    \Phi(\mathbb{X}):=\left [  
    \begin{array}{ccc}
    1_{\square_{j,1}}(x_1)& \cdots& 1_{\square_{j,1}}(x_m)\\
    \vdots &\ddots & \vdots \\
    1_{\square_{j,n}}(x_1)& \cdots & 
    1_{\square_{j,n}}(x_m)
    \end{array}
    \right ],
\end{align*}
and 
\begin{align*}
    \Phi(\mathbb{Y}):=\left [  
    \begin{array}{ccc}
    1_{\square_{j,1}}(y_1)& \cdots& 1_{\square_{j,n}}(y_m)\\
    \vdots &\ddots & \vdots \\
    1_{\square_{j,1}}(y_1)& \cdots & 
    1_{\square_{j,n}}(y_m)
    \end{array}
    \right ].
\end{align*}
Since the supports of the basis functions define a partition of the domain $\Omega$, each column of $\Phi(\mathbb{X})$, or for that matter  of $\Phi(\mathbb{Y})$, is zero except for one entry that is equal to one. 
For the purposes of this proof only, suppose that we permute the columns of $\Phi(\mathbb{X})$ so that the first $n_1$ columns correspond to the samples in $\mathbb{X}$  that fall in $\square_{j,1}$, the next $n_2$ columns correspond to samples in $\square_{j,2}$, and so forth, with $m=\sum_{i\leq n}n_i$. We permute the matrix $\mathbb{Y}$ based on the new ordering of $\mathbb{X}$. 
 With this reordering of samples, $\Phi(\mathbb{X})\Phi(\mathbb{X})'$ is a  diagonal matrix, and  each entry $(k,k)$ of the major diagonal of $\Phi(\mathbb{X})\Phi'(\mathbb{X})$ is equal to the number of samples in that lie in $\square_{j,k}$. It then holds that 
 $$
[ \left (\Phi(\mathbb{X})\Phi(\mathbb{X})' \right )^+]_{s,t}:=
 \left \{
 \begin{array}{ccc}
 0 &  & s\not = t, \\
 0 &  & s=t \text{ and } \{x_k\}_{k\leq m} \cap  \square_{j,s}=\emptyset, \\
 \left (\sum_{k\leq m} \square_{j,s}(x_k) \ \right)^{-1}
 & & s=t \text{ and }  \{x_k\}_{k\leq m} \cap  \square_{j,s}\not =\emptyset.
 \end{array}
 \right .
 $$
 The form in the above theorem now follows by explicitly expanding 
 $$\Phi(\mathbb{X})^{\dagger}:=\Phi(\mathbb{X})'\left (\Phi(\mathbb{X})\Phi(\mathbb{X})' \right )^{\dagger}.$$
\end{proof}
We note that the above derivation has been provided for completeness, and a slightly more general result of the same nature can be found in \cite{klus2016}. In the next section we show how the rates of convergence can derived from Theorem \ref{eq:error_EDMD}.

\subsection{Distribution-Free  Learning Theory}

Again, our interest is in approximating the action of the Koopman operator on a any function $f\in \mathcal{F}$.  We  we  given noisy measurements $\{(x_i,y_i)\}_{i\leq m}$ of the input-output behavior of the system. We do not known the distribution of this process since neither $w$ nor $\mu$ is known.  We will see that the EDMD approximation of the Koopman operator is closely related to approximations that arise in results of distribution-free learning theory.  \cite{gyorfy, vapnik} 

Note that, since it is assumed here that the observable function  $f$ is known, we can define another measurement process
$$
\tilde{y}_n:=f(y_n):=f(x_{n+1})=(f\circ w)(x_n):=(\mathcal{U}f)(x_n).
$$
Each $\tilde{y}_n$ can be computed from $(x_n,y_n)$ since $f$ is known. 
For the remainder of this section, let us just consider the measurement process $\{(x_i,\tilde{y}_i)\}_{i\leq m}$, which we will use to estimate the function $f\circ w$ in $\tilde{y}=f(w(x))$. Since the distribution of the  measurement process $\{(x_i,y_i)\}_{i\leq m}$ is unknown, so is the distribution of    the measurement process $\{(x_i,\tilde{y}_i)\}_{i\leq m}$.   In studies of the convergence of data-driven algorithms it is common to make assumptions regarding the underlying distribution  in order to establish convergence. These assumptions can be cast in terms of ergodicity or mixing of the process, for example. \cite{klusdmd} 

Here we depart from these conventional approaches in Koopman theory and employ a somewhat different strategy that is popular in statistical and distribution-free  learning theory. As discussed earlier, we are ultimately most interested in the case we are given observations of states along the sample path of a Markov chain. But again, error estimates are usually simpler to derive assuming that the samples are IID. We start by considering the IID assumption, and then discuss how the analysis can be extended to more general dependent processes such as Markov chains.

In view of the above we first suppose that the measurements $\{(x_i,\tilde{y}_i)\}_{i\leq m}$ are independent and identically distributed with respect to  some fixed but unknown  measure $\tilde{\mu}$ on 
 $Z:=\Omega \times \Omega$.  Such a collection can be obtained assuming that a set of initial conditions $\{x_i\}_{i\leq m}$ are selected randomly according to an unknown fixed (marginal) measure $\mu$ on $\Omega$, and then observing the state $y_i$ after a  single step of the dynamics. 
 If we define the unknown function $g:=f\circ w$, approaches in distribution-free learning theory seek to construct estimates $g^*\in L^2_\mu(\Omega)$ that approximately minimize the cost function 
$$
\mathcal{J}(g):= \frac{1}{2}  \int_{z:=(x,\tilde{y})\in Z} (\tilde{y}-g(x))^2\tilde{\mu}(dz),
$$
$$
g^*=\underset{g\in L^2_\mu(\Omega)}{\text{argmin}} \  \ {\mathcal{J}(g)}.
$$
 Any measure  $\tilde{\mu}$ on $Z$ can be factored as $\tilde{\mu}(dz):= \tilde{\mu}(dy,x){\mu}(dx)$ where $z:=(x,y)$, $\mu$ is the marginal probability 
$$
\mu(A)=\tilde{\mu}(A\times \Omega)
$$ 
for all measurable $A\subset \Omega$, and $\tilde{\mu}(A,x)$ is the conditional probability distribution of a measurable set $A\subset \Omega$ given $x\in \Omega$.
It is an elementary exercise that the function $g^*$ that achieves the minimum of the ideal functional $\mathcal{J}$ is the regressor function given by 
$$
g^*(x):=\int_\Omega \tilde{y} \tilde{\mu}(d\tilde{y},x).
$$
Since $\tilde{\mu}$ and $\tilde{\mu}(dy,x)$ are  unknown, however, it is not possible to use this expression or the cost function $\mathcal{J}$ directly for the construction  of estimates. \cite{cs2002}

It is one of the fundamental results of distribution-free learning theory that it is possible to construct useful estimates of the minimizer $g^*$ from the principle of empirical risk minimization.  That is, we seek a $g^*_{j,z}$ that minimizes the empirical risk
\begin{align*}
\mathcal{J}_z(g):=\mathcal{J}_z(g;f):=\frac{1}{m} \sum_{i\leq m} (\tilde{y}_i - g(x_i))^2=
\frac{1}{m} \sum_{i\leq m} (f({y}_i) - g(x_i))^2, 
\end{align*}
\begin{align*}
g^*_{j,z}&:=
\underset{g\in A_j}{\text{argmin}} \ \ \mathcal{J}_z(g) =\underset{g\in A_j}{\text{argmin}}\ \  \mathcal{J}_z(g;f).
\end{align*}
The notation $\mathcal{J}_z(g;f)$ emphasizes that here the empirical risk $\mathcal{J}_z(g)$ depends parametrically on the fixed function $f\in \mathcal{F}$. There is no such dependence in the classical problems of distribution-free learning theory, and in this respect, the approximation problem solved by  the EDMD algorithm is more general. 

It is shown in \cite{binevI} that, under the hypotheses of Theorem \ref{eq:error_EDMD},  the solution $g^*_{j,z}$ of the above empirical risk minimization problem is  
 written  as  
\begin{align}
g^*_{j,z}(x)
&:=  \sum_{k\in \Gamma^\phi_j}g_{j,k}
1_{{\square}_{j,k} }(x)
\end{align}
 with 
$$
g_{j,k}:=
\left \{ 
\begin{array}{ccc}
\frac{ 
\sum_{\ell\leq m}  1_{\square_{j,k}}(x_\ell)  \tilde{y}_\ell
}  
{ \sum_{s\leq m}1_{{\square}_{j,k} }(x_s)}
 &  & \{x_s\}_{s\leq m}\cap \square_{j,k}\not = \emptyset, \\
 &  & \\ 
0 & & \{x_s\}_{s\leq m}\cap \square_{j,k} = \emptyset .
\end{array} 
\right .
$$
The similarity of the solution $g^{*}_{j,z}$ of the empirical risk minimization problem to the construction of the EDMD approximation $\mathcal{U}^{edmd}_{n,m}$ is immediate when we note that, for any $f\in A_j$ having a representation $f:=\sum_{i\leq n}c_i 1_{\square_{j,i}}$, we have 
$$
\tilde{y}_\ell =f(y_\ell)=\sum_{i\leq n} 1_{\square_{j,i}}(x_\ell) c_i.
$$
We conclude that if $f\in A_j$, the approximant obtained by the EDMD algorithm coincides with that obtained by empirical risk minimization. 
Based on empirical risk minimization, we therefore can define 
\begin{align*}
(\mathcal{U}^{erm}_{j,z}f)&:=g^*_{j,z}, 
\end{align*}
for the function $f\in \mathcal{F}$. The above discussion illustrates that 
$$
\mathcal{U}^{edmd}_{j,z}:=\mathcal{U}^{erm}_{j,z}\Pi_j
$$
when $A_j:=\text{span}\left \{1_{\square_{j,k}} \ | \ k\in \Gamma_{j}^\phi \right\}$ and $\Pi_j$ is the $L^2(\Omega)$-orthonormal projection onto $A_j$.

\subsection{Error Bounds for EDMD}
We see from the above study that when $f\in A_j$, the approximations   generated by the EDMD algorithm and that constructed from empirical risk minimization coincide. This result is instructive and of interest in its own right. In addition, it is possible to use this analogy to obtain rates of convergence of approximation that rely on choices of priors, just as in the last few sections.  Recall again that $U:=L^2_\mu(U)$ in this section. Using the decomposition of the error in terms of approximation space and variance contributions as discussed in Section \ref{sec:mu_unknown_p_known}, we have 
\begin{align}
    \|\mathcal{U}f&-\mathcal{U}^{edmd}_{j,z}f\|_U  \notag  \\
    &\leq 
    \underbrace{\|\mathcal{U}f-\mathcal{U}_jf\|_U}_{\text{bias}} + \underbrace{\|\mathcal{U}_jf-\mathcal{U}^{erm}_{j,z}f\|_U}_{\text{variance}}+ \|\mathcal{U}^{erm}_{j,z}f-\mathcal{U}_{j,z}^{edmd}f\|_U. \label{eq:vb_edmd}
\end{align}
The first term will be bounded by a constant $O(2^{-rj/2})$ when $\mathcal{U}{f}\in {A}^{r,2}(U)$, and the second term will be bounded by a constant $\epsilon$ that holds for all samples outside a bad set of samples.  As before the  measure of the bad set of samples decays exponentially as a function of the number of samples, but grows as a function of the dimension of the approximant space $A_j$.
We have discussed the analysis of these terms and the form of the bounds that arise for them.  Only the last term is new in this inequality, so we study it more carefully. 
We have 
\begin{align*}
\|\mathcal{U}^{erm}_{j,z}f-\mathcal{U}_{j,z}^{edmd}f\|_U
&= \|\mathcal{U}^{erm}_{j,z}(I-\Pi_j)f\|_U.
\end{align*}

But, in this section  it is assumed that the function $f\in C(\Omega)$ to ensure that the evaluation operator $f\mapsto f(x_{\ell})$ is well-defined for all samples. It follows that for any fixed $z=\{(x_\ell,y_\ell)\}_{1\leq m}$ and $j\in \mathbb{N}$ we have 
\begin{align*}
\|\mathcal{U}_{j,z}^{erm}f\|_{L^\infty(\Omega)}:=\sup_{x\in \Omega} \left |  \sum_{k\in \Gamma_j^\phi} 
\frac{\sum_{\ell\leq m}1_{\square_{j,k}}(x_\ell)f(y_\ell)}{\sum_{s\leq m}1_{\square_{j,k}}(x_s)} 1_{\square_{j,k}}(x) \right | \leq \|f\|_{L^\infty(\Omega)},
\end{align*}
and this inequality holds for all $f\in C(\Omega)$. 
Now, since we have assumed that $A^{r,2}(U)\subset C(\Omega) \subset U$,  we have 
\begin{align*}
\|\mathcal{U}^{erm}_{j,z}f-\mathcal{U}_{j,z}^{edmd}f\|^2_U & 
= \int \left | \left (\mathcal{U}^{erm}_{j,z} (I-\Pi_j)f  \right)(x)  \right |^2 \mu(dx),
\\ 
&\lesssim \mu^2(\Omega) \|\mathcal{U}^{erm}_{j,z} (I-\Pi_j)f \|_{L^\infty(\Omega)}^2, \\
&\lesssim \|(I-\Pi_j)f\|_{L^\infty(\Omega)}^2.
\end{align*}
In this section the approximant space $A_j$ is the span of the Haar scaling functions on a grid having resolution level $j$, and the operators $\Pi_j$ are the $L^2(\Omega)$-orthonormal projections onto $A_j$.  Just for this particular basis, the orthonormal projection $\Pi_jf$ is identical to the averaging operator studied  in Example \ref{ex:two_cases}.  From that example, we know that $\|(I-\Pi_j)f\|_{L^\infty(\Omega)} \approx 2^{-rj}$ when $f\in \text{Lip}^*(r,L^\infty(\Omega))$. 
The overall estimate of the error in the EDMD approximation now results from combining this bound with the triangle inequality in Equation \ref{eq:vb_edmd}. One way to derive an   error bound based on Equation \ref{eq:vb_edmd} assumes   IID samples $\{(x_i,y_i)\}_{i\leq m}$, such as a those generated from a  random collection of initial conditions and measurement of the resulting one step response state. The analysis for this case is essentially the same as in our previous discussion.  We really are interested with the samples are along the sample path of a Markov chain. We briefly discuss this case next. 

There are a number of precise upper bounds in the literature on the second term  in Equation \ref{eq:vb_edmd}. A good overview can be found in \cite{binevI}. Given the tradeoff between the approximation space and bias errors, much effort has been devoted to ``equilibrating'' the leftmost two terms in Equation \ref{eq:vb_edmd} to achieve a balanced  overall rate of convergence.  
Let us suppose that $f\in A_j$ to keep the exposition simple, so that the rightmost term in Equation \ref{eq:vb_edmd} is zero. Furthermore, also suppose that $f$ is in a compact, convex subset of $C(\Omega)$.  The equilibration techniques work by choosing a relationship between the number of samples $m$ and the dimension of  approximant space $n_j$ as these numbers  grow. One example  rule in \cite{binevI}  calculates the  required mesh resolution $j$ of the space $A_j$, and therefore the number of basis functions to be used,  in terms of the number of samples $m$. 
With that rule, Equation \ref{eq:vb_edmd} can be  used following  \cite{binevI} to show that 
$$
\mathbb{E}_{{\nu}^m}\left (
\|\mathcal{U}f - \mathcal{U}^{edmd}_{j,z}f\|_{U}
\right ) \lesssim \left (\frac{\log m}{m}\right )^{2r/(2r+1)}.
$$
Here $\mathbb{E}_{{\nu}^m}(\cdot)$ is the expectation over the samples $z:=\{(x_i,\tilde{y}_i)\}_{i\leq m}$ with respect to the product measure ${\nu}^m$ on $Z^m$. It is known that this estimate is the best possible rate of convergence except for the $\log m$ term, and therefore it is referred to as a semi-optimal bound. \cite{devore2006approxmethodsuperlearn}
As discussed in more detail in \cite{kurdiladep}, an analogous result follows for some strongly mixing Markov chains. By modifying the rule derived in \cite{binevI} to employ the effective number of samples $e(m)$ as in  Reference \cite{kurdiladep}, it follows that 
$$
\mathbb{E}_{\mathbb{P}^m_{\{z\}}} \left (\|\mathcal{U}f - \mathcal{U}^{edmd}_{j,z}f\|_{U}  \right )\lesssim 
\left (\frac{\log e(m)}{e(m)}\right )^{2r/(2r+1)}.
$$

\bibliographystyle{siamplain}
\bibliography{references}

\clearpage 
\appendix
%
%
\section{Detailed Examples of $L^2(\Omega)-$Orthonormal Wavelets}
\label{sec:app_A}
\begin{framed}
 \begin{example}[Daubechies' Orthonormal Wavelets]
\label{ex:ON_wave}
We initially study the case when $\Omega\subseteq \mathbb{R}$ and $f:\Omega \rightarrow \mathbb{R}$.  Once the conventions for the one dimensional domain are understood, we briefly discuss the generalization for certain sets  $\Omega\subseteq \mathbb{R}^d$ and for vector-valued functions $f:\Omega\rightarrow \mathbb{R}^d$.   In this example  we  consider bases for the approximation space that are given by Daubechies orthonormal, compactly supported wavelets. This example can be considered a model or prototype for more general multiscale bases discussed later. It introduces the specifics of the  conventional numbering, indexing, and definition of bases when the approximation space is defined in terms of a wavelet basis. This case has perhaps the most limited applicability among the examples we discuss: the  limitations for the most part arise due to the fact that the class of orthonormal, compactly supported  wavelets in \cite{daubechies88} do not have a closed form expression. Still, computations can be carried out using quadratures that can derived as described in \cite{dahmenmultiscale}.  Despite these limitations, the approach discussed in this example can be applied directly  to Perron-Frobenius and Koopman operators defined over the torus $T^d\subseteq \mathbb{R}^d$.  This is an important application in itself. 

When $\Omega\subseteq \mathbb{R}$, we choose the orthonormal scaling function  to be $\phi:= \presuper{N}{\phi}$,  the Daubechies orthonormal scaling function of order $N\geq 2 $.  \cite{daubechies88} The associated  orthonormal wavelet is denoted as $\psi:=\presuper{N}{\psi}$. The support of the functions $\presuper{N}{\phi}$ and $\presuper{N}{\psi}$ is $[0,2N-1]$.  It is well-known that these wavelets can be constructed so that they have  arbitrarily high Lipschitz smoothness by choosing the index $N$ large enough.  Moreover, the   scaling functions reproduce polynomial functions of order $N-1$ and the wavelets have $N-1$ vanishing moments. These properties are  crucial in establishing the approximation rates for spaces of piecewise polynomials contained in the span of these functions. In fact, we find that approximation by the order $N$ Daubechies scaling functions and wavelets yield convergence that is at least  $O(2^{-rj})$ for functions in $A^{r,2}(L^2(\Omega)) \approx W^{r}(L^2(\Omega))$ with $N=r$.  This topic is discussed in more detail in Sections \ref{sec:approx_splines} and \ref{sec:zero_moments}. 

The scaling functions and wavelets over a dyadic grid of cubes having side length $2^{-j}$ are defined to be $\phi_{j,k}(x):=2^{jd/2}\phi(2^jx-k)$ and $\psi_{j,k}(x):=2^{jd/2}\psi(2^jx-k)$ for $k\in \mathbb{Z}^d$ and $j\in \mathbb{N}_0$ with $d=1$. In the conventional notation of multiresolution analyses, we set the scaling spaces $V_j:={\text{span}}\left \{\phi_{j,k} \ | \ k\in \Gamma^\phi_j \right \}$ and detail spaces $W_j:={\text{span}}\left \{\psi_{j,k} \ | \   k\in \Gamma^\psi_j \right \}$ with $\Gamma^\phi_j$ and $\Gamma^\psi_j$ the set of admissible translates for the scaling functions and wavelets on grid level $j$. In this first example, when   we assume  $\Omega:= \mathbb{R}$, we have $\Gamma^\phi_j=\Gamma^\psi_j=\mathbb{Z}$. But this will change when we consider compact domains $\Omega\subset \mathbb{R}^d$ below or multiwavelets in the next example. 
We have the representation 
\begin{equation}
f:= \sum_{k\in \Gamma^\phi_0}(f,\phi_{0,k})_{U}\phi_{0,k} + \sum_{j\in \mathbb{N}_0} \sum_{k\in \Gamma^\psi_j} (f,\psi_{j,k})_{U} \psi_{j,k}
\end{equation}
for any $f\in U$. 
It is standard to modify this representation and create  the multiscale summation 
\begin{equation}
f:= \sum_{j\in \mathbb{N}_0} \sum_{\tilde{k} \in \Gamma^{\tilde{\psi}}_j} (f,\tilde{\psi}_{j,\tilde{k}})_{L^2} \tilde{\psi}_{j,\tilde{k}} \label{eq:multiscale_ON}
\end{equation}
with the modified set of indices 
\begin{align*}
    \begin{array}{llll}
    \tilde{\psi}_{j,k}:=\psi_{j,k},
    &  k\in \tilde{\Gamma}^\psi_{j}:=\Gamma_j^\psi
    & \text{ for }  
    & 1\leq j \leq \infty
    \\
    \tilde{\psi}_{0,\tilde{k}}:=\phi_{0,k}, 
    &\tilde{k}:=(k,0) \in \tilde{\Gamma}_0^\psi:= \Gamma^\phi_{0}\times  \Gamma^\psi_0, 
    & \text{ for } 
    & j=0, 
    \\
    \tilde{\psi}_{0,\tilde{k}}:=\psi_{0,k}, 
    &\tilde{k}:=(0,k) \in \tilde{\Gamma}_0^\psi:= \Gamma^\phi_{0}\times  \Gamma^\psi_0, 
    & \text{ for } 
    & j=0.
     \end{array}
\end{align*}
%

\noindent 
In the discussion that follows, whenever we refer to a multiscale representation, we suppress the $\tilde{\cdot}$ notation and assume that the wavelets, scaling functions, and index sets have been modified so as to obtain a form similar to that in  Equation \ref{eq:multiscale_ON}, or just as in Equation \ref{eq:Wspace}. 

So far we have selected $\Omega=\mathbb{R}$. In this case the index sets $\Gamma_j^\phi,\Gamma_j^\psi:=\mathbb{Z}$ are infinite. However, suppose that we instead consider the compact domain $\Omega\subset \mathbb{R}$. For illustration suppose $\Omega:=[0,1]$. When we identify basis functions modulo periodization  over $[0,1]$,  we can arrange  that the index sets $\Gamma^\psi_j$ and $\Gamma^\phi_j$  contain $O(2^j)$ functions for all $j$ larger than some fixed level $j_0$. For example, for $\phi:= \presuper{N}{\phi}$ with $N=3$,   we obtain $2^{j}$ indices in $\Gamma^\phi_{j}$ for $j\geq 3$. 
The  approximation spaces are defined to be $A_{j+1}:=V_{j+1}:=V_j \oplus W_j$ for each $j_0\leq j\in \mathbb{N}_0$, and  now $V_j,W_j$ are finite dimensional. In the one dimensional case we have 
\begin{align*}
    \#(\Gamma_j^\phi)=2^j \quad \text{ and } \quad 
    \#(\Gamma_j^\psi)=2^j
\end{align*}
for all $j\geq j_0$ large enough. 
The orthogonal projections $\Pi_{j}:= U \rightarrow A_j$ for each $j\geq j_0$ are written as 
\begin{align*}
\Pi_{j}f&:=\sum_{k\in \Gamma^\phi_j} (\phi_{j,k},f)_{U} \phi_{j,k}
=  \sum_{ j_0\leq m \leq j} \sum_{k\in \Gamma^\psi_m} (\psi_{j,k},f)_U \psi_{j,k},
\end{align*}
where again the latter multiscale decomposition is modified to include both wavelets and scaling functions in the lowest level resolution grid $j_0$. 
We thereby obtain easily computed representations of the norm and inner product on the  approximation spaces $A^{r,2}(U)$. We have 
\begin{align*}
(f,g)_{A^{r,2}(U)}&= \sum_{j_0\leq j\in \mathbb{N}_0} 2^{2jr} \sum_{k\in \Gamma^\psi_{j}} (f,\psi_{j,k})_{U}(g,\psi_{j,k})_{U},\\
\|f\|_{A^{r,2}(U)}^2&= \sum_{j_0\leq j\in \mathbb{N}_0} 2^{2jr} \sum_{k\in \Gamma^\psi_j} |(f,\psi_{j,k})_{U}|^2,
\end{align*}
for $f,g\in A^{r,2}(U)$. No generality is lost in that the summation above starts for a coarsest grid level $j_0$ that may be greater that $0$. 

The modification of the above framework for domains $\Omega\subseteq \mathbb{R}^d$ for $d>1$ often employs tensor products of bases. First, consider only the case when $f:\Omega \rightarrow \mathbb{R}$ and $\Omega:=\mathbb{R}^d$. We set $\psi^{0}=\phi$ and $\psi^1:=\psi$. Let $e:=(e_1,e_2,\ldots,e_d)$ be the coordinates of the corners of the unit cube $\left \{0,1 \right \}^d$ in $\mathbb{R}^d$. We set the tensor product basis $\psi_e(x):=\psi^{e_1}(x_1)\cdots \psi^{e_d}(x_d)$, and analogously to the setup above we define the dilated and translated wavelets as 
$ \psi_{j,(e,k)}:=2^{jd/2}\psi_{e}(2^jx-k)$. We then have an expansion of the form in Equation \ref{eq:multiscale_ON} with $\tilde{k}:=(e,k)\in \tilde{\Gamma}_j^{\tilde{\psi}}:= \left \{0,1\right \}^d \times  \Gamma_j^\psi  
$ for $j\geq 1$.  Modifications are also made for $\tilde{\Gamma}^{\tilde{\psi}}_0$, similar to the above, to include scaling functions and wavelets on the coarsest level $j_0\geq 0$. 
The restriction of the domain $\Omega\subseteq \mathbb{R}^d$ above to generate bases for periodic functions on the torus $\mathbb{T}^d$ is straightforward. 

Finally, we note that when the family of  functions  
$$\left \{ \psi_{j,(e,k)}\ | \  j\in \mathbb{N}_0, e\in \left \{0,1\right \}^d, k\in \Gamma_j^\psi \right \}
$$
form an orthonormal basis for $U:=L^2(\Omega)$, it also generates a component-wise orthonormal basis for the functions  $f:= \Omega \to \mathbb{R}^d$ on $\Omega\subseteq \mathbb{R}^d$. If $\left \{ E_i\right \}_{1\leq i\leq d}$ is the canonical basis for $\mathbb{R}^d$, then 
$$
\left \{\psi_{j,(e,k)}E_\ell \biggl | \ j\in \mathbb{N}_0, e\in \left \{ 0,1\right \}^d, k\in {\Gamma}_{j}^{{\psi}}, 1\leq \ell\leq d \right \}
$$ 
is an orthonormal basis for $U:=(L^2(\Omega))^d$ in its usual inner product. This family of vector-valued functions serves as the basis for the approximation spaces $A^{r,2}(U)$ in this case. 
\end{example}
%
%
\begin{example}[Orthonormal  Multiwavelets]
\label{ex:ON_multi}
In our discussion of Example \ref{ex:ON_wave} it was noted that the orthonormal wavelets, not having a closed form expression, can be more difficult to employ than other wavelet constructions. The definition of wavelets that have many of the desirable properties of orthonormal, compactly supported wavelets has been studied carefully over the past 25 years.  In this example we consider compactly supported multiwavelets \cite{keinert,dgh}. These wavelets have closed form expressions, are written in terms of piecewise polynomials, exhibit useful symmetry and antisymmetry properties, and have good approximation properties. The essential difference in this framework is that there is a finite family of multiscaling functions $\left \{ \phi_{p}\right \}_{p=1,\ldots,b}$ and an associated family of multiwavelets $\left \{ \psi_q\right\}_{q=1,\ldots,a}$. The multiscaling functions and multiwavelets are then defined as 
$
\phi_{j,(p,k)}:=2^{jd/2}\phi_p(2^{j}x-k) $ 
and
$
\psi_{j,(q,k)}:=2^{jd/2}\psi_q(2^jx-k)
$
, respectively. We set  
 the spaces 
\begin{align*}
V_j&:=\text{span}\left \{ \phi_{j,(p,k)} \ | \   (p,k) \in \Gamma^\phi_j \right \},
\\ 
W_j&:=\text{span}\left \{ \phi_{j,(q,k)} \ | \  (q,k) \in \Gamma^\psi_j \right \},
\end{align*}
for $j\in \mathbb{N}_0$. For any $f\in L^2(\mathbb{R}^d)$, when the level $j=0$ is redefined to include both scaling functions and wavelets, the multiscale  expansion takes the form 
\begin{equation}
f:= \sum_{j\in \mathbb{N}_0}
\sum_{(q,k)\in \Gamma^{{\psi}}_j} (f,\psi_{j,(q,k)})_{L^2} \psi_{j,(q,k)},  \label{eq:multiwavelet_ON2}
\end{equation}
which again can be recast in the form of Equations  \ref{eq:Wspace} or \ref{eq:multiscale_ON} with an appropriate re-definition of the indices and index set. 

The multiscale expansion in Equation \ref{eq:multiwavelet_ON2} contains an infinite number of terms in each $\Gamma_j^\psi$ when $\Omega=\mathbb{R}^d$, and it must be modified for representations of functions over a compact domain $\Omega$.  This will result in a finite number of terms in each $\Gamma_j^\psi$.  One of the nice features of these multiwavelets is that it is a simple matter to adapt the  expansion over compact domains that are the union of dyadic cubes on some fixed resolution level $j_0$. The index sets $\Gamma_{j}^\psi$ and $\Gamma_{j}^\phi$ are then modified by retaining only those linearly independent restrictions of basis functions whose support intersects  $\Omega$ and satisfy the desired boundary conditions on $\partial \Omega$. This is particularly easy task for domains $\sim [0,1]^d$ using symmetry and antisymmetry of the functions. \cite{dgh}

Once the modifications of the index sets $\Gamma_{j}^\phi,\Gamma_{j}^\psi$ is complete, so that they are finite dimensional, we again define the finite dimensional approximation spaces 
$A_{j+1}:=V_{j+1}:=V_j\oplus W_j$. 
The orthogonal projections $Q_{j}: L^2(\Omega) \rightarrow A_j$ are now written in the form  
\begin{align*}
Q_{j}f:=\sum_{(q,k)\in \Gamma^\psi_j} (\psi_{j,(q,k)},f)_{U} \psi_{j,(q,k)}.
\end{align*}
When $\Omega \subset \mathbb{R}$, that is $d=1$, the index set $\Gamma^\psi_j$ contains $O(a2^j)$ functions in this case. The norm on the approximations is written as 
\begin{align*}
\|f\|_{A^{r,2}(L^2)}^2&= \sum_{j\in \mathbb{N}_0} 2^{2jr} \sum_{(q,k)\in \Gamma^\psi_j} |(f,\psi_{j,(p,k)})_{U}|^2,
\end{align*}
which again just gives a weighted sum of the generalized Fourier coefficients of the function $f$ in the orthonormal multiwavelet basis. Modifications for $d>1$ are constructed similarly to the discussion above. 
\end{example}
\end{framed}

%
%
\section{Discrete Stochastic Processes  and Markov Chains}
\label{sec:app_processes}
This section summarizes some of the basic definitions encountered in the study of stochastic processes and Markov chains. The reader is referred to \cite{meyn} for a full treatment of the details. Let $(Z,\Sigma_Z)$ be a measurable space and $(Q,\mu,\Sigma_Q)$ be a measure space.  A random variable is a function $f:Q \rightarrow Z$ such that $A\in \Sigma_Z$ implies that $f^{-1}(A)\in \Sigma_Q$. The probability distribution $\mu_f$ on $Z$  of the random variable $f$ is given by $\mu_f(A):=\mu(f^{-1}(A))$, or sometimes is expressed  in the form 
$$
\mu_f(A)=\text{Prob}\left \{ q \in Q \ | \ f(q) \in A \right \}=\mu\left \{ q \in Q \ | \ f(q) \in A \right \}
$$
for any measurable set $A\in \Sigma_Z$.
A discrete stochastic process $z:=\{z_i\}_{i\in \mathbb{N}_0}$  indexed by $\mathbb{N}_0$ is a sequence of random variables $z_i: Q\rightarrow Z$. If we denote the product space 
$
\mathbb{Z}:= Z^{\mathbb{N}_0},
$
then $z$ may be understood as a random variable $z:Q \rightarrow \mathbb{Z}$ from the measure space $(Q,\mu,\Sigma_Q)$ to the measurable space $(\mathbb{Z},\Sigma_{\mathbb{Z}})$. The probability distribution of the process $z$   is the probability measure $\mu_z:=\mu(z^{-1}(\mathbb{A}))$ for each
measurable $\mathbb{A} \subset \mathbb{Z}$.
We denote this probability measure on $\mathbb{A}$ by 
$$
\mathbb{P}_{\{z\}} 
(\mathbb{A})=\text{Prob} \left \{ \mathbb{A} \subset \mathbb{Z}
\right \}
$$
and refer to it as the distribution
of the process. 
A cylinder set  $C$ contained in $Z^m$  is a set that  has the form  $C^\ell =C_1 \times C_2\times \cdots \times C_m $
with each $C_i\in  \Sigma_Z$ for $1\leq i \leq m$. 
 We define  the law or probability distribution $\mu^m_{\{i_1,\ldots,i_m\}}$ on $Z^m$ for a  finite collection of coordinates indexed by $\{i_1,\ldots, i_m\}$ by setting 
$$
\mu^m_{\{i_1,\ldots,i_m\}}(C)
:=\mu\left \{ q\in Q \ | \  z_{i_k}(q) \in C_k ,\ 1\leq  k m \ell \right  \}
$$
for each cylinder set $C^\ell \subset Z^m$.
The statistical properties of the process $z$ is determined by the laws of all such  finite dimensional subsets of  coordinates. When the index set $\{i_1,\cdots,i_m\}$ is just the first $m$ coordinates,  we follow the convention above and  define the distribution of the first $m$ states of the process by 
$$
\mathbb{P}^m_{\{z\}}(C^m) := \mu^\ell_{\{1,\ldots,m\}}(C^m).
$$
We say that the  stochastic process is independent and identically distributed whenever  $\mu_{z_i}:=\mu$ for some fixed measure $\mu$ on $Z$  and all the $z_i$ are independent of one another. In this case for any finite 
set of indices 
$\{i_1,\ldots,i_m\}$ we have  
$$
\mu^m_{i_1,\ldots, i_m}(d\xi_1\cdots d\xi_m)=\mu_{z_{i_1}}(d\xi_1)\cdots \mu_{z_{i_m}}(d\xi_m)= \underset{k\leq \ell}{\otimes} \mu(d\xi_k)
$$
with $\mu^m(d\xi_1,\dots,d\xi_m):=\underset{k\leq \ell}\mu(d\xi_k)$ the product measure on $Z^\ell$. In this case we simply write 
$$
\mathbb{P}^m(C^m):=\mathbb{P}^\ell_{\{z\}}(C^m) = \mu^m(C^m),
$$
and the distribution of the first $m$ states of the process is just the product measure $\mu^m$.

We will limit considerations in this paper to stochastic processes that are Markov chains. A Markov chain on a state space $Z$ is a discrete  stochastic process whose statistics are entirely determined by a transition probability kernel. The transition probability kernel is a map on $\mathbb{P}:\Sigma_Z \times Z\mapsto [0,1]$. For each fixed measurable set $A\subset Z$, $\xi\mapsto \mathbb{P}(A,\xi)$ is a measurable 
function. And, for  each fixed $\xi\in Z$ the mapping $A\mapsto \mathbb{P}(A,\xi)$ is a probability measure. Intuitively, the transition probability kernel $\mathbb{P}(A,\xi)$ defines what the probability is that the next single step in the chain lands in the  set $A$ given that the current state is $\xi$. As discussed in \cite{meyn}, p. 66, if the probability distribution of the initial state $x_0$  is the probability measure $\mu_0$ on $Z$, the probability distribution of the first $m$ steps of the chain belonging to the cylinder set 
$
C:=C_1\times \cdots \times C_m
$
 is given by  
\begin{align*}
&\mathbb{P}^m_{\{x\}}(C_1 \times \cdots \times C_m)= \\
&\hspace*{.25in}\int_{\xi_{m-1}\in C_{m-1}} \cdots \int_{\xi_0\in A_0} \mathbb{P}(C_m, \xi_{m-1})\mathbb{P}(d\xi_{m-1},\xi_{m-2})  \cdots \mathbb{P}(d\xi_1,\xi_0)\mu_0(d\xi_o). 
\end{align*}
We are interested specifically in this paper in integrating expressions such as 
$$
\mathbb{E}_{\mathbb{P}^m_{\{x\}}}(g)
$$
for a function $g:Z^m\rightarrow \mathbb{R}$. For a general Markov chain this can be a formidable task. 

However, there are a few   cases in which the general expression simplifies. The first case is an  independent and identically distributed, or IID,  stochastic process.   For this process, each of the coordinates $z_i$ is independent of all other coordinates and  has the probability distibution  $\mu_{z_i}=\mu$  for some fixed measure $\mu$. In the language   of Markov chains the transition kernel is just $\mathbb{P}(dz,\xi)=\mu(dz)$.  This means that the  transition kernel does not care where it starts.  If we want to compute the expectation of a function $g:Z^n\rightarrow \mathbb{R}$ over the sequence of random variables $(z_1,\ldots, z_{m})$, we just have 
$$
\mathbb{E}_{\mathbb{P}^m}(g)= \mathbb{E}_{\mu^m}(g) = \mathbb{E}_{\mathbb{P}^m_{\{x\}}}(g)=
 \int_{Z^m} 
g({\xi_{1},\ldots, \xi_{m}})\mu(\xi_1)\mu(\xi_2)\cdots\mu(\xi_m)
$$
with $\mu^n$ the product measure on $Z^m$. 

In the second case,  if the process is just the deterministic system having the transition probability $\mathbb{P}(dy,x):=\delta_{w(x)}(dy)$, it can be shown that the probability  $\{z_1,\cdots,z_m\}\in C:=C_1\times \cdots C_m$ that $\{z_i\}_{i\leq m}$ is contained in the cylinder set $C$  is given by 
$$
\mathbb{P}^m_{\{x\}}(C_1\times \cdots \times  C_m)=
\delta_{w(x_0)}(C_1)\cdot\delta_{w^2(x_0)}(C_2)\cdots \cdot \delta_{w^n(x_0)}(C_n),
$$
and the expectation becomes 
\begin{align*}
\mathbb{E}_{\mathbb{P}^m_{\{x\}}}(g)&= \int_{Z^m} g(\xi_1,\ldots,\xi_m)\delta_{w(x_0)}(d\xi_1)\cdot\delta_{w^2(x_0)}(d\xi_2)\cdots \cdot \delta_{w^m(x_0)}(d\xi_m)\\
&=g(w^1(x_0),w^2(x_0), \cdots , w^m(x_0)).
\end{align*}

%
%
\section{Spectral Decomposition: Compact, Self-Adjoint Operators}
\label{app:spectral}
In this section we review the fundamental properties of compact, self-adjoint operators and their spectral decompositions, which are used extensively
in Section \ref{sec:rkhs}.  We first review the construction of the singular value decomposition for matrices, and subsequently discuss the generalization to the 
Schmidt decomposition of compact operators acting between Hilbert spaces.

For any matrix $T\in \mathbb{R}^{m\times n}$,  the symmetric and positive semidefinite  matrix $T^* T$  has a collection 
of real eigenvalues $\lambda_k:= \lambda_k(T^*T)$ that can be arranged  with multiplicities in 
nonincreasing order $\lambda_1 \geq \cdots \geq \lambda_n \geq 0$. When the corresponding 
eigenvectors $v_1,\ldots,v_n \in \mathbb{R}^{n\times n} $ are chosen to be of unit length and mutually orthogonal, we obtain the spectral factorization of $T^* T$ in the form 
\begin{align*}
T^*T & = V \Sigma^2_n V'
\end{align*}
where $\Sigma_n\in \mathbb{R}^{n \times n}$ is the diagonal matrix $\Sigma_n=\text{diag}(\sigma_1, \ldots, \sigma_n)$ and $V\in \mathbb{R}^{n\times n}$ is an orthogonal matrix whose columns are the associated eigenvectors $V=[v_1,\ldots,v_n]$. If we carry out the same construction for the matrix $TT^* \in \mathbb{R}^{m\times m}$, we obtain
\begin{align*}
T T^* & = U \Sigma^2_m U'
\end{align*}
where $\Sigma^2_m \in \mathbb{R}^{m\times m}$ is the diagonal matrix of non-increasing eigenvalues $\sigma_m:= \sqrt{\lambda_m}:=\sqrt{\lambda_m(TT^*)}$ and $U=[u_1,\ldots, u_m] \in \mathbb{R}^{m \times m}$ is an orthogonal matrix whose columns are the corresponding orthonormal eigenvectors of $TT^*$. It is a quick calculation to show that these decompositions are further related: we always have
\begin{equation*}
u_k = \frac{1}{\sigma_k} T v_k
\end{equation*}
for $k=1,\ldots,\text{min}(m,n)$.

The singular value decomposition of the matrix $T$ is then defined as 
\begin{equation*}
T=U\Sigma V'
\end{equation*}
where $\Sigma\in \mathbb{R}^{m\times n}$ is diagonal with entries $\sigma_1, \ldots, \sigma_{\text{min}(m,n)}$ on the diagonal.  It is convenient for comparison to the infinite dimensional operators discussed later to express these decompositions in terms of operators acting on arbitrary vectors.  We have 
\begin{align*}
Tx &= \sum_{k=1,\ldots,m} \sigma_k (v_k,x)_{\mathbb{R}^{n}} u_k, \\ 
TT^* y&= \sum_{k=1,\ldots,m} \sigma_k^2 (u_k,y)_{\mathbb{R}^{m}} u_k, \\ 
T^*T x &= \sum_{k=1,\ldots,n} \sigma_k^2 (v_k,x)_{\mathbb{R}^{n}} v_k, 
\end{align*}
for each $x\in \mathbb{R}^n$ and $y\in \mathbb{R}^m$.

Now consider the case when $U,V$ are Hilbert spaces and $T:V \rightarrow U$ is a linear, compact operator.  Since $T^*T$ is a compact, self-adjoint operator, all of its eigenvalues are real and
are contained in the interval $[0,\|T^*T\|]$. 
The number of eigenvalues greater than any given positive constant is finite, so the only possible accumulation point of the eigenvalues is zero.    Each
eigenspace corresponding to a nonzero eigenvalue is finite dimensional.  The eigenvalues $\lambda_k(T^*T)\equiv \lambda_k(TT^*)$ are assumed to be arranged 
in an extended enumeration that includes multiplicities in nonincreasing order
\begin{equation*}
\lambda_1 \geq \lambda_2 \geq \cdots \geq 0.
\end{equation*}
  The $k^{th}$ singular value $\sigma_k(T)$  of  $T: V \rightarrow U$  is defined as 
\begin{equation*}
\sigma_k(T):=\lambda^{1/2}_k( T^* T).
\end{equation*}
There exist orthonormal collections of eigenvectors $\left \{ u_k \right \}_{k=1}^\infty $ for $TT^*$ and $\left \{ v_k \right \}_{k=1}^\infty $ for $T^*T$ such that 
\begin{align*}
T f &= \sum_{k=1}^\infty  \sigma_k (v_k,f)_{V} u_k \quad \quad \text{convergence in } U,\\ 
TT^* g&= \sum_{k=1}^\infty \sigma_k^2 (u_k,g)_{U} u_k \quad \quad \text{convergence in } U, \\ 
T^*T f &= \sum_{k=1}^\infty  \sigma_k^2 (v_k,f)_{V} v_k  \quad \quad \text{convergence in } V,
\end{align*}
for each $f\in V$ and $g\in U$.

The expansions above are vital to the development in Section \ref{sec:rkhs}. One further result is also needed to define the square root operator $\sqrt{T}$. Suppose that $g:\mathbb{R}\rightarrow \mathbb{R}$. For  any self-adjoint compact operator $T:U \rightarrow U$, the function $g(T)$ can be defined in terms of the spectral expansion 
\begin{align*}
(g(T))(f):=\sum_{k=1}^\infty g(\lambda_i)(u_k,f)_H u_k,
\end{align*}
provided that $g$ is continuous on the spectrum of $T$. 
%
%
\section{Schatten Operators}
The spectral decomposition summarized in Section \ref{app:spectral} motivates the definition of the Schatten class of operators $S^p(V,U)$ acting between the Hilbert spaces $V$ and $U$. 
Whenever $T:V\rightarrow U$ is compact, the operator $T^*T:V\rightarrow V$ is compact, nonnegative, and self-adjoint.  By the spectral theory presented in Section  \ref{app:spectral}, there is a unique square root operator $(T^*T)^{1/2}$. We define  $|T|:=(T^*T)^{1/2}:V\rightarrow V$ given by
$$
|T|f:=\sum_{k=1}^\infty \sigma_k(v_k,f)_V v_k, \quad \quad \text{convergence in } V,
$$
for each $v\in V$. 
We say the  operator $T\in S^p(V,U)$ for $0<p<\infty$ provided that $T$ is compact and  the $S^p$-norm defined as 
$$
\|T\|_{S^p}:=\left ( \sum_{k=1}^\infty \sigma_k^p(T)\right )^{1/p} :=\left \| \left \{ \sigma_j \right \}_{j=1}^\infty \right \|_{\ell^p}< \infty 
$$
with $\sigma_j(T)$ the singular values of $|T|$.\
We also define $$S^\infty(V,U):=\left \{T\in \mathcal{L}(V,U) \ | \ T \text{ is compact } \right \}.$$ Perhaps the two most well-known subclasses of the Schatten operators are the Hilbert-Schmidt operators $S^2(V,U)$ and the trace class operators $S^1(V,U)$. 
For these two cases we have 
\begin{align*}
\|T\|_{S^1}&:=\sum_{j=1}^\infty (Tv_i,v_i)_V, \\
\|T\|_{S^2}&:= \| \left \{\sigma_j \right \}_{j=1}^\infty \|_{\ell^2}=\left (\sum_{j=1}^\infty (Tv_i,v_j)_{V}\right)^{1/2}
\end{align*}
%
%
%
\section{Approximation Spaces}
\label{sec:approx_spaces}
Our error estimates will be based on either linear  approximation methods as they arise 
in the construction of approximations spaces.  See \cite{piestch1981approxspaces}, or more recently \cite{devore1993constapprox}, for a thorough theoretical discussion of these spaces.  An approximation
method is a pair $(X,\left \{A_n\right\}_{n\in \mathbb{N}_0})$ where $X$ is a (quasi-)Banach space and $\left \{ A_n \right \}_{n\in \mathbb{N}_0}$ is a sequence of subsets of $X$ that satisfy the following conditions:
\begin{itemize}
\item[(i)] $ \left \{0\right\} = A_0 \subseteq A_1 \subseteq A_2 \subseteq \cdots \subseteq X$,
\item[(ii)] $a A_n = A_n$ for all $a\in \mathbb{R}$ and $n=1,2,\ldots$, 
\item[(iii)] there is a constant $c$ such that  $ A_n + A_n \subset A_{cn}$ for $m,n=1,\ldots$, 
\item[(iv)] $\cup_{n\in \mathbb{N}_0} A_n$ is dense in X,
\item[(v)] for each $f\in X$ there exists a best approximation of $f$ in $A_n.$
\end{itemize}
In general the subsets $\{A_n\}_{n\in \mathbb{N}_0}$ need not be linear spaces, and the assumptions above define the foundation of nonlinear approximation methods.  
In this paper we only consider the case when they are linear subspaces and  refer to the $A_n$ as the approximant subspaces.  
When we define the  $n^{th}$ approximation number as 
\begin{equation}
a_n(f,X):=E_{n-1}(f,X):= \inf_{a \in A_{n-1}} \| f-a \|_X 
\end{equation}
for each $n>0$,  the approximation space $A^{r,q}:=A^{r,q}(X)$ for $r>0, 1\leq q\leq \infty$ is defined to be the collection of $f\in X$ for which the sequence
$$\seq{n^{r-1/q} a_n(f,X)}_{n\in \mathbb{N}} \in \ell^q.$$  The approximation space $A^{r,q}$  is a (quasi-)Banach space when we set 
\begin{align*}
\| f\|_{A^{r,q}}:&=\left \| \seq{n^{r-1/q} a_n(f,X) }_{n\in \mathbb{N}} \right \|_{\ell^q} 
= \left ( \sum_{n=1}^\infty \left [ n^r E_{n-1}(f,X) \right ]^q \frac{1}{n} \right )^{1/q}.  
\end{align*}
Note that since the first term in the sequence above is $\|f\|_X$, so that if $\|f\|_{A^{r,u}}=0$, we have $f=0$.  
Many of the basic properties of approximation spaces are described in \cite{piestch1981approxspaces,piestch1982tensorproducsequenfunctoperat, devore1993constapprox}. 

One of the important properties that we will use in this paper is a theorem that guarantees the representation of elements in an 
approximation space in terms of quasi-geometric sequences.  A sequence of integers $\seq{n_k}_{k=0}^\infty$ is said to be quasi-geometric
if $n_0=1$ and there are two positive constants $a,b$ such that 
\begin{equation*}
1< a \leq \frac{n_{k+1}}{n_k} \leq b <\infty
\end{equation*}
for all $k\geq 0$. The representation theorem \cite{piestch1982tensorproducsequenfunctoperat} is stated below:
\begin{prop}
\label{prop:representer}
Let $\seq{n_k}_{k=0}^\infty$ be a quasi-geometric series. 
A function $f\in X$ belongs to $A^{r,q}$ if and only if there exists a representation of the form  
\begin{equation*}
f=\sum_{k=0}^\infty  x_{n_k}  
\end{equation*} 
such that $x_{n_k}\in A_{n_k}$ and 
\begin{equation*}
\seq{n_k^r \|x_{n_k} \|_X}_k \in \ell^q.
\end{equation*}
 In this case we have
 \begin{equation*}
 \|f\|_{A^{r,q}} \approx \inf_{\overset{f=\sum_{k}x_{n_k}}{x_{n_k} \in A_{n_k}} } \left \| \seq{ n_k^r x_{n_k} }_k\right \|_{\ell^q} .
 \end{equation*}
\end{prop}
\noindent This last characterization of the approximation spaces $A^{r,q}$ sometimes are presented in slightly different forms.  As noted in \cite{devore1993constapprox},
we can derive an equivalent expression for the seminorm $|f|_{A^{r,u}}$  choosing the quasigeometric sequence $n_k\equiv 2^k$ for $k=0,1,\ldots.$ 
\begin{equation*}
|f|_{A^{r,q}}=\left (  \sum_{n=0}^\infty \left [ 2^{nr}E_{2^n}(f,X)\right ]^q\right )^{1/q}.
\end{equation*}

Approximation spaces have been studied for a wide range of choices of the underlying space $X$ and sequence of 
approximating subsets $\seq{A_n}_{n\in \mathbb{N}_0}$.  In this paper we will restrict consideration 
to approximation spaces $X$ is in fact a Hilbert space and the subsets $\seq{A_n}_{n\in \mathbb{N}_0}$ are defined in terms of an orthonormal basis for $X$.  See \cite{kerkyacharian2003entropunivercodinapproxbasesproper}
Section (2) for  discussion of this classical choice, where it is compared to the more general cases in which the family 
is generated from an unconditional or greedy basis for a (quasi-)Banach space  $X$.  
\subsection{Linear Approximation Methods}
\label{sec:linear_approx}
Suppose that $\left \{ x_k \right \}_{k=1}^\infty$ is an orthonormal basis for the Hilbert space $X$.  A linear approximation method 
chooses the sets $A_n:=\text{span}\seq{x_k}_{k\leq n}$. In this case we have
\begin{equation*}
a_n = \inf_{g\in A_{n-1}} \| f- g\|_X = \| (I-\Pi_{n-1})f \|_X = E_{n-1}(f,X)
\end{equation*}
where $\Pi_n$ is the orthogonal projection onto $A_{n}$.  We always have
\begin{equation*}
f=\sum_{k\in \mathbb{N}_0}^\infty  Q_k f := \sum_{k} ( \Pi_{2^{k}} - \Pi_{2^{k-1}}) f
\end{equation*}
when we define $\Pi_{2^{-1}}:=0$. 
It follows that for $1\leq q< \infty$ we have a  seminorm on $A^{r,q}$
\begin{equation*}
|f|_{A^{r,q}} =
\left \{
\begin{array}{ccc}
\left \| \seq{2^{kr} \| Q_{k} f \|_X }_{k\in\mathbb{N}_0}  \right \|_{\ell^q} 
&  & 1\leq q <\infty\\
&  &   \\
\sup_{n\in \mathbb{N}_0} [2^{nr}E_{2^n}(f)] & & q=\infty
\end{array}
\right .
\end{equation*}
%
%
\subsection{Linear Approximation Spaces for Dyadic Splines}
\label{sec:approx_splines}
A various points in the paper, we refer to  the connection of the linear approximation spaces 
$$A^{r,q}(X;\{A_j\}_{j\in \mathbb{N}_0})$$ 
to other more common spaces such as the generalized Lipschitz spaces  $\text{Lip}(r,L^p(\Omega))$ and Sobolev spaces $W^{r}(L^q(\Omega))$. Here we summarize some of these relationships for approximations constructed from families of dyadic splines on $\Omega = [0,1]$.  A dyadic B-spline approximation  over the domain $\Omega=[0,1]$ on a grid of  level $j$ selects the knots of the spline to be the dyadic grid points $\Delta_{j}:=\left\{ k2^{-j} \ | \ 1 \leq k \leq 2^j-1\right \}$. A spline of order $r$ on grid level $j$ is a piecewise polynomial of order $r-1$ over each of the dyadic intervals of the mesh.  We denote by $\mathcal{S}_r(\Delta_{j})$ the linear space of the dyadic splines of order $r$ having dyadic knots $\Delta_{j}$ in $[0,1]$. The normalized B-splines $N^r$ over the interval $[0,r]$ having  integer knots $\left \{0, \ldots, r \right\}$ are used to construct bases for $\mathcal{S}_r(\Delta_{j})$. The functions $N^r$ can be defined by induction. Let  $N^1:=\chi_{[0,1]}$ be the characteristic function of $[0,1]$. Then $N^r$ is defined by recursion of the convolution  $N^r:=N^{r-1}*\chi_{[0,1]}$ for $r\in \mathbb{N}$.   Given $N^r$, we set   $N^r_{j,k}:=N^r(2^jx-k)$  for $x\in [0,1]$ and $0\leq k\leq n_j-1$.  We have  $n_j:=2^j=\#(\mathcal{S}_r(\Delta_j))$.  
Each function $f$ in $\mathcal{S}_r(\Delta_{j})$ has a spline series representation 
$$
f= \sum_{k\leq n_j} c_{j,k}N^r_{j,k}
$$
for a unique set of coefficients $\{c_{j,k}\}_{k\leq n_j}$. The following theorem gives a concise characterization of the approximation spaces in terms of the Besov spaces \cite{devore1993constapprox,runst}, a class of smoothness spaces that contain many well-known function spaces. 
\begin{theorem}[Theorem 3.3, page 361, \cite{devore1993constapprox}]
\label{th:approx_besov}
Let $r\in \mathbb{N}$, $1\leq p\leq \infty$, $s:=r-1+1/p$, and 
$$
A^{\alpha,q}(L^p(\Omega)):=A^{\alpha,q}(L^p(\Omega);\{\mathcal{S}_r(\Delta_j)\}_{j\in \mathbb{N}_0})
$$ 
be the approximation spaces generated by the approximant spaces $A_j:=\mathcal{S}_r(\Delta_j)$ of dyadic splines. For all $0<\alpha<s$ and $0<q\leq  \infty$, the approximation space  $A^{\alpha,q}(L^p(\Omega))$ is equivalent to the Besov space $B^{\alpha,q}(L^p(\Omega))$. 
\end{theorem}
The scale of  Besov spaces $B^{\alpha,q}(L^p(\Omega))$ contains a wide variety of other common function spaces, and it is useful to us  specifically because it contains the Sobolev and (generalized) Lipschitz spaces. From \cite{runst}, page 14, we see that 
\begin{align*}
\text{Lip}^*(\alpha,L^p(\Omega))=B^{\alpha,\infty}(L^p(\Omega)) & \quad \alpha>0, \ 1<p\leq \infty, \\
\text{Lip}(\alpha,L^p(\Omega))=B^{\alpha,\infty}(L^p(\Omega)) & \quad \text{noninteger } \ \alpha>0, \ 1< p\leq \infty, \\
W^{\alpha}(L^p(\Omega))= B^{\alpha,p}(L^p(\Omega)) & \quad \text{noninteger } \alpha >0 , \  1<p\leq \infty,  \\
W^{\alpha}(L^2(\Omega))=B^{\alpha,2}(L^2(\Omega)) &  \quad \alpha > 0.
\end{align*}

\subsubsection{Approximation by Haar Wavelets}
\label{sec:approx_by_haar}
The results above improve our  understanding of several special cases used in our examples.  The first of these cases are the examples that employ Haar scaling functions and wavelets for approximation of the Koopman or Perron-Frobenius operators. The span of scaled and translated Haar scaling functions are equivalent to the span of first order splines over $\Omega=[0,1]$, 
$$
A_j:=\text{span}\left \{
\phi_{j,k} \ | \  k\in \Gamma_j^\phi 
\right \} \equiv \text{span} \left \{  N^1_{j,k} \ | \ k\in \Gamma_{j}^\phi 
\right \}
$$
From Theorem \ref{th:approx_besov}, it follows that 
\begin{align*}
A^{\alpha,\infty}(L^p(\Omega);\{A_j \}_{j\in \mathbb{N}_0} )& = \text{Lip}^*(\alpha, L^p(\Omega)) , \quad \text{and} \\
A^{\alpha,2}(L^2(\Omega),\{A_j\}_{j\in \mathbb{N}_0} )& = W^\alpha(L^2(\Omega)),
\end{align*}
for $0<\alpha<1/2$.  

\subsubsection{Approximation by Orthonormal Wavelets that Reproduce Polynomials}
\label{sec:zero_moments}
In this paper we have elected to employ families of $L^2(\Omega)-$orthonormal wavelets  that reproduce certain polynomials. This general property is referred to the degree of exactness of a wavelet or multiwavelet system, and it is often described in terms of ``zero moment'' conditions on wavelets \cite{burrus, daubechies88, daubechiesbook}.  This property states, roughly speaking,  that if $\phi,\psi$ are orthonormal scaling functions and wavelets,  then if the wavelets satisfy $N$ moment conditions of the form 
$$
\int_\mathbb{R} x^k \psi(x)dx =0 \quad 0\leq k \leq N-1,
$$
then the scaling function $\phi$  reproduces polynomials in the sense that 
$$
x^k=\sum_{\ell} a_p \phi(x-\ell) \quad 0\leq k \leq N-1.
$$
When the scaling functions $\phi$ are compactly supported, the summation at a fixed $x$ in the reproduction formula involves only a finite number of terms. 

Recall that each B-spline $N^r$ of order $r$ is a piecewise polynomial of degree at most $r-1$ between its knots. This means that each $N^r$ can be expressed in terms of a finite linear combination of the scaling functions $\phi(x-k)$, provided the wavelets satisfy zero moment conditions for $0,\ldots, r-1$.  From this we conclude that the spaces of approximants
$$
\tilde{A}_j:=\mathcal{S}_r(\Delta_j)\subseteq \{ \phi_{j,k}\ | \ k\in \Gamma_j^\phi \} :=A_j
$$
are nested. From the definition of the seminorm
$$
|f|^q_{A^{r,q}(X;\{A_j\}_{j\in \mathbb{N}_0}}:=\sum_{n\in \mathbb{N}_0} [2^{nr}E_{2^n}(f,X;A_{2^n})]^q, 
$$
we conclude that 
$$|f|_{A^{r,q}(X;\{A_j\}_{j\in \mathbb{N}_0})}\leq 
|f|_{A^{r,q}(X;\{\tilde{A}_j\}_{j\in \mathbb{N}_0})}.
$$
Among other things, this inclusion implies that linear approximation by the wavelet functions converge at least as fast as approximations by the B-splines that are contained in the span of the wavelets. By Theorem \ref{th:approx_besov} we conclude that the approximation rates using wavelets is at least $O(2^{-\alpha j})$ for the range of $0<\alpha < s=r-1+1/p$ with $r$ the order of the splines contained in the range of the wavelets. 
\end{document}